\documentclass[18pt,a4paper,twoside]{book}
\usepackage[english]{babel}
\usepackage{amsfonts}
\usepackage{amsthm}
\usepackage{amsmath}
\usepackage{mathrsfs}

\usepackage{fancyhdr}
\fancypagestyle{plain}{%
\fancyhead{} %get rid of the headers on plain pages
} 
\fancypagestyle{empty}{%
\fancyhead{} %get rid of the headers on plain pages
}

\usepackage{amsmath,amsthm, amssymb}
\usepackage{url}
\usepackage{graphicx}
\usepackage{makeidx}
\usepackage[colorlinks=true]{hyperref}
\usepackage{mathptmx}
\usepackage{paralist}
\usepackage{verbatim}

\theoremstyle{plain}
\newtheorem{thm}{\bf Theorem}[section]

\newtheorem{prop}[thm]{\bf Proposition}
\newtheorem{lemma}[thm]{\bf Lemma}
\newtheorem{corollary}[thm]{\bf Corollary}

\theoremstyle{definition}
\newtheorem{definition}[thm]{\bf Definition}

\theoremstyle{remark}
\newtheorem{remark}[thm]{\bf Remark}
\newtheorem{question}[thm]{\bf Question}
\newtheorem{conj}[thm]{\bf Conjecture}

\theoremstyle{example}
\newtheorem{example}[thm]{\bf Example}

\numberwithin{equation}{chapter}

\def \Hs{{\operatorname{HS}}}
\def \Hf{{\operatorname{HF}}}
\def \Hp{{\operatorname{HP}}}
\def \HF{{\operatorname{HF}_{\bar{A}(\Delta)}}}

\def \HFo{{\operatorname{HF}_{\bar{A}(\Delta ,\omega)}}}

\def\ara{\operatorname{ara}}
\def\c{\operatorname{c}}
\def \chara{{\operatorname{char}}}
\def \cd{{\operatorname{cd}}}
\def \ecd{{\operatorname{\acute{e}cd}}}
\def \height{{\operatorname{ht}}}

\def \reg{{\operatorname{reg}}}
\def \depth{{\operatorname{depth}}}
\def \grade{{\operatorname{grade}}}
\def \Spec{{\operatorname{Spec}}}
\def \Proj{{\operatorname{Proj}}}
\def\init{\operatorname{in}}
\def\inito{\operatorname{in}_{\oo}}
\def\homo{\operatorname{hom}_{\oo}}
\def\dego{\deg_{\oo}}
\def\Supp{\operatorname{Supp}}
\def\Ass{\operatorname{Ass}}
\def\Min{\operatorname{Min}}
\def\Tor{\operatorname{Tor}}
\def\Ext{\operatorname{Ext}}
\def\Op{\operatorname{Op}}
\def\reg{\operatorname{reg}}

\def\codim{\operatorname{codim}}
\def\diag{\operatorname{diag}}

\def\GL{\operatorname{GL}}

\def \mm{{\mathfrak{m}}}
\def \ll{{\lambda}}
\def \oo{{\omega}}
\def \nn{{\mathfrak{n}}}
\def \xx{{\mathfrak{x}}}
\def \xxx{{\mathbf{x}}}
\def \yyy{{\mathbf{y}}}
\def \et{{\acute{e}t}}

\def \kk{{\kappa}}
\def \kkk{\Bbbk}
\def \PP{\mathbb P}
\def \Ab{\mathbf{Ab}}
\def \CC{\mathbb C}
\def \RR{\mathbb R}
\def \KK{\mathbb K}
\def \SS{\mathbb S}
\def \Spol{\widetilde{S}}
\def \CCC{\mathbf C}
\def\AA{\ensuremath{\mathbb A}}
\def \C{\mathcal C}
\def \Jj{\mathcal J}
\def \G{\mathcal G}
\def \R{\mathcal R}
\def \V{\mathcal V}
\def \H{\mathcal H}
\def \J{\mathcal Symb}
\def \M{\mathfrak M}

\def \MM{\mathcal M}
\def \B{\mathcal B}
\def \K{\mathcal K}
\def \Z{\mathcal Z}
\def \L{\mathcal L}
\def \O{\mathcal O}
\def \F{\mathcal F}
\def \de{\partial}
\def \FD{\mathcal{F}(\Delta)}
\def \D{\Delta}

\def \Id{I_{\Delta}}
\def \Idm{I_{\Delta}^{(k)}}
\def \JD{J(\Delta)}
\def \JDo{J(\Delta , \omega)}
\def \JDm{J(\Delta)^{(k)}}
\def \JDom{J(\Delta , \omega)^{(k)}}
\def \AD{\bar{A}(\Delta)}
\def \ADo{\bar{A}(\Delta ,\omega)}

\def \NN{{\mathbb{N}}}
\def \AA{{\mathbb{A}}}
\def \ZZ{{\mathbb{Z}}}
\def \QQ{{\mathbb{Q}}}
\def \GG{{\mathbb{G}}}
\def \aa{{\mathfrak{a}}}
\def \a{{\alpha}}
\def \aaa{{{\bf a}}}
\def \bbb{{{\bf b}}}
\def \bb{{\mathfrak{b}}}
\def \qq{{\mathfrak{q}}}
\def \Ga{\Gamma_{\mathfrak{a}}}
\def \Ha{H_{\mathfrak{a}}}
\def \pd{{\operatorname{projdim}}}
\def \id{{\operatorname{id}}}
\def \Im{{\operatorname{Im}}}
\def \Ker{{\operatorname{Ker}}}
\def \Hom{{\operatorname{Hom}}}
\def \Ext{{\operatorname{Ext}}}
\def \Sym{{\operatorname{Sym}}}
\def \GCD{{\operatorname{GCD}}}

\def \Vl{L_{\lambda}V}
\def \Kl{c_{\lambda}}

\def \tl{{}^{t}}

\def \Wc{L_{\gamma}W}
\def \Wl{L_{\lambda}W}
\def\ini{\operatorname{in}_{\prec}}

\makeindex

\begin{document}

\frontmatter
%-----------------------------------------------------------------------
% LATEX FILE   titolo.tex
%-----------------------------------------------------------------------
\begin{titlepage}%{\textwidth{360pt}}
\oddsidemargin 1.8cm
\evensidemargin 0in
\begin{center}
%
% SELEZIONARE centerline o rightline per spostare intestazione

{\bf
UNIVERSIT\`A DEGLI STUDI DI GENOVA}\\
\bigskip
FACOLT\`A DI SCIENZE
MATEMATICHE FISICHE E NATURALI\\
\bigskip
Dipartimento di Matematica
\vskip 2cm
%  Titolo della tesi:\\
\rule{\linewidth}{0.5mm}\\[.4cm]
\Large{\bf{Cohomological {\large and} Combinatorial Methods {\large in the} Study {\large of} Symbolic Powers {\large and} Equations {\large defining} Varieties}}\\[.8cm]
\rule{\linewidth}{0.5mm}\\[2cm]

%\ \vfill
%
%\end{center}
%%
%\begin{center}
Tesi per il conseguimento del \\ {\large \textbf{Dottorato di Ricerca in Matematica e
Applicazioni, ciclo XXIII}}
  \vskip .2cm
 Sede amministrativa: Universit\`a di Genova\\
\vskip .2cm
Settore Scientifico Disciplinare: MAT 02 - Algebra\\[2cm]
\large{\emph{Autore}}\\
  \large{Matteo \textsc{Varbaro}}\\
\vfill

%\begin{minipage}[t]{0.4\textwidth}
%\begin{flushleft} \large
%\emph{Author:}\\
%John \textsc{Smith}
%\end{flushleft}
%\end{minipage}
%\begin{minipage}[t]{0.4\textwidth}
%\begin{flushright} \large
%\emph{Supervisor:} \\
%Dr. Mark \textsc{Brown}
%\end{flushright}
%\end{minipage}

\emph{Supervisore} \hfill \emph{Coordinatore}\\
{\large Prof. Aldo	
\textsc{Conca}}, \hfill {\large Prof. Michele \textsc{Piana}}\\
Universit\`a di Genova \hfill Universit\`a di Genova
\end{center}
\end{titlepage}
\thispagestyle{empty}

\tableofcontents
\chapter{Introduction}

In this thesis we will discuss some aspects of Commutative Algebra, which have interactions with Algebraic Geometry, Representation Theory and Combinatorics. In particular, we will use local cohomology, combinatorial methods and the representation theory of the general linear group to study symbolic powers of monomial ideals, connectedness properties, arithmetical rank and the defining equations of certain varieties. Furthermore, we will focus on understanding when certain local cohomology modules vanish.
The heart of the thesis is in Chapters \ref{chapter1}, \ref{chapter2}, \ref{chapter3} and \ref{chapter4}. They are based on papers published by the author, but each chapter contains some results never appeared anywhere.

\vskip 3mm

Chapter \ref{chapter1} is an outcome of our papers \cite{va1,va2}. The aim of this chapter is to inquire on a question raised by Grothendieck in his 1961 seminar on local cohomology at Harvard (see the notes written by Hartshorne \cite{gro1}). He asked to find conditions on an ideal $\aa$ of a ring $R$ to locate the top-nonvanishing local cohomology module with support in $\aa$, that is to compute the cohomological dimension of $\aa$. We show, under some additional assumptions on the ring $R$, that the vanishing of $\Ha^i(R)$ for $i>c$ forces some connectedness properties, depending on $c$, of the topological space $\Spec(R/\aa)$, generalizing earlier results of Grothendieck \cite{SGA2} and of Hochster and Huneke \cite{ho-hu}. When $R$ is a polynomial ring over a field of characteristic $0$ and $\aa$ is a homogeneous ideal such that $R/\aa$ has an isolated singularity, we give a criterion to compute the cohomological dimension of $\aa$, building ideas from Ogus \cite{ogus}. Such a criterion leads to the proof of: 
\begin{compactitem}
\item[(i)] A relationship between cohomological dimension of $\aa$ and depth of $R/\aa$, which is a small-dimensional version of a result of Peskine and Szpiro \cite{PS} in positive characteristic.
\item[(ii)] A remarkable case of a conjecture stated in \cite{ly2} by Lyubeznik, concerning an inequality about the cohomological dimension and the \'etale cohomological dimension of an open subscheme of $\PP^n$.
\end{compactitem}
In order to show these results we use also ideas involving Complex Analysis and Algebraic Topology. 

\vskip 3mm

In Chapter \ref{chapter2} we show some applications of the results gotten in the first chapter, using interactions between Commutative Algebra, Combinatorics and Algebraic Geometry. It shapes, once again, on our papers \cite{va1,va2}. In the first half of the chapter we generalize a theorem of Kalkbrener and Sturmfels obtained in \cite{ka-st}. In simple terms, we show that Gr\"obner deformations preserve connectedness properties: As a consequence, we get that 
if $S/I$ is Cohen-Macaulay, then $\D(\sqrt{\init(I)})$ is a strongly connected simplicial complex. This fact induced us to investigate on the question whether $\D(\init(I))$ is Cohen-Macaulay provided that $S/I$ is Cohen-Macaulay and $\init(I)$ is square-free, supplying some evidences for a positive answer to this question. Another consequence is the settlement of the Eisenbud-Goto conjecture, see \cite{eisenbudgoto}, for a new class of ideals.
In the second half of the chapter, we compute the number of equations needed to cut out set-theoretically some algebraic varieties. For an overview on this important problem we refer the reader to Lyubeznik \cite{ly4}. The needed lower bounds for this number usually come from cohomological arguments: This is our case, in fact we get them from the results of the first chapter. For what concerns upper bounds, there is no general method, rather one has to invent ad hoc tricks, depending on the kind of varieties one deals with. For the varieties we are interested in, we get ideas from the Theory of Algebras with Straightening Laws, exhibiting the necessary equations. Among other things, we extend a result of Singh and Walther \cite{siwa}.

\vskip 3mm

Chapter \ref{chapter3} is based on a preprint in preparation with Bruns and Conca, namely \cite{BCVminors}. This time, beyond Commutative Algebra, Algebraic Geometry and Combinatorics, one of the primary roles is dedicated to Representation Theory.
We investigate on a problem raised by Bruns and Conca in \cite{BC1}: To find out the minimal relations between $2$-minors of a generic $m\times n$-matrix. When $m=2$ (without loss of generality we can assume $m\leq n$), such relations define the Grassmannian of all the $2$-dimensional subspaces of an $n$-dimensional vector space. It is well known that the demanded equations are, in this case, the celebrated Pl\"ucker relations. In particular, they are quadratic. Differently, apart from the case $m=n=3$, we show that there are always both quadratic and cubic minimal relations. We describe them in terms of Young tableux, exploiting the natural action of the group $\GL(W)\times \GL(V)$ on $X$, the variety defined by the relations (here $W$ and $V$ are vector spaces, respectively, of dimension $m$ and $n$). Furthermore, we show that these are the only minimal relations when $m=3,4$, and we give some reasons to believe that this is true also for higher $m$. Moreover, we compute the Castelnuovo-Mumford regularity of $X$, we exhibit a finite Sagbi bases associated to a toric deformation of $X$ (solving a problem left open by Bruns and Conca in \cite{BC2}) and we show a finite system of generators of the coordinate ring of the subvariety of $U$-invariants of $X$, settling the ``first main problem" of invariant theory in this case. The most of the results are proved more in general, considering $t$-minors of an $m\times n$-matrix.

\vskip 3mm

Chapter \ref{chapter4} concerns our paper \cite{va3}. The aim is to compare algebraic properties with combinatorial ones. The main result is quite surprising: We prove that the Cohen-Macaulayness of all the symbolic powers of a Stanley-Reisner ideal is equivalent to the fact that the related simplicial complex is a matroid! The beauty of this result is that the concepts of ``Cohen-Macaulay" and of ``matroid", a priori unrelated to each other, are both fundamental, respectively, in Commutative Algebra and in Combinatorics. Another interesting fact is that the proof is not direct, passing through the study of the symbolic fiber cone of the Stanley-Reisner ideal. It is fair to say that the same result has been proven independently and with different methods by Minh and Trung \cite{MT}. Actually, here we prove the above result for a class of ideals more general than the square-free ones, generalizing the result of our paper and of \cite{MT}. As a consequence, we show that the Stanley-Reisner ideal of a matroid is a set-theoretic complete intersection after localizing at the maximal irrelevant ideal. We end the chapter presenting a strategy, through an example, to produce non-Cohen-Macaulay square-free monomial ideals of codimension $2$ whose symbolic power's depth is constant. Here the crucial fact is that in the codimension $2$-case the structure of the mentioned fiber cone has been understood better thanks to our work with Constantinescu \cite{CV}.

\vskip 5mm

In addition to the four chapters described above, we decided to include a preliminary chapter entitled ``A quick survey of local cohomology". It is a collection of known results on local cohomology, the first topic studied at a certain level by the author. Local cohomology appears throughout the thesis in some more or less evident form, so we think that such a preliminary chapter will be useful to the reader. We added also five appendixes at the end of the thesis, which should make the task of searching for references less burdensome.

\section*{Acknowledgements}

First of all, I wish to thank my advisor, Aldo Conca. Some years ago, I decided that I wanted to do math in life. I started my PhD three years ago and, for reasons not only related to mathematics, I did it in Genova. In retrospect it could not get any better. Aldo was always able to stimulate me with new problems, giving me, most of the times, lighting suggestions. He let me do all the experiences that might have been useful to my intellectual growth. And he always granted me a certain freedom in choosing my research topics, something that not all the advisors are willing to do. So, I want to thank him for all these reasons, hoping, and being confident of it, to have the opportunity to collaborate with him again in the future.  

During these three years (actually a little more) I attended conferences, seminars and summer schools in different institutes and universities around the world, where I met many commutative algebraists. I learned that, even if I don't know where, as long as I do math I'll have a coffee with each of them at least once every year. I find that there is something poetic in this! Apart from the coffee, with some of these persons I have also had useful discussions about mathematics, and I am grateful to them for this reason. Special thanks are due to Winfried Bruns, who I met in several occasions. Especially, I spent the spring of 2009 in Osnabr\"uck, to learn from him. He taught me a lot of mathematics, and from such a visit came a close collaboration, which led to the results of Chapter \ref{chapter3}. Besides, he patiently helped me for my stay there, making such an experience possible for me. Another professor who deserves a special mention, for having taken the trouble to proofread the present thesis, is Anurag Singh.

I want to thank also four guys that, besides sharing with me pleasant conversations about mathematics, which led to the works \cite{BCV,shavar,BeneVarb,CV,NV,CV1}, I consider dear friends. I am talking about Le Dinh Nam, Alexandru Constantinescu, Bruno Benedetti and Leila Sharifan.

I'd like to say thank you to my parents, Annalisa and Rocco. Not just because they are my parents, which of course would be enough by itself. Somehow, they really helped me in my work, dealing with a lot of bureaucratic stuff for me: This way they exempted me from facing my ``personal enemy number 1"! In the same manner, thanks to my brother, Roberto. He has always instilled me a sort of serenity, which is an essential ingredient for doing research. 

Eventually, let me thank all the diploma students, PhD students, PostDoc students, professors, secretaries and technical assistants of the Department of Mathematics of the University of Genova, namely the  DIMA, who came positively into my life. It would be nice to thank them one by one, but it would be hard as well: So far, I have been haunting the DIMA for more than eight years, and in this time I met a lot of people. Many of them belong to the so-called ``bella gente", whose meaning is ``people whose company is good". I am lucky to say that they are too many to be mentioned all.

%Eventually, thanks to the extraordinary timelessness of mathematics, the place where this thesis has been done!

\chapter{Notation}

Throughout the thesis we will keep notation more or less faithful to some textbooks:
\begin{compactitem}
\item[(i)] For what concerns Commutative Algebra: Matsumura \cite{matsu}, Atiyah and Macdonald \cite{atimac} and Bruns and Herzog \cite{BH}.
\item[(ii)] For what concerns Algebraic Geometry: Hartshorne \cite{hart}.
\item[(iii)] For what concerns Algebraic Combinatorics: Stanley \cite{St}, Miller and Sturmfels \cite{MS} and Bruns and Herzog \cite[Chapter 5]{BH}.
\item[(iv)] For what concerns Representation Theory: Fulton and Harris \cite{FH}, Fulton \cite{Fu} and Procesi \cite{procesi}.
\end{compactitem}

\begin{itemize}

\item We will often say ``a ring $R$" instead of ``a commutative, unitary and Noetherian ring $R$". Whenever some of these assumptions fail, we will explicitly remark it.

\item The set of natural numbers will include $0$, i.e. $\NN=\{0,1,2,3,\ldots \}$.

\item An ideal $\aa$ of $R$ will always be different from $R$.

\item If $R$ is a local ring with maximal ideal $\mm$, we will write $(R,\mm)$ or $(R,\mm,\kk)$, where $\kk = R/\mm$.

\item Given a ring $R$, an $R$-module $M$ and a multiplicative system $T\subseteq R$, we will denote by $T^{-1}M$ the localization of $M$ at $T$. If $\wp$ is a prime ideal of $R$, we will write $M_{\wp}$ instead of $(R\setminus \wp)^{-1}M$.

\item Given a ring $R$, an ideal $\aa \subseteq R$ and an $R$-module $M$, we will denote by $\widehat{M^{\aa}}$ the completion of $M$ with respect to the $\aa$-adic topology. If $(R,\mm)$ is local we will just write $\widehat{M}$ instead of $\widehat{M^{\mm}}$.

\item The spectrum of a ring $R$ is the topological space 
\[ \Spec(R):=\{\wp \subseteq R \ : \ \wp \mbox{ is a prime ideal}\}\] 
supplied with the Zariski topology. I.e. the closed sets are $\V(\aa):=\{\wp \in \Spec(R) \ : \ \wp \supseteq \aa\}$ with $\aa$ varying among the ideals of $R$ and $R$ itself.

\item Given a ring $R$ and an $R$-module $M$, the support of $M$ is the subset of $\Spec(R)$
\[\Supp(M):=\{\wp \in \Spec(R) \ : \ M_{\wp}\neq 0\}.\]

\item We say that a ring $R$ is graded if it is $\NN$-graded, i.e. if there is a decomposition $R=\oplus_{i\in \NN}R_i$ such that each $R_i$ is an $R_0$-module and $R_iR_j\subseteq R_{i+j}$. The irrelevant ideal of $R$ is $R_+:=\oplus_{i>0}R_i$. An element $x$ of $R$ has degree $i$ if $x\in R_i$. Therefore $0$ has degree $i$ for any $i\in \NN$. 

\item The projective spectrum of a graded ring $R$ is the topological space 
\[ \Proj(R):=\{\wp \subseteq R \ : \ \wp \mbox{ is a homogeneous prime ideal not containing $R_+$}\},\] 
where the closed sets are $\V_+(\aa):=\{\wp \in \Proj(R) \ : \ \wp \supseteq \aa\}$ with $\aa$ varying among the homogeneous ideals of $R$.

\item When we speak about the grading on the polynomial ring $S:=\kkk[x_1,\ldots ,x_n]$ in $n$ variables over a field $\kkk$ we always mean the standard grading, unless differently specified: I.e. $\deg(x_1)=\deg(x_2)=\ldots =\deg(x_n)=1$.

\item For a poset $\Pi:=(P,\leq)$ we mean a set $P$ endowed with a partial order $\leq$.
\end{itemize}

\chapter{A quick survey of local cohomology}\label{chapter0}\index{local cohomology|(}

This is a preliminaries chapter, which contains no original results. Our aim is just to collect some facts about local cohomology we will use throughout the thesis. This is a fascinating subject, essentially introduced by Grothendieck during his 1961 Harvard University seminar, whose notes have been written by Hartshorne in \cite{gro1}. Besides it, other references  we will make use of are the book of Brodmann and Sharp \cite{BS}, the one of Iyengar et al. \cite{24hours} and the notes written by Huneke and Taylor \cite{hu-ta}.

 \section{Definition}\index{local cohomology}\label{defloccoh}
 
 Let $R$ be a commutative, unitary and Noetherian ring. Given an ideal $\aa \subseteq R$, we can consider the {\it $\aa$-torsion functor}\index{torsion functor} $\Ga$, from the category of $R$-modules to itself, namely
\[\Ga(M):=\bigcup_{n\in \NN}(0:_M \aa^n):=\{m\in M \ : \ \mbox{there exists }n \in \NN \mbox{ for which }\aa^n m =0\},\]
for any $R$-module $M$. Clearly $\Ga(M)\subseteq M$. Furthermore if $\phi:M \rightarrow N$ is a homomorphism of $R$-modules then $\phi(\Ga(M))\subseteq \Ga(N)$. So $\Ga(\phi)$ will simply be the restriction of $\phi$ to $\Ga(M)$. It is easy to see that the functor $\Ga$ is left exact. Thus, for any $i\in \NN$, we denote by $\Ha^i$ the $i$th derived functor of $\Ga$, which will be referred to as the $i$th {\it local cohomology functor with support in $\aa$}\index{local cohomology!functor}. For an $R$-module $M$, we will refer to $\Ha^i(M)$ as the $i$th {\it local cohomology module of $M$ with support in $\aa$} \index{local cohomology!module}. The following remark is as important as elementary. 
 \begin{remark}\label{cohomuptorad}
 If $\aa$ and $\bb$ are ideals of $R$ such that $\sqrt{\aa}=\sqrt{\bb}$, then $\Ga = \Gamma_{\bb}$. Therefore $\Ha^i = H_{\bb}^i$ for every $i\in \NN$.
 \end{remark}
 \begin{example}\label{cohomologyofZ}
We want to compute the local cohomology modules of $\ZZ$ as a $\ZZ$-module with respect to the ideal $(d)\subseteq \ZZ$, where $d$ is a nonzero integer. An injective resolution of $\ZZ$ is:
\[0\longrightarrow \ZZ \xrightarrow{\iota} \QQ \xrightarrow{\pi} \QQ /\ZZ \longrightarrow 0.\]
From the above resolution, we immediately realize that $H_{(d)}^i(\ZZ) = 0$ for any $i\geq 2$. Furthermore $H_{(d)}^0(\ZZ)=\Gamma_{(d)}(\ZZ)=0$. So we have just to compute $H_{(d)}^1(\ZZ)$. To this aim let us apply $\Gamma_{(d)}$ to the above exact sequence, getting:
\[0\longrightarrow \Gamma_{(d)}(\ZZ) \xrightarrow{\Gamma_{(d)}(\iota)} \Gamma_{(d)}(\QQ) \xrightarrow{\Gamma_{(d)}(\pi)} \Gamma_{(d)}(\QQ /\ZZ) \longrightarrow 0.\]
By definition $H_{(d)}^1(\ZZ)=\Gamma_{(d)}(\QQ /\ZZ)/\Gamma_{(d)}(\pi)(\Gamma_{(d)}(\QQ))$. Since $\Gamma_{(d)}(\QQ)=0$, we have $H_{(d)}^1(\ZZ)=\Gamma_{(d)}(\QQ /\ZZ)$. Finally, $\Gamma_{(d)}(\QQ /\ZZ)$ consists in all elements $\overline{s/t}\in \QQ /\ZZ$ with $s,t\in \ZZ$ and $t$ dividing $d^n$ for some $n\in \NN$. In particular, notice that $H_{(d)}^1(\ZZ)$ is not finitely generated as a $\ZZ$-module.
\end{example}

It is worthwhile to remark that from the definition of local cohomology we get the following powerful instrument: Any short exact sequence of $R$-modules
\[0\longrightarrow P\longrightarrow M \longrightarrow N \longrightarrow 0\]
yields the long exact sequence of $R$-modules:
\begin{eqnarray}\label{longexact}
0 \rightarrow \Ha^0(P)\longrightarrow \Ha^0(M) \longrightarrow \Ha^0(N)\longrightarrow \Ha^1(P)\longrightarrow \Ha^1(M)\rightarrow \ldots \nonumber \\
\ldots \rightarrow \Ha^{i-1}(N)\longrightarrow \Ha^i(P)\longrightarrow \Ha^i(M)\longrightarrow \Ha^i(N)\longrightarrow \Ha^{i+1}(P) \rightarrow \ldots
\end{eqnarray}

In Example \ref{cohomologyofZ} we had $H_{(d)}^i(\ZZ)=0$ for each $i>1= \dim \ZZ$. Actually Grothendieck showed that this is a general fact (for instance see \cite[Theorem 6.1.2]{BS}):
\begin{thm} \label{vanishinggro}(Grothendieck)
Let $R$ and $\aa$ be as above. For any $R$-module $M$, we have $\Ha^i(M)=0$ for all $i>\dim M$.
\end{thm}
Anyway, we should say that in Exemple \ref{cohomologyofZ} we were in a special case. In fact, an $R$-module $M$ has a finite injective resolution if and only if $M$ is Gorenstein. 

\section{Basic properties}

In addition to the long exact sequence \eqref{longexact}, we want to give other three basic tools which are essential in working with local cohomology. The first property we want to list is that local cohomology commutes with direct limits (for example see \cite[Theorem 3.4.10]{BS}).

\begin{lemma}\label{directcommut}
Let $R$ and $\aa$ be as usual. Let $\Pi$ be a poset and $\{M_{\alpha} \ : \ \alpha \in \Pi\}$ a direct system over $\Pi$ of $R$-modules. Then
\[\Ha^i(\varinjlim M_{\alpha})\cong \varinjlim \Ha^i(M_{\alpha}) \ \ \forall \ i\in \NN.\] 
\end{lemma}
%\begin{proof}
%Since direct limits commute with localizations {\cite[Appendix A, Theorem A.1]{matsu}}, the interpretation of local cohomology by means of the \u Cech complex (Theorem \ref{cech}) immediately yields the statement.
%\end{proof}

As an immediate consequence of Lemma \ref{directcommut}, we have that
\[\Ha^i(M\oplus N)\cong \Ha^i(M)\oplus \Ha^i(N)\]
for any $i\in \NN$ and for any $R$-modules $M$ and $N$.

The second property we would like to state allows us, in many cases, to control local cohomology when the base ring changes. The proofs can be found, for instance, in \cite[Theorem 4.3.2, Theorem 4.2.1]{BS}.

\begin{lemma}\label{basetheorems}
Let $R$ and $\aa$ be as above, $M$ an $R$-module, $R\xrightarrow{\phi} S$ a homomorphism of Noetherian rings and $N$ an $S$-module.
\begin{compactitem}
\item[(i)] (Flat Base Change)\index{Flat Base Change} If $\phi$ is flat, then $H_{\phi(\aa)S}^i(M\otimes_R S)\cong \Ha^i(M) \otimes_R S$ for any $i\in \NN$. The above isomorphism is functorial.
\item[(ii)] (Independence of the Base)\index{Independence of the Base} $\Ha^i(N)\cong H_{\phi(\aa)S}^i(N)$ for any $i\in \NN$, where the first local cohomology is computed over the base field $R$ and the second over $S$. The above isomorphism is functorial.
\end{compactitem}
\end{lemma}

Lemma \ref{basetheorems} has some very useful consequences: Let $\aa$ and $\bb$ be ideals of $R$, $M$ an $R$-module and $T\subseteq R$ a multiplicative system. Then, for any $i\in \NN$, there are functorial isomorphisms
\begin{eqnarray}
T^{-1}\Ha^i(M) \cong H_{\aa T^{-1}R}^i(T^{-1}M) \label{cohomloc}\\
\Ha^i(M)\otimes_R \widehat{R^{\bb}}\cong H_{\aa \widehat{R^{\bb}}}^i(M\otimes_R \widehat{R^{\bb}}) \label{cohomcompl}
\end{eqnarray}
Furthermore, if $R$ is local, then its completion $\widehat{R}$ is faithfully flat over $R$. So, for any $i\in \NN$,
\begin{equation}
\Ha^i(M) = 0 \iff H_{\aa \widehat{R}}^i(M\otimes_R \widehat{R})=0. \label{cohomcompl1}
\end{equation}
%In particular, in such a case
%\begin{equation}
%\cd(M,\aa) = \cd(M\otimes_{R}\widehat{R},\aa \widehat{R}). \label{cdcompl}
%\end{equation} 

The last basic tool we want to present in this section is the {\it Mayer-Vietories sequence}\index{Mayer-Vietories sequence}, for instance see \cite[3.2.3]{BS}. Sometimes, it allows us to handle with the support of the local cohomology modules. Let $\aa$ and $\bb$ be ideals of $R$ and $M$ an $R$-module. There is  the following exact sequence:

\begin{eqnarray}
0\rightarrow H_{\aa +\bb}^0(M)\rightarrow \Ha^0(M)\oplus H_{\bb}^0(M)\rightarrow H_{\aa \cap \bb}^0(M)\rightarrow H_{\aa +\bb}^1(M) \rightarrow \ldots \nonumber \\
\ldots \rightarrow H_{\aa +\bb}^i(M)\rightarrow \Ha^i(M)\oplus H_{\bb}^i(M)\rightarrow H_{\aa \cap \bb}^i(M)\rightarrow H_{\aa +\bb}^{i+1}(M) \rightarrow \ldots \label{mayervietoris}
\end{eqnarray}

\section{Equivalent definitions}

In this section we want to introduce two different, but equivalent, definitions of local cohomology. Depending on the issue one has to deal with, among these definitions one can be more suitable than the other.

\subsection{Via Ext}

Let $R$ and $\aa$ be as in the previous section. If $n$ and $m$ are two positive integers such that $n\geq m$, then there is a well defined homomorphism of $R$-modules $R/\aa^n \longrightarrow R/\aa^m$. Given an $R$-module $M$, the above map yields a homomorphism $\Hom_R(R/\aa^m,M)\longrightarrow \Hom_R(R/\aa^n,M)$, and more generally, for all $i\in \NN$,
\[\Ext_R^i(R/\aa^m,M)\longrightarrow \Ext_R^i(R/\aa^n,M).\]
The above maps clearly yield a direct system over $\NN$, namely $\{\Ext_R^i(R/\aa^n,M) \ : \ n\in \NN\}$. So, we can form the direct limit $\displaystyle \varinjlim \Ext_R^i(R/\aa^n,M)$. Since $\Hom_R(R/\aa^n,M)\cong 0:_M \aa^n$, we get the isomorphism:
\begin{displaymath}
\Ha^0(M)=\Ga(M)\cong \varinjlim \Hom_R(R/\aa^n,M).
\end{displaymath}
More generally the following theorem holds true (for instance see \cite[Theorem 1.3.8]{BS}):
\begin{thm}\label{limext}\index{local cohomology}
Let $R$ and $\aa$ be as above. For any $R$-module $M$ and $i\in \NN$, we have
\begin{displaymath}
\Ha^i(M)\cong \varinjlim \Ext_R^i(R/\aa^n,M).
\end{displaymath}
Furthermore, the above isomorphism is functorial.
\end{thm}
Actually, Theorem \ref{limext} can be stated as
\begin{equation}\label{cofinal}
\Ha^i(M)\cong \varinjlim \Ext_R^i(R/\aa_n,M),
\end{equation}
where $(\aa_n)_{n\in\NN}$ is an inverse system of ideals cofinal with $(\aa^n)_{n\in \NN}$, i.e. for any $n\in \NN$ there exist $k,m\in \NN$ such that $\aa_k\subseteq \aa^n$ and $\aa^m\subseteq \aa_n$. Several inverse systems of ideals cofinal with $(\aa^n)_{n\in \NN}$ have been used by many authors together with \eqref{cofinal} to prove important theorems. Let us list three examples:
\begin{itemize}
\item[(i)]  If $R$ is a complete local domain and $\wp$ is a prime ideal such that $\dim R/\wp = 1$, one can show that the inverse system of symbolic powers $(\wp^{(n)})_{n\in \NN}$ is cofinal with $(\wp^n)_{n\in \NN}$. On this fact is based the proof of the Hartshorne-Lichtenbaum Vanishing Theorem \cite{Ha}, which essentially says the following:
\begin{thm}\label{hartlich}(Hartshorne-Lichtenbaum)
Let $R$ be a $d$-dimensional complete local domain and $\aa\subseteq R$ an ideal. The following are equivalent:
\begin{compactitem}
\item[(a)] $\Ha^d(R)=0$;
\item[(b)] $\dim R/\aa > 0$. 
\end{compactitem}
\end{thm}
\item[(ii)] In the study of the local cohomology modules in characteristic $p>0$ a very powerful tool is represented by the inverse system $(\aa^{[p^n]})_{n\in \NN}$ of Frobenius powers of $\aa$: The ideal $\aa^{[p^n]}$ is generated by the $p^n$th powers of the elements of $\aa$. The above inverse system played a central role in the works of Peskine and Szpiro \cite{PS}, Hartshorne and Speiser \cite{ha-sp} and Lyubeznik \cite{ly1}. For instance, it was crucial in the proof of the following result of Peskine and Szpiro:
\begin{thm}\label{prelimpeskineszpiro}(Peskine-Szpiro)
Let $R$ be a regular local ring of positive characteristic and $\aa\subseteq R$ an ideal. Then
\[\Ha^i(R)=0 \ \ \forall \ i > \pd(R/\aa).\]
\end{thm}
\item[(iii)] The last example of inverse system we want to show was used by Lyubeznik in \cite{Ly} when $R$ is a polynomial ring and $\aa$ is a monomial ideal. The (pretended) Frobenius $n$th power of $\aa$, denoted by $\aa^{[n]}$, is the ideal generated by the $n$th powers of the monomials in $\aa$. The inverse system $(\aa^{[n]})_{n\in \NN}$ is obviously cofinal with $(\aa^n)_{n\in \NN}$. It was the fundamental ingredient in the proof of the following result: 
\begin{thm}\label{lyubmon}(Lyubeznik)
Let $R$ be a polynomial ring and $\aa\subseteq R$ be a square-free monomial ideal. Then
\[\Ha^i(R)=0 \ \ \forall \ i > \pd(R/\aa) \mbox{ \ \ \ and \ \ \ }\Ha^{\pd(R/\aa)}(R)\neq 0.\]
\end{thm}
\end{itemize}

\subsection{Via the \u Cech complex}

Let $\aaa = a_1,\ldots ,a_k$ be elements of $R$. The {\it \u Cech complex of $R$ with respect to $\aaa$} \index{Cech complex@\u Cech complex}, $C(\aaa)^{\bullet}$, is the complex of $R$-modules
\[0\xrightarrow{d^{-1}} C(\aaa)^0 \xrightarrow{d^0} C(\aaa)^1 \rightarrow \ldots \rightarrow C(\aaa)^{k-1}\xrightarrow{d^{k-1}}C(\aaa)^k\xrightarrow{d^k} 0,\]
where:
\begin{compactitem}
\item[(i)] $C(\aaa)^0=R$;
\item[(ii)] For $p=1,\ldots ,k$, the $R$-module $C(\aaa)^p$ is $\displaystyle \bigoplus_{1\leq i_1<i_2<\ldots <i_p\leq k} R_{a_{i_1}\cdots a_{i_p}}$;
\item[(iii)] The composition
\[C(\aaa)^0 \xrightarrow{d^0} C(\aaa)^1 \longrightarrow R_{a_i}\]
has to be the natural map from $R$ to $R_{a_i}$;
\item[(iv)] The maps $d^p$ are defined, for $p=1,\ldots ,k-1$, as follows: The composition
\[R_{a_{i_1}\cdots a_{i_p}}\longrightarrow C(\aaa)^{p}\xrightarrow{d^p} C(\aaa)^{p+1}\longrightarrow R_{a_{j_1}\cdots a_{j_{p+1}}}\]
has to be $(-1)^{s-1}$ times the natural map from $R_{a_{i_1}\cdots a_{i_p}}$ to $R_{a_{j_1}\cdots a_{j_{p+1}}}$ whenever $(i_1,\ldots ,i_p)=(j_1,\ldots ,\widehat{j_s},\ldots , j_{p+1})$; it has to be $0$ otherwise. 
\end{compactitem}

It is straightforward to check that the \u Cech complex is actually a complex. Thus we can define the cohomology groups $H^i(C(\aaa)^{\bullet})=\Ker(d^i)/\Im(d^{i-1})$. For any $R$-module $M$ we can also define $C(\aaa ,M)^{\bullet}=C(\aaa)^{\bullet}\otimes_R M$. We have the following result (see \cite[Theorem 5.1.19]{BS}):

\begin{thm}\label{cech}\index{local cohomology}
Let $\aa=(a_1,\ldots ,a_k)$ be an ideal of $R$. For any $R$-module $M$ and $i\in \NN$, we have
\[ \Ha^i(M)\cong H^i(C(\aaa ,M)^{\bullet}).\]
Furthermore, the above isomorphism is functorial.
\end{thm}
Probably the \u Cech complex supplies the ``easiest" way to compute the local cohomology modules. Anyway, even the computation of the local cohomology modules of the polynomial ring with respect to the maximal irrelevant ideal requires some effort, as the following example shows.

\begin{example}\label{cohompolring}\index{local cohomology!of the polynomial ring with support in the maximal irrelevant}
Let $S=\kkk[x,y]$ be the polynomial ring on two variables over a field $\kkk$, $a_1=x$ and $a_2=y$. Obviously $H_{(x,y)}^0(S)=0$. We want to compute $H_{(x,y)}^1(S)$ and $H_{(x,y)}^2(S)$. In fact, we know by Theorem \ref{vanishinggro} that $H_{(x,y)}^i(S)=0$ for any $i>2=\dim S$. The \u Cech complex $C(\aaa)^{\bullet}$ is
\[0\rightarrow S\xrightarrow{d^0} S_{x}\oplus S_{y}\xrightarrow{d^1}S_{xy}\xrightarrow{d^2} 0.\]
where the maps are:
\begin{compactitem}
\item[(i)] $d^0(f)=(f/1,f/1)$;
\item[(ii)] $d^1(f/x^n,g/y^m)=x^mg/(xy)^m-y^nf/(xy)^n$.
\end{compactitem}
One can easily realize that $d^1(f/x^n,g/y^m)=0$ if and only if $x^ng=y^mf$. Since $x^n,y^m$ is a regular sequence, $d^1(f/x^n,g/x^m)=0$ if and only if there exists $h\in S$ such that $f=x^nh$ and $g=y^mh$. This is the case if and only if $d^0(h)=(f/x^n,g/y^m)$. Therefore $\Ker(d^1)=\Im(d^0)$ and Theorem \ref{cech} implies $H_{(x,y)}^1(S)=H^1(C(\aaa)^{\bullet})=0$.

Notice that $S_{xy}$ is generated by $x^ny^m$ as an $S$-module, where $n$ and $m$ vary in $\ZZ$. Furthermore, if $n\geq 0$, then $x^ny^m=d^1(0,x^ny^m)$. Analogously, if $m\geq 0$, then $x^ny^m=d^1(-y^mx^n,0)$. Moreover, if both $n$ and $m$ are negative, then it is clear that $x^ny^m$ does not belong to the image of $d^1$. Therefore, using Theorem \ref{cech}, $H_{(x,y)}^2(S)\cong H^2(C(\aaa)^{\bullet})=S_{xy}/\Im(d^1)$ is the $S$-module generated by the elements $x^{-n}y^{-m}$ with both $n$ and $m$ bigger than $0$ supplied with the following multiplication:
\begin{align*}
(x^py^q)\cdot (x^{-n}y^{-m}) =
\begin{cases}
x^{p-n}y^{q-m} & \text { if } \ p<n \mbox{ and }q<m,\\
0 & \text {otherwise}
\end{cases}
\end{align*} 
\end{example}
What done in Example \ref{cohompolring} can be generalized to compute the local cohomology modules $H_{(x_1,\ldots ,x_n)}^i(\kkk[x_1,\ldots ,x_n])$.

\section{Why is local cohomology interesting?}

Local cohomology is an interesting subject on its own. However, it is also a valuable tool to face many algebraic and geometric problems. In this section, we list some subjects in which local cohomology can give a hand. 

\subsection{Depth}\index{depth}

Let $R$ and $\aa$ be as usual and $M$ be a finitely generated $R$-module such that $\aa M \neq M$. We recall that 
\[\grade (\aa , M) := \min \{i \ : \ \Ext_R^i(R/\aa,M)\neq 0\}.\] 
By the celebrated theorem of Rees (for instance see the book of Bruns and Herzog \cite[Theorem 1.2.5]{BH}), $\grade(\aa ,M)$ is the lenght of any maximal regular sequence on $M$ consisting of elements of $\aa$. On the other hand, if $a_1,\ldots ,a_k$ is a regular sequence on $M$, then $a_1^n,\ldots ,a_k^n$ is a regular sequence on $M$ as well. So, \eqref{cofinal} implies that $\Ha^i(M) = 0$ for any $i<\grade(\aa ,M)$. Actually, it is not difficult to show the following (for instance see \cite[Theorem 6.2.7]{BS}):
\begin{equation}\label{gradeloccoh}
\grade (\aa,M) = \min \{i \ : \ \Ha^i(M)\neq 0\}.
\end{equation}
In particular, if $R$ is local with maximal ideal $\mm$, then
\begin{equation}\label{depthloccoh}
\depth M = \min \{i \ : \ H_{\mm}^i(M)\neq 0\}.
\end{equation}
Therefore, one can read off the depth of a module by the local cohomology.

\subsection{Cohomological dimension and arithmetical rank}\label{secara}

The {\it cohomological dimension} \index{cohomological dimension} of an ideal $\aa\subseteq R$ is the natural number
\[ \cd(R, \mathfrak{a}) := \sup \{i \in \mathbb{N} \ : \ H_{\mathfrak{a}}^i(R) \neq 0 \} \]
By Theorem \ref{vanishinggro} we have $\cd(R,\mathfrak{a}) \leq \dim R$. 

\begin{remark}\label{explcohomdim}
If $\cd(R,\aa) = s$ then for any $R$-module $M$ we have $\Ha^i(M)=0$ for all $i>s$. To see this, by contradiction suppose that there exists $c>s$ and an $R$-module $M$ with $\Ha^c(M)\neq 0$. By Theorem \ref{vanishinggro}, we have $c \leq \dim R$, so we can assume that $c$ is maximal (considering all the $R$-modules).  Consider an exact sequence 
\[0\longrightarrow K \longrightarrow F \longrightarrow M \longrightarrow 0,\]
where $F$ is a free $R$-module. The above exact sequence yields the long exact sequence 
\[\ldots \longrightarrow \Ha^c(F)\longrightarrow \Ha^c(M) \longrightarrow \Ha^{c+1}(K) \longrightarrow \ldots .\]
By the maximality of $c$ it follows that $\Ha^{c+1}(K)=0$. Furthermore $\Ha^c(F)=\Ha^c(R)=0$ because local cohomology commutes with the direct limits (Lemma \ref{directcommut}). So, we get a contradiction. 
\end{remark}

The {\it arithmetical rank} \index{arithmetical rank} of an ideal $\aa \subseteq R$ is the natural number
\[\ara(\aa):=\min \{k \ : \ \exists \ r_1,\ldots ,r_k \in R \mbox{ such that }\sqrt{(r_1,\ldots ,r_k)}=\sqrt{\aa}\}.\]
If $R=\kkk[x_1,\ldots ,x_n]$ is the polynomial ring over an algebraically closed field $\kkk$, by Nullstellensatz $\ara(\aa)$ is the minimal natural number $k$ such that
\[\Z(\aa)=H_1\cap H_2 \cap \ldots \cap H_k\]
set-theoretically, where $\Z(\cdot)$ means the zero-locus of an ideal and the $H_j$'s are hypersurfaces of $\AA_{\kkk}^n$. From now on, anyway, we will consider varieties just in the schematic sense: For instance, $\AA_{\kkk}^n = \Spec(\kkk[x_1,\ldots ,x_n])$, $\V(\aa)=\{\wp \in \AA_{\kkk}^n \ : \ \wp \supseteq \aa\}$ and for $H\subseteq \AA_{\kkk}^n$ to be a hypersurface we mean that $H=\V((f))$, where $f\in \kkk[x_1,\ldots ,x_n]$. In this sense we have that, even if $\kkk$ is not algebraically closed, $\ara(\aa)$ is the minimal natural number $k$ such that
\[\V(\aa)=H_1\cap H_2 \cap \ldots \cap H_k,\]
where the $H_j$'s are hypersurfaces of $\AA_{\kkk}^n$. Actually, the same statement is true whenever $R$ is a UFD. 

If $R$ is graded and $\aa$ homogeneous we can also define the {\it homogeneous arithmetical rank}\index{arithmetical rank!homogeneous}, namely:
\[ \ara_h(\aa):=\min \{ k: \ \exists \ r_1, \ldots, r_k \in R \mbox{ homogeneous such that } \sqrt{\aa}=\sqrt{(r_1, \ldots, r_k)} \}. \]
Obviously we have
\[ \ara(\aa) \leq \ara_h(\aa). \]
If $R=\kkk[x_0,\ldots ,x_n]$ is the polynomial ring over a field $\kkk$, then $\ara_h(\aa)$ is the minimal natural number $k$ such that
\[\V_+(\aa)=H_1\cap H_2 \cap \ldots \cap H_k\]
set-theoretically, where the $H_j$'s are hypersurfaces of $\PP_{\kkk}^n$. It is an open problem whether $\ara(\aa)=\ara_h(\aa)$ in this case (see the survey article of Lyubeznik \cite{ly4}).

\begin{remark}
The reader should be careful to the following: If the base field is not algebraically closed, the number $\ara(\aa)$, where $\aa$ is an ideal of a polynomial ring, might not give the minimal number of polynomials whose zero-locus is $\Z(\aa)$. For instance, if $\aa=(f_1,\ldots ,f_m)\subseteq \mathbb{R}[x_1,\ldots ,x_n]$, clearly 
\[ \Z(\aa)=\Z(f_1^2+\ldots +f_m^2).\]
However $\ara(\aa)$ can be bigger than $1$ (for instance by \eqref{hauptidealsatz}). 
\end{remark}

From what said up to now the arithmetical rank has a great geometrical interest. Intuitively, everybody would say that there is no hope to define a curve of $\AA^4$ as the intersection of $2$ hypersurfaces. Actually this intuition is a consequence of the celebrated Krull's Hauptidealsatz (for instance see Matsumura \cite[Theorem 13.5]{matsu}), which essentially says: Given an ideal $\aa$ in a ring $R$,
\begin{equation}\label{hauptidealsatz}
\ara(\aa)\geq \height(\aa). 
\end{equation}
If equality holds in \eqref{hauptidealsatz}, $\aa$ is said to be a {\it set-theoretic complete intersection}\index{set-theoretic complete intersection}. Not all the ideals are set-theoretic complete intersections, but how can we decide it? One lower bound for the arithmetical rank is supplied by the cohomological dimension, namely
\begin{equation}\label{aragcd}
\ara(\aa)\geq \cd(R,\aa).
\end{equation}
To see inequality \eqref{aragcd} recall that the local cohomology modules $\Ha^i(R)$ are the same of $H_{\sqrt{\aa}}^i(R)$ (Remark \ref{cohomuptorad}). Then inequality \eqref{aragcd} is clear from the interpretation of local cohomology by meaning of the \u Cech complex (Theorem \ref{cech}). Actually, inequality \eqref{aragcd} is better than \eqref{hauptidealsatz} thanks to the following theorem of Grothendieck (for example see \cite[Theorem 6.1.4]{BS}):
\begin{thm}\label{nonvanishinggro} (Grothendieck)
Let $\aa$ be an ideal of a ring $R$ and $M$ be a finitely generated $R$-module. Let $\wp$ be a minimal prime ideal of $\aa$ and $\dim M_{\wp} = q$. Then 
\[\Ha^q(M)\neq 0.\]
In particular $\cd(R,\aa)\geq \height(\aa)$.
\end{thm} 

\begin{example}
Inequality in Theorem \ref{nonvanishinggro} can be strict. Let $R=\kkk[[x,y,v,w]]$ be the ring of formal series in $4$ variables over a field $\kkk$ and $\aa = (xv,xw,yv,yw) = (x,y)\cap (v,w)$. We have that $\height(\aa)=2$. Let us consider the following piece of the Mayer-Vietoris sequence \eqref{mayervietoris}
\[\Ha^3(R)\longrightarrow H_{(x,y,v,w)}^4(R)\longrightarrow H_{(x,y)}^4(R)\oplus H_{(v,w)}^4(R).\]
By Theorem \ref{nonvanishinggro} $H_{(x,y,v,w)}^4(R)\neq 0$, whereas $H_{(x,y)}^4(R)\oplus H_{(v,w)}^4(R)=0$ by \eqref{aragcd}. Therefore $\Ha^3(R)$ has to be different from zero, so that $\cd(R,\aa)\geq 3>2=\height(\aa)$. In particular $\ara(\aa)\geq 3 > \height(\aa)$, thus $\aa$ is not a set-theoretic complete intersection. Actually $\ara(\aa)=3$, in fact one can easily show that $\aa = \sqrt{(xv, \ yw, \ xw+yv)}$.
\end{example}

In Chapters \ref{chapter1} and \ref{chapter2}, we will see many examples of ideals which are not set-theoretic complete intersections.

\subsection{The graded setting}\label{subcastmum}

If $R$ is a graded ring, $\aa\subseteq R$ is a graded ideal and $M$ is a $\ZZ$-graded $R$-module, then the local cohomology modules $\Ha^i(M)$ carry on a natural $\ZZ$-grading. For instance, we can choose homogeneous generators of $\aa$, say $\aaa = a_1,\ldots ,a_k$. Then we can give a $\ZZ$-grading to the $R$-modules in the \u Cech complex $C(\aaa ,M)^{\bullet}$ in the obvious way. Moreover, it is straightforward to check that, this way, the maps in $C(\aaa ,M)^{\bullet}$ are homogeneous. Therefore, the cohomology modules of $C(\aaa ,M)^{\bullet}$ are $\ZZ$-graded, and so $\Ha^i(M)$ is $\ZZ$-graded for any $i\in \NN$ by Theorem \ref{cech}. For any integer $d\in \ZZ$, we will denote by $[\Ha^i(M)]_d$, or simply by $\Ha^i(M)_d$, the graded piece of degree $d$ of $\Ha^i(M)$. Of course one could try to give a $\ZZ$-grading to $\Ha^i(M)$ also using Theorem \ref{limext}, or computing directly the local cohomology functors as derived functors of $\Ga(\cdot)$ restricted to the category of $\ZZ$-graded $R$-modules. Well, it turns out that all the above possible definitions agree (see \cite[Chapter 12]{BS}). 

Suppose that $R$ is a standard graded ring (i.e. $R$ is graded and it is generated by $R_1$ as an $R_0$-algebra). The {\it Castelnuovo-Mumford regularity} \index{Castelnuovo-Mumford regularity} of a $\ZZ$-graded module $M$ is defined as
\[\reg(M):=\sup \{d+i \ : \ [H_{R_+}^i(M)]_d \neq 0 \mbox{ and }i\in \NN\}.\]
Unless $H_{R_+}^i(M)=0$ for all $i\in \NN$, it turns out that the Castelnuovo-Mumford regularity is always an integer (Theorem \ref{vanishinggro} and \cite[Proposition 15.1.5 (ii)]{BS}). In the case in which $R_0$ is a field, if $M$ is a finitely generated Cohen-Macaulay $\ZZ$-graded module, then, using \eqref{depthloccoh} and Theorem \ref{vanishinggro} we have that
\[\reg(M)=\sup \{d+\dim M \ : \ [H_{R_+}^{\dim M}(M)]_d \neq 0\}.\]

At this point it is worthwhile to introduce an invariant related to the regularity. Let $R$ be a Cohen-Macaulay graded algebra such that $R_0$ is a field. The {\it $a$-invariant} \index{a-invariant@$a$-invariant|(} of $R$ is defined to be 
\[a(R) := \sup \{d \in \ZZ \ : \ [H_{R_+}^{\dim R}(R)]_d \neq 0\}.\]
Therefore, if $R$ is furthermore standard graded, then
\begin{equation}\label{regfroma-inv}
\reg(R)=a(R)+\dim R.
\end{equation} 

\begin{example}
Let $S=\kkk[x,y]$ be the polynomial ring in two variables over a field $\kkk$. By Example \ref{cohompolring} we immediately see that $a(S)=-2$ and $\reg(S)=0$. More generally, if $S=\kkk[x_1,\ldots ,x_n]$ is the polynomial ring in $n$ variables over a field $\kkk$, then $a(S)=-n$ and $\reg(S)=0$.
\end{example}

Up to now, we have not yet stated one of the fundamental theorems in the theory of local cohomology: The {\it Grothendieck local duality theorem}\index{local duality}. We decided to do so because such a statement would require to introduce Matlis duality theory, which would take too much space, for our purposes. However, if the ambient ring $R$ is a graded Cohen-Macaulay ring such that $R_0$ is a field, we can state a weaker version of local duality theorem, which actually turns out to be useful in many situations and which avoids the introduction of Matlis duality. So let $R$ be as above. By $\omega_R$ we will denote the canonical module of $R$ (for the definition see \cite[Definition 3.6.8]{BH}, for the existence and unicity in this case put together \cite[Example 3.6.10, Proposition 3.6.12 and Proposition 3.6.9]{BH}).  

\begin{thm}(Grothendieck)\label{weaklocduality}
Let $R$ be an $n$-dimensional Cohen-Macaulay graded ring such that $R_0=\kkk$ is a field and let $\mm=R_+$ be its maximal irrelevant ideal. For any finitely generated $\ZZ$-graded $R$-module $M$, we have
\[\dim_{\kkk}(H_{\mm}^i(M)_{-d})=\dim_{\kkk}(\Ext_R^{n-i}(M,\omega_R)_{d}) \ \ \ \forall \ d\in \ZZ.\]
In particular, if $R=S=\kkk[x_1,\ldots ,x_n]$ is the polynomial ring in $n$ variables over a field $\kkk$, we have
\[\dim_{\kkk}(H_{\mm}^i(M)_{-d})=\dim_{\kkk}(\Ext_S^{n-i}(M,S)_{d-n}) \ \ \ \forall \ d\in \ZZ.\]
\end{thm}  
\begin{proof}
The theorem follows at once by \cite[Theorem 3.6.19]{BH}. For the part concerning the polynomial ring $S$, simply notice that $\omega_S = S(-n)$.
\end{proof}

A consequence of Theorem \ref{weaklocduality} is that
\[\dim_{\kkk}(H_{\mm}^n(R)_{-d})=\dim_{\kkk}(\Hom_R(R,\omega_R)_{d})=\dim_{\kkk}([\omega_R]_d) \ \ \ \forall \ d\in \ZZ.\]
So we get another interpretation of the $a$-invariant of a Cohen-Macaulay graded algebra such that $R_0$ is a field, namely
\begin{equation}\label{a-invcanonical}
a(R)=-\min \{d\in \ZZ \ : \ [\omega_R]_d \neq 0\}.
\end{equation}

If $R=S$ is the polynomial ring in $n$ variables over a field $\kkk$, local duality supplies another interpretation of the Castelnuovo-Mumford regularity. For any nonzero finitely generated $\ZZ$-graded $S$-module $M$ there is a graded minimal free resolution of $M$: 
\[0\rightarrow \bigoplus_{j\in \ZZ}S(-j)^{\beta_{p,j}}\rightarrow \bigoplus_{j\in \ZZ}S(-j)^{\beta_{p-1,j}}\rightarrow \ldots \rightarrow \bigoplus_{j\in \ZZ}S(-j)^{\beta_{0,j}}\rightarrow M \rightarrow 0\]
with $p\leq n$. Actually the natural numbers $\beta_{i,j}$ are numerical invariants of $M$, and they are called the {\it graded Betti numbers}\index{Betti numbers}\index{Betti numbers!graded} of $M$. Let us define
\[r:=\max \{j-i \ : \ \beta_{i,j}\neq 0, \ i\in \{0,\ldots ,p\}\}.\] 
Then, for any $i=1,\ldots ,p$, the $S$-module $F_i:=\bigoplus_{j\in \ZZ}S(-j)^{\beta_{i,j}}$ has not minimal generators in degree bigger than $r+i$. This implies that $\Hom_S(F_i,S)_{d}=0$ whenever $d<-r-i$. So, since $\Ext_S^{i}(M,S)$ is a quotient of a submodule of $\Hom_S(F_i,S)$, we get that
\[d<-r-i \implies \Ext_S^i(M,S)_d=0. \]
Actually, playing more attention, we can find an index $i\in \{0,\ldots ,p\}$ such that 
\[ \Ext_S^i(M,S)_{-r-i}\neq 0\] 
(for the details see Eisenbud \cite[Proposition 20.16]{eisenbud}). Therefore, by Theorem \ref{weaklocduality}, we get $H_{\mm}^{n-i}(M)_{-d-n}=0$ for any $d<-r-i$, which means that $H_{\mm}^j(M)_k=0$ whenever $j+k>r$. Moreover there exists a $j$ such that $H_{\mm}^j(M)_{r-j}\neq 0$. Eventually, we get: 
\begin{equation}\label{regularityfreeres}\index{Castelnuovo-Mumford regularity}
\reg(M)=\max \{j-i \ : \ \beta_{i,j}\neq 0, \ i\in \{0,\ldots ,p\}\}.
\end{equation}
From \eqref{regularityfreeres} it is immediate to check  that a minimal generator of $M$ has at most degree $\reg(M)$. This fact is true over much more general rings than $S$.

\begin{thm}\label{reggdegree}
Let $R$ be a standard graded ring and $M$ be a finitely generated $\ZZ$-graded module. A minimal generator of $M$ has at most degree $\reg(M)$. 
\end{thm}

We end this chapter remarking a relationship between local cohomology modules and Hilbert polynomials. Let $R$ be a graded ring such that $R_0=\kkk$ is a field and let $M$ be a finitely generated $\ZZ$-graded $R$-module. We denote by $\Hf_{M}:\ZZ \rightarrow \NN$ the {\it Hilbert function of $M$}\index{Hilbert function}, i.e.
\[\Hf_{M}(d)=\dim_{\kkk}(M_d).\]
A result of Serre (see \cite[Theorem 4.4.3 (a)]{BH}) states that there exists a unique quasi-polynomial function $P:\ZZ \rightarrow \NN$ such that $P(m)=\Hf_M(m)$ for any $m\gg 0$. Recall that ``$P$ quasi-polynomial function" means that there exists a positive integer $g$ and polynomials $P_i\in \QQ[X]$ for $i=0,\ldots  ,g-1$ such that $P(m)=P_i(m)$ for all $m=kg+i$ with $k\in \ZZ$. Such a quasi-polynomial function is called the {\it Hilbert quasi-polynomial of $M$}\index{Hilbert quasi-polynomial}, and denoted by $\Hp_M$. If $R$ is standard graded, then $\Hp_M$ is actually a polynomial function of degree $\dim M -1$ by a previous theorem by Hilbert (see \cite[Theorem 4.1.3]{BH}), and in this case $\Hp_M$ is simply called the {\it Hilbert polynomial of $M$}\index{Hilbert polynomial}.  The promised relationship between local cohomology modules and Hilbert polynomials is once again work of Serre (\cite[Theorem 4.4.3 (b)]{BH}).

\begin{thm}(Serre)\label{serrehilb}
Let $R$ be a graded ring such that $R_0=\kkk$ is a field and let $\mm=R_+$. If $M$ is a nonzero finitely generated $\ZZ$-graded $R$-module of dimension $d$, then
\[\Hf_M(m)-\Hp_M(m)=\sum_{i=0}^d (-1)^i \dim_{\kkk}H_{\mm}^i(M)_m \ \ \ \forall \ m\in \ZZ.\]
\end{thm}

Eventually, notice that Theorem \ref{serrehilb} yields a third interpretation of the $a$-invariant. If $R$ is a Cohen-Macaulay graded ring such that $R_0$ is a field, then:
\begin{equation}\label{a-invandhilb}\index{a-invariant@$a$-invariant|)}
a(R) = \sup \{i\in \ZZ \ : \ \Hf_R(i)\neq \Hp_R(i)\}.
\end{equation} 
\index{local cohomology|)}

\mainmatter
\chapter{Cohomological Dimension}\label{chapter1}\index{cohomological dimension|(}

In his seminair on local cohomology (see the notes written by Hartshorne \cite[p. 79]{gro1}), Grothendieck raised the problem of finding conditions under which, fixed a positive integer $c$, the local cohomology modules $\Ha^i(R)=0$ for every $i>c$, where $\aa$ is an ideal in a ring $R$. In other words, looking for terms under which $\cd(R,\aa)\leq c$. Ever since many mathematicians worked on this question: For instance, see Hartshorne \cite{Ha,hartshorne4}, Peskine and Szpiro \cite{PS}, Ogus \cite{ogus}, Hartshorne and Speiser \cite{ha-sp}, Faltings \cite{faltings1}, Huneke and Lyubeznik \cite{hu-ly}, Lyubeznik \cite{ly1}, Singh and Walther \cite{siwa1}.

This chapter is dedicated to discuss the question raised by Grothendieck. Essentially, we will divide the chapter in two parts. In the first one we will treat of necessary conditions for $\cd(R,\aa)$ being smaller than a given integer, whereas in the second one we will discuss about sufficient conditions for it. However, before describing the contents and results of the chapter, we want to give a brief summary of the classical results existing in the literature around this problem.

First of all, the reader might object: Why do not ask conditions about the lowest nonvanishing local cohomology module, rather than the highest one? The reason is that the smallest integer $i$ such that $\Ha^i(R)\neq 0$ is well understood: It is the maximal length of a regular sequence on $R$ consisting of elements of $\aa$, see \eqref{gradeloccoh}. The highest nonvanishing local cohomology module, instead, is much more subtle to locate. Besides, the number $\cd(R,\aa)$ is charge of informations. For instance, it supplies one of the few known obstructions for an ideal being generated up to radical by a certain number of elements, see \eqref{aragcd}. The two starting results about ``the problem of the cohomological dimension" are both due to Grothendieck: They are essential, as they fix the range in which we must look for the natural number $\cd(R,\aa)$ (Theorems \ref{nonvanishinggro} and \ref{vanishinggro}):
$$\height(\aa)\leq \cd(R,\aa)\leq \dim R.$$
Let us set $n:=\dim R$. In light of the results of Grothendieck, as a first step it was natural to try to characterize when $\cd(R,\aa)\leq n-1$. In continuing the discussion let us suppose the ring $R$ is local and complete. Such an assumption is reasonable, since for any ring $R$:
\begin{compactitem} 
\item[(i)] $\cd(R,\aa)=\sup \{\cd(R_{\mm},\aa R_{\mm}) \ : \ \mm \mbox{ is a maximal ideal}\}$ by \eqref{cohomloc}.
\item[(ii)] $\cd(R_{\mm},\aa R_{\mm})=\cd(\widehat{R_{\mm}},\aa \widehat{R_{\mm}})$  by \eqref{cohomcompl1}.
\end{compactitem}
First Lichtenbaum and then, more generally, Hartshorne \cite[Theorem 3.1]{Ha}, settled the problem of characterizing when $\cd(R,\aa)\leq n-1$. Essentially, they showed that the necessary and sufficient condition for this to happen is that $\dim R/\aa > 0$, see Theorem \ref{hartlich}. By then, the next step should have been to describe when $\cd(R,\aa)\leq n-2$. In general this case is still not understood well as the first one. However, if $R$ is regular, a necessary and sufficient condition, besides being known, is very nice. Such a  condition, which has been proven under different assumptions in \cite{Ha,PS,hu-ly,siwa1}, essentially is that the punctured spectrum of $R/\aa$ is connected\index{connected}. Actually, if the ``ambient ring" $R$ is regular, $\cd(R,\aa)$ can always be characterized in terms of the ring $R/\aa$ (\cite{ogus,ha-sp,ly1}). However, in any of these papers, the described conditions are quite difficult to verify. Our task in this chapter will be to give some necessary and/or sufficient conditions, as easier as possible to verify, for $\cd(R,\aa)$ being smaller than a fixed integer.   

\vspace{4mm}

In the first section we will focus on giving necessary conditions for $\Ha^i(R)$ vanishing for any $i$ bigger than a fixed integer. The  framework of this section comes from the first part of our paper \cite{va1}. As we remarked in \ref{cohomuptorad}, the local cohomology functors $\Ha^i$ depend just on the radical of $\aa$. Therefore, a way to face the Grothendieck's problem could be to study the topological properties of $\Spec(R/\aa)$. In fact, such a topological space and $\Spec(R/\sqrt{\aa})$ are obviously homeomorphic. Actually in this section we will act in this way; particularly, we will study the connectedness properties of $\Spec(R/\aa)$, in the sense explained in Appendix \ref{connectednesssection}. In \cite{ho-hu}, Hochster and Huneke proved that if $\aa$ is an ideal of an $n$-dimensional, $(n-1)$-connected \index{connected!r-@$r$-}, local, complete ring $R$ such that $\cd(R,\aa)\leq n-2$, then $R/\aa$ is $1$-connected, generalizing a result previously obtained by Faltings in \cite{faltings}. We prove, under the same hypotheses on the ring $R$, that 
\[\cd(R,\aa)\leq n-s \ \implies \ R/\aa \mbox{ is $(s-1)$-connected}.\] 
More generally, we prove that if $R$ is $r$-connected, with $r<n$, then
\begin{equation}\label{barteq}
\cd(R,\aa)\leq r-s \ \implies \ R/\aa \mbox{ is $s$-connected},
\end{equation}
see Theorem \ref{bart}. This result also strengthens a connectedness theorem due to Grothendieck \cite[Expos\'{e} XIII, Th\'{e}or\`{e}me 2.1]{SGA2}, who got the same thesis under the assumption that $\aa$ could be generated by $r-s$ elements. To prove \eqref{barteq} we drew on the proof of Grothendieck's theorem given in the book of Brodmann and Sharp \cite[19.2.11]{BS}.
(Theorem \ref{bart} has been proved independently in the paper of Divaani-Aazar, Naghipour and Tousi \cite[Theorem 2.8]{D-N-T}. We illustrate a relevant mistake in \cite[Theorem 3.4]{D-N-T} in Remark \ref{patty}).

We prove \eqref{barteq} also for certain noncomplete rings $R$, such as local rings satisfying the $S_2$ Serre's condition (Proposition \ref{serre}), or graded rings over a field (in this case the ideal $\aa$ must be homogeneous), see Theorem \ref{winchester}. This last version of \eqref{barteq} allows us to translate the result into a geometric point of view, dicscussing about the cohomological dimension of an open subscheme $U$ of a scheme $X$ projective over a field (Theorem \ref{kaprapal}). Particularly, we give topological necessary conditions on the closed subset $X\setminus U$ for $U$ being affine\index{affine schemes} (Corollary \ref{abrham}).

We end the section discussing whether the implication of \eqref{barteq} can be reversed. In general the answer is no, and we give explicit counterexamples.

\vspace{2mm}

The aim of the second section, whose results are part of our paper \cite{va2}, is to explore the case in which $R:=\kkk[x_1,\ldots ,x_n]$ is a polynomial ring over a field $\kkk$, usually of characteristic $0$, and $\aa$ is a homogeneous ideal. In this setting we have a better understanding of the cohomological dimension. The characteristic $0$ assumption allows us to reduce the issues, the most of the times, to the case $\kkk = \CC$: This way we can borrow results from Algebraic Topology and Complex Analysis. A key result we prove is Theorem \ref{1}. An essential ingredient in its proof is the work of Ogus \cite{ogus}, combined with a comparison theorem of Grothendieck obtained in \cite{gro} and classical results from Hodge theory. Theorem \ref{1} gives a criterion to compute the cohomological dimension of a homogeneous ideal $\aa$ in a polynomial ring $R$ over a field of characteristic $0$, provided that $R/\aa$ has an isolated singularity (see also Theorem \ref{ccdalg} for a more algebraic interpretation). Roughly speaking, such a criterion relates the cohomological dimension $\cd(R,\aa)$ with the dimensions of the finite $\kkk$-vector spaces
\[[H_{\nn}^i(\Lambda^j \Omega_{A/\kkk})]_0,\]
where $A:=R/\aa$, $\nn$ is the irrelevant maximal ideal of $A$ and $\Omega_{A/\kkk}$ is the module of K\"ahler differentials of $A$ over $\kkk$. In what follows, let us denote the $A$-module $\Lambda^j \Omega_{A/\kkk}$ by $\Omega^j$. The advantages of such a characterization are essentially two:
\begin{compactitem}
\item[(i)] The cohomological dimension $\cd(R,\aa)$, in this case, is an intrinsic invariant of $A=R/\aa$, which is not obvious a priori.
\item[(ii)] The dimensions $\dim_{\kkk}[H_{\nn}^i(\Omega^j)]_0$ are moderately easy to compute, thanks to the Grothendieck's local duality (see Theorem \ref{weaklocduality}):
\[\dim_{\kkk}[H_{\nn}^i(\Omega^j)]_0=\dim_{\kkk}[\Ext_R^{n-i}(\Omega^j,R)]_{-n}.\]
\end{compactitem}
Even if Theorem \ref{1} will essentially be obtained putting together earlier results without upsetting ideas, it has at least two amazing consequences. The first one regards a relationship between depth and cohomological dimension. Before describing it, let us remind the result of Peskine and Szpiro mentioned in Theorem \ref{prelimpeskineszpiro}. It implies that, if $\chara(\kkk)>0$, then
\begin{equation}\label{intropeskineszpiro}\index{depth}
\depth(R/\aa)\geq t \implies \cd(R,\aa)\leq n-t.
\end{equation}
In characteristic $0$, the above fact does not hold true already for $t=4$ (see Example \ref{t=4ps}). However, for $t\leq 2$, \eqref{intropeskineszpiro} holds true also in characteristic $0$ from \cite{Ha} (more generally see Proposition \ref{hunekelyubeznik}). Well, Theorem \ref{1} enables us to settle the case $t=3$ of \eqref{intropeskineszpiro} also in characteristic $0$, provided that $R/\aa$ has an isolated singularity (see Theorem \ref{depth}). From this fact it is natural to raise the following problem:

\vspace{1mm}

\noindent \textbf{Question} \ref{depth-coho}. Suppose that $R$ is a regular local ring, and that $\aa\subseteq R$ is an ideal such that $\depth(R/\aa)\geq 3$. Is it true that $\cd(R,\aa)\leq \dim R -3$?

\vspace{1mm}

The second consequence of Theorem \ref{1} consists in the solution of a remarkable case of a conjecture done by Lyubeznik in \cite{ly2}. It concerns a relationship between cohomological dimension and \'etale cohomological dimension\index{etale cohomological dimension@\'etale cohomological dimension} of a scheme, already wondered by Hartshorne in \cite{hartshorne4}. The truth of Lyubeznik's guess would imply that \'etale cohomological dimension provides a better lower bound for the arithmetical rank than the one supplied by cohomological dimension (see \eqref{aragecd} and \eqref{aragcd}). In Theorem \ref{lyub} we give a positive answer to the conjecture in characteristic $0$, under a smoothness assumption. On the other hand, Lyubeznik informed us that recently he found a counterexample in positive characteristic.

\section{Necessary conditions for the vanishing of local cohomology}\index{connected|(}\index{connected!r-@$r$-|(}
\label{chap1}

Let $\aa$ be an ideal of an ring $R$, and $c$ be a positive integer such that $\height(\aa)\leq c < \dim R$. As we anticipated in the introduction, this section is dedicated in finding necessary conditions for $\Ha^i(R)$ vanishing for any $i>c$. Mainly, we will care about the topological properties which $\Spec(R/\aa)$ must have to this aim.

\subsection{Cohomological dimension vs connectedness}
\label{subchap1.1}

The purpose of this subsection is to prove Theorem \ref{bart}. It fixes the connectedness properties that $\Spec(R/\aa)$ must own in order to $\Ha^i(R)$ vanish for all $i>c$. Theorem \ref{bart} has as consequences many previously known theorems. Especially, we remark a theorem by Grothendieck (Theorem \ref{palladineve}) and one by Hochster and Huneke (Theorem \ref{frink}).

Let us start with a lemma which shows a sub-additive property of the cohomological dimension. It will be crucial to prove Proposition \ref{homer}. 

\begin{lemma}\label{subadditivecohomdim}
Let $R$ be any ring, and $\aa$, $\bb$ ideals of $R$. Then
\[\cd(R,\aa + \bb)\leq \cd(R,\aa)+\cd(R,\bb).\]
\end{lemma}
\begin{proof}
By \cite[Proposition 2.1.4]{BS}, $\Gamma_{\bb}$ maps injective $R$-modules into injective ones. Furthermore notice that $\Ga \circ \Gamma_{\bb} = \Gamma_{\aa+\bb}$. So the statement follows at once by the spectral sequence 
\[E_2^{i,j} = \Ha^i(H_{\bb}^j(R))\implies H_{\aa + \bb}^{i+j}(R),\]
for instance see the book of Gelfand and Manin \cite[Theorem III.7.7]{GM}.
\end{proof}

The next two propositions strengthen \cite[Proposition 19.2.7]{BS} and \cite[Lemma 19.2.8]{BS}.

\begin{prop}\label{homer}
Let $(R,\mathfrak{m})$ be a complete local domain and let $\aa$ and
$\mathfrak{b}$ be ideals of $R$ such that $\dim
R/\mathfrak{a}
>\dim R/(\mathfrak{a}+\mathfrak{b})$ and $\dim
R/\mathfrak{b}
>\dim R/(\mathfrak{a}+\mathfrak{b})$. Then
\[ \cd(R, \mathfrak{a} \cap \mathfrak{b}) \geq \dim R - \dim R/(\mathfrak{a}+\mathfrak{b}) - 1 \]
\end{prop}
\begin{proof}
Set $n := \dim R$ and $d := \dim R/(\mathfrak{a}+\mathfrak{b})$. We make an induction on $d$. If $d=0$ we consider the Mayer-Vietoris
sequence \eqref{mayervietoris}
\[ \ldots \longrightarrow H_{\mathfrak{a} \cap \mathfrak{b}}^{n-1}(R)  \longrightarrow H_{\mathfrak{a}+\mathfrak{b}}^n(R) \longrightarrow H_{\mathfrak{a}}^n(R) \oplus H_{\mathfrak{b}}^n(R) \longrightarrow \ldots . \]
Hartshorne-Lichtenbaum Vanishing Theorem \ref{hartlich} implies that $H_{\mathfrak{a}}^n(R)=H_{\mathfrak{b}}^n(R)=0$ and $H_{\mathfrak{a}+\mathfrak{b}}^n(R) \neq
0$. So the above exact sequence implies $\cd(R, \mathfrak{a}\cap
\mathfrak{b}) \geq n-1$.

Consider the case in which $d>0$. By prime avoidance we can choose an element $x \in \mathfrak{m}$ such that $x$ does not belong to any minimal prime
of $\mathfrak{a}$, $ \mathfrak{b}$ and $\mathfrak{a}+\mathfrak{b}$.
Then let $\mathfrak{a}'=\mathfrak{a}+(x)$ and
$\mathfrak{b}'=\mathfrak{b}+(x)$. From the choice of $x$ and by the Hauptidealsatz of Krull it follows
that $\dim R/(\mathfrak{a}'+\mathfrak{b}') = d-1$, $\dim
R/\mathfrak{a}' = \dim R/\mathfrak{a}-1
>d-1$ and $\dim R/\mathfrak{b}'=\dim R/\mathfrak{b}-1>d-1$; hence by induction we have
$\cd(R,\mathfrak{a}' \cap \mathfrak{b}') \geq n-d$. Because $\sqrt{\mathfrak{a} \cap \mathfrak{b} + (x)}=
\sqrt{\mathfrak{a}' \cap \mathfrak{b}'}$, then $H_{\mathfrak{a}'
\cap \mathfrak{b}'}^i(R)=H_{\mathfrak{a} \cap \mathfrak{b} +
(x)}^i(R)$ for all $i \in \mathbb{N}$. From Lemma \ref{subadditivecohomdim} we have
\[n-d \leq \cd(R,\aa' \cap \bb') = \cd(R,\aa \cap \bb + (x)) \leq \cd(R,\aa \cap \bb) + 1.\]
Therefore the proof is completed.
\end{proof}

The next result is a generalization of Proposition \ref{homer} to complete local rings which are not necessarily domains.

\begin{prop}\label{lisa}
Let $R$ be an $r$-connected complete local ring with $r<\dim R$. Consider two ideals $\aa$, $\mathfrak{b}$ of $R$ such that $\dim R/\mathfrak{a}
>\dim R/(\mathfrak{a}+\mathfrak{b})<\dim R/\mathfrak{b}$. Then
\begin{equation}\label{cdconn1} 
\cd(R, \mathfrak{a} \cap \mathfrak{b})  \geq r - \dim R/(\mathfrak{a}+\mathfrak{b}). 
\end{equation}
\end{prop}
\begin{proof}
Set $d := \dim R/(\mathfrak{a}+\mathfrak{b})$, and let $\wp_1,
\ldots , \wp_m$ be the minimal prime ideals of $R$.

First we consider the following case:
\[\exists \ i\in \{1, \ldots ,m\}\mbox{ \ such that \ } \dim R/(\mathfrak{a}+\wp_i)>d\mbox{ \ and \ } \dim R/(\mathfrak{b}+\wp_i)>d.\] 
In this case we can use Proposition \ref{homer},
considering $R/\wp_i$ as $R$, $(\mathfrak{a}+\wp_i)/\wp_i$ as  $\mathfrak{a}$ and
$(\mathfrak{b}+\wp_i)/\wp_i$ as $\mathfrak{b}$. Then
\[ \cd( R/\wp_i,( (\mathfrak{a}+\wp_i) \cap (\mathfrak{b}+\wp_i) ) / \wp_i) \geq \dim R/\wp_i- d-1.\]
Notice that 
\[ \cd( R/\wp_i,( (\mathfrak{a}+\wp_i) \cap (\mathfrak{b}+\wp_i)
) / \wp_i) = \cd( R/\wp_i,( (\mathfrak{a} \cap \mathfrak{b})+\wp_i
) / \wp_i) \leq \cd( R, \mathfrak{a} \cap \mathfrak{b})\] 
by Lemma \ref{basetheorems} (ii). Moreover $\dim R/\wp_i > r$: If $\dim R/\wp_i = \dim R$, then this follows by the assumptions. Otherwise $R$ is a reducible ring, thus Lemma \ref{milhouse} (ii) implies $\dim R/\wp_i>r$. So we get the thesis.

Thus we can suppose that 
\[\forall \ i \in \{ 1, \ldots, m \} \mbox{ \ either \ }\dim R/(\mathfrak{a}+ \wp_i) \leq d \mbox{ \ or \ }\dim R/(\mathfrak{b}+\wp_i) \leq d.\] 
After a rearrangement we can pick $s\in \{1,\ldots ,m-1\}$ such that for all $i \leq s$ we have $\dim
R/(\mathfrak{a}+\wp_i) \leq d$ and  for all $i > s$ we have $\dim
R/(\mathfrak{a}+\wp_i) > d$.
The existence of such an $s$ is guaranteed from the fact that $\dim
R/\mathfrak{a}
> d$ and  $\dim R/\mathfrak{b}
> d$. Let us consider the ideal of $R$
\[ \qq=(\wp_1 \cap \ldots \cap \wp_s) + (\wp_{s+1} \cap \ldots \cap
\wp_m) \] and let $\wp$ be a minimal prime of $\qq$ such that $\dim
R / \wp= \dim R/\qq$. Lemma \ref{milhouse} (i) says that $\dim R/\wp \geq
r$. Moreover, since there exist $i \in \{ 1, \ldots, s \}$ and
$j \in \{ s+1, \ldots, m \}$ such that $\wp_i \subseteq \wp$ e
$\wp_j \subseteq \wp$, we have
\begin{eqnarray} 
\dim R/(\mathfrak{a} + \wp) \leq \dim R/(\mathfrak{a} + \wp_i) \leq d & \mbox{and} \nonumber \\
\dim R/(\mathfrak{b} + \wp) \leq \dim R/(\mathfrak{b} + \wp_j) \leq d \nonumber
\end{eqnarray}
Therefore we deduce that $\dim R/ ((\mathfrak{a} \cap \mathfrak{b}) + \wp) \leq d$.
But $R/\wp$ is catenary, (see Matsumura \cite[Theorem
29.4 (ii)]{matsu}), then
\[ \height(((\mathfrak{a} \cap \mathfrak{b}) + \wp)/\wp) = \dim R/ \wp - \dim R/ ((\mathfrak{a} \cap \mathfrak{b}) + \wp). \]
and hence $\height(((\mathfrak{a} \cap \mathfrak{b}) + \wp)/\wp) \geq r - d$.
So Lemma \ref{basetheorems} (ii) and Theorem \ref{nonvanishinggro}
\[ \cd(R,\mathfrak{a} \cap
\mathfrak{b})\geq \cd(R/\wp,((\mathfrak{a} \cap \mathfrak{b}) +
\wp)/\wp) \geq r-d.
\]
\end{proof}

The hypothesis $r<\dim R$ in Proposition \ref{lisa} is crucial, as we are going to show in the following example.
 
\begin{example}
Let $R=\kkk[[x,y]]$ the ring of formal power series over a field $\kkk$, $\aa = (x)$ and $\bb = (y)$. Being a domain, $R$ is $2$-connected by Remark \ref{bprc} ($\dim R = 2$). Moreover $\cd(R,\aa \cap \bb)=\cd(R,(xy))\leq 1$ by \eqref{aragcd}. Finally $\dim R/(\aa +\bb) = 0$, therefore Proposition \ref{lisa} does not hold if we chose $r=2$. However in this case, and in general when $R$ is irreducible, we can use Proposition \ref{lisa} choosing $r=\dim R -1$.
\end{example}

Eventually, we are ready to prove the main theorem of this section.

\begin{thm}\label{bart}
Let $R$ be an $r$-connected complete local ring with $r<\dim R$. Given an ideal $\aa\subseteq R$, we have
\[\cd(R,\aa)\leq r-s \ \implies \ R/\aa \mbox{ is $s$-connected}.\]
\end{thm}
\begin{proof}
Let $\wp_1, \ldots , \wp_m$ be the minimal primes of
$\mathfrak{a}$. If $m=1$, then $R/\aa$ is $(\dim R/\aa)$-connected. Let $\wp$ be a minimal prime of
$R$ such that $\wp \subseteq \wp_1$. Using Lemma \ref{basetheorems} (ii) and Theorem \ref{vanishinggro} we have $\cd(R,\aa) = \cd(R, \wp_1) \geq \cd(R/\wp, \wp_1/\wp) \geq
\height(\wp_1/\wp) $. So, since $R/ \wp$ is catenary, we have
\[ \dim R/\aa = \dim R/\wp_1 = \dim R/\wp - \height(\wp_1/\wp) \geq \dim R/\wp - \cd(R , \mathfrak{a}). \]
By Lemma \ref{milhouse} (ii) we get the thesis.
If $m > 1$, let $A$ and $B$ be a pair of disjoint subsets of $\{1,\ldots ,m\}$ such that $A\cup B=\{1,\ldots ,m\}$ and, setting
\[ c:= \dim \left( \frac{R}{ (\cap_{i \in A}\wp_i) + (\cap_{j \in B} \wp_j)} \right), \]
$R/\aa$ is $c$-connected (the existence of $A$ and $B$ is ensured by Lemma \ref{milhouse} (i)). Set $\mathcal{J}:= \cap_{i \in A} \wp_i$ and $\mathcal{K}:=
\cap_{j \in B} \wp_j$. Since $\dim R/\mathcal{J}>c$ and $\dim
R/\mathcal{K}>c$, Proposition \ref{lisa} implies
\[ c = \dim R/(\mathcal{J} + \mathcal{K}) \geq r - \cd(R, \mathcal{J} \cap \mathcal{K}). \]
Since $\sqrt{\mathfrak{a}} = \mathcal{J} \cap \mathcal{K}$, the
theorem is proved.
\end{proof}

As we have already mentioned Theorem \ref{bart} unifies many previous results concerning relationships between connectedness properties and the cohomological dimension. For instance, since $\cd(R,\aa)$ bounds from below the number of generators of $\aa$ by \eqref{aragcd}, we immediately get a theorem obtained in \cite[Expos\'{e} XIII, Th\'{e}or\`{e}me 2.1]{SGA2} (see also \cite[Theorem 19.2.11]{BS}).

\begin{thm}\label{palladineve}(Grothendieck) Let $(R,\mathfrak{m})$ be a complete local $r$-connected ring with $r<\dim R$. If an ideal $\aa\subseteq R$ is generated by $r-s$ elements, then $R/\aa$ is $s$-connected.
\end{thm}

Theorem \ref{bart} implies also \cite[Theorem 3.3]{ho-hu}, which is in turn a generalization of a result of \cite{faltings}. See also Schenzel \cite[Corollary 5.10]{schenzel}.

\begin{thm}\label{frink}(Hochster-Huneke) Let $(R, \mathfrak{m})$ be a
complete equidimensional local ring of dimension $d$ such that
$H_{\mathfrak{m}}^d(R)$ is an indecomposable $R$-module. If $\cd(R, \mathfrak{a}) \leq d-2$, then the punctured
spectrum $\Spec(R/\mathfrak{a}) \setminus \{ \mathfrak{m} \}$ is connected.
\end{thm}
\begin{proof}
Since $H_{\mm}^d(R)$ is indecomposable, \cite[Theorem 3.6]{ho-hu} implies that $R$ is connected in codimension $1$. Because $\Spec(R/\aa)\setminus \{\mm\}$ is connected if and only if $R/\aa$ is $1$-connected, the statement follows at once from Theorem \ref{bart}. 
\end{proof}

\begin{remark}\label{patty}
In \cite[Theorem 3.4]{D-N-T}  the authors claim  that Theorem
\ref{frink} holds  without the assumption  that
$H_{\mathfrak{m}}^d(R)$ is indecomposable. This is not correct, as we are going to show:
Set $R:=\kkk[[x,y,u,v]]/(xu,xv,yu,yv)$ and let $\mathfrak{a}$ be the zero
ideal of $R$. The minimal prime ideals of $R$ are
$(\overline{x},\overline{y})$ and $(\overline{u},\overline{v})$,
so $R$ is a complete equidimensional local ring of dimension $d=2$, and $\cd(R,\aa)=0=d-2$.
From Lemma \ref{lemmaconn} (i) one can see that $R/\aa = R$ is not $1$-connected. Therefore the punctured spectrum of $R/\aa$ is not connected.
\end{remark}

%
%\begin{remark}\label{spiderpork}
%
%If $(R,\mathfrak{m})$ is complete and local and $M$ is a finitely
%generated $R$-module, we have
%\[ \c(M/\mathfrak{a}M) \geq \min \{ \c(M), \sdim M - 1 \} - \cd(M , \mathfrak{a}), \]
%by the following argument: we can consider the complete local ring $S:= R/ (0 :_R
%M)$. Then we easily have
%$\c(M/\mathfrak{a}M)=\c(S/\mathfrak{a}S)$, $\c(M) = \c(S)$ and
%$\sdim M = \sdim S$. Moreover, by \cite[Theorem 2.2]{D-N-T}, or
%\cite[Lemma 2.1]{schenzel}, we have
%\[ \cd(M , \mathfrak{a})=\cd(S , \mathfrak{a})=\cd(S ,
%\mathfrak{a}S)\] and thus the result follows from applying Theorem \ref{bart}
%to $S$ and $\mathfrak{a}S$.
%
%\end{remark}
%
%By Remark \ref{spiderpork} and the part (b) of Lemma
%\ref{milhouse} we obtain the following corollary.
%
%\begin{corollary} \label{burns}
%
%Let $(R,\mathfrak{m})$ be complete and local, and $M$ a finitely
%generated $R$-module. Then
%\[ \c(M/\mathfrak{a}M) \geq  \c(M) - \cd(M , \mathfrak{a}) - 1. \]
%Moreover, if $M$ has more than one minimal prime ideal, the
%inequality is strict.
%
%\end{corollary}

\subsection{Noncomplete case}
\label{subchap1.2}

So far, we have obtained a certain understanding of the problem assuming the ambient ring $R$ to be complete. Of course, in the results of the previous subsection we could have avoided the completeness assumption. The inconvenient would have been that we should have complicated the statements asking for properties of the rings after completing, e.g.: {\it Let $R$ be a local ring such that $\widehat{R}$ is $r$-connected.} The aim of the present subsection is to introduce rings for which, even if noncomplete, the results of the previous section still hold true, without using ``completing-hypothesis". We start with a lemma.

\begin{lemma}\label{willy}
Let $(R, \mathfrak{m})$ be a $d$-dimensional local analytically
irreducible ring (i.e. $\widehat{R}$ is irreducible). Given an ideal $\aa\subseteq R$, we have
\[\cd(R,\aa)\leq d-s-1 \ \implies \ R/\aa \mbox{ is $s$-connected}.\]
\end{lemma}
\begin{proof}
Since $\widehat{R}$ is irreducible, it is $d$-connected. In particular it is $(d-1)$-connected. Moreover, $\cd(R,\aa)=\cd(\widehat{R},\aa \widehat{R})$ by \eqref{cohomcompl1}. Thus Theorem \ref{bart} yields that $\widehat{R}/\aa \widehat{R}$ is $s$-connected. Eventually, Lemma \ref{boe} (i) implies that $R/\aa$ is $s$-connected as well, and we conclude.
\end{proof}

The following proposition gives a class of rings for which we do not have to care about their properties after completion. Let us remark that such a class comprehends Cohen-Macaulay local rings.

\begin{prop}\label{serre}
Let $R$ be a $d$-dimensional local ring satisfying Serre's condition $S_2$, which is a quotient of a Cohen-Macaulay local ring. Given an ideal $\aa\subseteq R$, we have
\[\cd(R,\aa)\leq d-s-1 \ \implies \ R/\aa \mbox{ is $s$-connected}.\]
\end{prop}
\begin{proof}
The completion of $R$ satisfies $S_2$ as well as $R$ (see \cite[Exercise 23.2]{matsu}). Then $\widehat{R}$ is connected in codimension 1 by Corollary \ref{serre2}. Arguing as in the proof of Lemma \ref{willy}, we conclude.
\end{proof}
%
%\begin{prop}\label{secco}
%Let $R$ be a normal $\kk$-algebra finitely generated of dimension $d$ and $\aa$ be an ideal of $R$. Set
%\[s = d - \cd(R,\aa).\] 
%If $R/\aa$ is connected, then it is $(s-1)$-connected.
%\end{prop}
%\begin{proof}
%Let $\wp_1, \ldots, \wp_m$ be the minimal prime ideals of
%$R/\mathfrak{a}$, and let $A$ and $B$ be a partition of $\{1,\ldots ,m\}$ such that $R/\aa$ is $r$-connected but not $(r+1)$-connected, where
%\[r = \dim (\mathcal{V}((\cap_{i \in A}\wp_i) + (
%\cap_{j \in B}\wp_j))),\] 
%see Lemma \ref{milhouse} (i).
%Since $r \geq 0$ and since $\dim (\emptyset)=
%-1$, there exists a maximal ideal $\mathfrak{m}$ of $R$ which
%contains $\mathfrak{b}= (\cap_{i \in A}\wp_i) + ( \cap_{j \in
%B}\wp_j)$. Hence, again by Lemma \ref{milhouse} (i) we
%have that $R_{\mathfrak{m}}/\mathfrak{a}R_{\mathfrak{m}}$ cannot be $(r+1)$-connected.
%However $R$ is a local analytically irreducible ring by a result of Nagata \cite[Theorem 37.8]{nagata}. So $R_{\mm}/\aa R_{\mm}$ is $(d - \cd(R_{\mm},\aa R_{\mm})-1)$-connected by Proposition \ref{willy}.
%However $\cd(R_{\mathfrak{m}}, \mathfrak{a}R_{\mathfrak{m}}) \leq \cd(R, \mathfrak{a})$ by \eqref{cohomloc}, thus $R_{\mm}/\aa R_{\mm}$ is $(s-1)$-connected. By what said above $R/\aa$ is $(s-1)$-connected.
%\end{proof}

Now we prove a version of Theorem \ref{bart} in the case when $R$ is a graded algebra over a field. Such a version will be useful also in the next chapter.

\begin{thm}\label{winchester}
Let $R$ be a graded $r$-connected ring such that $R_0$ is a field, with $r<\dim R$. Given a homogeneous ideal $\aa\subseteq R$, we have
\[\cd(R,\aa)\leq r-s \ \implies \ R/\aa \mbox{ is $s$-connected}.\]
\end{thm}
\begin{proof}
Let $\mathfrak{m}$ be the irrelevant maximal ideal of $R$. Using Lemma \ref{milhouse} (i), since the minimal prime ideals of $R$ are in $\mm$, we have that $R_{\mm}$ is $r$-connected. Furthermore, let us notice that $\widehat{R_{\mm}}\cong \widehat{R^{\mm}}$.
Since the minimal prime ideals of $R_{\mm}$ come from homogeneous prime ideals of $R$, by Lemma \ref{krusty} their extension
$\wp \widehat{R_{\mathfrak{m}}}$ belongs to
$\Spec(\widehat{R_{\mathfrak{m}}})$. Then Lemma \ref{boe} (ii) yields that
$\widehat{R_{\mm}}$ is $r$-connected.
Notice that, by \eqref{cohomcompl1} and \eqref{cohomloc} we have
\[\cd(\widehat{R_{\mm}},\aa \widehat{R_{\mm}})\leq \cd(R,\aa),\]
so $\widehat{R_{\mm}}/\aa \widehat{R_{\mm}}$ is $s$-connected by Theorem \ref{bart}. Lemma \ref{boe} (i) implies that $R_{\mm}/\aa R_{\mm}$ is $s$-connected as well. Eventually we conclude by Lemma \ref{milhouse} (i): In fact the minimal primes of $\aa$, being homogeneous, are in $\mm$, so $R/\aa$ is $s$-connected as well as $R_{\mm}/\aa R_{\mm}$.
\end{proof}

%\begin{remark}\label{nelson}
%
%Proceeding in a similar way as in Remark \ref{spiderpork} we can
%deduce from Theorem \ref{winchester} the following more general
%fact.
%
%Let $k$ be a field, $R$ a $k$-algebra finitely generated
%positively graded on $\mathbb{Z}$ and $M$ a $\mathbb{Z}$-graded
%finitely generated $R$-module; then, if $\mathfrak{a}$ is graded,
%\[ \c(M/\mathfrak{a}M) \geq \c(M) - \cd(M, \mathfrak{a}) - 1. \]
%Moreover, if $M$ has more than one minimal prime ideal, the
%inequality is strict.
%
%To prove this we only have to note that $0:_R M \subseteq R$ is a
%graded ideal (\cite[Lemma 1.5.6]{B-H}).
%
%\end{remark}
%
%Remark \ref{nelson} implies easily the following corollary.
%
%\begin{corollary}\label{marge}
%Let $R$ be a graded ring such that $R_0$ is a field and $\Proj (R)$ is $r$-connected. Let
%$\mathfrak{a}$ be a graded ideal of $R$, and set
%\[s=r-\cd(R,\aa).\]
%Then $\V_+(\aa)$ is $(s-1)$-connected. Moreover, if $\Proj(R)$ is reducible, then $\V_+(\aa)$ is $s$-connected.
%\end{corollary}

\subsection{Cohomological dimension and connectedness of open subschemes of projective schemes}
\label{subchap1.3}
\index{affine schemes|(}
In this subsection we give a geometric interpretation of the results
obtained in Subsection \ref{subchap1.2}, using \eqref{cdcd2}.
More precisely, given a projective scheme $X$ over a field $\kkk$  and an open subscheme $U\subseteq X$, our purpose is to
find necessary conditions for which the cohomological dimension of
$U$ is less than a given integer.
By a well known result of Serre, there is a characterization of
noetherian affine schemes in terms of the cohomological dimension, namely:
A noetherian scheme $X$ is affine if and only if $\cd(X)=0$ (see
Hartshorne \cite[Theorem 3.7]{hart}). Hence, as a
particular case, in this subsection we  give necessary conditions
for the affineness of an open subscheme of a projective scheme
over $\kkk$. This is an interesting theme in algebraic geometry, and
it was studied from several mathematicians (see for example
Goodman \cite{goodman} or Brenner
\cite{brenner}). For instance, it is well known that if $X$ is a noetherian
separated scheme and $U \subseteq X$ is an affine open subscheme, then every irreducible component of $Z=X\setminus U$ has
codimension less than or equal to 1. Here it follows a quick proof.

\begin{prop}\label{otto}
Let $X$ a noetherian separated scheme, $U \subseteq X$ an affine
open subscheme and $Z = X \setminus U$. Then every irreducible
component of $Z$ has codimension less than or equal to 1.
\end{prop}
\begin{proof}
Let $Z_1$ be an irreducible component of $Z$ and let $V$ be an open
affine such that $V \cap Z_1 \neq \emptyset$. Since $X$ is
separated, $U \cap V$ is also affine, and since codim$(Z_1, X)=$
codim$(Z_1 \cap V, V)$, we can suppose that $X$ is a notherian
affine scheme. So let $R$ and $\aa \subseteq R$ be such that $X = \Spec(R)$ and $Z =
\mathcal{V}(\mathfrak{a})$. By the affineness criterion of Serre,
we have $H^i(U, \mathcal{O}_X) =0$ for all $i
> 0$. Then \eqref{cdcd1} implies $\cd(R, \mathfrak{a}) \leq 1$. This implies
$\height (\wp) \leq 1$ for all $\wp$
minimal prime of $\mathfrak{a}$ (Theorem \ref{nonvanishinggro}).
\end{proof}

In light of this result it is natural to ask: {\it What can we say
about the codimension of the intersection of the various components of $Z$?} To answer  this question we investigate on the connectedness of $Z$.
%\end{itemize}

%Now we prove a general well known result.

%\begin{thm}

%Let $(R,\mathfrak{m})$ be complete and local, $X= \Spec(R)$, $U
%\subseteq X$ an open subscheme and $Z=X \setminus U$. Then
%\begin{itemize}
%\item[(a)] if $\mathcal{F}$ is a coherent sheaf on $X$ and $H^i(U
%, \mathcal{F})=0$ for all $i>k$, then $\c(\Supp(\mathcal{F}) \cap
%Z) \geq \c(\Supp(\mathcal{F}))-k-1$. Moreover, if
%$\Supp(\mathcal{F})$ is reducible, the inequality is strict;
%\item[(b)] if $\cd(U) \leq k-1$, then for all $\mathcal{F}$
%coherent sheaf on $X$, $\c(\Supp(\mathcal{F}) \cap Z) \geq
%\c(\Supp(\mathcal{F}))-k-1$, where the inequality is strict if
%$\Supp(\mathcal{F})$ is reducible. In particular, if $\dim(X) \geq
%k+1$ and connected in codimension 1 (for example, $X$ irreducible
%or $R$ Cohen-Macaulay (see Proposition \ref{barney})), then $Z$ is
%connected; \item[(c)] if $U$ is affine, then $\c(Z) \geq \c(X)
%-2$, where the inequality is strict if $X$ is reducible. In
%particular, if $\dim(X) \geq 2$ and $X$ is connected in
%codimension 1, then $Z$ is connected.
%\end{itemize}

%\end{thm}

%\begin{proof}
%(a). Let $M$ be the finite generated $R$-module $\Gamma(X,
%\mathcal{F})$ and $\mathfrak{a}$ the ideal which determines $Z$.
%Then $\widetilde{M}= \mathcal{F}$, hence by (8) we have
%$H_{\mathfrak{a}}=0$ for all $i>k+1$. Hence $\cd(M, \mathfrak{a})
%\leq k$, and since $\Supp(\mathcal{F})= \Supp(M)$, from Corollary
%\ref{burns} follows that
%\[ \c(\Supp(\mathcal{F}) \cap
%Z) \geq \c(\Supp(\mathcal{F}))-k-1. \]

%Moreover, if $\Supp(\mathcal{F})$ is reducible, the inequality is
%strict.
%\end{proof}

\begin{thm}\label{kaprapal}
Let $X$ be an $r$-connected projective scheme over a field $\kkk$ with $r<\dim X$, $U \subseteq X$
an open subscheme and $Z=X \setminus U$. Then
\[\cd(U)\leq r-s-1 \ \implies \ Z \mbox{ is $s$-connected}.\]
\end{thm}
\begin{proof}
Let $X= \Proj(R)$, with $R$ a graded finitely generated
$\kkk$-algebra, and let $\mathfrak{a}\subseteq R$ be a graded ideal
defining $Z$. Then, from \eqref{cdcd2} it follows that $\cd(R, \mathfrak{a}) \leq r-s$. By Example \ref{punctureddu} (ii) we have that $R$ is $(r+1)$-connected. Therefore, from Theorem \ref{winchester} we get that $R/\aa$ is $(s+1)$-connected. So $Z=\V_+(\aa)$ is $s$-connected, using Example \eqref{punctureddu} (ii) once again.
\end{proof}

From Theorem \ref{kaprapal} we can immediately obtain the
following corollary.

\begin{corollary}\label{abrham}
Let $X$ be an $r$-connected projective scheme over a field $\kkk$ with $r<\dim X$, $U \subseteq X$
an open subscheme and $Z=X \setminus U$. If $U$ is affine, then $Z$ is $(r-1)$-connected. In particular, if $X$ is connected in codimension 1 and $\codim_X Z=1$, then $Z$ is connected in
codimension 1.
\end{corollary}
\begin{proof}
By the affineness criterion of Serre $\cd(U)=0$, so we conclude by
Theorem \ref{kaprapal}.
\end{proof}

\begin{example} Let $R:=\kkk[X_0,\ldots ,X_4]/(X_0X_1X_4-X_2X_3X_4)$ and denote by $x_i$ the residue class of $X_i$ modulo $R$. Consider the ideal  
$\mathfrak{a}=(x_0x_1,x_0x_3,x_1x_2,x_2x_3)$
$= \wp_1 \cap \wp_2 \subseteq R$, where  $\wp_1=(x_0,x_2)$
and $\wp_2= (x_1,x_3)$. Clearly $\wp_1$ and $\wp_2$ are prime ideals of height $1$, therefore $U=\Proj(R)\setminus \V_+(\aa)$ might be affine. Since $\Proj(R)$ is a complete intersection of $\mathbb{P}^4$, it is connected in codimension 1 (Proposition \ref{barney}). But $\dim \V_+(\wp_1+ \wp_2)=0$, therefore $\V_+(\aa)$ is not $1$-connected by Lemma \ref{milhouse} (i). Thus, by
Corollary \ref{abrham} we conclude that $U$ is not affine.
\end{example}
\index{affine schemes|)}
%\begin{remark}

%Using (8), in a similar way as for Theorem \ref{kaprapal}, we
%can translate in to the language of algebraic geometry Theorem
%\ref{bart}, Theorem \ref{winchester} and Remark \ref{secco}. In
%fact it suffices to work with affine schemes $\Spec(R)$ where $R$
%is a ring as in the above results.

%For example, let $X$ be an $r$-dimensional normal affine scheme of
%finite type over a field $k$ and let $U \subseteq X$ be an open
%subscheme such that $\cd(U) \leq s$. Then,  using point (2) of
%Remark \ref{secco}, if $X \setminus U$ is connected, then
%\[ \c(X \setminus U) \geq r - s - 2. \]

%\end{remark}

\subsection{Discussion about the sufficiency for the vanishing of local cohomology}\label{discussionconn}

Theorem \ref{bart} gives a necessary condition for the cohomological dimension to be smaller than a given integer. In order to make the below discussion easier, suppose that the local complete ring $(R,\mm)$ is an $n$-dimensional domain. In this case Theorem \ref{bart} says the following: {\it If $\aa$ is an ideal of $R$ such that $\cd(R,\aa) < c$, then $R/\aa$ is $(n-c)$-connected.} One question which might come in mind is if the condition $R/\aa$ is $s$-connected is also sufficient to $\cd(R,\aa)$ being smaller than $n-s$. If $s=1$, with some further assumptions on $R$, Theorem \ref{bart} can actually be reversed. This fact has been proven by various authors, with different assumptions on the ring $R$ and different tools, see \cite{Ha}, \cite{ogus}, \cite{PS} and \cite{siwa1}. The more general version of the result has been settled in \cite[Theorem 2.9]{hu-ly}. First recall that a local ring $(A,\mm)$ is $1$-connected if and only if its punctured spectrum $\Spec(A)\setminus \{\mm\}$ is connected (Example \ref{punctureddu} (i)).

\begin{thm}\label{huneke&lyubeznik}
(Huneke and Lyubeznik) Let $R$ be an $n$-dimensional local, complete, regular ring containing a separably closed field. For an ideal $\aa \subseteq R$,
\[\cd(R,\aa)\leq n-2 \ \ \ \iff \ \ \ \dim R/\aa \geq 2 \mbox{ \ and \ } R/\aa \mbox{ \ is $1$-connected}.\]
\end{thm}

The next step we might consider is the case in which $R/\aa$ is $2$-connected. From Theorem \ref{bart} we know the following implication:
\begin{equation}\label{bart3conn}
\cd(R,\aa)\leq n-3 \ \ \ \implies \ \ \ \dim R/\aa \geq 3 \mbox{ \ and \ } R/\aa \mbox{ \ is $2$-connected}.
\end{equation}
Unfortunately, this time, the converse does not hold, neither assuming $R$ to be regular. We show this in the following example.
 
\begin{example}\label{notlocconn}
Let $R:=\kkk[[x_1,x_2,x_3,x_4,x_5]]$ be the ($5$-dimensional) ring of formal power series in $5$ variables over a field $\kkk$. Let us consider the ideal of $R$
\[\aa := (x_1x_2x_4,x_1x_3x_4,x_1x_3x_5,x_2x_3x_5,x_2x_4)=(x_1,x_2)\cap (x_2,x_3)\cap (x_3,x_4)\cap (x_4,x_5).\]
One can see, for instance using Lemma \ref{hilbert}, that $A=R/\aa$ is $2$-connected. However we are going to show that $\cd(R,\aa)\geq 3=5-2$. To this aim, let us localize $A$ at the prime ideal $\wp = (x_1,x_2,x_4,x_5)$. Since $x_3$ is invertible in $R_{\wp}$, we have that
\[A_{\wp} \cong R_{\wp}/(x_1x_4,x_1x_5,x_2x_4,x_2x_5).\]
Moreover, $\bb =\aa R_{\wp}= (x_1x_4,x_1x_5,x_2x_4,x_2x_5)=(x_1,x_2)\cap (x_3,x_4)$. By Lemma \ref{lemmaconn} it follows that $A_{\wp}$ is not $1$-connected. Moreover notice that $R_{\wp}$ is analytically irreducible, since $\widehat{R_{\wp}}$ is a $4$-dimensional regular domain as well as $R_{\wp}$. Therefore, Proposition \ref{willy} implies that $\cd(R_{\wp},\bb)\geq 3$. So \eqref{cohomloc} implies that $\cd(R,\aa)\geq 3$.
\end{example}

In Example \ref{notlocconn}, we exploited the fact that the connectedness properties of a ring, in general, do not pass to its localizations. More precisely, even if a ring $A$ is $r$-connected, in general it might happen that $A_{\wp}$ is not $s$-connected for some prime ideal $\wp$ of $A$, where $s=\height(\wp)+r-\dim A$ is the expected number. This fact was the peculiarity of Example \ref{notlocconn} which allowed us to use Theorem \ref{bart}. So it is understandable if the reader is not appeased from the above example, especially because there are rings in which such a pathology cannot happen. For instance if the ring $A$ is a domain, and therefore $\dim A$-connected, then $A_{\wp}$ is a domain for any prime $\wp$, and therefore $\height{\wp}$-connected. So, we decided to give a counterexample to the converse of \eqref{bart3conn} also in this situation, i.e. when $\aa$ is a prime ideal.

\begin{example}\label{notformlocconn}
Consider the polynomial ring $R''':=\CC[x_1,x_2,x_3,x_4]$ and the $\CC$-algebra homomorphism
\begin{displaymath}
\begin{array}{rll}
\phi : R''' & \rightarrow & \CC[t,u] \\
x_1 & \mapsto & t \\
x_2 & \mapsto & tu \\ 
x_3 & \mapsto & u(u-1) \\
x_4 & \mapsto & u^2(u-1)
\end{array}
\end{displaymath}
The image of $\phi$ is the coordinate ring of an irreducible surface in $\AA^4$, and it has been considered by Hartshorne in  \cite[Example 3.4.2]{hartshorne}. It is not difficult to see that
\[\aa''' := \Ker(\phi)=(x_1x_4-x_2x_3,x_1^2x_3+x_1x_2-x_2^2,x_3^3+x_3x_4-x_4^2).\]
Obviously $\aa'''$ is a prime ideal of $R'''$. Let us denote by $R''$ the ring of formal power series $\CC[[x_1,x_2,x_3,x_4]]$. Doing some elementary calculations, one can show that there exist $g_1,h_1,g_2,h_2 \in R''$ such that $(g_1,h_1,g_2,h_2)=(x_1,x_2,x_3,x_4)$ (we are considering ideals in $R''$) and
\[\sqrt{\aa''}=\sqrt{(g_1,h_1)}\cap \sqrt{(g_2,h_2)},\]
where $\aa'' := \aa'''R''$ (see \cite[Example 4.3.7]{BS}). Using Lemma \ref{lemmaconn}, we infer that $R''/\aa''$ is not $1$-connected. 

Now let us homogenize $\aa'''$ to $\aa'$, i.e. set
\[\aa' = (x_1x_4-x_2x_3,x_1^2x_3+x_0x_1x_2-x_0x_2^2,x_3^3+x_0x_3x_4-x_0x_4^2)\subseteq R'=R'''[x_0].\]
The ideal $\aa'$ is prime as well as $\aa'''$ (for instance see the book of Bruns and Vetter \cite[Proposition 16.26 (c)]{BrVe}). Furthermore $\dim R'/\aa' = \dim R'''/\aa'''+1=3$. Eventually, set $R:=\CC[[x_0,\ldots ,x_5]]$ and $\aa := \aa'R$. The ideal $\aa$ is prime by Lemma \ref{krusty}. Moreover, $\dim R/\aa = \dim R'/\aa' = 3$. Thus $R/\aa$ is $3$-connected. In particular it is $2$-connected! Let us consider the prime ideal $\wp = (x_1,x_2,x_3,x_4) \subseteq R$. Since $R_{\wp}$ is a $4$-dimensional regular local ring containing $\CC$, then Cohen's structure theorem implies $\widehat{R_{\wp}}\cong R''$. Furthermore, since $x_0$ is invertible in $R_{\wp}$, the ideal $\aa R_{\wp}\subseteq R_{\wp}$ is generated by $x_1x_4-x_2x_3$, \ $x_1^2x_3+x_1x_2-x_2^2$ and $x_3^3+x_3x_4-x_4^2$. Therefore, the ideal $\aa$ read in $R''$ is the same ideal of $\aa''$. So, $\widehat{R_{\wp}}/\aa \widehat{R_{\wp}}$ is not $1$-connected as $R''/\aa''$ is not. 
At this point, Theorem \ref{bart} implies 
\[\cd(\widehat{R_{\wp}},\aa \widehat{R_{\wp}})\geq 3.\] 
So, using \eqref{cohomcompl1} and \eqref{cohomloc}, we have
\[\cd(R,\aa)\geq 3 >5-3.\]
This shows that the converse of \eqref{bart3conn} does not hold, neither if we assume that $R$ is the ring of formal power series over the complex numbers and that $\aa\subseteq R$ is a prime ideal.
\end{example}

As the reader can notice, also in the above example we used Theorem \ref{bart}. In fact, with that notation, the ideal $\aa\subseteq R$, even if prime, was still a bit pathological. That is, we found a prime ideal $\wp$ of $A=R/\aa$ such that $\widehat{A_{\wp}}$ was not $\height{\wp}$-connected, at the contrary of $A_{\wp}$. Once again, the reader might demand an example in which all the completion of the localizations of $A$ are domain. To this aim, for instance, it would be enough to consider an ideal $\aa\subseteq R$ such that $A=R/\aa$ has an isolated singularity, i.e such that $A_{\wp}$ is a regular ring for any non maximal prime ideal $\wp$. In fact the completion of a regular ring is regular, thus a domain. Also in this situation, counterexamples to \eqref{bart3conn} still exist. However to construct them are needed techniques different from those developed in this section. How to establish these examples will be clear from the results of the next section.
\index{connected|)}\index{connected!r-@$r$-|)}

\section{Sufficient conditions for the vanishing of local cohomology}

In this section, we will focus on giving sufficient conditions for $\cd(R,\aa)\leq c$. However the conditions we find will often be also necessary. To this aim we have to move to a regular ambient ring. Actually, we will even suppose that the ring $R$ is a polynomial ring over a field of characteristic $0$. Moreover the ideal $\aa$ will be homogeneous and such that $R/\aa$ has an isolated singularity, which is a case of great interest by the discussion of Subsection \ref{discussionconn}. Let us recall that the assumption ``$R/\aa$ has an isolated singularity" is equivalent to ``$X=\Proj(R/\aa)$ is a regular scheme", i.e. the local rings $\O_{X,P}$ are regular for all points $P$ of $X$. If the field is algebraically closed, this is in turn equivalent to $X$ being nonsingular.

\subsection{Open subschemes of the projective spaces over a field of characteristic 0}\label{s1}

In this subsection we prove the key-theorem of the section (Theorem \ref{1}), which will let us achieve considerable consequences. We start with a remark which allows us to use complex coefficients in many situations.  

\begin{remark}\label{0}
Let $A$ be a ring, $B$ a flat $A$-algebra, $R$ an $A$-algebra and $M$ an $R$-module. Set $R_B:=R \otimes_A B$ and $M_B := M\otimes_A B$. Clearly $R_B$ is a flat $R$-algebra. Therefore Theorem \ref{basetheorems} (i) implies that, for any ideal $\aa \subseteq R$ every $j \in \NN$:
\begin{equation}\label{redtocomplex0}
H_{\aa}^j(M)\otimes_A B \cong H_{\aa}^j(M)\otimes_R R_B \cong H_{\aa R_B}^j(M\otimes_R  R_B)\cong H_{\aa R_B}^j(M_B)
\end{equation}

We want to use \eqref{redtocomplex0} in the following particular case. Let $S:=\KK[x_0, \ldots, x_n]$ be the polynomial ring in $n+1$ variables over a field $\KK$ of characteristic $0$, and $I \subseteq S$ an ideal. Since $I$ is finitely generated we can find a field $\kkk$ of characteristic $0$ such that, setting $S_{\kkk}:=\kkk[x_0, \ldots, x_n]$, the following properties hold:
\[ \kkk \subseteq \KK, \ \ \QQ \subseteq \kkk \subseteq \CC, \ \ (I \cap S_{\kkk})S=I \] 
(to this aim we have simply to add to $\QQ$ the coefficients of a set of generators of $I$).
Since $\KK$ and $\CC$ are faithfully flat $\kkk$-algebras, \eqref{redtocomplex0} implies that
\begin{equation}\label{redtocomplex1}
\cd(S,I)=\cd(S_{\kkk}, I \cap S_{\kkk})=\cd(S_{\CC}, (I\cap S_{\kkk})S_{\CC}),
\end{equation}
where $S_{\CC}:=\CC[x_1, \ldots, x_n]$.

In the above situation, assume that $I$ is homogeneous and that $X:=\Proj (S/I)$ is a regular scheme. Then set $X_{\kkk}:=\Proj (S_{\kkk} / (I\cap S_{\kkk}))$ and $X_{\CC}:=\Proj(S_{\CC}/((I\cap S_{\kkk})\CC))$. Notice that $X \cong X_{\kkk} \times _{\kkk} \Spec(\KK)$ and $X_{\CC}\cong X_{\kkk} \times _{\kkk} \Spec(\CC)$. Furthermore $X_{\kkk}$ and $X_{\CC}$ are regular schemes. Since $\KK$ and $\CC$ are both flat $\kkk$-algebras, we get, for all natural numbers $i,j$,
\[ H^i(X, \Omega_{X/\KK}^j) \cong H^i(X_{\kkk}, \Omega_{X_{\kkk} /\kkk}^j) \otimes_{\kkk}\KK \] 
and
\[ H^i(X_{\CC}, \Omega_{X_{\CC}/\CC}^j) \cong H^i(X_{\kkk}, \Omega_{X_{\kkk} /\kkk}^j) \otimes_{\kkk}\CC \]
(for example see the book of Liu \cite[Chapter 6, Proposition 1.24 (a) and Chapter 5, Proposition 2.27]{liu}).
Particularly we have 
\begin{equation}\label{redtocomplex2}  \dim_{\KK}(H^i(X, \Omega_{X/\KK}^j))=\dim_{\CC}(H^i(X_{\CC}, \Omega_{X_{\CC}/\CC}^j)). 
\end{equation}
We will call the dimension of the $\KK$-vector space $H^i(X, \Omega_{X/\KK}^j)$ the {\it (i,j)-Hodge number} of $X$,\index{Hodge numbers|(} denoting it by $h^{ij}(X)$.
\end{remark}

In the next remark, for the convenience of the reader, we collect some well known facts which we will use throughout the paper.

\begin{remark}\label{betti}
Let $X$ be a projective scheme over $\CC$: Let us recall that $X^h$ means $X$ regarded as an analytic space, as explained in Section \ref{GAGAsec}, and $\underline{\CC}$ denotes the constant sheaf associated to $\CC$. We will denote by $\beta_i(X)$ the {\it topological Betti number} \index{Betti numbers}\index{Betti numbers!topological} 
\[\beta_i(X):=\operatorname{rank}_{\ZZ}(H_i^{Sing}(X^h,\ZZ))=\operatorname{rank}_{\ZZ}(H_{Sing}^i(X^h,\ZZ))=\]
\[=\dim_{\CC}(H_{Sing}^i(X^h,\CC))=\dim_{\CC}(H^i(X^h,\underline{\CC})).\] 
The equalities above hold true from the universal coefficient Theorems \ref{unihom} and \ref{unicohom}, and by the isomorphism \ref{singsheaf}.
Pick another projective scheme over $\CC$, say $Y$, and denote by $Z$ the Segre product $X \times Y$. The K\"unneth formula for singular cohomology (for instance see Hatcher \cite[Theorem 3.16]{hatcher}) yields 
\[H_{Sing}^i(Z^h,\CC)\cong \bigoplus_{p+q=i}H_{Sing}^p(X^h,\CC)\otimes_{\CC}H_{Sing}^q(Y^h,\CC).\]
Thus at the level of topological Betti numbers we have
\begin{equation}\label{betti1}
\beta_i(Z)=\sum_{p+q=i}\beta_p(X)\beta_q(Y).
\end{equation}

Now assume that $X$ is a regular scheme projective over $\CC$. It is well known that $X^h$ is a K\"ahler manifold (see \cite[Appendix B.4]{hart}), so the Hodge decomposition (see the notes of Arapura \cite[Theorem 10.2.4]{arap}) is available. Therefore, using Theorem \ref{GAGA}, we have
\[H_{Sing}^i(X^h,\CC)\cong \bigoplus_{p+q=i}H^p(X^h,\Omega_{X^h}^q)\cong \bigoplus_{p+q=i}H^p(X,\Omega_{X/\CC}^q).\] 
Thus
\begin{equation}\label{betti2}
\beta_i(X)=\sum_{p+q=i}h^{pq}(X)
\end{equation}

Finally, note that the restriction map on singular cohomology
\begin{equation}\label{injective}
H_{Sing}^i(\PP(\CC^{n+1}),\CC)\longrightarrow H_{Sing}^i(X^h,\CC)
\end{equation}
is injective provided that $i=0, \ldots ,2 \dim X$ (see Shafarevich \cite[pp. 121-122]{sha}). So Example  \ref{bettiproj} implies that
\begin{equation}\label{betti3}
\beta_{2i}(X)\geq 1 \mbox{ \ \ provided that }i\leq \dim X
\end{equation}
\end{remark}

The contents of the next remark do not concern the cohomological dimension. We just want to let the reader notice a nice consequence of Remark \ref{betti} when combined with a result of Barth. 

\begin{remark}
Let $X$ and $Y$ be two positive dimensional regular scheme projective over $\CC$. Chose any embedding of $Z:=X \times Y$ in $\PP^N$. Then the dimension of the secant variety of $Z$ in $\PP^N$ is at least $2 \dim Z - 1$.

To show this notice that \eqref{betti3} yields $\beta_0(X) \geq 1$, \ $\beta_2(X) \geq 1$, \ $\beta_0(Y) \geq 1$ and $\beta_2(Y) \geq 1$.
Therefore using \eqref{betti1} we have 
\[ \beta_2(Z) \geq \beta_2(X)\beta_0(Y)+\beta_2(Y)\beta_0(X)\geq 2.\] 
By a theorem of Barth (see Lazarsfeld \cite[Theorem 3.2.1]{la}), it follows that $Z$ cannot be embedded in any $\PP^M$ with $M<2 \dim X -1$. If the dimension of the secant variety were less than $2 \dim X-1$, it would be possible to project down in a biregular way $X$ from $\PP^N$ in $\PP^{2 \dim X-2}$, and this would be a contradiction.

Note that the above lower bound is the best possible. For instance consider the classical Segre embedding $Z := \PP^r\times \PP^s\subseteq \PP^{rs+r+s}$. The defining ideal of $Z$ is generated by the 2-minors of a generic $(r+1)\times (s+1)$-matrix. The secant variety of $Z$ is well known to be generated by the $3$-minors of the same matrix, therefore its dimension is $2(r+s)-1$.
%in fact $\PP^r \times \PP^s$ can be embedded in $\PP^{2(r+s)-1}$ (see Hartshorne \cite[p. 1026]{hart1}).
\end{remark}

The following theorem provides some necessary and sufficient conditions for the cohomological dimension of the complement of a smooth variety in a projective space being smaller than a given integer. For the proof we recall the concept of {\it DeRham-depth} introduced in \cite[Definition 2.12]{ogus}. The general definition requires the notion of local algebraic DeRham cohomology (see Hartshorne \cite[Chapter III.1]{hartder}). However for the case we are interested in there is an equivalent definition in terms of local cohomology (\cite[proof of Theorem 4.1]{ogus}). Let $X\subseteq \PP^n$ be a projective scheme over a field $\kkk$ of characteristic $0$. Let $S:=\kkk[x_0,\ldots ,x_n]$ be the polynomial ring and $\aa \subseteq S$ be the defining ideal of $X$. Then the DeRham-depth of $X$ is equal to $s$ if and only if $\operatorname{Supp}(H_{\aa}^i(S))\subseteq \{\mm \}$ for all $i>n-s$ and $\operatorname{Supp}(H_{\aa}^{n-s}(S))\nsubseteq \{ \mm\}$, where $\mm$ is the maximal irrelevant ideal of $S$.

\begin{thm}\label{1}
Let $X\subseteq \mathbb{P}^n$ be a regular scheme projective over a field $\kkk$ of characteristic $0$. Moreover, let $r$ be an integer such that $\operatorname{codim}_{\PP^n}X\leq r \leq n$ and $U:=\mathbb{P}^n \setminus X$. Then $\cd(U) < r$ if and only if
\begin{displaymath}
h^{pq}(X) = \left\{ \begin{array}{cc} 0 & \mbox{if  } p \neq q, \ p+q < n - r\\
1 &  \mbox{if  } p=q, \ p+q <n-r \end{array}  \right.
\end{displaymath}
Moreover, if $\kkk=\CC$, the above conditions are equivalent to:
\begin{displaymath}
\beta_i(X) = \left\{ \begin{array}{cc} 1 & \mbox{if  $i< n-r$ and $i$ is even} \\
0 &  \mbox{if  $i< n-r$ and $i$ is odd} \end{array}  \right.
\end{displaymath}
\end{thm}
\begin{proof}
Set $X_{\CC}\subseteq \PP_{\CC}^n$ as in Remark \ref{0} and $U_{\CC}:=\PP_{\CC}^n\setminus X_{\CC}$. Then $\cd(U)=\cd(U_{\CC})$ by \eqref{redtocomplex1} and \eqref{cdcd2}. Moreover, notice that $h^{ij}(X)=h^{ij}(X_{\CC})$ by \eqref{redtocomplex2}. Finally, since $X_{\CC}$ is regular, we can reduce to the case in which $\kkk=\CC$. If $\cd(U)<r$ then $r\geq \operatorname{codim}_{\PP^n}X$ by Theorem \ref{nonvanishinggro} and \eqref{cdcd2}. Thus the ``only if"-part follows by \cite[Corollary 7.5, p. 148]{hartshorne4}.

So, it remains to prove the ``if"-part. By \eqref{groder2} algebraic De Rham cohomology agrees with singular cohomology. Therefore by \eqref{injective} the restriction maps
\begin{equation}\label{derham}
H_{DR}^i(\PP^n) \longrightarrow H_{DR}^i(X)
\end{equation}  
are injective for all $i \leq 2 \dim X$. By the assumptions, equation (\ref{betti2}) yields
$\beta_i(X)=1$ if $i$ is even and $i<n-r$, $0$ otherwise. Moreover by Example \ref{bettiproj} $\beta_i(\PP^n)=1$ if $i$ is even and $i\leq 2n$, $0$ otherwise. So the injections in (\ref{derham}) are actually isomorphisms for all $i < n-r$.

We want to use \cite[Theorem 4.4]{ogus}, and to this aim we will show that the De Rham-depth of $X$ is greater than or equal to $n-r$. This means that we have to show that $\operatorname{Supp}(H_{\aa}^i(S))\subseteq \mm$ for all $i > r$, where $S=\CC[x_0, \ldots ,x_n]$, $\aa \subseteq S$ is the ideal defining $X$ and $\mm$ is the maximal irrelevant ideal of $S$. But this is easy to see, because if $\wp$ is a graded prime ideal containing $\aa$ and different from $\mm$, being $X$ regular, $\aa S_{\wp}$ is a complete intersection in $S_{\wp}$: Therefore $(H_{\aa}^i(S))_{\wp}\cong H_{\aa S_{\wp}}^i(S_{\wp})=0$ for all $i > r \ (\geq \operatorname{ht(\aa S_{\wp})})$. Hence \cite[Theorem 4.4]{ogus} yields the conclusion.

Finally, if $\kkk=\CC$, the last condition in the statement is a consequence of the first one by \eqref{betti2}. Moreover, the restriction maps on singular cohomology
\[H_{Sing}^i(\PP(\CC^{n+1}),\CC) \longrightarrow H_{Sing}^i(X^h,\CC) \]
are compatible with the Hodge decomposition (see \cite[Corollary 11.2.5]{arap}). The Hodge numbers of the projective space are well known (for instance see \cite[Exercise 7.8, p. 150]{hartshorne4}), namely 
\begin{displaymath}
h^{pq}(\PP^n) = \left\{ \begin{array}{cc} 0 & \mbox{if  } p \neq q,\\
1 &  \mbox{if  } p=q, \ p\leq n \end{array}  \right.
\end{displaymath}
Since the restriction maps on singular cohomology are injective for $i<n-r$ by \eqref{injective}, the last condition in the statement implies the first one. The theorem is proved. 
\end{proof}

\begin{remark}\label{newinvariant}
Notice that Theorem \ref{1} implies the following surprising fact: {\it If $X\subseteq \PP^n$ is a regular scheme projective over a field of characteristic $0$, then the natural number $n-\cd(\PP^n \setminus X)$ is an invariant of $X$, i.e. it does not depend on the embedding.} Actually the same statement is true also in positive characteristic, using \cite[Theorems 5.1 and 5.2]{ly1}.
\end{remark}

Theorem \ref{1} is peculiar of the characteristic-zero-case. For instance pick an elliptic curve $E$ over a field of positive characteristic, whose Hasse invariant is $0$. This means that the Frobenius morphism acts as $0$ on $H^1(E,\O_E)$ (for a more explanatory definition see \cite[p. 332]{hart}). Such a curve exists in any positive characteristic by \cite[Corollary 4.22]{hart}. Then set $X:=E \times \PP^1 \subseteq \PP^5$ and $U:=\PP^5 \setminus X$. From the K\"unneth formula for coherent sheaves (see Sampson and Washnitzer \cite[Section 6, Theorem 1]{SaWa}), one can deduce that the Frobenius morphism acts as 0 on $H^1(X,\O_X)$ as well as on $H^1(E,\O_E)$. Therefore \cite[Theorem 2.5]{ha-sp} implies $\cd(U)=2$ (see also \cite[Theorem 5.2]{ly1}). However notice that, using again \cite[Section 6, Theorem 1]{SaWa}, we have $h^{10}(X)=h^{10}(E)=1$.

\vskip1mm

%\begin{remark}\label{bar}
%Theorem \ref{1} let us partially answer a question left open by Barile in her paper \cite[p. 81]{barile}. In that work she computed the arithmetical rank of the ideal $I_t(X)\subseteq k[X]$ generated by the $t$-minors of a symmetric matrix of indeterminates $X$, for all $t$. To bound from below the number of defining equations she used \'etale cohomology, and she found the same bounds, in characteristic $0$, also using local cohomology, but only for $t$ odd. The question about the cohomological dimension $\cd(k[X],I_t(X))$ was left open for $t$ even.\\
%Actually $I_2(X) \subseteq k[X]$ defines the Veronese variety $Y=v_2(\PP^n)\subseteq \PP^{\binom{n+2}{n}-1}$. Clearly $h^{ij}(Y)=h^{ij}(\PP^n)$, and $h^{ij}(\PP^n)=1$ if $i=j$ and $i+j \leq 2n$ and $0$ otherwise. So Theorem \ref{1} implies that $\cd(k[X],I_2(X))=\cd(\PP^{\binom{n+2}{n}-1}\setminus Y)+1=\binom{n+2}{n}-n$.\\
%The number of defining equation of a non-singular variety $X$ embedded in a fixed $\PP^n$, in spite of the cohomological dimension of $\PP^n \setminus X$, is not an invariant of $X$. For example consider $\PP^2$ embedded in $\PP^5$ as a linear subspace: clearly $\PP^2$ is defined (indeed also scheme-theoretically) by 3 equations. However to define set-theoretically the Veronese $v_2(\PP^2)\cong \PP^2$ in $\PP^5$ one needs 4 equations  provided that the characteristic of the field is different from 2 (see \cite[Theorem 3.1]{barile}).  
%\end{remark}
We want to state a consequence of Theorem \ref{1} which, roughly speaking, says that in characteristic $0$ the cohomological dimension of the complement of the Segre product of two projective schemes is always large. We will also show, in Remark \ref{gowa}, that in positive caharacteristic, at the contrary, such a cohomological dimension can be as small as possible. 

\begin{prop}\label{mah}
Let $X$ and $Y$ be two positive dimensional regular schemes projective over a field $\kkk$ of characteristic $0$. Choose an embedding of $Z:=X \times Y$ in some $\PP^N$ and set $U:=\PP^N \setminus Z$. Then $\cd(U) \geq N-3$. In particular if $\dim Z \geq 3$, $Z$ is not a set-theoretic complete intersection.
\end{prop}
\begin{proof}
We can reduce the problem to the case $\kkk=\CC$ by \eqref{redtocomplex1}. From \eqref{betti1} we have 
\begin{equation}\label{eqmah} 
\beta_2(Z) \geq \beta_2(X)\beta_0(Y)+\beta_2(Y)\beta_0(X).
\end{equation} 
Furthermore, \eqref{betti3} yields $\beta_0(X) \geq 1$, \ $\beta_2(X) \geq 1$, \ $\beta_0(Y) \geq 1$ and $\beta_2(Y) \geq 1$. If $\cd(U)<N-3$, then Theorem \ref{1} would imply that $\beta_2(Z)=1$. But, by what said above, one reads from \eqref{eqmah} that $\beta_2(Z)\geq 2$, so we get the thesis. 
\end{proof}

In the next remark we will show that the characteristic-$0$-assumption is crucial in Proposition \ref{mah}. In fact in positive characteristic can happen that, using the notation of Proposition \ref{mah}, $\cd(U)$ is as small as possible, i.e. $\codim_{\PP^N}(Z)-1$. 

\begin{remark}\label{gowa}
Let $\kkk$ be a field of positive characteristic. Let us consider two Cohen-Macaulay graded $\kkk$-algebras $A$ and $B$ of negative $a$-invariant (for instance two polynomial rings over $\kkk$). Set $R:=A\#B:=\oplus_{i\in \NN}A_i \otimes_{\kkk}B_i$ their Segre product (with the notation of the paper of Goto and Watanabe \cite{GW}). By \cite[Theorem 4.2.3]{GW}, $R$ is Cohen-Macaulay. So, presenting $R$ as a quotient of a polynomial ring of $N+1$ variables, say $R\cong P/I$, a theorem in \cite{PS} (see Theorem \ref{prelimpeskineszpiro}) implies that $\cd(P,I)=N+1- \dim R$. Translating in the language of Proposition \ref{mah} we have $X:=\Proj(A)$, $Y:=\Proj(B)$, $Z:=\Proj(R) \subseteq \PP^N=\Proj(P)$ and $\cd(\PP^N \setminus Z)=\cd(P,I)-1=N - \dim Z -1=\codim_{\PP^N}(Z)-1$.
\end{remark}

\index{Hodge numbers|)}

\subsection{\'Etale cohomological dimension}\label{1.2}\index{etale cohomological dimension@\'etale cohomological dimension|(}

If the characteristic of the base field is $0$ we have seen in the previous subsection that we can, usually, reduce the problem to $\kkk=\CC$. In this context, thanks to the work of Serre described in Section \ref{GAGAsec}, is available the complex topology, so we can use methods from algebraic topology and from complex analysis.

Unfortunately, when the characteristic of $\kkk$ is positive, the above techniques are not available. Moreover some of the results obtained in Subsection \ref{s1} are not true in positive characteristic, as we have shown just below Theorem \ref{1} and in Remark \ref{gowa}. In this subsection we want to let the reader notice that, in order to have, in positive characteristic, cohomological results similar to those gettable in characteristic $0$, sometimes it is better to use the \'etale site rather than the Zariski one. Moreover, as well as local cohomology, \'etale cohomology gives a lower bound for the number of set-theoretically defining equations of a projective scheme by \eqref{aragecd}, therefore the results of this subsection will be useful also in the next chapter.

We recall the following result of Lyubeznik \cite[Proposition 9.1, (iii)]{ly3}, that can be seen as an \'etale version of \cite[Theorem 8.6, p. 160]{hartshorne4}. 

\begin{thm}\label{13}(Lyubeznik)
Let $\kkk$ be a separably closed field of arbitrary characteristic, $Y \subseteq X$ two projective schemes over $\kkk$ such that $U:=X \setminus Y$ is nonsingular. Set $N:=\dim X$ and let $\ell \in \ZZ$ be coprime with the characteristic of $\kkk$. If  $\ecd(U) < 2N-r$, then the restriction maps
\[H^i(X_{\et}, \underline{\ZZ/\ell\ZZ}) \longrightarrow H^i(Y_{\et},\underline{\ZZ/\ell\ZZ})\] 
are isomorphism for $i < r$ and injective for $i=r$.
\end{thm}
%\begin{proof}
%The Gysin sequence applied to the smooth pair $(Y,X)$ (\cite[Corollary 16.2]{milnel}) yields
%\[  \ldots  \longrightarrow  H^{2N-i-1}(U_{\et}, \ZZ/l\ZZ) \longrightarrow H^{2N-2c-i}(Y_{\et}, \ZZ/l\ZZ) \longrightarrow \]
%\[ \longrightarrow H^{2N-i}(X_{\et}, \ZZ/l\ZZ(-c)) \longrightarrow H^{2N-i}(U_{\et}, \ZZ/l\ZZ)  \longrightarrow \ldots  \]
%for all $i \leq 2N-2c-1$ (where $c= \operatorname{codim}_X Y$). So the assumptions imply that the maps
%\[ H^{2N-2c-i}(Y_{\et}, \ZZ/l\ZZ) \rightarrow H^{2N-i}(X_{\et}^N, \ZZ/l\ZZ(-c)) \]
%are isomorphisms for $i < r$ and surjective for $i=r$. By Poincare duality (\cite[Theorem 24.1]{milnel}), the dual maps
%\[ H_c^i(X_{\et}, \ZZ/l\ZZ(n)) \longrightarrow H_c^i(Y_{\et},\ZZ/l\ZZ(n)) \]
%are isomorphisms for $i < r$ and injective for $i=r$ ($H_c$ denotes cohomology with compact support and $n=\dim Y$). Since $X$ and $Y$ are projective, cohomology with compact support is the same as usual cohomology. Moreover the choice of a $l^{th}$ root of 1 in $k$ defines an isomorphism of sheaves $\ZZ/l \ZZ(n) \cong \ZZ/l \ZZ$, so the above maps yield those we were looking for.
%\end{proof}

\begin{remark}
An \'etale version of Theorem \ref{1} cannot exist. In fact, if $X\subseteq \PP^n$ is a regular scheme projective over a field $\kkk$, the natural number $\ecd(\PP^n \setminus X)$ is not an invariant of $X$ and $n$, as instead is for $\cd(\PP^n \setminus X)$, see Remark \ref{newinvariant}. For instance, we can consider $\PP^2 \subseteq \PP^5$ embedded as a linear subspace and $v_2(\PP^2)\subseteq \PP^5$ (where $v_2(\PP^2)$ is the $2$nd Veronese embedding): the first one is a complete intersection, so $\ecd(\PP^5 \setminus \PP^2)\leq 7$ by \eqref{aragecd}. Instead, $\ecd(\PP^5 \setminus v_2(\PP^2))=8$ by a result of Barile \cite{barile}.  

Notice that the above argument shows that the number of defining equations of a projective scheme $X\subseteq \PP^n$ depends on the embedding, and not only on $X$ and on $n$. This suggests the limits of the use of local cohomology on certain problems regarding the arithmetical rank.
\end{remark}

We want to end this subsection stating a result similar to Proposition \ref{mah}. First we need a remark, which is analogous to the last part of Remark \ref{betti}.

\begin{remark}\label{chow}
Let $X$ be a regular scheme projective over a field $\kkk$ and $\ell$ an integer coprime to $\chara(\kkk)$. The cycle map is a graded homomorphism
\[\bigoplus_{i\in\NN}\operatorname{CH}^i(X)\otimes \QQ_{\ell} \longrightarrow \bigoplus_{i\in \NN}H^{2i}(X_{\et},\QQ_{\ell}(i)), \]
where by $\operatorname{CH}^*(X)$ we denote the Chow ring of $X$ (for the definition of the cycle map see the book of Milne  \cite[Chapter VI]{milne}).
Let $\operatorname{N}^i(X)$ be the group $\operatorname{CH}^i(X)$ modulo numerical equivalence. To prove \cite[Chapter VI, Theorem 11.7]{milne}, it is shown that the degree $i$-part of the kernel of the cycle map is contained in the kernel of
\[\operatorname{CH}^i(X)\otimes \QQ_{\ell}\longrightarrow \operatorname{N}^i(X)\otimes \QQ_{\ell}.\]
As $X$ is projective, $\operatorname{N}^i(X)$ is nonzero for any $i=0,\ldots ,\dim X$. Therefore, $H^{2i}(X_{\et},\ZZ_{\ell}(i))$ is nonzero for any $i=0,\ldots ,\dim X$. Finally, by the definition of $H^{2i}(X_{\et},\ZZ_{\ell}(i))$, there exists an integer $m$ such that 
\[H^{2i}(X_{\et},\underline{\ZZ/\ell^m\ZZ}(i))\neq 0\]
for any $i=0,\ldots ,\dim X$.
\end{remark}

\begin{prop}\label{14}
Let $\kkk$ a separably closed field of arbitrary characteristic. Let $X$ and $Y$ be two positive dimensional regular schemes projective over $\kkk$. Let us choose an embedding $Z:=X \times Y \subseteq \PP^N$, and set $U:=\PP^N \setminus Z$. Then $\ecd(U) \geq 2N-3$. In particular if $\dim Z \geq 3$, $Z$ is not a set-theoretic complete intersection.
\end{prop}
\begin{proof}
By Remark \ref{chow} there is an integer $\ell$ coprime with $\chara(\kkk)$ such that the modules 
\[ H^{2i}(X_{\et},\underline{\ZZ/\ell \ZZ}(i))\neq 0 \mbox{ \ \ and \ \ } H^{2i}(Y_{\et},\underline{\ZZ/\ell \ZZ}(i))\neq 0\]
for $i=0,1$. Therefore, using K\"unneth formula for \'etale cohomology, \cite[Chapter VI, Corollary 8.13]{milne}, one can easily show that $H^2(Z_{\et},\underline{\ZZ/\ell \ZZ}(1))$ cannot be isomorphic to $\ZZ/\ell \ZZ$. However $H^2(\PP_{\et}^N,\underline{\ZZ/\ell \ZZ}(1))\cong \ZZ/\ell\ZZ$ (see \cite[p. 245]{milne}). Now Theorem \ref{13} implies the conclusion.
\end{proof}

\begin{remark}
With the notation of Proposition \ref{14}, if $\dim Z \geq 3$, then it is not a set-theoretic complete intersection even if $\kkk$ is not separably closed. In fact, if it were, $Z\times_{\kkk}\Spec(\kkk_s)$ would be a set-theoretic complete intersection ($\kkk_s$ stands for the separable closure of $\kkk$). This would contradict Proposition \ref{14}.
\end{remark}
%
%\begin{prop}\label{curve}
%Let $k$ an algebraically closed field, $C$ a smooth projective curve of positive genus, $X$ a projective scheme and $Y=C \times X \subseteq \PP^N$ (any embedding). Then $\ecd(\PP^N \setminus Y) \geq 2N-2$. In particular, if $\dim X \geq 1$, then $Y$ is not a set-theoretic complete intersection.
%\end{prop}
%\begin{proof}
%Set $g$ the genus of $C$. By \cite[Proposition 14.2 and Remark 14.4]{milnel} it follows that $H^1(C_{\et},\ZZ/l\ZZ)\cong (\ZZ/l\ZZ)^{2g}$. Moreover $H^0(X_{\et},\ZZ/l \ZZ)\neq 0$ and $H^1(\PP^N_{\et},\ZZ/l \ZZ)=0$. But by K\"unneth formula for \'etale cohomology $H^1(Y_{\et},\ZZ/l \ZZ)\neq 0$, therefore Theorem \ref{13} let us conclude.
%\end{proof}
\index{etale cohomological dimension@\'etale cohomological dimension|)}
\subsection{Two consequences}
\index{depth|(}
In this subsection we draw two nice consequences of the investigations we made in the first part of the section.

The first fact we want to present is a consequence of Theorem \ref{1}, and regards a relationship between cohomological dimension of an ideal in a polynomial ring  and the depth of the relative quotient ring. Let us recall a result by Peskine and Szpiro, already mentioned in the Preliminaries' chapter.

\begin{thm}\label{peskineszpiro} (Peskine and Szpiro)
Let $R$ be an $n$-dimensional regular local ring of positive characteristic and $\aa\subseteq R$ an ideal. If $\depth(R/\aa)\geq t$, then $\cd(R,\aa)\leq n-t$.
\end{thm}

The same assertion does not hold in characteristic $0$. In the following example we show how, if the characteristic is not positive, Theorem \ref{peskineszpiro} can fail already with $t=4$.

\begin{example}\label{t=4ps}
Let $A:=\kkk[x,y,z]$, $B:=\kkk[u,v]$ and $T:=A\sharp B$ their Segre product. Clearly $T$ is a quotient of $S:=\kkk[X_0,\ldots ,X_5]$. Let $I$ be the kernel of the surjection $S\rightarrow T$. Arguing like in Remark \ref{gowa}, $T$ is Cohen-Macaulay. Since $\dim T = 4$, we have that $\depth(S/I)=\depth(T)=4$. If $\chara(\kkk)=0$, Proposition \ref{mah} implies that $\cd(\PP^5\setminus \Proj(T))\geq 2$. Using \eqref{cdcd2}, we get $\cd(S,I)\geq 3$. Now we can easily carry this example in the local case. Let $\mm:=(X_0,\ldots ,X_5)$ denote the maximal irrelevant ideal of $S$, and set $R:=S_{\mm}$. Clearly $R$ is a $6$-dimensional regular local ring. If $\aa:=IR$, we have $\depth(R/\aa)=\depth(S/I)=4$. Furthermore $\Ha^i(R)\cong (H_I^i(S))_{\mm}$ for any $i\in \NN$ by \eqref{cohomloc}. Since $H_I^i(S)=0$ if and only if $(H_I^i(S))_{\mm}=0$ (see Bruns and Herzog \cite[Proposition 1.5.15 c)]{BH}), 
\[\cd(R,\aa)=\cd(S,I)\geq 3>6-4.\] 
\end{example}

When $t=2$, Theorem \ref{peskineszpiro} holds true also for regular local rings containing a field of characteristic $0$. This fact easily follows from \cite[Theorem 2.9]{hu-ly}, as we are going to show below.

\begin{prop}\label{hunekelyubeznik}
Let $R$ be an $n$-dimensional regular local ring containing a field and $\aa\subseteq R$ an ideal. If $\depth(R/\aa)\geq 2$, then $\cd(R,\aa)\leq n-2$.
\end{prop}
\begin{proof}
Let $A:=((\widehat{R})^{sh})^{\wedge}$ the completion of the strict Henselization of $\widehat{R}$. Since $A$ is faithfully flat over $R$, we have $\cd(A,\aa A)=\cd(R,\aa)$ using Lemma \ref{basetheorems} (i). Moreover, since $A/\aa A$ is faithfully flat over $R/\aa$, by Lemma \ref{basetheorems} (i) and \eqref{depthloccoh} we get $\depth(A/\aa A)=\depth(R/\aa)\geq 2$. Therefore, by Proposition \ref{barney}, $A/\aa A$ is $1$-connected. At this point the thesis follows by \cite{hu-ly}, see Theorem \ref{huneke&lyubeznik}.
\end{proof}

Notice that Theorem \ref{peskineszpiro} and Proposition \ref{hunekelyubeznik} can be stated also if $R$ is a polynomial ring over a field and $\aa$ is a homogeneous ideal. In fact it would be enough to localize at the maximal irrelevant ideal to reduce the problem to the local case. From Theorem \ref{1} we are able to settle another case, in characteristic $0$, of Theorem \ref{peskineszpiro}. 

\begin{thm}\label{depth}
Let $S:=\kkk[x_1,\ldots ,x_n]$ be a polynomial ring over a field $\kkk$. Let $I\subseteq S$ be a homogeneous prime ideal such that $(S/I)_{\wp}$ is a regular local ring for any homogeneous prime ideal $\wp \neq \mm:=(x_1, \ldots ,x_n)$. If $\depth(S/I)\geq 3$, then $\cd(S,I)\leq n-3$.
\end{thm} 
\begin{proof}
If the characteristic of $\kkk$ is positive, the statement easily follows from Theorem \ref{peskineszpiro}. Therefore we can assume $\chara(\kkk)=0$.
Suppose by contradiction that $\cd(S,I)\geq n-2$. Set $X:=\Proj(S/I)\subseteq \PP^{n-1}=\Proj(S)$. So we are supposing that $\cd(\PP^{n-1}\setminus X)\geq n-3$ by \eqref{cdcd2}. By the assumptions, $X$ is a regular scheme projective over $\kkk$, therefore Theorem \ref{1} implies that $h^{10}(X)\neq 0$ or that $h^{01}(X)\neq 0$. But, with the notation of Remark \ref{0}, we have that $h^{10}(X)=h^{10}(X_{\CC})$ and $h^{01}(X)=h^{01}(X_{\CC})$. Since $X_{\CC}^h$ is a compact K\"ahler manifold, it is a well known fact of Hodge theory that
\[H^1(X_{\CC}^h,\O_{X_{\CC}^h})\cong H^0(X_{\CC}^h,\Omega_{X_{\CC}^h})\] 
(for instance see \cite[Theorem 10.2.4]{arap}). Therefore Theorem \ref{GAGA} implies $h^{10}(X_{\CC})=h^{01}(X_{\CC})$. So 
\[ h^{10}(X)\neq 0.\] 
By \eqref{shafloc4} we have $H^1(X,\O_X)=[H_{\mm}^2(S/I)]_0 \subseteq H_{\mm}^2(S/I)$. But therefore \eqref{depthloccoh} implies $\depth(S/I)\leq 2$, which is a contradiction.
\end{proof}

Together with Theorem \ref{peskineszpiro} and Proposition \ref{hunekelyubeznik}, Theorem \ref{depth} raises the following question:

\begin{question}\label{depth-coho}
Suppose that $R$ is a regular local ring, and that $\aa\subseteq R$ is an ideal such that $\depth(R/\aa)\geq 3$. Is it true that $\cd(R,\aa)\leq \dim R -3$?
\end{question}

But Theorem \ref{depth}, another  supporting fact for a positive answer to Question \ref{depth-coho} is provided  by the following theorem.

\begin{thm}
Let $(R,\mm)$ be an $n$-dimensional regular local ring containing a field, and let $\aa$ be an ideal of $R$. If $\depth(R/\aa)\geq 3$, then:
\begin{compactitem}
\item[(i)] $\cd(R,\aa)\leq n-2$;
\item[(ii)] $\Supp(H_{\aa}^{n-2}(R))\subseteq \{\mm\}$.
\end{compactitem}
\end{thm}
\begin{proof}
By Proposition \ref{hunekelyubeznik} we have $\cd(R,\aa)\leq n-2$. So we have just to show that, given a prime ideal $\wp$ different from $\mm$, we have $(H_{\aa}^{n-2}(R))_{\wp}=0$. Using \eqref{cohomloc} this is equivalent to show that
\begin{equation}\label{proof:toshow}
H_{\aa R_{\wp}}^{n-2}(R_{\wp})=0 \ \ \ \forall \ \wp \in \Spec(R)\setminus \{\mm\}.
\end{equation} 
Let us denote by $h$ the height of $\wp$. Then $h\leq n-1$ because $\wp \neq \mm$. We can furthermore suppose that $h\geq n-2$, since otherwise $\dim R_{\wp} < n-2$ and $H_{\aa R_{\wp}}^{n-2}(R_{\wp})$ would be zero from Theorem \ref{vanishinggro}. In such a case $\wp$ cannot be a minimal prime ideal of $\aa$ because $\depth(R/\aa)\geq 3$ (for instance, see \cite[Theorem 17.2]{matsu}). Particularly, $\dim R_{\wp}/\aa R_{\wp} > 0$.

First let us suppose that $h=n-2$. Notice that both $R_{\wp}$ and $\widehat{R_{\wp}}$ are $(n-2)$-dimensional regular local domains. Moreover, 
\[\dim \widehat{R_{\wp}}/\aa \widehat{R_{\wp}} = \dim R_{\wp}/\aa R_{\wp} > 0.\]
Therefore Hartshorne-Lichtenbaum Theorem \ref{hartlich} implies that $H_{\aa \widehat{R_{\wp}}}^{n-2}(\widehat{R_{\wp}})=0$. So, by \eqref{cohomcompl1} we get \eqref{proof:toshow}.

So we can suppose $h=n-1$. A theorem of Ishebeck (\cite[Theorem 17.1]{matsu}) yields
\[\Ext_{R}^0(R/\wp ,R/\aa)=\Ext_{R}^1(R/\wp ,R/\aa)=0.\]
This means that $\grade(\wp ,R/\aa) > 1$ and so that $H_{\wp}^0(R/\aa)=H_{\wp}^1(R/\aa)=0$ by \eqref{gradeloccoh}. Using \eqref{cohomloc} this implies
\[H_{\wp R_{\wp}}^0(R_{\wp}/\aa R_{\wp})=H_{\wp R_{\wp}}^1(R_{\wp}/\aa R_{\wp})=0.\]
This means that $\depth(R_{\wp}/\aa R_{\wp})\geq 2$ by \eqref{depthloccoh}. Since $R_{\wp}$ is an $(n-1)$-dimensional regular local ring containing a field, Proposition \ref{hunekelyubeznik} yields
\[H_{\aa R_{\wp}}^{n-2}(R_{\wp})=0.\]
This concludes the proof.
\end{proof}
\index{depth|)}
The second fact we want to present in this subsection concerns a possible relationship between ordinary and \'etale cohomological dimension. Such a kind of question was already done by Hartshorne in \cite[p. 185, Problem 4.1]{hartshorne4}. In \cite[Conjecture, p. 147]{ly2} Lyubeznik made a precise conjecture about this topic:

\begin{conj}(Lyubeznik).\label{conjly}\index{etale cohomological dimension@\'etale cohomological dimension}
If $U$ is an $n$-dimensional scheme of finite type over a separably closed field, then $\ecd(U)\geq n+\cd(U)$.
\end{conj}

Using Theorems \ref{1} and \ref{13}, we are able to solve the conjecture in a special case.

\begin{thm}\label{lyub}
Let $X \subseteq \PP^n$ be a regular scheme projective over an algebraically closed field $\KK$ of characteristic $0$, and $U:=\PP^n \setminus X$. Then
\[\ecd(U)\geq n + \cd(U)\]
\end{thm}
\begin{proof} 
Let $\kkk$ be the field obtained from $\QQ$ adding the coefficients of a set of generators of the defining ideal of $X$, like in Remark \ref{0}, and let $\overline{\kkk}$ denote the algebraic closure of $\kkk$ (recall that in characteristic $0$ separable and algebraic closures are the same thing). So we have
\[\QQ \subseteq \overline{\kkk} \subseteq \CC \mbox{ \ \ and \ \ }\QQ \subseteq \overline{\kkk} \subseteq \KK.\]
Therefore, from \cite[Chapter VI, Corollary 4.3]{milne}, we have
\[\ecd(U)=\ecd(U_{\overline{\kkk}})=\ecd(U_{\CC}).\]
Moreover, \eqref{redtocomplex1} yields $\cd(U)=\cd(U_{\CC})$. So, from now on, we suppose $\KK=\CC$.

Set $\cd(U):=s$, and define an integer $\rho_s$ to be $0$ (resp. $1$) if $n-s-1$ is odd (resp. if $n-s-1$ is even). By Theorem \ref{1} and by equation (\ref{betti3}), it follows that $\beta_{n-s-1}(X)>\rho_s$. Consider, for a prime number $p$, the $\ZZ/p\ZZ$-vector space 
\[\operatorname{Hom}_{\ZZ}(H_i^{Sing}(X^h,\ZZ), \ZZ/p\ZZ).\] 
Since $H_i^{Sing}(X^h,\ZZ)$ is of rank bigger than $\rho_s$, the above $\ZZ/p\ZZ$-vector space has dimension greater than $\rho_s$. Therefore by the surjection given by Theorem \ref{unicohom} 
\[ H_{Sing}^{n-s-1}(X^h,\ZZ/p\ZZ) \longrightarrow \operatorname{Hom}_{\ZZ}(H_{n-s-1}^{Sing}(X^h,\ZZ), \ZZ/p\ZZ),\] 
we infer that $ \dim_{\ZZ/p\ZZ}H_{Sing}^{n-s-1}(X^h,\ZZ/p\ZZ) > \rho_s$.
Now a comparison theorem due to Grothendieck (see Theorem \ref{etvsangro}) yields 
\[ \dim_{\ZZ/p\ZZ}H^{n-s-1}(X_{\et},\ZZ/p\ZZ)> \rho_s .\] 
Since $\dim_{\ZZ/p\ZZ}(H^{n-s-1}(\PP_{\et}^n,\ZZ/p\ZZ))=\rho_s$ by \cite[p. 245]{milne}, Theorem \ref{13} implies that 
\[\ecd(U) \geq 2n- (n-s)=n+s.\]
\end{proof}

Theorem \ref{lyub} might look like a very special case of Conjecture \ref{conjly}. However the case when $U$ is the complement of a projective variety in a projective space is a very important case. In fact the truth of Conjecture \ref{conjly} would ensure that to bound the homogeneous arithmetical rank from below it would be enough to work just with the \'etale site, and not with the Zariski one. Since usually one is interested in computing the number of (set-theoretically) defining equations of a projective scheme in the projective space, in some sense the most interesting case of Conjecture \ref{conjly} is when $U=\PP^n\setminus X$ for some projective scheme $X$. From this point of view, one can look at Theorem \ref{lyub} in the following way: {\it In order to give a lower bound for the minimal number of hypersurfaces of $\PP_{\CC}^n$ cutting out set-theoretically a regular scheme projective over a field of characteristic $0$, say $X\subseteq \PP^n$, it is better to compute $\ecd(\PP^n\setminus X)$ rather than $\cd(\PP^n \setminus X)$.}

Recently Lyubeznik, who informed us by a personal communication, found a counterexample to Conjecture \ref{conjly} when the characteristic of the base field is positive: His counterexample, not yet published, consists in a scheme $U$ which is the complement in $\PP^n$ of a reducible projective scheme.

\subsection{Translation into a more algebraic language}

We want to end this chapter translating Theorem \ref{1}, for the convenience of some readers, in a more algebraic language. So, we leave the geometric notation, setting $S:=\kkk[x_1,\ldots ,x_n]$ the polynomial ring in $n$ variables over a field $\kkk$ of characteristic $0$.

\begin{thm}\label{ccdalg}
Let $I\subseteq S=\kkk[x_1,\ldots ,x_n]$ be a nonzero homogeneous ideal such that $A_{\wp}$ is a regular ring for any homogeneous prime ideal \ $\wp$ of \ $A:=S/I$  different from the irrelevant maximal ideal $\nn:=A_+$. If $r$ is an integer such that $\height(I)\leq r \leq n-1$, then the following are equivalent:
\begin{compactitem}
\item[(i)] Local cohomology with support in $I$ vanish beyond $r$, that is $\cd(S,I)\leq r$.
\item[(ii)] Denoting by $\Omega_{A/\kkk}$ the module of K\"ahler differentials of $A$ over $\kkk$ and by $\Omega^j$ its $j$th exterior power $\Lambda^j\Omega_{A/\kkk}$,
{\footnotesize
\begin{displaymath}
\dim_{\kkk}(H_{\nn}^i(\Omega^j)_0) = \left\{ \begin{array}{ll} 0 & \mbox{if  } i\geq 2, \ i\neq j+1, \ i+j<n-r\\
1 &  \mbox{if  } i\geq 2, \ i=j+1, \ i+j<n-r\\
\dim_{\kkk}(H_{\nn}^0(\Omega^j)_0) -1 &  \mbox{if  } i=1, \ j\neq 0, \ j<n-r-1\\
\dim_{\kkk}(H_{\nn}^0(A)_0) &  \mbox{if  } i=1, \ j=0, \ r<n-1\\
\end{array}  \right.
\end{displaymath}
}
\item[(iii)]
Viewing $A$ and $\Omega^j$ as $S$-modules,
{\footnotesize
\begin{displaymath}
\dim_{\kkk}(\Ext_S^i(\Omega^j,S)_{-n}) = \left\{ \begin{array}{ll} 0 & \mbox{if  } i\leq n- 2, i\neq n-j-1, i-j>r\\
1 & \mbox{if  } i\leq n- 2, i= n-j-1, i-j>r\\
\dim_{\kkk}(\Ext_S^n(\Omega^j,S)_{-n}) -1 & \mbox{if  } i=n-1, j\neq 0, j<n-r-1\\
\dim_{\kkk}(\Ext_S^n(A,S)_{-n}) & \mbox{if  } i=n-1, j=0, r<n-1\\
\end{array}  \right.
\end{displaymath}
}
\end{compactitem}
\end{thm}
\begin{proof}
Set $X:=\V_+(I)\subseteq \PP^{n-1}$ and $U:=\PP^{n-1}\setminus X$. Then $\cd(S,I)\leq r \ \iff \ \cd(U)<r$ by \eqref{cdcd2}. By the assumptions $X$ is a regular scheme, therefore we can use the characterization of Theorem \ref{1} as follows: First of all notice that $\Omega^q_{X/\kkk}\cong \widetilde{\Omega^q}$. So, \eqref{shafloc4} implies:
\[h^{pq}(X)=\dim_{\kkk}(H_{\nn}^{p+1}(\Omega^q)_0) \ \ \forall \ p\geq 1.\]
Moreover, by \eqref{shafloc3}, we have:
\[h^{0q}(X)=\dim_{\kkk}(H_{\nn}^1(\Omega^q)_0)-\dim_{\kkk}([H_{\nn}^0(\Omega^q)]_0)+1.\]
Thus, using Theorem \ref{1} we get the equivalence between (i) and (ii). The equivalence between (ii) and (iii) follows by local duality, see Theorem \ref{weaklocduality}. In fact, viewing $\Omega^q$ both as  $A$- and $S$-modules, point (ii) of Theorem \ref{basetheorems} guarantees that $H_{\nn}^p(\Omega^q)\cong H_{\mm}^p(\Omega^q)$, where $\mm$ is the irrelevant maximal ideal of $S$. Therefore Theorem \ref{weaklocduality} yields
\[\dim_{\kkk}(H_{\nn}^p(\Omega^q)_0)=\dim_{\kkk}(\Ext_S^{n-p}(\Omega^q,S)_{-n}).\]
\end{proof}

At first blush, conditions (ii) and (iii) of Theorem \ref{ccdalg} may seem even more difficult than understanding directly whether $H_I^i(S)=0$ for each $i>r$. However, one should focus on the fact that the modules $H_{\nn}^i(\Omega^j)$ depend only on the ring $A$, and not by the chosen presentation as a quotient of a polynomial ring. Furthermore, the $S$-modules $\Ext_S^i(\Omega^j,S)$, being finitely generated, are much more wieldy with respect to the local cohomology modules $H_I^i(S)$.

\index{cohomological dimension|)}

\chapter{Properties preserved under Gr\"obner deformations and arithmetical rank}
\label{chapter2}

This chapter is devoted to some applications of the results of Chapter \ref{chapter1}. Essentially, it is structureted in two different parts, which we are going to describe. Most of the results appearing in this chapter are borrowed from our works \cite{va1,va2}. 

\vskip 4mm

\index{connected|(}\index{connected!r-@$r$-|(}\index{connected!in codimension $1$|(}\index{connected!strongly|see{simplicial complex}}\index{initial ideal|(}\index{simplicial complex|(}\index{simplicial complex!strongly connected|(}

In Section \ref{secgrodef}, we study the connectedness behavior under Gr\"obner deformations. That is, if $I$ is an ideal of the polynomial ring $S:=\kkk[x_1,\ldots ,x_n]$ and $\prec$ is a monomial order on $S$, we study whether the ring $S/\ini(I)$ inherits some connectedness properties of $S/I$. The answer is yes, in fact we prove that, if $r$ is an integer less than $\dim S/I$, then $S/\ini(I)$ is $r$-connected whenever $S/I$ is $r$-connected (Corollary \ref{ptdmv1}). Actually, more generally, we prove the analog version for initial objects with respect to weights (Theorem \ref{ptdmv}). This fact generalizes the result of Kalkbrener and Sturmfels \cite[Theorem 1]{ka-st}, which says that, if $I$ is a geometrically prime ideal, in the sense that it is prime and it remains prime once we tensor with the algebraic closure of $\kkk$, then the simplicial complex $\D(\sqrt{\ini(I)})$ is pure and strongly connected (Theorem \ref{kalkstu}). As a consequence they settled a fact conjectured by Kredel and Weispfenning \cite{kr-we}. However, our proof is different from the one of Kalkbrener and Sturmfels: In fact, it leads to their result just assuming the primeness of $I$, and not the geometrically primeness. If anything, the proof of Theorem \ref{ptdmv} is inspired to the approach used in the lectures of Huneke and Taylor \cite{hu-ta} to show the theorem of Kalkbrener and Sturmfels. Another immediate consequence of Corollary \ref{ptdmv1} is that, if $I$ is homogeneous and $\V_+(I)$ is a positive dimensional connected projective scheme, then $\V_+(\ini(I))$ is connected too (Corollary \ref{conninit}). Eventually, we prove an even more general version of Corollary \ref{ptdmv1}, replacing $S$ with any $\kkk$-subalgebra $A\subseteq S$ such that $\ini(A)$ is finitely generated (Corollary \ref{sagbiconn}).

Another nice consequence of Corollary \ref{ptdmv1} is the solution of the conjecture of Eisenbud and Goto, formulated in their paper \cite{eisenbudgoto}, for a new class of ideals. Namely, we prove the conjecture for ideals $I$ which have an initial ideal with no embedded primes (Theorem \ref{eisenbudgotonoemb}). In particular, we show that the Eisenbud-Goto conjecture is true for ideals defining ASL\index{Algebra with Straightening Laws}, see Corollary \ref{eisenbudgotoasl}.

In Subsection \ref{subchap2.1}, we noticed that Corollary \ref{ptdmv1} also implies that the simplicial complex $\D:=\D(\sqrt{\ini(I)})$ is  pure and strongly connected whenever $I$ is a homogeneous ideal such that $S/I$ is Cohen-Macaulay (Theorem \ref{skinner}). At first blush, the reader might ask whether $\kkk[\D]$ is even Cohen-Macaulay whenever $S/I$ is Cohen-Macaulay. However, Conca showed, using the computer algebra system Cocoa \cite{cocoa}, that this is not the case (see Example \ref{esempioaldo}). A more reasonable question could be the following: 

\vskip 1mm

\noindent {\bf Question} \ref{squarefreecm} If $S/I$ is Cohen-Macaulay and $\ini(I)$ is square-free, is $S/\ini(I)$ Cohen-Macaulay as well as $S/I$?

\vskip 1mm 

Maybe the answer to this question is negative, and the reason why we cannot find a counterexample to it is that the property of $\ini(I)$ being square-free is very rare. However, in some cases we can give an affirmative answer to the above question, namely:
\begin{compactitem}
\item[(i)] If $\dim(S/I)\leq 2$ (Proposition \ref{answerdim2}).
\item[(ii)] If $S/I$ is Cohen-Macaulay with minimal multiplicity (Proposition \ref{answerminmult}).
\item[(iii)] For certain ASL\index{Algebra with Straightening Laws} (Proposition \ref{answersomeasl}).
\end{compactitem} 

\vskip 4mm
\index{arithmetical rank}\index{arithmetical rank!homogeneous}\index{set-theoretic complete intersection}\index{Segre product}

Section \ref{secarasegre} is dedicated to the study of the arithmetical rank\index{arithmetical rank} of certain algebraic varieties.
The beauty of finding the number of defining equations of a variety is expressed by Lyubeznik in \cite{ly} as follows:

\vskip 1mm

\emph{Part of what makes the problem about the number of defining equations so interesting is that it can be very easily stated, yet a solution, in those rare cases when it is known, usually is highly nontrivial and involves a fascinating interplay of Algebra and Geometry.}

\vskip 1mm

The varieties whose we study the number of defining equations are certain Segre products of two projective varieties. In Subsection \ref{seclowerbounds} we prove that, if $X$ and $Y$ are two smooth projective schemes over a field $\kkk$, then, to define set-theoretically the Segre product  $X\times Y$ embedded in some projective space $\PP^N$, are needed at least $N-2$ equations (Proposition \ref{lowerara}). Furthermore, if $X$ is a curve of positive genus, then the needed equations are at least $N-1$ (Proposition \ref{curve}). Both these results come from cohomological considerations, and they are consequences of the work done in Chapter \ref{chapter1}. On the other hand, the celebrated theorem of Eisenbud and Evans \cite[Theorem 2]{eisenbudevans} tells us that, in any case, $N$ homogeneous equations are enough to define set-theoretically any projective scheme in $\PP^N$. Unfortunately, to decide, once fixed $X$ and $Y$, whether the minimum number of equations is $N$, \ $N-1$ or $N-2$, is a very hard problem. We will solve this issue in some special cases. Let us list some works that already exist in this direction.
\begin{compactitem}
\item[(i)] In their paper \cite{bruns-schwanzl}, Bruns and Schw\"anzl studied the number of defining equations of a determinantal variety. In particular they proved that the Segre product 
\[\PP^n \times \PP^m \subseteq \PP^N \mbox{ \ \ \ where \ \ \ } N:=nm+n+m\] 
can be defined set-theoretically by $N-2$ homogeneous equations and not less. In particular, it is a set-theoretic complete intersection if and only if $n=m=1$.
\item[(ii)] In their work \cite{siwa}, Singh and Walther gave a solution in the case of 
\[E \times \PP^1 \subseteq \PP^5,\] 
where $E$ is a smooth elliptic plane curve: The authors proved that the arithmetical rank of this Segre product is $4$. Later, in \cite{song}, Song proved that the arithmetical rank of $C \times \PP^1$, where $C \subseteq \PP^2$ is a Fermat curve (i.e. a curve defined by the equation $x_0^d+x_1^d+x_2^d$), is $4$ provided that $d\geq 3$. In particular both $E\times \PP^1$ and $C\times \PP^1$ are not set-theoretic complete intersections.  
\end{compactitem}
%\begin{os}
%Notice that the notion of the arithmetic rank is relative, so for study the arithmetic rank of a variety is necessary to give an embedding in another variety. In fact, the difference between the codimension and the arithmetical rank is not a biregular invariant (see Examples \ref{30}).
%\end{os}
In light of these results (especially we have been inspired to \cite{siwa}), it is natural to study the following problem.

\vskip 1mm

\emph{Let $n,m,d$ be natural numbers such that $n \geq 2$ and $m,d \geq 1$, and let $X\subseteq \PP^n$ be a smooth hypersurface of degree $d$. Consider the Segre product $Z := X \times \PP^m \subseteq \PP^N$, where $N:=nm+n+m$. What can we say about the number of defining equations of $Z$?}

\vskip 1mm

Notice that the arithmetical rank of $Z$ can depend, at least a priori, by invariant different from $n,m,d$: In fact we will need other conditions on $X$. However, for certain $n,m,d$, we can provide some answers to this question. 

In the case $n=2$ and $m=1$, we introduce, for every $d$, a locus of special smooth projective plane curves of degree $d$, that we will denote by $\V_d$ (see Remark \ref{hyperflexes}): This locus consists in those smooth projective curves $X$ of degree $d$ which have a $d$-flex, i.e. a point $p$ at which the multiplicity intersection of $X$ and the tangent line in $p$ is equal to $d$. Using methods coming from ``ASL theory"\index{Algebra with Straightening Laws} (see De Concini, Eisenbud and Procesi \cite{DEP2} or Bruns and Vetter \cite{BrVe}), we prove that for such a curve $X$ the arithmetical rank of the Segre product $X \times \PP^1 \subseteq \PP^5$ is $4$, provided that $d\geq 3$ (see Theorem \ref{ara1}). It is easy to show that every smooth elliptic curve belongs to $\V_3$ and that every Fermat's curve of degree $d$ belongs to $\V_d$, so we recover the results of \cite{siwa} and of \cite{song} (see Corollaries \ref{araelliptic} and \ref{arafermat}). However, the equations that we will find are different from  the ones found in those papers, and our result is characteristic free. Moreover, a result of Casnati and Del Centina \cite{casnati-del centina} shows that the codimension of $\V_d$ in the locus of all the smooth projective plane curves of degree $d$ is $d-3$, provided that $d\geq 3$ (Remark \ref{hyperflexes}). So we compute the arithmetical rank of $X\times \PP^1\subseteq \PP^5$ for a lot of new plane curves $X$.

For a general $n$, we can prove that if $X\subseteq \PP^n$ is a general smooth hypersurface of degree not bigger than $2n-1$, then the arithmetical rank of $X\times \PP^1\subseteq \PP^{2n+1}$ is at most $2n$ (Corollary \ref{bello}). To establish this we need a higher-dimensional version of $\V_d$ and Lemma \ref{cil}, suggested us by Ciliberto. This result is somehow in the direction of the open question whether any connected projective scheme of positive dimension in $\PP^N$ can be defined set-theoretically by $N-1$ equations. 

With some similar tools we can show that, if $F:=x_n^d+\sum_{i=0}^{n-3}x_iG_i(x_0,\ldots ,x_n)$ and $X:=\V_+(F)$ is smooth, then the arithmetical rank of $X\times \PP^1\subseteq \PP^{2n+1}$ is exactly $2n-1$ (Theorem \ref{45}).

Eventually, using techniques similar to those of \cite{siwa}, we are able to show the following: The arithmetical rank of the Segre product $X \times \PP^m \subseteq \PP^{3m+2}$, where $X$ is a smooth conic of $\PP^2$, is equal to $3m$, provided that $\chara(\kkk)\neq 2$ (Theorem \ref{conic}). In particular, $X \times \PP^m$ is a set-theoretic complete intersection if and only if $m=1$.
\index{arithmetical rank}\index{arithmetical rank!homogeneous}\index{set-theoretic complete intersection}\index{Segre product}

\section{Gr\"obner deformations}\label{secgrodef}

\subsection{Connectedness preserves under Gr\"obner deformations}

The main result of this subsection is Theorem \ref{ptdmv}, which will launch the results of the two next subsections.

\begin{lemma}\label{piccoloaiutantedibabbonatale}
Let $I$ be an ideal of $S$ and $\omega \in \mathbb{N}_{\geq 1}^n$. For all $r\in \ZZ$, if $S/I$ is $r$-connected, then $S[t]/\homo(I)$ is $(r+1)$-connected.
\end{lemma}
\begin{proof}
Let $\wp_1, \ldots, \wp_m$ be the minimal prime ideals of $I$. Then, by Lemma \ref{ralph} (v), it follows that $\homo(\wp_1), \ldots
, \homo(\wp_m)$ are the minimal prime ideals of $\homo(I)$. Suppose by contradiction that $S[t]/\homo(I)$ is not $(r+1)$-connected. Then, by Lemma \ref{milhouse} (i), we can choose $A,B \subseteq \{1,\ldots ,m\}$ disjoint, such that $A\cup B=\{1,\ldots ,m\}$ and such that, setting $\mathcal{J}:= \cap_{i \in A}
\homo(\wp_i)$ and $\mathcal{K}:= \cap_{j \in B}\homo(\wp_j)$,
\[\dim S[t]/(\mathcal{J} + \mathcal{K})\leq r. \]
Set $J:= \cap_{i \in A} \wp_i$ and $K:= \cap_{j \in B}\wp_j$. Lemma \ref{ralph} (i) implies that $\homo(J)=\Jj$ and $\homo(K)=\K$.
Obviously, $\Jj + \K \subseteq \homo(J +K)$. So, using Lemma \ref{ralph} (vi), we get 
\[r\geq\dim S[t]/(\Jj + \K)\geq \dim S[t]/\homo(J+K)= \dim S/(J+K)+1.\] 
At this point, Lemma \ref{milhouse} (i) implies that $S/I$ is not $r$-connected, which is a contradiction.
\end{proof}

\begin{thm}\label{ptdmv}
Let $I$ be an ideal of $S$, $\omega \in \mathbb{\NN}_{\geq 1}^n$ and $r$ an integer less than $\dim S/I$. If $S/I$ is $r$-connected, then $S/\inito(I)$ is $r$-connected.
\end{thm}
\begin{proof}
Note that $S[t]/\homo(I)$ is a graded ring with $\kkk$ as degree $0$-part (considering the $\oo$-graduation). Moreover, $(\bar{t}):=(\homo(I) +
(t))/\homo(I) \subseteq S[t]/\homo(I)$ is a homogeneous ideal
of $S[t]/\homo(I)$. The ring $S[t]/\homo(I)$ is $(r+1)$-connected by Lemma \ref{piccoloaiutantedibabbonatale}, and
\[\cd(S[t]/\homo(I),(\bar{t}))\leq 1=(r+1)-r\]
by \eqref{aragcd}. So Theorem \ref{winchester} yields that $S[t]/(\homo(I)+(t))$ is $r$-connected. Eventually, we get the conclusion, because Proposition \ref{flatdef} says that
\[S[t]/(\homo(I)+(t))\cong S/\inito(I).\]
\end{proof}

Using Theorem \ref{telespallabob}, we immediately get the following Corollary.
\begin{corollary}\label{ptdmv1}
Let $I$ be an ideal of $S$, $\prec$ a monomial order on $S$ and $r$ an integer less than $\dim S/I$. If $S/I$ is $r$-connected, then $S/\ini(I)$ is $r$-connected.
\end{corollary}

Given a monomial order $\prec$ on $S$, for any ideal $I$ let $\D_{\prec}(I)$ denote the simplicial complex $\D(\sqrt{\ini(I)})$ on $[n]$ (see \ref{appendixd}). Corollary \ref{ptdmv1} gets immediately \cite[Theorem 1]{ka-st}.

\begin{thm}\label{kalkstu}
(Kalkbrener and Sturmfels). Let $\prec$ a monomial order on $S$. If
$I$ is a prime ideal, then $\D_{\prec}(I)$ is pure of dimension $\dim S/I-1$ and strongly connected.
\end{thm}
\begin{proof}
Since $I$ is prime, $S/I$ is $d$-connected, where $d:=\dim S/I$. So $S/I$ is $(d-1)$-connected, and Corollary \ref{ptdmv1} implies that $S/\ini(I)$ is $(d-1)$-connected. So $S/\sqrt{\ini(I)}$ is connected in codimension $1$, which is equivalent to $\D_{\prec}(I)$ being strongly connected by Remark \ref{sccd1}.
\end{proof}

The following is another nice consequence of Theorem \ref{ptdmv}.

\begin{corollary}\label{conninit}
Let $I$ be a homogeneous ideal of $S$ and $\oo \in \mathbb{\NN}_{\geq 1}^n$. If $\V_+(I)\subseteq \PP^{n-1}$ is a positive dimensional connected projective scheme, then
$\V_+(\inito(I))\subseteq \PP^{n-1}$ is connected as well.
\end{corollary}
\begin{proof}
We recall that the fact that $\V_+(I)$ is connected is equivalent to $\V_+(I)$ being $0$-connected. Moreover, Remark \ref{punctureddu}, since $\V_+(I)$ is homeomorphic to $\Proj(S/I)$, yields that $S/I$ is $1$-connected. Because $\V_+(I)$ is positive dimensional, we have $\dim S/I\geq 2$. So, Theorem \ref{ptdmv} implies that $S/\inito(I)$ is $1$-connected. At this point, we can go backwards getting the connectedness of $\V_+(\inito(I))$. 
\end{proof}

\begin{remark}\label{abu}
Once fixed an ideal $I\subseteq S$, an integer $r<\dim S/I=:d$ and a monomial order $\prec$, Corollary \ref{ptdmv1} yields
\[S/I \mbox{ $r$-connected }\implies S/\ini(I) \mbox{ $r$-connected.}\]
In general, the reverse of such an implication does not hold true. To see this, let us recall that there
exists a nonempty Zariski open set $U \subseteq \GL(V)$, where $V$ is a $\kkk$-vector space of dimension $n$, and a Borel-fixed ideal $J \subseteq S$, i.e an ideal fixed under the action of the subgroup of all upper triangular matrices $B_+(V)\subseteq \GL(V)$, such that $\ini(gI)=J$ for all $g \in U$.  The ideal $J$ is called the generic initial
ideal of $I$, for instance see the book of Eisenbud \cite[Theorem 15.18, Theorem
15.20]{eisenbud}. It is known that, since $J$ is Borel-fixed,
$\sqrt{J}= (x_1,\dots,x_c)$  where $c=\height(I)$,
see \cite[Corollary 15.25]{eisenbud}). Hence $S/J=S/\ini(gI)$ (for $g\in U$) is $d$-connected. On the other hand, for all $g\in \GL(V)$, $S/gI$ is $r$-connected if and only if $S/I$ is $r$-connected, and clearly, for any $r\geq 1$, we could have started from a homogeneous ideal not $r$-connected. 
\end{remark}

The last corollary of Theorem \ref{ptdmv} we present in this subsection actually strengthens Theorem \ref{ptdmv}. 

\begin{corollary}\label{sagbiconn}\index{initial algebra}
Let $A$ be a $\kkk$-subalgebra of $S$ and $\oo\in \NN_{\geq 1}^{n}$ be such that $\inito(A)$ is finitely generated. If $J$ is an ideal of $A$ and $r$ is an integer less than $\dim A/J$, then $\inito(A)/\inito(J)$ is $r$-connected whenever $A/J$ is $r$-connected.
\end{corollary}
\begin{proof}
We can reduce the situation to Theorem \ref{ptdmv} using the result \cite[Lemma 2.2]{BC5} of the notes of Bruns and Conca.
\end{proof}

\subsection{The Eisenbud-Goto conjecture}

Theorem \ref{ptdmv} gives also a new class of ideals for which the Eisenbud-Goto conjecture holds true. First of all we recall what the conjecture claims. Let $I$ be a homogeneous ideal of $S$. The {\it Hilbert series}\index{Hilbert series} of $S/I$ is:
\[\Hs_{S/I}(z):=\sum_{k\in \NN}\Hf_{S/I}(k)z^d\in \NN[[z]],\]
where $\Hf$ is the Hilbert function\index{Hilbert function}, see \ref{subcastmum}.
It is  well known that 
\[\displaystyle \Hs_{S/I}(z)=\frac{h_{S/I}(z)}{(1-z)^d}\] 
where $d$ is the dimension of $S/I$ and $h_{S/I}\in \ZZ[z]$, usually called the {\it $h$-vector}\index{h-vector@$h$-vector} of $S/I$, is such that $h_{S/I}(1)\neq 0$ (for instance see the book of Bruns and Herzog \cite[Corollary 4.1.8]{BH}). The {\it multiplicity}\index{multiplicity|(} of $S/I$ is:
\[e(S/I)=h_{S/I}(1).\]
Eisenbud and Goto, in \cite{eisenbudgoto}, conjectured an inequality involving the multiplicity, the Castelnuovo-Mumford regularity and the height of a homogeneous ideal, namely:
\index{Castelnuovo-Mumford regularity|(}
\begin{conj}(Eisenbud-Goto).\label{conjeigo}
Let $I\subseteq S$ be a homogeneous radical ideal contained in $\mm^2$, where $\mm:=(x_1,\ldots ,x_n)$. If $S/I$ is connected in codimension $1$, then
\[\reg(S/I)\leq e(S/I)-\height(I).\]
\end{conj}

\begin{remark}
Conjecture \ref{conjeigo} can be easily shown when $S/I$ is Cohen-Macaulay. For an account of other known cases see the book of Eisenbud \cite{eisenbud1}.
\end{remark}

We recall that, for a monomial ideal $I\subseteq S$, we denote $\widetilde{I}\subseteq \Spol$ the polarization of $I$  (see \ref{polarization}).\index{polarization}

\begin{lemma}\label{noembprimes}
Let $I$ be a monomial ideal of $S$ with no embedded prime ideals. Then, fixed $c>0$, $S/I$ is connected in codimension $c$ if and only if $\Spol/\widetilde{I}$ is connected in codimension $c$.
\end{lemma}
\begin{proof}
Since $I$ has not embedded prime ideals, it has a unique primary decomposition, which is of the form:
\[I=\bigcap_{F\in \FD} I_F,\]
where $\D=\D(\sqrt{I})$ and for any facet $F\in \FD$, the ideal $I_F$ is a primary monomial ideal with $\sqrt{I_F}=\wp_F=(x_i:i\in F)$. Because polarization commutes with intersections, see \eqref{intersectionpol1}, we have that: 
\[\widetilde{I}=\bigcap_{F\in \FD}\widetilde{I_F}\subseteq \Spol.\]
Being each $I_F$ a primary monomial ideal, it turns out that $S/I_F$ is Cohen-Macaulay for any $F\in \FD$. Therefore $\Spol/\widetilde{I_F}$ is Cohen-Macaulay for all $F\in \FD$ by Theorem \ref{cmpol}. In particular, for all facets $F\in \FD$, the ring $\Spol/\widetilde{I_F}$ is connected in codimension $1$ by Proposition \ref{barney}. So, using Lemma \ref{hilbert}, for any two minimal prime ideals $\wp$ and $\wp'$ of $\widetilde{I_F}$, there is a sequence $\wp=\wp_0,\ldots ,\wp_s=\wp'$ of minimal prime ideals of $\widetilde{I_F}$ such that $\height(\wp_i+\wp_{i-1})\leq \height(\widetilde{I_F})+1=\height(I_F)+1=|F|+1$ for any $i=1,\ldots ,s$. Furthermore, one can easily show that the prime ideal
\[\widetilde{\wp_F}=(x_{i_1,1},x_{i_2,1},\ldots ,x_{i_{|F|},1})\subseteq \Spol ,\]
where $F=\{i_1,\ldots ,i_{|F|}\}$, is a minimal prime ideal of $\widetilde{I_F}$, and therefore of $\widetilde{I}$. Set $h:=\height(I)$. Using Lemma \ref{milhouse} (ii) and the fact that $c>0$, one can easily show:
\[h\leq |F|<h+c.\]
Suppose that $S/I$ is connected in codimension $c$. Let us consider two minimal prime ideals $\wp$ and $\wp'$ of $\widetilde{I}$. To show that $\Spol/\widetilde{I}$ is connected in codimension $c$, by Lemma \ref{hilbert} we have to exhibit a sequence $\wp=\wp_0,\ldots ,\wp_s=\wp'$ of minimal prime ideals of $\widetilde{I}$ such that $\height(\wp_i+\wp_{i-1})\leq h+c$ for any $i=1,\ldots ,s$. Assume that $\wp$ is a minimal prime of $\widetilde{I_F}$ and that $\wp'$ is a minimal prime of $\widetilde{I_G}$, where $F,G\in \FD$.
First of all, from what said above, there are two sequences $\wp=\wp_0,\ldots ,\wp_t=\widetilde{\wp_F}$ and $\wp'=\wp_0',\ldots ,\wp_q'=\widetilde{\wp_G}$ such that $\wp_i\in \Min(\widetilde{I_F})$, \ $\wp_j'\in \Min(\widetilde{I_G})$, \ $\height(\wp_i+\wp_{i-1})\leq |F|+1\leq h+c$ and $\height(\wp_j'+\wp_{j-1}')\leq |G|+1\leq h+c$. Then, since $S/I$ is connected in codimension $c$, there exists a sequence $F=F_0,\ldots ,F_s=G$ of facets of $\D$ such that $\height(\wp_{F_i}+\wp_{F_{i-1}})\leq h+c$. Eventually, the desired sequence connecting $\wp$ with $\wp'$ is:
\[\wp=\wp_0,\wp_1,\ldots ,\wp_t=\widetilde{\wp_F}=\widetilde{\wp_{F_0}},\widetilde{\wp_{F_1}},\ldots ,\widetilde{\wp_{F_s}}=\widetilde{\wp_G}=\wp_q',\wp_{q-1}',\ldots ,\wp_0'=\wp'.\]
For the converse, the same argument works backwards.
\end{proof}

\begin{thm}\label{eisenbudgotonoemb}
Let $I\subseteq S$ be a homogeneous radical ideal contained in $\mm^2$, and suppose that $S/I$ is connected in codimension $1$. If there exists a monomial order $\prec$ on $S$ such that $\ini(I)$ has no embedded prime ideals, then
\[\reg(S/I)\leq e(S/I)-\height(I).\]
\end{thm}
\begin{proof}
Let $\prec$ be a monomial order such that $J:=\ini(I)$ has no embedded prime ideals. The following are standard facts: $\reg(S/I)\leq \reg(S/J)$ (\cite[Corollary 3.5]{BC5}), $\height(I)=\height(J)$ (\cite[Theorem 3.9 (a)]{BC5}) and $e(S/I)=e(S/J)$ (\cite[Proposition 1.4 (e)]{BC5}). Also, obviously we have that $J\subseteq \mm^2$ as well as $I$. At last, Corollary \ref{ptdmv1} implies that $S/J$ is connected in codimension $1$. 

Because $S/J$ has no embedded prime ideals, Lemma \ref{noembprimes} implies that $\Spol/\widetilde{J}$ is connected in codimension $1$. From Theorem \ref{cmpol}, we have $\reg(S/J)=\reg(\Spol/\widetilde{J})$ (recall the interpretation of the Castelnuovo-Mumford regularity in terms of the graded Betti numbers \eqref{regularityfreeres}). Moreover, using Remark \ref{dimunderpol}, we get $\height(J)=\height(\widetilde{J})$ and $e(S/J)=e(\Spol/\widetilde{J})$. Besides, obviously $\widetilde{J}\subseteq \M^2$, where by $\M$ we denote the maximal irrelevant ideal of $\Spol$. Nevertheless, Eisenbud-Goto conjecture has been showed for square-free monomial ideals by Terai in \cite[Theorem 0.2]{terai}, so that:
\[\reg(\Spol/\widetilde{J})\leq e(\Spol/\widetilde{J})-\height(\widetilde{J}).\]

Since, from what said above, $\reg(S/I)\leq \reg(\Spol/\widetilde{J})$, $\height(I)=\height(\widetilde{J})$ and $e(S/I)=e(\Spol/\widetilde{J})$, we eventually get the conclusion.
\end{proof}

In view of Theorem \ref{eisenbudgotonoemb}, it would be interesting to discover {\it classes of ideals for which there exists a monomial order such that the initial ideal with respect to it has no embedded prime ideals}. Actually, we already have for free a class like that: Namely, the homogeneous ideals defining an Algebra with Straightening Laws (ASL for short)\index{Algebra with Straightening Laws}\index{ASL|see{Algebra with Straightening Laws}} over $\kkk$. For the convenience of the reader we give the definition here: Let $\sqsubset$ be a partial order on $[n]$, and let us denote by $\Pi$ the poset $([n],\sqsubset)$. To $\Pi$ is associated a Stanley-Reisner ideal, namely the one of the order complex $\D(\Pi)$ whose faces are the chains of $\Pi$:
\[ I_{\Pi} := I_{\D(\Pi)} = ( x_ix_j~:~i \textrm{~and~} j \textrm{~are incomparable elements of~}\Pi).\]
Given a homogeneous ideal $I\subseteq S$, the standard graded algebra $R:=S/I$ is called a {\it ASL} on $\Pi$ over $\kkk$ if:
\begin{compactitem}
\item[(i)] The residue classes of the monomials not in $I_{\Pi}$ are linearly independent in $R$.
\item[(ii)] For every $i,j \in \Pi$ such that $i$ and $j$ are incomparable the ideal $I$ contains a polynomial of the form
\[ x_ix_j - \sum\lambda x_hx_k\]
with $\lambda \in \kkk$, $h,k \in \Pi$, $h\sqsubseteq k$, $h\sqsubset i$ and $h\sqsubset j$. The above sum is allowed to run on the empty-set.
\end{compactitem}
The polynomials in (ii) give a way of rewriting in $R$ the product of two incomparable elements. These relations are called the \emph{straightening relations}\index{straightening relations}.
Let $\prec$ be a degrevlex monomial order on a linear extension of $\sqsubset$. Then the polynomials in (ii) form a Gr\"obner basis of $I$ and $\ini(I)=I_{\Pi}$. Particularly, being square-free, $\ini(I)$ has no embedded prime ideals. So, we have the following:

\begin{corollary}\label{eisenbudgotoasl}
Let $I\subseteq S$ be a homogeneous ideal defining an ASL. Then, the Eisenbud-Goto conjecture \ref{conjeigo} holds true for $I$. That is, if $S/I$ is connected in codimension 1, then
\[\reg(S/I)\leq e(S/I)-\height(I).\]
\end{corollary}
\begin{proof}
Let $\prec$ be a degrevlex monomial order on a linear extension of the partial order given by the poset underlining the structure of ASL of $S/I$. Since $\ini(I)$ is square-free, Theorem \ref{eisenbudgotonoemb} implies the thesis.
\end{proof}

\index{Castelnuovo-Mumford regularity|)}\index{multiplicity|)}

%\begin{remark}
%Let $I$ be a homogeneous ideal of $S$ such that $I\subseteq \mm^2$, where $\mm=(x_1,\ldots ,x_n)$, and such that $S/I$ is Cohen-Macaulay. If $\kkk$ is infinite, we can find an $S/I$-regular sequence ${\bf f}=f_1,\ldots ,f_d$ of elements in $S_1$, where $d=\dim S/I$. We have that $h_{S/I}=h_{S/(I+({\bf f}))}$ (for example see \cite[Remark 4.1.11]{BH}). Moreover, we have that $\reg(S/I)=\reg(S/(I+({\bf f})))$. Notice also that $S/(I+({\bf f})) = S'/J$ where $S'=\kkk[y_1,\ldots ,y_c]$ is a polynomial ring in $c:=n-d$ variables over $\kkk$ and $J$ is a homogeneous ideal of $S'$ contained in $(y_1,\ldots ,y_c)^2$. So, since $\dim S'/J=0$, we have:
%\[\Hs_{S'/J}(z)=h_{S'/J}(z)=1+cz+h_2z^2+\ldots +h_tz^t.\]
%
%\end{remark}

\subsection{The initial ideal of a Cohen-Macaulay ideal}
\label{subchap2.1}

Combining Theorem \ref{ptdmv} with Proposition \ref{barney} we immediately get the following nice consequence:

\begin{thm}\label{skinner}
Let $I$ be a homogeneous ideal of $S$, and $\omega \in \NN_{\geq 1}^n$. If $\depth(S/I)\geq r+1$, then $S/\inito(\oo)$ is $r$-connected. In particular, if $S/I$ is Cohen-Macaulay, then $S/\init_{\omega}(I)$ is connected in codimension 1.
\end{thm}

At first blush, the reader might ask whether Theorem \ref{skinner} could be strengthened saying that ``$S/\sqrt{\inito(I)}$ is Cohen-Macaulay whenever $S/I$ is Cohen-Macaulay". However, this is far to be true, as we are going to show in the following example due to Conca.

\begin{example}\label{esempioaldo}
Consider the graded ideal:
\[I = (x_1x_5+x_2x_6+x_4^2, \ x_1x_4+x_3^2-x_4x_5, \ x_1^2+x_1x_2+x_2x_5) \subseteq \mathbb{C}[x_1, \ldots ,x_6]=S.\]
Using some computer algebra system, for instance Cocoa \cite{cocoa}, one can verify that $I$ is a prime ideal which is a complete intersection of height $3$. In particular, $S/I$ is a Cohen-Macaulay domain of dimension 3. Moreover, $S/I$ is normal. At last, the radical of the initial ideal of $I$ with respect to the lexicographical monomial order is:
\[\sqrt{\init(I)} = (x_1,x_2,x_3)\cap (x_1,x_3,x_6)\cap (x_1,x_2,x_5)\cap (x_1,x_4,x_5).\]
In accord with Theorem \ref{skinner}, $R:=S/\sqrt{\init(I)}$ is connected in codimension 1. However, $R$ is not Cohen-Macaulay. Were it,  $R_{\wp}$ would be Cohen-Macaulay, where $\wp$ is the homomorphic image of the prime ideal $(x_1,x_3,x_4,x_5,x_6)$. In particular, $R_{\wp}$ would be $1$-connected by Propsition \ref{barney}. The minimal prime ideals of $R_{\wp}$ are the homomorphic image of those of $R$ which are contained in $\wp$, namely $(x_1,x_3,x_6)$ and $(x_1,x_4,x_5)$. Since their sum is $\wp$, the ring $R_{\wp}$ cannot be $1$-connected.

\index{cohomological dimension}
This example shows also as the cohomological dimension of an ideal, in general, cannot be compared with the one of its initial ideal. In fact $\cd(S,I)=3$ because $I$ is a complete intersection of height $3$ (see \eqref{aragcd}). But $\cd(S,\init(I))=\operatorname{projdim}(R)>3$, where the equality follows by a result of Lyubeznik in \cite{Ly} (it is reported in Theorem \ref{lyubmon}). On the other hand, examples of ideals $J\subseteq S$ such that $\cd(S/J)>\cd(S/\init(J))$ can be easily produced following the guideline of Remark \ref{abu}.
\end{example}

In Example \ref{esempioaldo}, the dimension of $S/I$ is $3$. Moreover, starting from it, we can construct other similar example for any dimension greater than or equal to $3$. Oppositely, such an example cannot exist in the dimension $2$-case by the following result.

\begin{prop}\label{dim2squarefreein}
Let $I\subseteq S$ be a homogeneous ideal such that $\V_+(I)\subseteq \PP^{n-1}$ is a positive dimensional connected projective scheme and $\prec$ a monomial order. Then
\[\depth(S/\sqrt{\ini(I)})\geq 2.\]
In particular, $\depth(S/\sqrt{\ini(I)})\geq 2$ whenever $\depth(S/I)\geq 2$.
\end{prop}
\begin{proof}
Corollary \ref{conninit} implies that $\V_+(\sqrt{\ini(I)})\subseteq \PP^{n-1}$ is connected. As one can easily check, this is the case if and only if the associated simplicial complex $\D:=\D(\sqrt{\ini(I)})$ is connected. In turn, it is well known that this is the case if and only if the depth of $\kkk[\D]=S/\sqrt{\ini(I)}$ is at least $2$. The last part of the statement follows at once by Proposition \ref{barney}. 
\end{proof}

As we already said, the higher dimensional analog of the last part of the statement of Proposition \ref{dim2squarefreein} is not true. However, one can notice that the ideal of Example \ref{esempioaldo} is such that its initial ideal is not square-free. Actually, even if Conca attempted to find such an example with a square-free initial ideal, he could not find it. This facts lead him to formulate, in one of our informal discussions, the following question: 

\begin{question}\label{squarefreecm}
Let $I\subseteq S$ be a homogeneous ideal and suppose that $\prec$ is a monomial order on $S$ such that $\ini(I)$ is square-free. Is $\depth(S/I)=\depth(S/\ini(I))$?
\end{question}

Actually, we do not know whether to expect an affirmative answer to Question \ref{squarefreecm} or a negative one. In any case, we think it could be interesting to inquire into it. In the rest of this subsection, we are going to show some situations in which Question \ref{squarefreecm} has an affirmative answer. First of all, we inform the reader that it is well known that in general, even without the assumption that $\ini(I)$ is square-free, $\depth(S/I)\geq \depth(S/\ini(I))$ (for example see \cite[Corollary 3.5]{BC5}). So, the interesting part of Question \ref{squarefreecm} is whether, under the constraining assumption about the square-freeness of $\ini(I)$, the inequality $\depth(S/I)\leq \depth(S/\ini(I))$ holds true. The first result we present in this direction is that Question \ref{squarefreecm} has an affirmative answer in low dimension.

\begin{prop}\label{answerdim2}
Let $I\subseteq S$ be a homogeneous ideal and suppose that $\prec$ is a monomial order on $S$ such that $\ini(I)$ is square-free. If $\dim S/I \leq2$, then 
\[\depth(S/I)=\depth(S/\ini(I)).\] 
\end{prop}
\begin{proof}
Since $\depth(S/\ini(I))\geq 1$, being $\ini(I)$ square-free, the only nontrivial  case is when $\depth(S/I)=2$. In this case Proposition \ref{dim2squarefreein} supplies the conclusion.
\end{proof}

For the next result, we need to recall a definition: Let us suppose that $I\subseteq S$ is a homogeneous ideal, contained in $\mm^2$, such that $S/I$ is Cohen-Macaulay $d$-dimensional ring. The $h$-vector\index{h-vector@$h$-vector} of $S/I$ will have the form:
\[h_{S/I}(z)=1+(n-d)z+h_2z^2+\ldots+h_sz^s,\]
where the Cohen-Macaulayness of $S/I$ forces all the $h_i$'s to be natural numbers. Therefore, we have that
\[e(S/I)=h_{S/I}(1)\geq n-d+1.\]
The ring $S/I$ is said Cohen-Macaulay with {\it minimal multiplicity}\index{multiplicity}\index{multiplicity!minimal} when $e(S/I)$ is precisely $n-d+1$. We have the following:

\begin{prop}\label{answerminmult}
Let $I\subseteq S$ be a homogeneous ideal contained in $\mm^2$, and suppose that $\prec$ is a monomial order on $S$ such that $\ini(I)$ is square-free. If $S/I$ is Cohen-Macaulay with minimal multiplicity, then $S/\ini(I)$ is Cohen-Macaulay too.
\end{prop}
\begin{proof}
Obviously, $\ini(I)\subseteq \mm^2$. Furthermore, $e(S/\ini(I))=e(S/I)$ (for example see \cite[Proposition 1.4 (e)]{BC5}). So, using Theorem \ref{skinner}, we have that $S/\ini(I)$ is a $d$-dimensional Stanley-Reisner ring (where $d=\dim S/I$) connected in codimension $1$ such that $e(S/\ini(I))=n-d+1$. Then, a result of our paper with Nam \cite[Proposition 4.1]{NV} implies that $S/\ini(I)$ is Cohen-Macaulay with minimal multiplicity.
\end{proof}
\index{Algebra with Straightening Laws|(}

A particular case of Question \ref{squarefreecm} is whether $\depth(S/I)=\depth(S/I_{\Pi})$, where $S/I$ is a homogeneous ASL on a poset $\Pi$ over $\kkk$. The result below provides an affirmative answer to this question in a particular case. First, we need to recall that, if $\Pi=([n],\sqsubset)$ is a poset, then the {\it rank} of an element $i\in \Pi$, is the maximum $k$ such that there is a chain
\[i=i_k\sqsupset i_{k-1}\sqsupset \ldots \sqsupset i_1, \ \ \ i_j\in \Pi.\]

\begin{prop}\label{answersomeasl}
Let $I\subseteq S$ be a homogeneous ideal such that $S/I$ is an ASL on a poset $\Pi$ over $\kkk$. Assume that, for each $k\in [n]$, there are at most two elements of $\Pi$ of rank $k$. Then $S/I$ is Cohen-Macaulay if and only if $S/I_{\Pi}$ is Cohen-Macaulay.
\end{prop}
\begin{proof}
The only implication to show is that if $S/I$ is Cohen-Macaulay, then $S/I_{\Pi}$ is Cohen-Macaulay. Since $I_{\Pi}$ is the initial ideal of $I$ with respect to a degrevlex monomial order on a linear extension on $[n]$ of the poset order, Theorem \ref{skinner} yields that $S/I_{\Pi}$ is connected in codimension $1$. At this point, the result got by the author and Constantinescu \cite[Theorem 2.3]{CV1} implies that, because the peculiarity of $\Pi$, $S/I_{\Pi}$ is Cohen-Macaulay.
\end{proof}

\index{Algebra with Straightening Laws|)}
\index{connected|)}\index{connected!r-@$r$-|)}\index{connected!in codimension $1$|)}\index{simplicial complex|)}\index{simplicial complex!strongly connected|)}\index{initial ideal|)}

\section{The defining equations of certain varieties}\label{secarasegre}\index{arithmetical rank|(}\index{arithmetical rank!homogeneous|(}\index{set-theoretic complete intersection|(}\index{Segre product|(}

As we already said in the introduction, in this section, making use of results from the second part of Chapter \ref{chapter1}, we will estimate the number of equations needed to define certain projective schemes.

\subsection{Notation and first remarks}\label{aranotation}

We want to fix some notation that we will use throughout this section. By $\kkk$ we denote an algebraically closed field of arbitrary characteristic.
We recall that the {\it Segre product} of two finitely generated graded $\kkk$-algebra $A$ and $B$ is defined as
\[ A  \sharp B := \bigoplus_{k \in \NN} A_k \otimes_{\kkk} B_k. \]
This is a graded $\kkk$-algebra and it is clearly a direct summand of the tensor product $A \otimes_{\kkk} B$. The name ``Segre product", as one can expect, comes from Algebraic Geometry. In the rest of this chapter, we leave the setting of the last section: In fact, instead of working with the polynomial ring in $n$ variables $S=\kkk[x_1,\ldots ,x_n]$, we will work with the polynomial ring in $n+1$ variables 
\[R:=\kkk[x_0,\ldots ,x_n].\]
This is due to the fact that in this section we will often have a geometric point of view: Thus, just for the sake of notation, we prefer to work with $\PP^n$ rather than $\PP^{n-1}$. We also need another polynomial ring: Fixed a positive integer $m$, let 
\[T:=\kkk[y_0,\ldots ,y_m]\]
denote the polynomial ring in $m+1$ variables over $\kkk$. Let $X \subseteq \PP^n$ and $Y \subseteq \PP^m$ be two projective schemes defined respectively by the standard graded ideals $\aa \subseteq R$ and $\mathfrak{b}\subseteq T$. Set $A:=R/\aa$ and $B:=T/\mathfrak{b}$. Then, we have the isomorphism 
\[X\times Y \cong \Proj (A \sharp B),\]
where $X\times Y$ is the Segre product of $X$ and $Y$. Moreover, if 
\[Q:= \kkk[x_iy_j : i=0, \ldots ,n; \ \ j=0, \ldots, m] \subseteq \kkk[x_0, \ldots, x_n, y_0, \ldots , y_m]=R \otimes_{\kkk} T,\]
then $A  \sharp B \cong Q/\mathcal{I}$ with $\mathcal{I}\subseteq Q$ the homogeneous ideal we are going to describe:
If $\aa=(f_1, \ldots, f_r)$ and $\mathfrak{b}=(g_1, \ldots, g_s)$ with $\deg f_i=d_i$ and $\deg g_j=e_j$, then $\mathcal{I}$ is generated by the following polynomials:

\vskip 2mm

\begin{compactitem}
\item[(i)] $M \cdot f_i$ where $M$ varies among the monomials in $T_{d_i}$ for every $i=1, \ldots, r$.
\item[(ii)] $g_j \cdot N$ where $N$ varies among the monomial in $R_{e_j}$ for every $j=1, \ldots, s$.
\end{compactitem}

\vskip 2mm

\noindent We want to present $A \sharp B$ as a quotient of a polynomial ring. So, consider the polynomial ring in $(n+1)(m+1)$ variables over $\kkk$: 
\[P:=\kkk[z_{ij}:i=0, \ldots,n: \ j=0, \ldots,m].\]
Moreover, consider the $\kkk$-algebra homomorphism 
\[\phi: P \xrightarrow{\psi} Q \xrightarrow{\pi} A \sharp B,\]
where $\psi(z_{ij}):=x_i y_j$ and $\pi$ is just the projection. Therefore, we are interested into describe $I:= \Ker(\phi)$.
In fact, $I$ is the defining ideal of $X\times Y$, since we have:
\[ X\times Y \cong \Proj (P/I) \subseteq \PP^N, \ \ \ \ N:=nm+n+m.\]
Let us describe a system of generators of $I$. For any $i=1,\ldots , r$ and for all monomials $M \in T_{d_i}$, let us choose a polynomial $f_{i,M} \in P$ such that $\psi(f_{i,M})=M \cdot f_i$. Analogously, pick a polynomial $g_{j,N} \in P$ for all $j=1, \ldots,s$ and for each monomial $N \in R_{e_j}$. It turns out that
\[ I=I_2(Z)+J, \]
where:

\vskip 2mm

\begin{compactitem}
\item[(i)] $I_2(Z)$ denotes the ideal generated by the $2$-minors of the matrix $Z:=(z_{ij})$.
\item[(ii)] $J:=(f_{i,M},g_{j,N} \ : \ i\in [r], \ \ j\in [s], \ \mbox{ $M$ and $N$ are monomials of $T_{d_i}$ and $R_{e_j}$})$.
\end{compactitem}

\vskip 2mm

\noindent Our purpose is to study the defining equations (up to radical) of $I$ in $P$, and so to compute the arithmetical rank of $I$. In general, this is a very hard problem. It is enough to think that the case in which $\aa = \bb = 0$ is nothing but trivial (see \cite{bruns-schwanzl}). We will give a complete answer to this question in some other special cases.
We end this subsection remarking that, if we are interested in defining $X\times Y$ set-theoretically rather than ideal-theoretically, the number of equations immediately lowers a lot.
\begin{remark}\label{8}
It turns out that the number of polynomials generating the kernel of $P\rightarrow A\sharp B$ is, in general, huge. In fact, for any minimal generator $f_i$ of the ideal $\aa\subseteq A$, we have to consider all the polynomials $f_{i,M}$ with $M$ varying in $T_{d_i}$: These are $\binom{m+d_i}{m}$ polynomials! The same applies  for the minimal generators of $\bb\subseteq B$. At the contrary, up to radical, it is enough to choose $m+1$ monomials for every $f_i$ and $n+1$ monomials for every $g_j$, in a way we are going to explain.

For every $i= 1,\ldots, r$ and $l =0, \ldots, m$, set $M:=y_l^{d_i}$. A possible choice for $f_{i,M}$ is: 
\[f_{i,l}:=f_i(z_{0l}, \ldots, z_{nl}) \in P.\] 
In the same way, for every $j= 1,\ldots, s$ and $k =0, \ldots, n$ we define: 
\[g_{j,k}:=g_j(z_{k0}, \ldots, z_{km}) \in P.\] 
If we call $J'$ the ideal of $P$ generated by the $f_{i,l}$'s and the $g_{j,k}$'s, then we claim that:
\[ \sqrt{I}=\sqrt{I_2(Z)+J'}. \]
Since $\kkk$ is algebraically closed, Nullstellensatz implies that it is enough to prove that $\Z(I)=\Z(I_2(Z)+J')$, where $\Z(\cdot)$ denotes the zero locus. Obviously, we have $\Z(I)\subseteq \Z(I_2(Z)+J')$. So, pick a point
\[p:=[p_{00},p_{10}, \ldots ,p_{n0},p_{01}, \ldots, p_{n1}, \ldots ,p_{0m}, \ldots , p_{nm}] \in \mathcal{Z}(I_2(Z)+J').\] 
It is convenient to write $p=[p_0, \ldots, p_m]$, where $p_h:=[p_{0h}, \ldots, p_{nh}]$ is $[0,0, \ldots ,0]$ or a point of $\PP^n$. Since $p \in \Z(I_2(Z))$, it follows that the nonzero points among the $p_h$'s are equal as points of $\PP^n$. Moreover, if $p_h$ is a nonzero point, actually it is a point of $X$, because $f_{i,h}(p)=0$ for all $i=1, \ldots ,r$. Then, we get that $f_{i,M}(p)=0$ for every $i, M$ and any choice of $f_{i,M}$. By the same argument, we can prove that also all the $g_{j,N}$'s vanish at $p$, so we conclude. 

The authors of \cite{bruns-schwanzl} described $nm+n+m-2$ homogeneous equations defining $I_2(Z)$ up to radical. So, putting this information together with the discussion above, we got \ $nm+(s+1)n+(r+1)m+s+r$ \ homogeneous equations defining set-theoretically $X\times Y\subseteq \PP^N$.
\end{remark}

\subsection{Lower bounds for the number of defining equations}\label{seclowerbounds}
\index{cohomological dimension|(}\index{etale cohomological dimension|@\'etale cohomological dimension(}

In this subsection, we provide the necessary lower bounds for the number of (set-theoretically) defining equations of the varieties we are interested in. The results of this subsection will be immediate consequences of those of Chapter \ref{chapter1}. As the reader will notice, the lower bounds derivable from cohomological considerations are more general than the gettable upper bounds. However, in general, while to obtain upper bounds one can think up a lot of clever ad hoc arguments, depending on the situation, to get lower bounds, essentially, the only available tools are cohomological considerations. In the cases in which they does not work we are, for the moment, completely helpless.

\begin{prop}\label{lowerara}
Let $X$ and $Y$ be smooth projective schemes over $\kkk$ of positive dimension. Let $X\times Y\subseteq \PP^N$ be any embedding of the Segre product $X\times Y$, and let $I$ be the ideal defining it (in a polynomial ring in $N+1$ variables over $\kkk$). Then:
\[N-2\leq \ara(I)\leq \ara_h(I)\leq N.\]
\end{prop}
\begin{proof}
The fact that $\ara_h(I)\leq N$ is a consequence of \cite[Theorem 2]{eisenbudevans}. Combining \eqref{aragecd} with \eqref{ecdecd}, we get that
\[\ara(I)\geq \ecd(U)-N+1,\]
where $U:=\PP^N\setminus (X\times Y)$. Eventually, Proposition \ref{14} implies $\ecd(U)\geq 2N-3$, so we conclude.
\end{proof}

The lower bound of Proposition \ref{lowerara} can be improved by one when $X$ is a smooth curve of positive genus.

\begin{prop}\label{curve}
Let $X$ be a smooth projective curve over $\kkk$ of positive genus and $Y$ a smooth projective scheme over $\kkk$. Let $X\times Y\subseteq \PP^N$ be any embedding of the Segre product $X\times Y$, and let $I$ be the ideal defining it (in a polynomial ring in $N+1$ variables over $\kkk$). Then:
\[N-1\leq \ara(I)\leq \ara_h(I)\leq N.\]
\end{prop}
\begin{proof}
Once again, the fact that $\ara_h(I)\leq N$ is a consequence of \cite[Theorem 2]{eisenbudevans}. Set $g$ the genus of $X$, and choose an integer $\ell$ prime with $\chara(\kkk)$. It is well known (see for instance the notes of Milne \cite[Proposition 14.2 and Remark 14.4]{milnel}) that $H^1(X_{\et},\ZZ/\ell\ZZ)\cong (\ZZ/l\ZZ)^{2g}$. Moreover $H^0(Y_{\et},\ZZ/\ell \ZZ)\neq 0$ and $H^1(\PP^N_{\et},\ZZ/\ell \ZZ)=0$. But by K\"unneth formula for \'etale cohomology (for instance see the book of Milne \cite[Chapter VI, Corollary 8.13]{milne}) $H^1((X\times Y)_{\et},\ZZ/\ell \ZZ)\neq 0$, therefore Theorem \ref{13} implies that $\ecd(U)\geq 2N-2$, where $U:=\PP^N\setminus (X\times Y)$. At this point we can conclude as in Proposition \ref{lowerara}, since
\[\ara(I)\geq \ecd(U)-N+1.\]. 
\end{proof}
\index{cohomological dimension|)}\index{etale cohomological dimension|@\'etale cohomological dimension)}
\subsection{Hypersurfaces cross a projective spaces}

Proposition \ref{lowerara} implies that the number of equations defining set-theoretically $X\times Y\subseteq \PP^N$ described in Remark \ref{8}, where $N=nm+n+m$, are still too much. In this subsection we will improve the upper bound given in Proposition \ref{lowerara} in some special cases. In some situations we will determine the exact number of equations needed to define $X\times Y$ set-theoretically. 

\begin{remark}\label{9}
Assume that $X:=\mathcal{V}_+(F) \subseteq \mathbb{P}^n$ is a projective hypersurface defined by a homogeneous polynomial $F$, $m:=1$ and $Y:=\PP^1$. In this case $X\times Y\subseteq \PP^{2n+1}$ and Remark \ref{8} gives us the same upper bound for the arithmetical rank of the ideal $I$ defining $X\times Y$ than Proposition \ref{lowerara}, namely 
\[\ara_h(I)\leq 2n+1.\]
Remark \ref{8} also gets an explicit set of homogeneous polynomials generating $I$ up to radical. Using the notation of the book of Bruns and Vetter \cite{BrVe}, we denote by $[i,j]$ the $2$-minor $z_{i0}z_{j1}-z_{j0}z_{i1}$ of the matrix $Z$, for every $i$ and $j$ such that $0 \leq i<j \leq n$. In \cite{bruns-schwanzl} is proven that
\[ I_2(Z) = \sqrt{(\sum_{i+j=k} [i,j] : k=1, \ldots, 2n-1)} \]
By Remark \ref{8}, to get a set of generator of $I$ up to radical, we have only to add 
\[F_0:=F(z_{00},\ldots ,z_{n0}) \mbox{ \ \ and \ \ }F_1:=F(z_{01},\ldots ,z_{n1}).\]
Notice that the generators up to radical we exhibited have the (unusual) property of being part of a set of minimal generators of $I$.
\end{remark}

\begin{thm}\label{10}
Let $X=\mathcal{V}_+(F) \subseteq \mathbb{P}^n$ be a hypersurface such that there exists a line $\ell \subseteq \mathbb{P}^n$ that meets $X$ only at a point $p$. If $I$ is the ideal defining $X \times \mathbb{P}^1 \subseteq \mathbb{P}^{2n+1}$, then:
\[ \ara_h(I) \leq 2n \]
\end{thm}

\begin{proof}
By a change of coordinates we can assume that $\ell=\mathcal{V}_+((x_0, \ldots, x_{n-2}))$. The set 
\[\Omega:=\{ [i,j] : 0\leq i<j\leq n, i+j \leq 2n-2 \}\]
is an ideal of the poset of the minors of the matrix $Z=(z_{ij})$: That is, for any two positive integers $h$ and $k$ such that $h <k$, if $[i,j]\in \Omega$ with $h\leq i$ and $k\leq j$, then $[h,k]\in \Omega$. So, \cite[Lemma 5.9]{BrVe} implies: 
\[ \ara(\Omega P) \leq \mbox{rank}(\Omega)=2n-2 \]
(let us remind that $P=\kkk[z_{ij}:i=0,\ldots ,n \ j=0,1]$). We want to prove that $I=\sqrt{K}$, where $K:=\Omega P + (F_0,F_1)$ (with the notation of Remark \ref{9}). To this aim, by Remark \ref{9} and Nullstellensatz, it is enough to prove $\mathcal{Z}(I_2(Z)+(F_0,F_1))=\mathcal{Z}(K)$. So, set 
\[q:=[q_0,q_1]=[q_{00}, \ldots, q_{n0},q_{01}, \ldots, q_{n1}] \in \mathcal{Z}(K).\]
If $q_0=0$ or $q_1=0$ trivially $q\in \Z(I_2(Z))$, so we assume that both $q_0$ and $q_1$ are points of $\mathbb{P}^n$. First let us suppose that $q_{ij}\neq 0$ for some $i \leq n-2$ and $j\in \{0,1\}$. Let us assume that $j=0$ (the case $j=1$ is the same). In such a case, notice that for any $h\in \{0,\ldots ,n\}$ different from $i$ we have that $[h,i]$ (or $[i,h]$) is an element of $\Omega$. Because $q\in \Z(K)$, we get that, setting $\ll:=q_{i1}/q_{i0}$:
\[q_{h1}=\ll q_{h0} \ \ \ \forall \ h=0,\ldots ,n.\] 
This means that $q_1$ and $q_0$ are the same point of $\PP^n$, so that $q \in \mathcal{Z}(I_2(Z))$. We can therefore assume that $q_{ij}=0$ for all $i\leq n-2$ and $j=0,1$. In this case, $q_0$ and $q_1$ belong to $ \ell \cap X=\{p\}$, so $q_0=q_1=p$. Once again, this yields $q\in \Z(I_2(Z))$.
\end{proof}

Combining Theorem \ref{10} and Proposition \ref{lowerara}, we get the following corollary. 

\begin{corollary}
Let $X \subseteq \mathbb{P}^n$ be a smooth hypersurface such that there exists a line $\ell \subseteq \mathbb{P}^n$ that meets $X$ only at a point $p$. If $I$ is the ideal defining $X \times \mathbb{P}^1 \subseteq \mathbb{P}^{2n+1}$, then:
\[ 2n-1\leq \ara(I)\leq \ara_h(I) \leq 2n \]
\end{corollary}

Besides, combining Theorem \ref{10} and Proposition \ref{curve}, we get the following result.

\begin{thm}\label{ara1}
Let $X \subseteq \mathbb{P}^2$ be a smooth projective curve of degree $d\geq 3$ over $\kkk$. Assume that there exists a line $\ell \subseteq \mathbb{P}^2$ that meets $X$ only at a point $p$, and let $I$ be the ideal defining the Segre product $X \times \mathbb{P}^1 \subseteq \mathbb{P}^{5}$.
Then
\[ \ara(I)=\ara_h(I) =4. \]
\end{thm}
\begin{proof}
We recall that the genus $g$ of the curve $X$ is given by the formula:
\[g=1/2(d-1)(d-2).\] 
In particular, under our assumptions, it is positive. Thus Propsition \ref{curve}, together with Theorem \ref{10}, lets us conclude.
\end{proof}

\begin{remark}
In Theorem \ref{ara1}, actually, we prove something more than the arithmetical rank being $4$. If $X=\V_+(F)$, then the four polynomials generating $I$ up to radical are the following:
\[F(z_{00},z_{10},z_{20}), \ \ F(z_{10},z_{11},z_{12}), \ \ z_{00}z_{11}-z_{01}z_{10}, \ \ z_{00}z_{12}-z_{02}z_{10}.\]
It turns out that these polynomials are part of a minimal generating set of $I$. This is a very uncommon fact.
\end{remark}

Theorem \ref{ara1} generalizes \cite[Theorem 1.1]{siwa} and \cite[Theorem 2.8]{song}. In fact we get them in the next two corollaries.

\begin{corollary}\label{araelliptic}
Let $X \subseteq \mathbb{P}^2$ be a smooth elliptic curve over $\kkk$ and let $I$ be the ideal defining the Segre product $X \times \mathbb{P}^1 \subseteq \mathbb{P}^{5}$.
Then
\[ \ara(I)=\ara_h(I) =4. \]
\end{corollary}
\begin{proof}
It is well known that any plane projective curve $X$ of degree at least three has an ordinary flex, or an ordinary inflection point (see the book of Hartshorne \cite[Chapter IV, Exercise 2.3 (e)]{hart}). That is, a point $p\in X$ such that the tangent line $\ell$ at $p$ has intersection multiplicity $3$ with $X$ at $p$. If $X$ has degree exactly $3$, which actually is our case, such a line $\ell$ does not intersect $X$ in any other point but $p$. Therefore we are under the assumptions of Theorem \ref{ara1}, so we may conclude.
\end{proof}

\begin{remark}
The equations exhibited in Corollary \ref{araelliptic} are different from those found in \cite[Theorem 1.1]{siwa}. In fact, our result is characteristic free, while the authors of \cite{siwa} had to assume $\chara(\kkk)\neq 3$.
\end{remark}

\begin{corollary}\label{arafermat}
Let $X=\V_+(F) \subseteq \mathbb{P}^2$ be a Fermat curve of degree $d\geq 2$ over $\kkk$, that is $F=x_0^d+x_1^d+x_2^d$,  and let $I$ be the ideal defining the Segre product $X \times \mathbb{P}^1 \subseteq \mathbb{P}^{5}$. If $\chara(\kkk)$ does not divide $d$,
then:
\[ \ara(I)=\ara_h(I) =4. \]
\end{corollary}
\begin{proof}
Let $\ll \in \kkk$ be such that $\ll^d=-1$. So, let $\ell\subseteq \PP^2$ be the line defined by the linear form $x_0-\ll x_1$. One can easily check that: 
\[X\cap \ell =\{[\ll ,1,0]\}.\]
Eventually, since $\chara(\kkk)$ does not divide $d$, $X$ is smooth, so Theorem \ref{ara1} lets us conclude. 
\end{proof}

\begin{remark}
In the situation of Corollary \ref{arafermat}, if $\chara(\kkk)$ divides $d$, then $X\times \PP^1$ is a set-theoretic complete intersection in $\PP^5$, that is $\ara(I)=\ara_h(I)=3$. In fact, more in general, this is always the case when $F=\ell^d$, where $\ell\in \kkk[x_0,x_1,x_2]_1$, independently of the characteristic of $\kkk$. To see this, up to a change of coordinates we can assume that $\ell=\V_+(x_2)$. Using the Nullstellensatz as usual, it is easy to see that: 
\[\sqrt{(z_{20}, \ z_{21}, \ z_{00}z_{11}-z_{01}z_{10})}=\sqrt{I}.\]
On the other hand, $\ara(I)\geq 3$ by the Hauptidealsatz \eqref{hauptidealsatz}. 
\end{remark}

Curves satisfying the hypothesis of Theorem \ref{ara1}, however, are much more than those of Corollaries \ref{araelliptic} and \ref{arafermat}. The next remark will clarify this point.

\begin{remark}\label{hyperflexes}
In light of Theorem \ref{10}, it is natural to introduce the following set. For every natural numbers $n,d \geq 1$ we define $\V_d^{n-1}$ as the set of all the smooth hypersurfaces $X\subseteq \PP^n$, modulo $\operatorname{PGL}_n(\kkk)$, of degree $d$ which have a point $p$ as in Theorem \ref{10}. Notice that any hypersurface of $\V_d^{n-1}$ can be represented, by a change of coordinates, by $\V_+(F)$ with 
\[ F=x_{n-1}^d + \sum_{i=0}^{n-2} x_i G_i(x_0, \ldots, x_{n}),\] 
where  the $G_i$'s are homogeneous polynomials of degree $d-1$.

We start to analyze the case $n=2$, and for simplicity we will write $\V_d$ instead of $\V_d^1$.  So our question is: \emph{How many smooth projective plane curves of degree $d$ do belong to $\V_d$?} Some plane projective curves belonging to $\V_d$ are:
\begin{compactitem}
\item[(i)] Obviously, every smooth conic belongs to $\V_2$.
\item[(ii)] Every smooth elliptic curve belongs to $\V_3$, as we already noticed in the proof of Corollary \ref{araelliptic}.
\item[(iii)] Every Fermat's curve of degree $d$ belongs to $\V_d$, as we already noticed in Corollary \ref{arafermat}.
\end{compactitem}
In \cite[Theorem A]{casnati-del centina}, the authors compute the dimension of the loci $\mathcal{V}_{d, \alpha}$, $\alpha=1,2$, of all the smooth plane curves of degree $d$ with exactly $\alpha$ points as in Theorem \ref{10}  (if these points are nonsingular, as in this case, they are called $d$-flexes). They showed that $\mathcal{V}_{d, \alpha}$ is an irreducible rational locally closed subvariety of the moduli space $\mathcal{M}_g$ of curves of genus $g=1/2(d-1)(d-2)$. Furthermore the dimension of $\mathcal{V}_{d, \alpha}$ is
\[ \dim (\mathcal{V}_{d, \alpha})={d+2-\alpha \choose 2}-8+3 \alpha. \]
Moreover, it is not difficult to show that $\mathcal{V}_{d, 1}$ is an open Zariski subset of $\mathcal{V}_d$, (see \cite[Lemma 2.1.2]{casnati-del centina}), and so
\[ \dim (\mathcal{V}_d)={d+1 \choose 2}-5.  \]
%Now, it is well known that $\mathcal{M}_g$ is an irreducible quasi-projective variety
%of dimension $3g-3$, provided $g \geq 2$ ($\mathcal{M}_1$ has dimension 1, using the $j$-invariant). So we obtain that the codimension of $\mathcal{V}_d$ in $\mathcal{M}_g$ ($g={d-1 \choose 2}$) is $d^2-5d+5$. However, in $\mathcal{M}_g$ there are also not-plane curves of genus $g$, 
The locus $\mathcal{H}_d$ of all smooth plane curves of degree $d$ up to isomorphism is a nonempty open Zariski subset of $\mathbb{P}^{{d+2 \choose 2}}$ modulo the group $\operatorname{PGL}_2(\mathbb{\kkk})$, so 
\[\dim(\H_d)={d+2 \choose 2}-9.\] 
Particularly, the codimension of $\mathcal{V}_d$ in $\mathcal{H}_d$, provided $d\geq 3$, is $d-3$. So, for example, if we pick a quartic $X$ in the {\it hypersurface} $\V_4$ of $\mathcal{H}_4$, Theorem \ref{ara1} implies that $X\times \PP^1\subseteq \PP^5$ can be defined by exactly four equations. However, {\it it remains an open problem to compute the arithmetical rank of $X\times \PP^1 \subseteq \PP^5$ for all the quartics $X\subseteq \PP^2$}. 
\end{remark}

In the general case ($n\geq 2$ arbitrary) we can state the following lemma.

\begin{lemma}\label{cil}
Let $X\subseteq \PP^n$ be a smooth hypersurface of degree $d$. If $d\leq 2n-3$, or if $d\leq 2n-1$ and $X$ is generic, then $X\in \V^{n-1}_d$.
\end{lemma}
\begin{proof}
First we prove the following:

\vskip 2mm

{\it Claim}. If $X\subseteq \PP^n$ is a smooth hypersurface of degree $d\leq 2n-1$ not containing lines, then $X\in \V^{n-1}_d$.

We denote by $\GG(2,n+1)$ the Grassmannian of lines of $\PP^n$. Let us introduce the following incidence variety: 
\[W_n:=\{(p,\ell)\in \PP^n \times  \GG(2,n+1): p\in \ell\}.\] 
It turns out that this is an irreducible variety of dimension $2n-1$. Now set 
\[T_{n,d}:=\{(p,\ell,F)\in W_n \times L_{n,d}: \ i(\ell,\V_+(F);p)\geq d\},\]
where by $L_{n,d}$ we denote the projective space of all the polynomials of $\kkk[x_0,\ldots ,x_n]_d$, and by $i(\ell,\V_+(F);p)$ the intersection multiplicity of $\ell$ and $\V_+(F)$ at $p$ (if $\ell\subseteq \V_+(F)$, then $i(\ell,\V_+(F);p):=+\infty$). 
Assume that $p=[1,0,\ldots ,0]$ and that $\ell$ is given by the equations 
\[x_1=x_2=\ldots =x_n=0.\] 
Then it is easy to see that, for a polynomial $F\in L_{n,d}$, the condition $(p,\ell,F)\in T_{n,d}$ is equivalent to the fact that the coefficients of $x_0^d, \  x_0^{d-1}x_1, \ \ldots , \ x_0x_1^{d-1}$ in $F$ are $0$. This implies that $T_{n,d}$ is a closed subset of $\PP^n\times \GG(2,n+1)\times L_{n,d}$: Therefore, $T_{n,d}$ is a projective scheme over $\kkk$.
Consider the restriction of the first projection $\pi_1:T_{n,d} \longrightarrow W_n$. Clearly $\pi_1$ is surjective; moreover, it follows by the above discussion that all the fibers of $\pi_1$ are projective subspaces of $L_{n,d}$ of dimension $\dim(L_{n,d})-d$. Therefore, $T_{n,d}$ is an irreducible projective variety of dimension 
\[\dim(T_{n,d})=2n-1 +\dim(L_{n,d})-d.\] 
Now consider the restriction of the second projection $\pi_2: T_{n,d} \longrightarrow L_{n,d}$. Clearly, if $X\in \pi_2(T_{n,d})$ is smooth and does not contain any line, then $X\in \V^{n-1}_d$. So, to prove the claim, we have to check the surjectivity of $\pi_2$ whenever $d\leq 2n-1$. To this aim, since both $T_{n,d}$ and $L_{n,d}$ are projective, it is enough to show that for a general $F\in \pi_2(T_{n,d})$, the dimension of the fiber $\pi_2^{-1}(F)$ is exactly $2n-1-d$. On the other hand it is clear that the codimension of $\pi_2(T_{n,d})$ in $T_{n,d}$ is at least $d-2n+1$ when $d\geq 2n$. We proceed by induction on $n$ (for $n=2$ we already know this).

First consider the case in which $d\leq 2n-3$. Let $F$ be a general form of $\pi_2(T_{n,d})$, and set $r=\dim(\pi_2^{-1}(F))$. By contradiction assume that $r>2n-1-d$. Consider a general hyperplane section of $\V_+(F)$, and let $F'$ be the polynomial defining it. Obviously, any element of $\pi_2(T_{n-1,d})$ comes from $\pi_2(T_{n,d})$ in this way, so $F'$ is a generic form of $\pi_2(T_{n,d})$. The condition for a line to belong to a hyperplane is of codimension $2$, so the dimension of the fiber of $F'$ is at least $r-2$. Since $F'$ is a polynomial of $K[x_0,\ldots ,x_{n-1}]$ of degree $d\leq 2(n-1)-1$, we can apply an induction getting $r-2\leq2n-3-d$, so that $r\leq 2n-1-d$, which is a contradiction. 

We end with the case in which $d=2n-1$ (the case $d=2n-2$ is easier). Let $F$ and $r$ be as above, and suppose by contradiction that $r\geq 1$. This implies that there exists a hypersurface $\H\subseteq \GG(n,n+1)$ such that for any general $H\in \H$ the polynomial defining $\V_+(F)\cap H$ belongs to $\pi_2(T_{n-1,d})$. This implies that the codimension of $\pi_2(T_{n-1,d})$ in $T_{n-1,d}$ is less than or equal to $1$, but we know that this is at least $2$.  

\vskip 2mm

So we proved the claim. Now, we prove the lemma by induction on $n$. Once again, if $n=2$, then it is well known to be true.

If $d\leq 2n-3$, then we cut $X$ by a generic hyperplane $H$. It turns out (using Bertini's theorem) that $X\cap H \subseteq \PP^{n-1}$ is the generic smooth hypersurface of degree $d\leq 2(n-1)-1$, so by induction there exist a line $\ell\subseteq H$ and a point $p\in \PP^n$ such that $(X\cap H)\cap \ell=\{p\}$. So we conclude that $X\in \V_d^{n-1}$.

It is known that the generic hypersurface of degree $d\geq 2n-2$ does not contain lines. So if $d=2n-2$ or $d=2n-1$ the statement follows by the claim.
\end{proof}

Combining Lemma \ref{cil} with Theorem \ref{10} we get the following.

\begin{corollary}\label{bello}
Let $X\subseteq \PP^n$ be a smooth hypersurface of degree $d$, and let $I$ be the ideal defining the Segre product $X \times \mathbb{P}^1 \subseteq \mathbb{P}^{2n+1}$. If $d\leq 2n-3$, or if $d\leq 2n-1$ and $X$ is generic, then 
\[ 2n-1\leq \ara_h(I) \leq 2n. \]
\end{corollary}

Putting some stronger assumptions on the hypersurface $X$, we can even compute the arithmetical rank of the ideal $I$ (and not just  to give an upper bound as in Theorem \ref{10}). 
%To this purpose, for every $n,d$ we introduce a new locus, which we will denote by $\C_d^{n-1}$,
%as the class of isomorphisms of smooth hypersurfaces of degree $d$ in $\PP^{n}$ such that there exists an hyperplane $H \subseteq \PP^{n}$ such that $X \cap H$ (only set-theoretically) is a linear space of codimension $2$. 
%Notice that 
%\[\C_d^{n-1} \subseteq \V_d^{n-1} \mbox{ \ \ and \ \ }\C_d^1= \V_d^1.\]
%Note also that all hypersurfaces in $\C_d^{n-1}$ can be represented, by a change of coordinate, by $\V_+(F)$ with $F=x_{n}^d + x_0 G(x_0, \ldots, x_{n})$, where $G$ is a homogeneous polynomial of degree $d-1$.

\begin{thm}\label{45}
Let $X=\V_+(F) \subseteq \PP^n$ be such that, $F=x_n^d+\sum_{i=0}^{n-3}x_iG_i(x_0,\ldots ,x_n)$ ($G_i$ homogeneous polynomials of degree $d-1$), and let $I$ be the ideal defining the Segre product $X \times \mathbb{P}^1 \subseteq \mathbb{P}^{2n+1}$. Then
\[ \ara_h(I) \leq 2n-1. \]
Moreover, if $X$ is smooth, then
\[\ara(I)= \ara_h(I) = 2n-1. \]
\end{thm}
\begin{proof}
If $X$ is smooth, Proposition \ref{lowerara} implies that $\ara(I)\geq 2n-1$.
Therefore,we need to prove that the upper bound holds true.
Consider the set 
\[\Omega:=\{ [i,j] : i<j, i+j \leq 2n-3 \}.\] 
As in the proof of Theorem \ref{10}, we have
\[ \ara(\Omega R) \leq \mbox{rank}(\Omega)=2n-3 .\]
The rest of the proof is completely analog to that of Theorem \ref{10}.
%We will prove, assuming $k$ algebraically closed, that $\mathcal{Z}(I)=\mathcal{Z}(J)$,  where $J=\Omega R + (F_0,F_1)$, and then we will use Nullstellensatz to conclude.\\
%Therefore, let be $Q=[Q_0;Q_1]=[Q_{01}, \ldots, Q_{0n},Q_{10}, \ldots, Q_{1n}] \in \mathcal{Z}(J)$. If $Q_{i0}\neq 0$ for some $i=0,1$, then $Q \in \mathcal{Z}(I_2(Z))$, and then it is forced to stay also in $\mathcal{Z}(I)$. However, if $Q_{i0}=0$ for $i=0,1$, then, by the equation of $X$, $Q_{in}=0$ too, and so $Q \in \mathcal{Z}(I_2(Z))$ and we can easily obtain the thesis.
\end{proof}

\begin{remark}
Notice that, if $n\geq 4$, the generic hypersurface of $\PP^n$ defined by the form $F=x_n^d+\sum_{i=0}^{n-3}x_iG_i(x_0,\ldots ,x_n)$ is smooth (whereas if $n\leq 3$ and $d\geq 2$ such a hypersurface is always singular).
\end{remark}

The below argument uses ideas from \cite{siwa}. Unfortunately, to use these kinds of tools, we have to put some assumptions to $\chara(\kkk)$.
\begin{thm}\label{conic}
Assume $\chara(\kkk)\neq 2$. Let $X=\V_+(F)$ be a smooth conic of $\PP^2$, and let $I$ be the ideal defining the Segre product $X\times \PP^m \subseteq \PP^{3m+2}$. Then
\[\ara(I)=\ara_h(I)=3m.\]
In particular $X\times \PP^m\subseteq \PP^{3m+2}$ is a set-theoretic complete intersection if and only if $m=1$.
\end{thm}
\begin{proof}
First we want to give $3m$ homogeneous polynomials of the polynomial ring $P=\kkk[z_{ij}: i=0,1,2, \ \ j=0, \ldots ,m]$ which define $I$ up to radical.
For $j=0, \ldots ,m$ choose set, similarly to  Remark \ref{8},
\[F_j:=F(z_{0j},z_{1j},z_{2j})\] 
Then, for all $0\leq j<i \leq m$, set
\[F_{ij}:=\sum_{k=0}^2 \displaystyle \frac{\de F}{\de x_k}(z_{0i}, z_{1i} , z_{2i})z_{jk}. \]
Eventually, for all $h=1, \ldots ,2m-1$, let us put: 
\[G_h:=\sum_{i+j=h}F_{ij}.\]  
We claim that
\[ I=\sqrt{J}, \mbox{ where } J:=(F_i,G_j:i=0, \ldots ,m, \ \ j=1,  \ldots ,2m-1). \]
The inclusion $J \subseteq I$ follows from the Euler's formula, since $\chara(\kkk)\neq 2$.
As usual, to prove $I\subseteq \sqrt{J}$, thanks to Nullstellensatz, we may prove that $\Z(J) \subseteq \Z(I)$. Pick $p\in \Z(J)$, and write it as 
\[p:=[p_0,p_1,\ldots ,p_m] \mbox{ \ \ where \ \ } p_j:=[p_{0j},p_{1j},p_{2j}].\]
Since  $F_i(p)=0$, for every $i=0, \ldots ,m$, the nonzero $p_i$'s are points of $X$. So, by Remark \ref{8}, it remains to prove that the nonzero $p_i$'s are equal as points of $\PP^2$.

By contradiction, let $i$ be the minimum integer such that $p_i \neq 0$ and there exists $k$ such that $p_k \neq 0$ and $p_i \neq p_k$ as points of $\PP^2$. Moreover let $j$ be the least among these $k$ (so $i<j$). Set $h:=i+j$. We claim that $p_k=p_l$ provided that $k+l=h$, $k<l$, $k\neq i$, $p_k \neq 0$ and $p_l\neq 0$.
In  fact, if $l<j$, then $p_i=p_l$ by the minimal property of $j$. For the same reason, also $p_k=p_i$, so $p_k=p_l$. On the other hand, if $l>j$, then $k<i$, so $p_k=p_l$ by the minimality of $i$. So $F_{lk}(p)=0$ for any $(k,l)\neq (i,j)$ such that $k+l=h$, because $p_k$ belongs to the tangent line of $X$ in $p_l$ (being $p_l=p_k$). Then $G_h(p)=F_{ji}(p)$, and so, since $p \in \Z(J)$, $F_{ji}(p)=0$: This means that $p_i$ belongs to the tangent line of $X$ in $p_j$, which is possible, being $X$ a conic, only if $p_i=p_j$, a contradiction.

For the lower bound, we just have to  notice that, in this case, Proposition \ref{lowerara} yields $\ara(I)\geq 3m$.
\end{proof}

\begin{remark}
B\u adescu and Valla, computed recently in \cite{bava}, independently from this work, the arithmetical rank of the ideal defining any rational normal scroll. Since the Segre product of a conic with $\PP^m$ is a rational normal scroll, Theorem \ref{conic} is a particular case of their result.
\end{remark}

We end this subsection with a result that yields a natural question.

\index{cohomological dimension}
\begin{prop}\label{1000}
Let $n\geq 2$ and $m\geq 1$ be two integers. Let $X=\V_+(F) \subseteq \PP^n$ be a hypersurface smooth over $\kkk$ and let $I \subseteq P=\kkk[z_0, \ldots ,z_N]$ be the ideal defining an embedding $X \times \mathbb{P}^m \subseteq \mathbb{P}^N$. Then
\begin{displaymath}
\cd(P,I) = \left\{ \begin{array}{ll} N-1 & \mbox{if \ } n=2 \mbox{ \ and \ }\deg(F)\geq 3\\
N-2 &  \mbox{otherwise} \end{array}  \right.
\end{displaymath}
\end{prop}
\begin{proof}
By Remark \ref{0} we can assume $\kkk=\CC$. If $Z:=X\times \PP^m\subseteq \PP^N$, Using equation \eqref{betti1} we have 
\[ \beta_0(Z)=1, \mbox{ \ \ } \beta_1(Z)=\beta_1(X) \mbox{ \ \ and \ \ } \beta_2(Z)=\beta_2(X)+1 \geq 2,\]
where the last inequality follows by \eqref{betti3}. If $n=2$, notice that $\beta_1(X)\neq 0$ if and only if $\deg(F)\geq 3$. In fact, equation \eqref{betti2} yields 
\[\beta_1(X)=h^{01}(X)+h^{10}(X)=2h^{01}(X),\]
where the last equality comes from Serre's duality (see \cite[Chapter III, Corollary 7.13]{hart}). But $h^{01}(X)$ is the geometric genus of $X$, therefore it is different from $0$ if and only if $\deg(F)\geq 3$. So if $n=2$ we conclude by Theorem \ref{1}.
If $n>2$, then we have $\beta_1(X)=0$ by the Lefschetz hyperplane theorem (see the book of Lazarsfeld \cite[Theorem 3.1.17]{la}, therefore we can conclude once again using Theorem \ref{1}.
\end{proof}

In light of the above proposition, it is natural the following question.

\begin{question}
With the notation of Proposition \ref{1000}, if we consider the classical Segre embedding of $X\times \PP^m$ (and so $N=nm+n+m$), do the integers $\ara(I)$ and $\ara_h(I)$ depend only on $n$, $m$ and $\deg(F)$?
\end{question}

\subsection{The diagonal of the product of two projective spaces}

In \cite{spe} Speiser, among other things, computed the arithmetical rank of the diagonal 
\[ \Delta = \Delta(\PP^n)\subseteq \PP^n \times \PP^n, \] 
provided that the characteristic of the base field is $0$. In positive characteristic he proved that the cohomological dimension of $\PP^n\times \PP^n \setminus \Delta$ is the least possible, namely $n-1$, but he did not compute the arithmetical rank of $\Delta$. In this short subsection we will give a characteristic free proof of Speiser's result. Actually Theorem \ref{13} easily implies that the result of Speiser holds in arbitrary characteristic, since the upper bound found in \cite{spe} is valid in any characteristic. However, since in that paper the author did not describe the equations needed to define set-theoretically $\Delta$, we provide the upper bound with a different method, that yields an explicit set of equations for $\Delta$.

To this aim, we recall that the coordinate ring of $\PP^n \times \PP^n$ is 
\[A:=\kkk[x_iy_j:i,j=0, \ldots ,n]\] 
and the ideal $I\subseteq A$ defining $\Delta$ is 
\[I:=(x_iy_j-x_jy_i:0\leq i<j\leq n).\]  

\begin{prop}\label{spei}
In the situation described above $\ara(I)=\ara_h(I)=2n-1$.
\end{prop}
\begin{proof}
As said above, by \cite[Proposition 2.1.1]{spe} we already know that $\ara_h(I)\leq 2n-1$. However, we want to exhibit a new proof of this fact: Let us consider the following over-ring of $A$ 
\[R:=\kkk[x_i,y_j:i,j=0, \ldots ,n].\]
Notice that $A$ is a direct summand of the $A$-module $R$. The extension of $I$ in $R$, namely $IR$, is the ideal generated by the 2-minors of the following $2\times (n+1)$ matrix:
\[
\begin{pmatrix}
x_0&\ldots&x_n\\
y_0&\ldots&y_n
\end{pmatrix}.
\]
So, by \cite[(5.9) Lemma]{BrVe}, a set of generators of $IR\subseteq R$ up to radical is
\[g_k := \sum_{{0\leq i<j\leq n}\atop {i+j=k}}(x_iy_j-x_jy_i), \ \ \ k=1, \ldots ,2n-1.\]
Since these polynomials belong to $A$ and since $A$ is a direct summand of $R$, we get 
\[ \sqrt{(g_1, \ldots ,g_{2n-1})A}=I, \]
therefore 
\[\ara(I)\leq\ara_h(I)\leq 2n-1.\]

For the lower bound choose $\ell$ coprime with $\chara(\kkk)$. K\"unneth formula for \'etale cohomology \cite[Chapter VI, Corollary 8.13]{milne} implies that 
\[H^2(\PP_{\et}^n\times \PP_{\et}^n, \ZZ/\ell\ZZ) \cong (\ZZ/\ell \ZZ)^2,\]
while $H^2(\Delta_{\et} , \ZZ/\ell\ZZ)\cong H^2(\PP_{\et}^n , \ZZ/\ell\ZZ)\cong \ZZ/\ell\ZZ$. So, Theorem  \ref{13} yields 
\[\ecd(U)\geq 4n-2,\]
where $U:=\PP^n \times \PP^n \setminus \Delta$. Therefore, combining \eqref{ecdecd} and \eqref{aragecd}, we get
\[2n-1\leq \ara(I)\leq\ara_h(I).\]
\end{proof}

\index{arithmetical rank|)}\index{arithmetical rank!homogeneous|)}\index{set-theoretic complete intersection|)}\index{Segre product|)}

\chapter{Relations between Minors}\label{chapter3}\index{generic matrix|(}\index{generic matrix!minors of|(}\index{relations|(}\index{relations!between minors|(}

In Commutative Algebra, in Algebraic Geometry and in Representation Theory
the minors of a matrix are an interesting object for many
reasons. Surprisingly, in general, the minimal relations among
the $t$-minors of a $m\times n$ generic matrix $X$ are still
unknown. In this chapter, which is inspired to our work joint with Bruns and Conca \cite{BCVminors}, we will investigate  them.

Let us consider the following matrix:
\begin{displaymath}
X:=\begin{pmatrix}
x_{11} & x_{12} & x_{13} & x_{14} \\
x_{21} & x_{22} & x_{23} & x_{24}
\end{pmatrix}
\end{displaymath}
where the $x_{ij}$'s are indeterminates over a field $\kkk$. We remind that, as in the second part of Chapter \ref{chapter2}, we denote the $2$-minor insisting on the columns $i$ and $j$ of $X$ by $[ij]$, namely
$[ij]:=x_{1i}x_{2j}-x_{1j}x_{2i}$. One can verify the identity: 
\[
[12][34]-[13][24]+[14][23]=0.
\]
This is one of the celebrated {\it Pl\"ucker relations}\index{relations!Pl\"ucker}, and in this case it is the only minimal
relation, i.e. it generates the ideal
of relations between the $2$-minors of $X$. Actually, the case $t=\min\{m,n\}$ is well
understood in general, even if anything but trivial: In fact in such a situation the Pl\"ucker relations are the only minimal relations
among the $t$-minors of $X$. In particular, there are only
quadratic minimal relations.
This changes already for $2$-minors of a $3\times 4$-matrix.
For $2$-minors of a matrix with more than $2$ rows we must
first modify our notation in order to specify the rows of a
minor: 
\[ [ij|pq]:=x_{ip}x_{jq}-x_{iq}x_{jp}.\] 
Of course, we keep the Pl\"ucker
relations, but they are no more sufficient: Some cubics
appear among the minimal relations, for example the following identity
\[
\det\begin{pmatrix}
[12|12]&[12|13]&[12|14]\\
[13|12]&[13|13]&[13|14]\\
[23|12]&[23|13]&[23|14]
\end{pmatrix}=0
\]
does not come from the Pl\"ucker relations, see Bruns \cite{B}.

One reason for which the case of maximal minors is easier than
the general case emerges from a representation-theoretic point
of view. Let $R$ be the polynomial ring over $\kkk$ generated by
the variables $x_{ij}$, and $A_t\subseteq R$ denote the
$\kkk$-subalgebra of $R$ generated by the $t$-minors of $X$. When
$t=\min \{m,n\}$, the ring $A_t$ is nothing but than the coordinate ring of the
Grassmannian\index{Grassmannian} $\GG(m,n)$ of all $\kkk$-subspaces of dimension $m$
of a $\kkk$-vector space $V$ of dimension $n$ (we assume $m\leq
n$). In the general case, $A_t$ is the coordinate ring of the
Zariski closure of the image of the following homomorphism:
\[
\Lambda_t \ : \ \Hom_{\kkk}(W,V) \to \Hom_{\kkk}(\bigwedge^t W,\bigwedge^t V),
 \ \ \Lambda_t(\phi):=\wedge^t \phi,
 \]
where $W$ is a $\kkk$-vector space of dimension $m$. Notice that
the group $G:=\GL(W)\times \GL(V)$ acts on each graded
component $[A_t]_d$ of $A_t$. If $t=\min\{m,n\}$, each
$[A_t]_d$ is actually an irreducible $G$-representation (for the terminology about Representation Theory see Appendix \ref{appendixc}).
This is far from being true in the general case,
complicating the situation tremendously.

\vspace{2mm}

In this paper, under the assumption that the characteristic of
$\kkk$ is $0$, we will prove that quadric and cubics are the only
minimal relations among the $2$-minors of a $3\times n$ matrix
and of a $4\times n$ matrix (Theorem \ref{settled cases}). This confirms
the impression of Bruns and Conca in \cite{BC1}. We will use
tools from the representation theory of the general linear group in order
to reduce the problem to a computer calculation. In fact, more generally, we will prove in
Theorem \ref{independencefromn} that a minimal relation between $t$-minors of a $m\times n$ matrix must already be in a $m\times (m+t)$-matrix. In general, apart from particular well known cases, we prove that
cubic minimal relations always exist (Corollary \ref{cubics}),
and we interpret them in a representation-theoretic fashion. A type of these relations can be written in a nice determinantal form\index{relations!determinantal}: A minimal cubic relation is given by the vanishing of the following determinant (just for a matter of space below we put $s=t-1$, $u=t+1$ and $v=t+2$):
\begin{center}
$\det \begin{pmatrix}
[1,\ldots ,s,t,u|1,\ldots ,s,t]&[1,\ldots ,s,t,u|1,\ldots ,s,u]&[1,\ldots ,s,t,u|1,\ldots ,s,v]\\
[1,\ldots ,s,t,v|1,\ldots ,s,t]&[1,\ldots ,s,t,v|1,\ldots ,s,u]&[1,\ldots ,s,t,v|1,\ldots ,s,v]\\
[1,\ldots ,s,u,v|1,\ldots ,s,t]&[1,\ldots ,s,u,v|1,\ldots ,s,u]&[1,\ldots ,s,u,v|1,\ldots ,s,v]
\end{pmatrix}.
$
\end{center}
The above polynomial actually corresponds to the highest weight vector of the irreducible $G$-representation $\Wc \otimes \Vl^*$, where 
\[ \gamma := (t+1,t+1,t-2) \mbox{ \ \ and \ \ } \ll := (t+2,t-1,t-1).\] 
However, for $t\geq 4$ there are also other irreducible $G$-representations of degree $3$ which correspond to minimal relations, as we point out in Theorem \ref{shapeink3}. We also prove that there are no other minimal relations for ``reasons of shape". The only minimal relations of degree more than $3$ which might exist would be for ``reasons of multiplicity" (Proposition \ref{noothershape}).

In Thoerem \ref{regularity} we can write down a formula for the
Castelnuovo-Mumford regularity\index{Castelnuovo-Mumford regularity} of $A_t$ in all the cases. This is quite surprising: We do not even know a set of minimal generators of the ideal of relations between minors, but we can estimate its regularity, which is usually computed by its minimal free resolution. Essentially this is possible thanks to a result in \cite{BC1} describing the canonical module of $A_t$, combined with the interpretation of the Castelnuovo-Mumford regularity in terms of local cohomology as explained in Subsection \ref{subcastmum}. The regularity yields an upper bound on the degree of a minimal relation. Even if such a bound, in general, is quadratic in $m$ (Corollary \ref{generalupper}), it is the best general upper bound known to us.

In the last section we also exhibit a finite Sagbi basis\index{Sagbi bases} of
$A_t$ (Theorem \ref{sagbi}). This problem was left open by
Bruns and Conca in \cite{BC2}, where they proved the existence
of a finite Sagbi basis without describing it. With similar
tools, in Theorem \ref{invariant}, we give a finite system of $\kkk$-algebra generators of
the ring of invariants\index{ring of invariants}\index{U-invariant@$U$-invariant} $A_t^U$, where $U:=U_-(W)\times U_+(V)$ is the subgroup of $G$ with $U_-(W)$ (respectively $U_+(V)$) the subgroup of lower (respectively upper) triangular matrices of $\GL(W)$ (respectively of $\GL(V)$) with $1$'s on the diagonals. This is part of a classical sort of problems in invariant theory, namely the ``first main problem" of invariant theory. 

\section{Some relations for the defining ideal of $A_t$}

Let $\kkk$ be a field of characteristic $0$, $m$ and $n$ two positive integers such that $m\leq n$ and
\begin{displaymath}
X := \left(\begin{array}{ccccc} x_{11} & x_{12} & \cdots & \cdots &  x_{1n} \\
x_{21} & x_{22} & \cdots & \cdots & x_{2n} \\
\vdots & \vdots & \ddots & \ddots & \vdots \\
x_{m1} & x_{m2} & \cdots & \cdots & x_{mn}
\end{array} \right)
\end{displaymath}
a $m\times n$ matrix of indeterminates over $\kkk$. Moreover let
\[R(m,n):=\kkk[x_{ij} \ : \ i=1,\ldots ,m, \ j=1,\ldots ,n ]\]
be the polynomial ring in $m\cdot n$ variables over $\kkk$. As said in the introduction, we are interested in understanding the relations between the $t$-minors of $X$. In other words, we have to consider the kernel of the following homomorphism of $\kkk$-algebras:
\[ \pi \ : \ S_t(m,n) \longrightarrow A_t(m,n),\]
where $A_t(m,n)$ is the {\it algebra of minors}\index{algebra of minors}, i.e. the $\kkk$-subalgebra of $R(m,n)$ generated by the $t$-minors of $X$, and $S_t(m,n)$ is the polynomial ring over $\kkk$ whose variables are indexed on the $t$-minors of $X$. Since all the generators of the $\kkk$-algebra $A_t(m,n)$ have the same degree, namely $t$, we can ``renormalize" them: That is, a $t$-minor will have degree $1$. This way $\pi$ becomes a homogeneous homomorphism. Let $W$ and $V$ be two $\kkk$-vector spaces of dimension $m$ and $n$. Of course we have the identification
\[S_t(m,n)\cong \Sym(\bigwedge^t W \otimes \bigwedge^t V^*)=\bigoplus_{i\in \NN}\Sym^d(\bigwedge^t W \otimes \bigwedge^t V^*),\]
thus $G:=\GL(W)\times \GL(V)$ acts on $S_t(m,n)$. The algebra of minors $A_t(m,n)$ is also a $G$-representation, as explained in Appendix \ref{appendixc}, Section \ref{sacminors}. Actually, for any natural number $d$, the graded components $S_t(m,n)_d$ and $A_t(m,n)_d$ are finite dimensional $G$-representations, and the map $\pi$ is $G$-equivariant. Let us denote 
\[ J_t(m,n):= \Ker(\pi).\] 
When it does not raise confusion we just write $R$, $S_t$, $J_t$ and $A_t$ in place of $R(m,n)$, $S_t(m,n)$, $J_t(m,n)$ and $A_t(m,n)$.

\begin{remark}\label{partcases}
Consider the following numerical situations:
\begin{align}
\label{easycases}
t =1 \mbox{ \ \ or \ \ }  n \leq t+1 \\
\label{grassmannian}   
t =m
\end{align}
In the cases \eqref{easycases} the algebra $A_t$ is a polynomial ring, so that $J_t=0$ (for the case $m=n=t+1$ look at the book of Bruns and Vetter, \cite[Remark 10.17]{BrVe}). In the case \eqref{grassmannian} $A_t$ is the coordinate ring of the Grassmannian\index{Grassmannian} $\GG(m,n)$ of $\kkk$-subspaces of $V$ of dimension $m$. In this case, if $2\leq m\leq n-2$, the ideal $J_t$ is generated by the Pl\"ucker relations\index{relations!Pl\"ucker}. In particular it is generated in degree two. Clearly the Pl\"ucker relations occur in all the remaining cases too. So the ideal $J_t$ is generated in degree at least $2$ in the cases different from \eqref{easycases}.
\end{remark}

Because Remark \ref{partcases}, throughout the paper we will assume to being in cases different from \eqref{easycases} and \eqref{grassmannian}. So, from now on, we feel free to assume
\[ 1<t<m \mbox{ \ \ and  \ \ }n>t+1. \]
Our purpose for this section is to show that in the above range minimal generators of degree $3$ always appear in $J_t$. Notice that, since $\pi$ is $G$-equivariant, then $J_t$ is a $G$-subrepresentation of $S_t$. Moreover, if $\Wc \otimes \Vl^*$ is an irreducible representation of $S_t$ then Schur's Lemma \ref{Schur Lemma} implies that either it collapses to zero or it is mapped $1$-$1$ to itself. In other words $\pi$ is a ``shape selection". Therefore, Theorem \ref{decat} implies that $\Wc \otimes \Vl^* \subseteq J_t$ whenever $\Wc \otimes \Vl^* \subseteq S_t$, and $\gamma$ and $\lambda$ are different partitions of a natural number $d$. However, it is difficult to say something more at this step. In fact, in contrast with $A_t$, a decomposition of $S_t$ as direct sum of irreducible representations is unknown. For instance a decomposition of $\Sym(\bigwedge^m V)$ as a $\GL(V)$-representation is known just for $m\leq 2$, see Weyman \cite[Proposition 2.3.8]{We}. When $m>2$, this is an open problem in representation theory, which is numbered among the plethysm's\index{plethysm} problems. In order to avoid such a difficulty, we go ``one step more to the left", in a way that we are going to outline. For any natural number $d$ we have the natural projection
\[p_d : \bigotimes^d (\bigwedge^t W \otimes \bigwedge^t V^*) \longrightarrow \Sym^d(\bigwedge^t W \otimes \bigwedge^t V^*) \cong S_t(m,n)_d.\]
The projection $p_d$ is a $G$-equivariant surjective map. We have also the following $G$-equivariant isomorphism
\[f_d: (\bigotimes^d \bigwedge^t W) \otimes (\bigotimes^d \bigwedge^t V^*) \longrightarrow \bigotimes^d (\bigwedge^t W \otimes \bigwedge^t V^*).\]
Therefore, for any natural number $d$, we have the $G$-equivariant surjective map
\[ \phi_d := p_d\circ f_d : (\bigotimes^d \bigwedge^t W) \otimes (\bigotimes^d \bigwedge^t V^*) \longrightarrow S_t(m,n)_d.  \]
Putting together the $\phi_d$'s we get the following ($G$-equivariant and surjective) homogeneous $\kkk$-algebra homomorphism
\[\phi : T_t(m,n) := \bigoplus_{d \in \NN} ((\bigotimes^d \bigwedge^t W) \otimes (\bigotimes^d \bigwedge^t V^*)) \longrightarrow S_t(m,n).\]
When it does not raise confusion we will write $T_t$ for $T_t(m,n)$.
\begin{remark}
Notice that $T_t$ is not commutative. The ideals $I\subseteq T_t$ we consider will always be two-sided, i.e. they are $\kkk$-vector spaces such that $st$ and $ts$ belong to $I$ whenever $t\in I$ and $s\in T_t$. Moreover, if we say that an ideal $I\subseteq T_t$ is generated by $t_1,\ldots ,t_q$ we mean that 
\[I = \{\sum _{i=1}^q f_it_ig_i \ : \ f_i, g_i \in T_t(m,n)\}.\]
In this sense $T_t$ is Noetherian. However, $T_t$ is neither left-Noetherian nor right-Noetherian.
\end{remark}
By definition and by meaning of the $f_d$'s, the kernel of $\phi$ is generated in degree $2$, namely
\[\Ker(\phi) = ((f\otimes f')\otimes (e\otimes e'))-(f'\otimes f)\otimes (e'\otimes e)) \ : \ f,f' \in \bigwedge^t W, \ e,e' \in \bigwedge^t V^*).\]
Finally, we have a $G$-equivariant surjective graded homomorphism
\[ \psi := \pi \circ \phi : T_t(m,n) \to A_t(m,n). \]
We call $K_t(m,n)$ ($K_t$ when it does not raise any confusion) the kernel of the above map. Since $\Ker(\phi)$ is generated in degree two and $J_t$ is generated in degree at least two, in order to understand which is the maximum degree of a minimal generator of $J_t$, we can study which is the maximum degree of a minimal generator of $K_t$. 

\begin{lemma}\label{redtotensor}
Let $d$ be an integer bigger than $2$. There exists a minimal generator of degree $d$ in $K_t$ if and only if there exists a minimal generator of degree $d$ in $J_t$.
\end{lemma}

The advantages of passing to $T_t$ are that it ``separates rows and columns" and that it is available a decomposition of it in irreducible $G$-representations, see Proposition \ref{tensordec}. The disadvantage is that we have to work in a noncommutative setting.

\begin{prop}\label{tensordec}
As a $G$-representation $T_t$ decomposes as
\[T_t(m,n)_d \cong \bigoplus_{{ \mbox{{\footnotesize $\gamma$ and $\lambda$ $d$-admissible}}}\atop {\mbox{ {\footnotesize $\height(\gamma)\leq m$,  $\height(\ll)\leq n$}}}} (\Wc \otimes \Vl^*)^{n(\gamma , \lambda)},\]
where the multiplicities\index{representation!multiplicities} $n(\gamma ,\lambda)$ are the nonzero natural numbers described recursively as follows:
\begin{compactenum}
\item If $d=1$, i.e. if $\gamma  = \lambda = (t)$, then $n(\gamma, \lambda)=1$;
\item If $d > 1$, then $n(\gamma , \lambda) = \sum n(\gamma' ,\ll')$ where the sum runs over the $(d-1)$-admissible diagrams $\gamma'$ and $\ll'$ such that $\gamma' \subseteq \gamma \subseteq \gamma'(t)$ and $\ll' \subseteq \ll \subseteq \ll'(t)$ (for the notation see Theorem \ref{pieri}).
\end{compactenum}
\end{prop}
\begin{proof}
Since $T_t(m,n)_1= \bigwedge^t W \otimes \bigwedge^t V$, we can prove the statement by induction. So suppose that 
\[T_t(m,n)_{d-1} \cong \bigoplus_{{ \mbox{{\footnotesize $\gamma$ and $\lambda$ $(d-1)$-admissible}}}\atop {\mbox{ {\footnotesize $\height(\gamma)\leq m$,  $\height(\ll)\leq n$}}}} (\Wc \otimes \Vl^*)^{n(\gamma , \lambda)}\]
for $d\geq 2$. Therefore, since $T_t(m,n)_d \cong T_t(m,n)_{d-1}\otimes (\bigwedge^t W \otimes \bigwedge^t V^*)$, we have
\[T_t(m,n)_d \cong \bigoplus_{{ \mbox{{\footnotesize $\gamma$ and $\lambda$ $(d-1)$-admissible}}}\atop {\mbox{ {\footnotesize $\height(\gamma)\leq m$,  $\height(\ll)\leq n$}}}} ((\Wc \otimes \bigwedge^t W) \otimes (\Vl^*\otimes \bigwedge^t V^*))^{n(\gamma , \lambda)}.\]
At this point we get the desired decomposition from Pieri's Theorem \ref{pieri}.
\end{proof}

As the reader will realize during this chapter, the fact that the $n(\gamma,\lambda)$'s may be, and in fact usually are, bigger than $1$, is one of the biggest troubles to say something about the relations between minors. We will say that a pair of $d$-admissible diagrams, a $d$-admissible {\it bi-diagram}\index{bi-diagram} for short, $(\gamma'|\ll')$ is a {\it predecessor}\index{predecessor} of a $(d+1)$-admissible bi-diagram\index{bi-diagram!admissible} $(\gamma|\lambda)$ if $\gamma'\subseteq \gamma \subseteq \gamma'(t)$ and $\ll'\subseteq \lambda \subseteq \ll'(t)$.
In order to simplify the notation, from now on a sentence like ``$(\gamma|\lambda)$ is an irreducible representation of $T_t$ (or $S_t$ or $J_t$ etc.)" will mean that $\Wc \otimes \Vl^*$ occurs with multiplicity at least $1$ in the decomposition of $T_t$ (or $S_t$ or $J_t$ etc.). Furthermore, since both $J_t$ and $K_t$ are generated, as ideals, by irreducible representations, we say that ``$(\gamma|\lambda)$ is a minimal irreducible representation\index{representation!irreducible!minimal} of $J_t$ (or of $K_t$)" if the elements in $\Wc \otimes \Vl^*$ are minimal generators of $J_t$ (or $K_t$).

\vspace{2mm}

We are going to list some equations of degree $3$ which are minimal generators of $J_t$ (in situations different from \eqref{easycases} and \eqref{grassmannian}). To this purpose we define some special bi-diagrams $(\gamma^i |\lambda^i)$ for any $\displaystyle i=1,\ldots , \left\lfloor \frac{t}{2} \right\rfloor$, for which both $\gamma^i$ and $\lambda^i$ have exactly $3t$ boxes. In Theorem \ref{shapeink3}, we will prove that some of these diagrams will be minimal irreducible representations of degree $3$ in $J_t$. It is convenient to consider separately the cases in which $t$ is even or odd.

Suppose that $t$ is even. For each $\displaystyle i=1, \ldots , \frac{t}{2}$ we set $\displaystyle a_i:=\frac{3t}{2}-i+1$, \ $b_i:= 2(i-1)$, \ $c_i:= 2(t-i+1)$ and $\displaystyle d_i:=\frac{t}{2}+i-1$. Then
\begin{align}
\label{kerdiage} \gamma^i & :=(a_i,a_i,b_i), \\
\lambda^i  & :=(c_i,d_i,d_i). \nonumber
\end{align}
It turns out that $\gamma^i$ and $\lambda^i$ are both partitions of $3t$.

If $t$ is odd, then for each $\displaystyle i=1, \ldots ,\frac{t-1}{2}$ set $\displaystyle a_i':=\frac{3t-1}{2}-i+1$, \ $b_i':=2(i-1)+1$, \ $c_i':= 2(t-i+1)-1$ and $\displaystyle d_i':=\frac{t+1}{2}+i-1$. Then
\begin{align}
\label{kerdiago} \gamma^i & :=(a_i',a_i',b_i'), \\
\lambda^i  & :=(c_i',d_i',d_i'). \nonumber
\end{align}
Once again it turns out that $\gamma^i$ and $\lambda^i$ are both partitions of $3t$.

\begin{example}
For $t=2$ there is only one $(\gamma^i|\lambda^i)$, namely $(\gamma^1|\lambda^1)$. The picture below features this bi-diagram.
\[
{\setlength{\unitlength}{1mm}
\begin{picture}(30,25)(-5,0)

\put(-37,15){$(\gamma^1|\lambda^1) \ =$}
\put(-15,20){\line(1,0){15}}
\put(-15,15){\line(1,0){15}}
\put(-15,10){\line(1,0){15}}

\put(-15,20){\line(0,-1){10}}
\put(-10,20){\line(0,-1){10}}
\put(-5,20){\line(0,-1){10}}
\put(0,20){\line(0,-1){10}}

\put(7.5,22.5){\line(0,-1){20}}

\put(15,20){\line(1,0){20}}
\put(15,15){\line(1,0){20}}
\put(15,10){\line(1,0){5}}
\put(15,5){\line(1,0){5}}

\put(15,20){\line(0,-1){15}}
\put(20,20){\line(0,-1){15}}
\put(25,20){\line(0,-1){5}}
\put(30,20){\line(0,-1){5}}
\put(35,20){\line(0,-1){5}}

\end{picture}}
\]
Notice that the above bi-diagram has only one predecessor, namely
\[
{\setlength{\unitlength}{1mm}
\begin{picture}(30,25)(-5,0)

\put(-37,15){$(\alpha^1|\alpha^1) \ =$}
\put(-15,20){\line(1,0){15}}
\put(-15,15){\line(1,0){15}}
\put(-15,10){\line(1,0){5}}

\put(-15,20){\line(0,-1){10}}
\put(-10,20){\line(0,-1){10}}
\put(-5,20){\line(0,-1){5}}
\put(0,20){\line(0,-1){5}}

\put(7.5,22.5){\line(0,-1){15}}

\put(15,20){\line(1,0){15}}
\put(15,15){\line(1,0){15}}
\put(15,10){\line(1,0){5}}

\put(15,20){\line(0,-1){10}}
\put(20,20){\line(0,-1){10}}
\put(25,20){\line(0,-1){5}}
\put(30,20){\line(0,-1){5}}

\end{picture}}
\]
It turns out that $(\alpha^1|\alpha^1)$ has multiplicity one and is symmetric\index{bi-diagram!symmetric}. As we are going to see in the proof of Theorem \ref{shapeink3} this facts holds true in general, and it is the key to find minimal relations in $J_t$.

\end{example}

\begin{thm}\label{shapeink3}
The bi-diagram $(\gamma^i|\lambda^i)$ is a minimal irreducible representation of $J_t$ of degree $3$ in the following cases:
\begin{compactenum}
\item if $t$ is even, when $\displaystyle \frac{t}{2}\geq i\geq \max\left\{0,\frac{3t}{2}-m+1,t-\frac{n}{2} + 1 \right\}$.
\item if $t$ is odd, when $\displaystyle \frac{t-1}{2}\geq i\geq \max\left\{0,\frac{3t-1}{2}-m+1,t- \frac{n+1}{2} \right\}$.
\end{compactenum}
\end{thm}
\begin{proof}
By Lemma \ref{redtotensor} it is enough to show that the bi-diagram $(\gamma^i | \lambda^i)$ is a minimal irreducible representation of degree $3$ of $K_t$.
Both $\gamma^i$ and $\lambda^i$ are $3$-admissible partitions. Furthermore, the assumptions on $i$ imply that $\gamma_1^i \leq m$ and $\lambda_1^i\leq n$. Therefore $(\gamma^i|\lambda^i)$ is an irreducible representation of $[T_t]_3$ by Proposition \ref{tensordec}.
Notice that $[T_t]_3$ decomposes as $[K_t]_3 \oplus [A_t]_3$. But $(\gamma^i|\lambda^i)$ is an {\it asymmetric bi-diagram}\index{bi-diagram!asymmetric}, i.e. $\gamma^i \neq \lambda^i$, thus it cannot be an irreducible representation of $A_t$ by Theorem \ref{decat}. This implies that it is an irreducible representation of $[K_t]_3$.

It remains to prove that $(\gamma^i|\lambda^i)$ is minimal. First of all we show that it has multiplicity $1$ in $T_t$. If $t$ is even the unique predecessor of $(\gamma^i|\lambda^i)$ is the symmetric bi-diagram $(\alpha^i|\alpha^i)$, where:
\begin{align*}
\displaystyle \alpha_1^i & = \frac{3t}{2}-i+1 \\
\displaystyle \alpha_2^i & = \frac{t}{2}+i-1.
\end{align*}
Also if $t$ is odd the unique predecessor of $(\gamma^i |\lambda^i)$, which we denote again by $(\alpha^i|\alpha^i)$, is symmetric. The $G$-representation $[T_t]_2$ decomposes as $[K_t]_2 \oplus [A_t]_2$. Now, the irreducible representations of $T_t$ of degree $2$ have obviously multiplicity $1$, since their unique predecessor is $((t)|(t))$ that has multiplicity $1$. So $(\alpha^i|\alpha^i)$ is not an irreducible representation of $[K_t]_2$, because it is an irreducible representation of $[A_t]_2$ by Theorem \ref{decat} and it has multiplicity $1$ in $[T_t]_2$. Therefore we conclude that $(\gamma^i|\lambda^i)$ is a minimal irreducible representation of $K_t$ of degree $3$.
\end{proof}

\begin{definition}
The bi-diagram $(\gamma^i|\lambda^i)$ such that $i$ satisfies the condition of Theorem \ref{shapeink3} are called {\it shape relations}\index{relations!shape}.
\end{definition}

\begin{corollary}\label{cubics}
In our situation, i.e. if $1<t<m$ and $n>t+1$, the ideal $J_t$ has some minimal generators of degree $3$.
\end{corollary}
\begin{proof}
It is enough to verify that there is at least one $i$ satysfying the conditions of Theorem \ref{shapeink3}.
\end{proof}

\subsection{Make explicit the shape relations}

The reader might recriminate that we have not yet described explicitly the degree $3$ minimal relations we found in Theorem \ref{shapeink3}. Actually we think that the best way to present them is by meaning of shape, as we did. But of course it is legitimate to pretend to know the polynomials corresponding to them, therefore we are going to explain how to pass from bi-diagrams to polynomials. Consider one of the bi-diagrams $(\gamma^i|\lambda^i)$ of Theorem \ref{shapeink3}. Let us call it $(\gamma|\lambda)$. Then choose two tableux $\Gamma$ and $\Lambda$ of shape, respectively, $\gamma$ and $\lambda$ and such that $c(\Gamma)=c(\Lambda)=3t$. We can consider the following map
\[e(\Gamma)\otimes e(\Lambda)^* : \bigotimes^{3t}W\otimes \bigotimes^{3t}V^*\longrightarrow \bigotimes^{3t}W\otimes \bigotimes^{3t}V^*,\]
where $e(\cdot)$ are the Young symmetrizers (see \ref{subsecshurmodules}). Actually the image of $e(\Gamma)\otimes e(\Lambda)^*$ is in $\bigotimes^3 \bigwedge^t W \otimes \bigotimes^3 \bigwedge^t V^*$, therefore we can compose it with the map
\[\bigotimes^3 \bigwedge^t W \otimes \bigotimes^3 \bigwedge^t V^* \longrightarrow [S_t]_3.\]
The polynomials in the image do not depend from the chosen tableu, since $\Wc \otimes \Vl^*$ has multiplicity one in $T_t$.

Actually, among the bi-diagrams $(\gamma^i|\lambda^i)$ of Theorem \ref{shapeink3}, there is one such that the relative $U$-invariant\index{U-invariant@$U$-invariant} (see Subsection \ref{sacminors}) can be written in a very plain way. Namely, the bi-diagram under discussion is $(\gamma|\ll):=(\gamma^{\lfloor t/2\rfloor}|\ll^{\lfloor t/2\rfloor})$, namely:
\[(\gamma|\ll)=((t+1,t+1,t-2)|(t+2,t-1,t-1)).\] 
The corresponding equation can be written in a determinantal form\index{relations!determinantal}: More precisely, it is given by the vanishing of the following determinant (just for a matter of space below we put $s=t-1$, $u=t+1$ and $v=t+2$):
\begin{equation}\label{determinantalrelation}
{\small \det \begin{pmatrix}
[1,\ldots ,s,t,u|1,\ldots ,s,t]&[1,\ldots ,s,t,u|1,\ldots ,s,u]&[1,\ldots ,s,t,u|1,\ldots ,s,v]\\
[1,\ldots ,s,t,v|1,\ldots ,s,t]&[1,\ldots ,s,t,v|1,\ldots ,s,u]&[1,\ldots ,s,t,v|1,\ldots ,s,v]\\
[1,\ldots ,s,u,v|1,\ldots ,s,t]&[1,\ldots ,s,u,v|1,\ldots ,s,u]&[1,\ldots ,s,u,v|1,\ldots ,s,v]
\end{pmatrix}=0.}
\end{equation}
We are going to prove that \eqref{determinantalrelation} gives actually the $U$-invariant of the irreducible representation $(\gamma|\ll)$ described above. In particular, this means that the determinant \eqref{determinantalrelation} generates $\Wc \otimes \Vl^*$ as a $G$-module. Furthermore, notice that in the cases $t=2,3$ the unique minimal irreducible representation of $J_t$ of degree $3$ we described in Theorem \ref{shapeink3} is $(\gamma |\ll)$. So, in these two cases we have a very good description of the (guessed) minimal equations between $t$-minors.

\begin{prop}\label{explicitrel}
If $1<t<m$ and $n>t+1$, the equation \eqref{determinantalrelation} supplies the $U$-invariant of the minimal irreducible representation $(\gamma|\ll)$ of $J_t$, where $\gamma = (t+1,t+1,t-2)$ and $\ll = (t+2,t-1,t-1)$. Particularly, it is a minimal cubic generator of $J_t$. 
\end{prop}
\begin{proof}
First of all notice that the polynomial of the determinant of \eqref{determinantalrelation} is nonzero in $S_t$. This is obvious, since it is the determinant of a $3\times 3$ matrix whose entries are variables all different between them. 

Let us call $f$ the polynomial associated to the determinant of \eqref{determinantalrelation}. To see that $f$ is an $U$-invariant, we have to show that $(A,B)\cdot f = \alpha f$ for some $\alpha \in \kkk\setminus \{0\}$, where $A\in U_-(W)\subseteq \GL(W)$ is a lower triangular matrices and $B\in U_+(V)\subseteq \GL(V)$ is an upper triangular matrices. To prove this, it suffices to show that $f$ vanishes formally every time that we substitute an index, of the columns or of the rows of the variables appearing in $f$, with a smaller one. This is easily checkable: For example, if we substitute the $(t+2)$th row with the $t$th one, that is, using the notation of \eqref{determinantalrelation}, if we substitute the $v$th row with the $t$th one, then $f$ becomes:

\begin{displaymath}
\det \begin{pmatrix}
[1,\ldots ,s,t,u|1,\ldots ,s,t]&[1,\ldots ,s,t,u|1,\ldots ,s,u]&[1,\ldots ,s,t,u|1,\ldots ,s,v]\\
0 & 0 & 0\\
- [1,\ldots ,s,u,t|1,\ldots ,s,t]&-[1,\ldots ,s,u,t|1,\ldots ,s,u]&-[1,\ldots ,s,u,t|1,\ldots ,s,v]
\end{pmatrix},
\end{displaymath}
which is obviously formally $0$. Or, if we substitute the $v$th column with the $t$th one, $f$ becomes:

\begin{displaymath}
\det \begin{pmatrix}
[1,\ldots ,s,t,u|1,\ldots ,s,t]&[1,\ldots ,s,t,u|1,\ldots ,s,u]&[1,\ldots ,s,t,u|1,\ldots ,s,t]\\
[1,\ldots ,s,t,v|1,\ldots ,s,t]&[1,\ldots ,s,t,v|1,\ldots ,s,u]&[1,\ldots ,s,t,v|1,\ldots ,s,t]\\
[1,\ldots ,s,u,v|1,\ldots ,s,t]&[1,\ldots ,s,u,v|1,\ldots ,s,u]&[1,\ldots ,s,u,v|1,\ldots ,s,t]
\end{pmatrix},
\end{displaymath}
which is formally $0$ since the first and third columns are the same.

To see that \eqref{determinantalrelation} becomes $0$ in $A_t$, under $\pi$, we have to observe that $f$ becomes the expansion in $t$-minors of the determinant of the following $3t\times 3t$ matrix:
\begin{displaymath}
{\small \left(\begin{array}{cccccccccccccccccc}
x_{11} & x_{12} & \ldots & x_{1s} & x_{1t} \hspace{3mm} & x_{11} & x_{12} & \ldots & x_{1s} & x_{1u} \hspace{3mm} & x_{11} & x_{12} & \ldots & x_{1s} & x_{1v}\\
x_{21} & x_{22} & \ldots & x_{2s} & x_{2t} \hspace{3mm} & x_{21} & x_{22} & \ldots & x_{2s} & x_{2u} \hspace{3mm} & x_{21} & x_{22} & \ldots & x_{2s} & x_{2v}\\
\vdots & \vdots & \vdots & \vdots & \vdots & \vdots & \vdots & \vdots & \vdots & \vdots & \vdots & \vdots & \vdots & \vdots & \vdots \\
x_{s1} & x_{s2} & \ldots & x_{ss} & x_{st} \hspace{3mm} & x_{s1} & x_{s2} & \ldots & x_{ss} & x_{su} \hspace{3mm} & x_{s1} & x_{s2} & \ldots & x_{ss} & x_{sv}\\
x_{t1} & x_{t2} & \ldots & x_{ts} & x_{tt} \hspace{3mm} & x_{t1} & x_{t2} & \ldots & x_{ts} & x_{tu} \hspace{3mm} & x_{t1} & x_{t2} & \ldots & x_{ts} & x_{tv}\\
x_{u1} & x_{u2} & \ldots & x_{us} & x_{ut} \hspace{3mm} & x_{u1} & x_{u2} & \ldots & x_{us} & x_{uu} \hspace{3mm} & x_{u1} & x_{u2} & \ldots & x_{us} & x_{uv}

\vspace{3mm}\\

x_{11} & x_{12} & \ldots & x_{1s} & x_{1t} \hspace{3mm} & x_{11} & x_{12} & \ldots & x_{1s} & x_{1u} \hspace{3mm} & x_{11} & x_{12} & \ldots & x_{1s} & x_{1v}\\
x_{21} & x_{22} & \ldots & x_{2s} & x_{2t} \hspace{3mm} & x_{21} & x_{22} & \ldots & x_{2s} & x_{2u} \hspace{3mm} & x_{21} & x_{22} & \ldots & x_{2s} & x_{2v}\\
\vdots & \vdots & \vdots & \vdots & \vdots & \vdots & \vdots & \vdots & \vdots & \vdots & \vdots & \vdots & \vdots & \vdots & \vdots \\
x_{s1} & x_{s2} & \ldots & x_{ss} & x_{st} \hspace{3mm} & x_{s1} & x_{s2} & \ldots & x_{ss} & x_{su} \hspace{3mm} & x_{s1} & x_{s2} & \ldots & x_{ss} & x_{sv}\\
x_{t1} & x_{t2} & \ldots & x_{ts} & x_{tt} \hspace{3mm} & x_{t1} & x_{t2} & \ldots & x_{ts} & x_{tu} \hspace{3mm} & x_{t1} & x_{t2} & \ldots & x_{ts} & x_{tv}\\
x_{v1} & x_{v2} & \ldots & x_{vs} & x_{vt} \hspace{3mm} & x_{v1} & x_{v2} & \ldots & x_{vs} & x_{vu} \hspace{3mm} & x_{v1} & x_{v2} & \ldots & x_{vs} & x_{vv}

\vspace{3mm}\\

x_{11} & x_{12} & \ldots & x_{1s} & x_{1t} \hspace{3mm} & x_{11} & x_{12} & \ldots & x_{1s} & x_{1u} \hspace{3mm} & x_{11} & x_{12} & \ldots & x_{1s} & x_{1v}\\
x_{21} & x_{22} & \ldots & x_{2s} & x_{2t} \hspace{3mm} & x_{21} & x_{22} & \ldots & x_{2s} & x_{2u} \hspace{3mm} & x_{21} & x_{22} & \ldots & x_{2s} & x_{2v}\\
\vdots & \vdots & \vdots & \vdots & \vdots & \vdots & \vdots & \vdots & \vdots & \vdots & \vdots & \vdots & \vdots & \vdots & \vdots \\
x_{s1} & x_{s2} & \ldots & x_{ss} & x_{st} \hspace{3mm} & x_{s1} & x_{s2} & \ldots & x_{ss} & x_{su} \hspace{3mm} & x_{s1} & x_{s2} & \ldots & x_{ss} & x_{sv}\\
x_{u1} & x_{u2} & \ldots & x_{us} & x_{ut} \hspace{3mm} & x_{u1} & x_{u2} & \ldots & x_{us} & x_{uu} \hspace{3mm} & x_{u1} & x_{u2} & \ldots & x_{us} & x_{uv}\\
x_{v1} & x_{v2} & \ldots & x_{vs} & x_{vt} \hspace{3mm} & x_{v1} & x_{v2} & \ldots & x_{vs} & x_{vu} \hspace{3mm} & x_{v1} & x_{v2} & \ldots & x_{vs} & x_{vv}
\end{array}\right).}
\end{displaymath}
Such a determinant is zero (for instance because the $1$st row is equal to the $(t+1)$th).

So far, we have shown that the polynomial $f\in S_t$ associated to \eqref{determinantalrelation} is a nonzero $U$-invariant of $J_t$. Now, it clearly has bi-weigth $(\tl\gamma |\tl\ll)$, so it is the $U$-invariant of the irreducible representation $(\gamma|\ll)$. Particularly, it is a minimal cubic generator of $J_t$ by Theorem \ref{shapeink3}.
\end{proof}

\begin{remark}
Actually, in the proof of Proposition \ref{explicitrel} it is not necessary to show that the determinant of \eqref{determinantalrelation} vanishes in $A_t$. In fact, once proved that it is the $U$-invariant of $(\gamma |\ll)$, that it goes to zero follows because $\gamma \neq \ll$. However we wanted to show that it vanishes because the huge matrix of Proposition \ref{explicitrel} has been the first mental image that suggested us the equation \ref{determinantalrelation}.
\end{remark}

\section{Upper bounds on the degrees of minimal relations among minors}

During the previous section we have found out some cubics among the minimal generators of $J_t$. The authors of \cite{BC1} said that there are indications that quadrics and cubics are enough for generating $J_t$, at least for $t=2$. In this section we are going to prove that their guess is right for $3\times n$ and $4\times n$ matrices. The way to get these results will be to notice that the ``minimal" $U$-invariants of a $J_2(3,n)$ must be already in $J_2(3,5)$, and a similar fact for a $4\times n$-matrix. These cases are then doable with a computer calculation. Furthermore, we will give some reasons to believe to the guess of \cite{BC1} in general, even for $t\geq 3$. First of all, anyway, we will give a formula for the Castelnuovo-Mumford regularity of $A_t(m,n)$ in all the cases. Since $A_t(m,n)=S_t(m,n)/J_t(m,n)$, we will also get a general upper bound for the degree of a minimal relation between minors (see \ref{subcastmum}).

\subsection{Tha Castelnuovo-Mumford regularity of the algebra of minors}
\index{Castelnuovo-Mumford regularity|(}\index{a-invariant@$a$-invariant|(}

The authors of \cite{BC1} noticed that $A_t$ is a Cohen-Macaulay $\kkk$-algebra with negative $a$-invariant. This implies that the degrees of the minimal generators of $J_t$ are less than or equal to $\dim A_t = mn$ (for the last equality see \cite[Proposition 10.16]{BrVe}). It would be desirable to know the exact value of $a(A_t)$, equivalently of $\reg(A_t)$. To this aim we will pass through a toric deformation of $A_t$. If $\prec$ is a diagonal term order on $R=R(m,n)$, i.e. such that $\init([i_1\ldots i_p|j_1\ldots j_p])=x_{i_1j_1}\cdots x_{i_pj_p}$, then the initial algebra\index{initial algebra} $\init(A_t)$ is a finitely generated normal Cohen-Macaulay $\kkk$-algebra, see \cite[Theorem 7.10]{BC2}. Furthermore in \cite[Lemma 3.3]{BC1}, the authors described the canonical module $\omega_{\init(A_t)}$ of $\init(A_t)$. By \eqref{a-invcanonical} we have that $a(\init(A_t))=-\min\{d:[\omega_{\init(A_t)}]_d\neq 0\}$. Thus it is natural to expect
to get the Castelnuovo-Mumford regularity of $\init(A_t)$ from $\omega_{\init(A_t)}$. Actually this is true, albeit anything but trivial, and we are going to prove it in Theorem \ref{regularity}. Eventually, it is easy to prove that $\reg(A_t)=\reg(\init(A_t))$.

\begin{thm}\label{regularity}
%Let $k_0$ be the smallest natural number $k$ satisfying the inequality
%\[\displaystyle m+n+k-1 < \left\lfloor \frac{m(n+k)}{t}\right\rfloor.\]
Suppose to be in cases different from \eqref{easycases} and \eqref{grassmannian}, that is $1<t<m$ and $n>t+1$. Then $A_t$ and $\init(A_t)$ are finitely generated graded Cohen-Macaulay algebras such that:
\begin{compactitem}
\item[(i)] If $m+n-1<\lfloor mn/t \rfloor$, then 
\begin{displaymath}
\begin{array}{c}
a(A_t) = a(\init(A_t)) = -\lceil mn/t \rceil ,\\
\reg(A_t) =  mn-\lceil mn/t \rceil . 
\end{array}
\end{displaymath}
\item[(ii)] Otherwise, i.e. if $m+n-1\geq \lfloor mn/t \rfloor$, we have 
\begin{displaymath}
\begin{array}{c}
a(A_t) = a(\init(A_t)) = -\lfloor m(n+k_0)/t \rfloor ,\\
\reg(A_t) = mn-\lfloor m(n+k_0)/t \rfloor . 
\end{array}
\end{displaymath}
where $k_0 = \lceil (tm+tn-mn)/(m-t) \rceil$.
\end{compactitem}
\end{thm}
\begin{proof}
The $\kkk$-algebras $A_t$ and $\init(A_t)$ are finitely generated and Cohen-Macaulay by \cite[Theorem 7.10]{BC2}. By \cite[Lemma 3.3]{BC1}, it turns out that the canonical module $\omega=\omega_{\init(A_t)}$ of the initial algebra of $A_t$ with respect to a diagonal term order $\prec$ is the ideal of $\init(A_t)$ generated by $\init(\Delta)$,  where $\Delta$ is a product of minors of $X$ of shape $\gamma=(\gamma_1,\ldots ,\gamma_h)$ where $|\gamma|=td$, $h<d$ and such that $\xx := \prod x_{ij}$ divides $\init(\Delta)$. To prove the theorem we need to find the least $d$ for which such a $\Delta$ exists. Of course such a $d$ must be such that $td\geq mn$, but in general this is not sufficient. First we prefer to illustrate the strategy we will use to locate $d$ with an example:
\begin{example}
Set $t=3$ and $m=n=5$. So our matrix looks like
\begin{displaymath}
X = \left(\begin{array}{ccccc} x_{11} & x_{12} & x_{13} & x_{14} &  x_{15} \\
x_{21} & x_{22} & x_{23} & x_{24} & x_{25} \\
x_{31} & x_{32} & x_{33} & x_{34} & x_{35} \\
x_{41} & x_{42} & x_{43} & x_{44} & x_{45} \\
x_{51} & x_{52} & x_{53} & x_{54} & x_{55} 
\end{array} \right).
\end{displaymath}
Notice that we are in the case (ii) of the theorem. We are interested in finding the least $d\in \NN$ such that there exists a product of minors $\Delta\in A_3(5,5)$ of shape $\gamma =(\gamma_1,\ldots ,\gamma_h)$ such that $|\gamma|=3d$, $h<d$ and $\xx = \prod x_{ij}$ divides $\init(\Delta)$. The first product of minors which comes in mind is
\[\mbox{{\small $\Delta_2:= [12345|12345][1234|2345][2345|1234][123|345][345|123][12|45][45|12][1|5][5|1]$.}}\]
Obviously $\init(\Delta_2)$ is a multiple of $\xx$, but its shape is $\gamma_2=(5,4,4,3,3,2,2,1,1)$. This is a partition of $25$, which is not divisible by $3$. This means that $\Delta_2$ does not even belong to $A_3(5,5)$. Moreover the parts of $\gamma_2$ are $9$, whereas $\lfloor 25/3\rfloor = 8$ (it should be bigger than $9$). To fix this last problem, it is natural to multiply $\Delta_2$ for $5$-minors till the desired result is gotten. For instance, setting $\Delta_1:=\Delta_2 \cdot [12345|12345]^3$, we have that $\Delta_1$ is a product of minors of shape $\gamma_1 =(5,5,5,5,4,4,3,3,2,2,1,1)$. This is a partition of $40$ with $12$ parts, and $\lfloor 40/3 \rfloor = 13>12$. However $\Delta_1$ is still not good, because $40$ is not a multiple of $3$. In some sense $\Delta_1$ is too big, in fact we can replace in its shape a $5$-minor by a $4$-minor, to get a partition of $39$. For instance, set
\[\Delta := \Delta_2\cdot [12345|12345]^2[1234|1234].\]
The shape of $\Delta$ is $\gamma = (5,5,5,4,4,4,3,3,2,2,1,1)$, which is a partition of $39$ with $12$ parts. Since $39=3\cdot 13$, $12<13$ and $\xx$ divides $\init(\Delta)$, the natural number $d$ we were looking for is at most $13$. Actually it is exactly $13$, and we will prove that the strategy used here to find it works in general. Before coming back to the general case, notice that in this case $k_0=3$, and $13=\lfloor 40/3 \rfloor = \lfloor m(n+k_0)/t\rfloor$.
\end{example}

First let us suppose to be in the case (i). Let us define the product of minors 
\[\Pi := \pi_1\cdots \pi_{m+n-1}\] 
where
\[\pi_i:= \begin{cases} [m-i+1,m-i+2,\ldots ,m|1,2,\ldots ,i]  & \mbox{if \ } 1\leq i< m \\
[1,2,\ldots ,m|i-m+1,i-m+2,\ldots ,i] & \mbox{if \ } m\leq i\leq n \\
[1,2,\ldots ,m+n-i|i-m+1,i-m+2,\ldots ,n] & \mbox{if \ } n<i\leq m+n-1
\end{cases}
\]
The shape of $\Pi$ is $\ll = (m^{n-m+1},(m-1)^2,(m-2)^2,\ldots ,1^2)$, which is a partition of $mn$. Moreover $\xx$ divides $\init(\Pi)$. Let $r_0$ be the unique integer such that $0\leq r_0<t$ and $mn+r_0=d_0t$. If $r_0=0$, then we put $\Delta := \Pi$ and $\gamma := \ll$. Then $\gamma$ has $m+n-1$ parts and $m+n-1< \lfloor mn/t \rfloor = mn/t =  d_0$. Since $\init(\Delta)$ is a multiple of $\xx$, we deduce that $\omega_{d_0}\neq 0$, i.e.
\begin{equation}\label{a-inv1}
a(A_t)\geq -d_0=-mn/t=-\lceil mn/t \rceil.
\end{equation}
Now suppose that $r_0>0$, i.e. $mn$ is not a multiple of $t$. So, in this case, $\lceil mn/t \rceil = \lfloor mn/t \rfloor +1$. We multiply $\Pi$ by an $r_0$-minor, for instance set
\[\Delta := \Pi \cdot [1,2,\ldots ,r_0|1,2,\ldots ,r_0].\]
The shape of $\Delta$   is $\gamma = (m^{n-m+1},(m-1)^2,(r_0+1)^2,r_0^3,(r_0-1)^2,\ldots ,1^2)$. This is a partition of $d_0t$ with $m+n$ parts. Since $m+n-1< \lfloor mn/t \rfloor$, we get $m+n<\lceil mn/t \rceil = d_0$. Since $\xx$ divides $\init(\Delta)$ (because it divides $\init(\Pi)$), we get $\omega_{d_0}\neq 0$, i.e.
\begin{equation}\label{a-inv2}
a(A_t)\geq -d_0=-\lceil mn/t \rceil.
\end{equation}
Clearly equality must hold true both in \eqref{a-inv1} and in \eqref{a-inv2}, since if $\xx$ divides a monomial $\init(\Gamma)$ for some $\Gamma \in A_t$, then $\deg(\Gamma)\geq  \lceil mn/t \rceil$.

Now let us assume to be in the case (ii). Notice that the integer $k_0$ of the assumption is bigger than $0$. Let $p_0$ be the unique integer such that $0\leq p_0<t$ and $m(n+k_0)=d_0t+p_0$. Let us define the product of minors
\[\Delta := \Pi\cdot [1,2,\ldots ,m|1,2,\ldots ,m]^{k_0-1}\cdot [1,2,\ldots , m-p_0|1,2,\ldots , m-p_0].\]
The shape of $\Delta$ is 
\[\gamma = (m^{k_0+n-m},(m-1)^2,(m-p_0+1)^2,(m-p_0)^3,(m-p_0-1)^2,\ldots ,1^2).\]
This is a partition of
$d_0t$ with $k_0+n+m-1$ parts. By the choice of $k_0$, one can verify that $k_0+n+m-1<d_0$. Furthermore, being $\Delta$ a multiple of $\Pi$, $\xx$ divides $\init(\Delta)$. So $\init(\Delta)\in \omega$, which implies $\omega_{d_0}\neq 0$ and
\begin{equation}\label{a-inv3}
a(A_t)\geq -d_0=-\lfloor m(n+k_0)/t \rfloor.
\end{equation}
To see that the $a$-invariant of $\init(A_t)$ is actually $- d_0$ in \eqref{a-inv3}, we need the following easy lemma.

\begin{lemma}\label{decofx}
With a little abuse of notation set $X:=\{x_{ij} \ : \ i=1,\ldots, m, \ j=1,\ldots ,n \}$. Define a poset structure on $X$ in the following way:
\[x_{ij}\leq x_{hk} \ \ \mbox{ if \ \ \ }i=h\mbox{ and $j=k$ \ \ or \ \ $i<h$ and $j<k$}. \]
Suppose that $X=X_1 \cup \ldots \cup X_h$ where each $X_i$ is a chain, i.e. any two elements of $X_i$ are comparable, and set $N:=\sum_{i=1}^h|X_i|$. Then
\[ \displaystyle h\geq \frac{N}{m}+m-1.\]
\end{lemma}
\begin{proof}
For any $\ell=1,\ldots ,m-1$ set 
\[X(\ell):=\{x_{ij}\in X \ : \ \mbox{or $m+{j-i}=\ell$ or $n+(i-j)=\ell$} \}.\] 
It is easy to see that for any $\ell$, $|X(\ell)|=2\ell$. Moreover, if $Y$ is a chain such that $X(\ell)\cap Y \neq \emptyset$ for some $\ell$, then $|Y|\leq \ell$. Notice that $x_{mk}\in X(k)$ for all $k=1,\ldots , m-1$. So, choosing an $i_k$ such that $x_{mk}\in X_{i_k}$, we have that $|X_{i_k}|\leq k$. Note that $i_k\neq i_h$ whenever $k\neq h$ since $x_{mk}$ and $x_{mh}$ are not comparable. In the same way we choose a $X_{j_k}$ containing $x_{1,n+1-k}$ for any $k=1,\ldots ,m-1$. Once again $|X_{j_k}|\leq k$ since $x_{1,n+1-k}\in X(k)$. Furthermore the $j_k$'s are distinct because different $x_{1,n+1-k}$'s are incomparable. Actually, for the same reason, all the $i_h$'s and $j_k$'s are distinct. In general, for any $i=1,\ldots ,h$ we have $|X_i|\leq m$. Thus, setting $A:=\{i_k,j_k \ : \ k=1,\ldots ,m-1\}$, we get
\[N=\sum_{i=1}^h |X_i|=\sum_{i\in A}|X_i| + \sum_{i\in \{1,\ldots ,h\}\setminus A}|X_i|\leq (m-1)m + m(h-2m+2),\]
which supplies the desired inequality $\displaystyle h\geq \frac{N}{m}+m-1$.
\end{proof}
Now take a product of minors $\Delta=\delta_1 \cdots \delta_h$ such that $\init(\Delta)\in \omega$. Let $\lambda$ be the shape of $\Delta$ and suppose that $|\lambda|=td$ with $d<d_0$. For any $i=1,\ldots ,h$ set 
\[X_i:=\{x_{pr}: x_{pr}|\init(\delta_i)\}.\]
Since $\xx$ divides $\init(\Delta)$, with the notation of Lemma \ref{decofx} we have that $X=\cup_{i=1}^h X_i$ where each $X_i$ is a chain with respect to the order defined on $X$. So by Lemma \ref{decofx} we have that 
\[\displaystyle h\geq \frac{dt}{m}+m-1.\]
We recall that $d_0t=mn+mk_0-p_0$, where $0\leq p_0 < t $. Of course we can write in a unique way $dt=mn+ms-q$, where $0\leq q < m$. Before going on, notice that $k_0$ is the smallest natural number $k$ satisfying the inequality
\[\displaystyle m+n+k-1 < \left\lfloor \frac{m(n+k)}{t}\right\rfloor.\]
By what said $s\leq k_0$. There are two cases:
\begin{compactenum}
\item If $s=k_0$, consider the inequalities
\[ \displaystyle m+n+(s-1)-1 = \frac{dt+q}{m} + m -2 < \frac{dt}{m}+m-1\leq h \leq d-1.\]
Notice that, since $d<d_0$, we have that $q\geq p_0+t$. Moreover $m<2t$, otherwise we would be in case (i) of the theorem. Thus
\[ d-1 = \frac{m(n+s)-q-t}{t} \leq \displaystyle \left\lfloor\frac{m(n+(s-1))}{t} \right\rfloor.\]
The inequalities above contradicts the minimality of $k_0$.
\item If $s<k_0$, then
\[ \displaystyle n+s+m -1 = \frac{dt+q}{m} + m -1 \leq h < d  = \frac{m(n+s)-q}{t} \leq \left\lfloor \frac{m(n+s)}{t}\right\rfloor.\]
Once again, this yields a contradiction to the minimality of $k_0$.
\end{compactenum}

Finally, it turns out that the Hilbert function of a graded $\kkk$-algebra and the one of its initial algebra (with respect to any term order) coincide, see Conca, Herzog and Valla \cite[Proposition 2.4]{CHV}. In particular we have $\Hf_{A_t}=\Hf_{\init(A_t)}$ and $\Hp_{A_t}=\Hp_{\init(A_t)}$. So by the characterization of the $a$-invariant given in \eqref{a-invandhilb}, we have
\[a(A_t)=a(\init(A_t)).\] 
Furthermore $\reg(A_t)=\dim A_t+a(A_t)$ from \eqref{regfroma-inv}, and $\dim A_t = mn$ by \cite[Proposition 10.16]{BrVe}.
\end{proof}

\begin{remark}
Let us look at the cases in Theorem \ref{regularity}.
\begin{compactitem}
\item[(i)] If $X$ is a square matrix, that is $m=n$, one can easily check that we are in case (i) of Theorem \ref{regularity} if and only if $m$ is at least twice the size of the minors, i.e. $m\geq 2t$.
\item[(ii)] The natural number $k_0$ of Theorem \ref{regularity} may be very large: For instance, let us consider the case $t=m-1$ and $n=m+1$. Since in the interesting situations $m\geq 3$, one can easily check that we are in the case (ii) of Theorem \ref{regularity}. In this case we have $k_0=m^2-2m-1$. Therefore Theorem \ref{regularity} yields
\[ a(A_{m-1}(m,m+1))=-m^2\]
and
\[\reg(A_{m-1}(m,m+1))=m.\]
\end{compactitem}
\end{remark}

Since $\reg(J_t)=\reg(A_t)+1$ and $\reg(J_t)$ bounds from above the degree of a minimal generator of $J_t$ by Theorem \ref{reggdegree}, as a consequence of Theorem \ref{regularity} we get the following:

\begin{corollary}\label{upperbound}
Let us call $d$ the maximum degree of a minimal generator of $J_t$. 
\begin{compactitem}
\item[(i)] If $m+n-1<\lfloor mn/t \rfloor$, then 
\[d \leq  mn-\lceil mn/t \rceil +1.\]
\item[(ii)] Otherwise, i.e. if $m+n-1\geq \lfloor mn/t \rfloor$, we have 
\[d\leq mn-\lfloor m(n+k_0)/t \rfloor +1, \]
where $k_0 = \lceil (tm+tn-mn)/(m-t) \rceil$.
\end{compactitem}
\end{corollary}

For instance, the degree of a minimal generator of $J_2$ can be at most $\lfloor mn/2 \rfloor + 1$. Instead, a minimal generator of $J_{m-1}(m,m+1)$ has degree at most $m+1$.
\index{Castelnuovo-Mumford regularity|)}\index{a-invariant@$a$-invariant|)}
\subsection{The independence on $n$ for the minimal relations}

Given a degree $d$ minimal relation between $t$-minors of the $m\times n$ matrix $X$, clearly it insists at most on $td$ columns of $X$. Therefore it must be a minimal relation already in a $m\times td$ matrix. This fact can be useful when $d$ is small: For instance, to see if there are minimal relations of degree $4$ between $2$-minors of a $3\times n$-matrix, it is enough to check if they are in a $3\times 8$ matrix. However, even once checked that there are no minimal relations of degree $4$, there might be anyway of degree $5$, $6$ or so on; and with $d$ growing up this observation is useless. In fact, unfortunately, so far the best upper bound we have for the degree of a minimal relation is given by the Castelnuovo-Mumford regularity of $A_t(m,n)$, see Theorem \ref{regularity}. For example, the degree of a minimal generator of $J_2(3,n)$ might be $d=\lfloor 3n/2 \rfloor +1$, provided $n\geq  6$. So we should control if such a minimal relation is in a $m\times td$ matrix; but $td=2\lfloor 3n/2 \rfloor +2 > n$, so we would be in a worse situation than the initial one.
In this subsection we will make an observation somehow similar, but finer, to the one discussed above. The size of the ``reduction-matrix", besides being independent on $n$, will not even depend on $d$. Precisely, in Theorem \ref{independencefromn} we will show that a a degree $d$ minimal relation between $t$-minors of an $m\times n$ matrix must be already in a $m\times (m+t)$ matrix! Thus, in principle, once fixed $m$ we  could check by hand the maximum degree of a minimal generator of $J_t(m,n)$. In practice, however, a computer can actually supply an answer just for small values of $m$. Actually the proof of Theorem \ref{independencefromn} is not very difficult: Essentially we have just to exploit the structure as $G$-representations of our objects.

\begin{thm}\label{independencefromn}
Let $d(t,m,n)$ denote the highest degree of a minimal generator of $J_t(m,n)$. Then
\[d(t,m,n)\leq d(t,m,m+t)\]
\end{thm}
\begin{proof}
We can assume to be in a numerical case different from \eqref{easycases}. So $d(t,m,n)$ is the highest degree of a minimal generator of $K_t(m,n)$, too (Lemma \ref{redtotensor}).
Therefore suppose that $(\gamma|\lambda)$ is a minimal irreducible representation of $K_t(m,n)$. Then we claim that $\lambda_1 \leq m+t$: If not, for any predecessor $(\gamma'|\ll')$ of $(\gamma |\lambda)$ we have that $\ll'_1>m$. On the other side $\gamma_1'\leq m$. This implies that any predecessor of $(\gamma |\lambda)$ is asymmetric, and so for any $(\gamma'|\ll')$ predecessor of $(\gamma |\lambda)$
\[L_{\gamma'}W \otimes L_{\ll'}V^* \subseteq K_t(m,n).\]
Since there must exist some predecessor $(\gamma'|\ll')$ such that
\[\Wc \otimes \Vl^* \subseteq (L_{\gamma'}W \otimes L_{\ll'}V^*)\otimes (\bigwedge^t W \otimes \bigwedge^t V^*),\]
it turns out that $(\gamma|\lambda)$ cannot be minimal.
So if $(\gamma|\lambda)$ is minimal in $K_t(m,n)$, then $\ll_1\leq m+t$. Let us consider the canonical bi-tableu of $(\gamma|\lambda)$ , namely $(c_{\gamma}|c_{\lambda})$. Such a bi-tableu corresponds to a minimal generator of $K_t(m,n)$. By the definition of the Young symmetrizers (see \ref{subsecshurmodules}), actually $(c_{\gamma}|c_{\ll})$ can be seen as an element of $T_t(m,m+t)$, because $\ll_1\leq m+t$. So it belongs to $K_t(m,m+t)$. Furthermore, if $(c_{\gamma}|c_{\ll})$ were not minimal in $K_t(m,m+t)$, all the more reason it would not be minimal in $K_t(m,n)$. This thereby implies the thesis.
\end{proof}

So, putting together Corollary \ref{upperbound} and Theorem \ref{independencefromn}, we get:

\begin{corollary}\label{generalupper}
Let $d(t,n,m)$ be as in Theorem \ref{independencefromn}.
\begin{compactitem}
\item[(i)] If $m+t-1 < \lfloor m^2/t \rfloor$, then
\[  d(t,m,n)\leq m^2 + m(t-1) - \lceil m^2/t \rceil +1. \]
\item[(ii)] Otherwise, we have
\[  d(t,m,n)\leq m^2 + m(t-1) - \lfloor m(m+k_0)/t \rfloor +1, \]
where $k_0=\lceil (t^2+tm-m^2)/(m-t)\rceil$.
\end{compactitem}
\end{corollary}

For instance, Corollary \ref{generalupper} implies that a minimal relation between $2$-minors of a $3\times n$ matrix  has degree at most $7$. Actually we will see in the next subsection that such a relation is at most a cubic.

\subsection{Relations between $2$-minors of a $3\times n$ and a $4\times n$ matrix}

In this paragraph we want to explain the strategy to prove the following result:

\begin{thm}\label{settled cases}
For all $n\geq 4$, the ideals $J_2(3,n)$ and $J_2(4,n)$ are generated by quadrics and cubics.
\end{thm}

Of course we have to use Theorem \ref{independencefromn}. It implies that $d(2,3,n)\leq d(2,3,5)$ and $d(2,4,n)\leq d(2,4,6)$. Since we already know that there are minimal relations of degree $2$ and $3$ by Corollary \ref{cubics}, we have just to show that $d(2,3,5)\leq 3$ and $d(2,4,6)\leq 3$. The strategy to prove these inequalities is the same; however from a computational point of view we must care more attentions to show the second one, since its verification might require some days. For this reason we will present the strategy to prove the second inequality:
\begin{compactitem}
\item[(1)] Set $J:=J_2(4,6)$, $S:=S_2(4,6)$ and, for any $d\in \NN$, let $J_{\leq d}\subseteq J$ denote the ideal generated by the polynomials of degree less than or equal to $d$ of $J$. By elimination (for instance see Eisenbud \cite[15.10.4]{eisenbud}) one can compute a set of generators of  $J_{\leq 3}$.
\item[(2)] Fixed some term order, for instance degrevlex, we compute a Gr\"obner basis of $J_{\leq 3}$ up to degree $13$. So we get $B:=\init(J_{\leq 3})_{\leq 13}$.
\item[(3)] Let us compute the Hilbert function of $S/B$. Clearly we have 
\[\Hf_{S/J_{\leq 3}}(d)\leq \Hf_{S/B}(d),\] 
where equality holds true provided that $d\leq 13$.
\item[(4)] We recall that a decomposition of $A_t(m,n)$ in irreducible $G$-representations is known by Theorem \ref{decat}. Moreover any irreducible $G$-module appearing in it is of the form $\Wl \otimes \Vl^*$ for some partition $\ll$. It turns out that there is a formula to compute the dimension of such vector spaces, namely the hook length formula (for instance see the book of Fulton \cite[p. 55]{Fu}). So we can quickly get the Hilbert function of $A_2(4,6)=S/J$. Let us compute it up to degree $13$.
\item[(5)] Since $J_{\leq 3}\subseteq J$, we will have that $\Hf_{S/B}(d)\geq \Hf_{S/J}(d)$. However, comparing the two Hilbert functions, one can check that $\Hf_{S/B}(d)= \Hf_{S/J}(d)$ for any $d\leq 13$. This implies that $J_{\leq 3}=J_{\leq 13}$.
\item[(6)] Corollary \ref{upperbound} yields that a minimal generator of $J$ has at most degree $13$. Therefore $J_{\leq 13}=J$. 
So we are done, since $J_{\leq 3}=J_{\leq 13}$.
\end{compactitem}

\vskip 1mm

We used the computer algebra system Singular, \cite{singular}. The employed machine took about $60$ hours to compute the generators of $J_{\leq 3}$, and about $15$ hours to compute a Gr\"obner basis of it up to degree $13$. This means that probably is necessary to bound the degrees, if not the computation might not finish. At the contrary, the computation of $J_2(3,5)$ can be done without any restrictions: It is just a matter of seconds.
%
%\begin{remark}
%To the aim of proving by computer that $d(2,3,5)=3$ we could also compare the Hilbert function of $J_2(3,5)$ with that of the ideal generated by the Pl\"ucker relations and the shape relations (we explained how to find the polynomials associated to them in the previous section). Of course this argument applies also in the other cases. Another possible refinement is to give an upper bound for $d(2,3,5)$ by studying the ideal defining the initial algebra of $A_2(3,5)$, whose a finite Sagbi basis will be described in the next section (since the initial algebra is toric, the computation will be quite fast). For instance, in this way we obtain $d(2,3,5)\leq 4$. So the Hilbert functions should be compared just up to degree $4$ (instead of $8$).
%\end{remark}

\subsection{The uniqueness of the shape relations}\index{relations!shape|(}\index{predecessor|(}

Theorem \ref{shapeink3} implies that some cubics lie in $J_t$. But are we sure that we cannot find some other higher degrees minimal relations with similar arguments? We recall that those cubics correspond to some bi-diagrams, namely the shape relations, which are irreducible representations of $T_t$. The crucial properties we needed to prove that these bi-diagrams actually correspond to minimal generators of $J_t$ have been that they are asymmetric diagrams of multiplicity one, whose only predecessor is symmetric.
In this subsection we are going to prove that, but the shape relations and those of degree less than $3$, there are no other bi-diagrams of $T_t$ satisfying these properties. In a certain sense, this implies that the bi-diagrams that are minimal irreducible representations of $J_t$ for {\it reasons of shape} are just in degree $2$ and $3$. So the question is: Are there any other minimal irreducible representations of $J_t$ for {\it reasons of multiplicity}?

To our purpose we introduce some notation and some easy facts. For the next lemma let us work just with $V$. Let $\lambda=(\lambda_1,\ldots , \lambda_k)$ be a partition. Pieri's formula (Theorem \ref{pieri})\index{Pieri's formula} implies that a diagram $\Vl$ is an irreducible $\GL(V)$-subrepresentation of $\bigotimes^d \bigwedge^t  V$ if and only if $|\lambda|=td$ and $k\leq d$. As in the case of bi-diagrams we say that a partition $\lambda'\vdash t(d-1)$ with at most $d-1$ parts is a predecessor\index{predecessor} of $\lambda$ if $\lambda'\subseteq \lambda \subseteq \ll'(d)$. Moreover we define the {\it difference sequence}\index{partition!difference sequence} $\Delta \lambda := (\Delta \lambda_1,\ldots , \Delta \lambda_k)$ where $\Delta \lambda_i:=\lambda_i-\lambda_{i+1}$, setting $\lambda_{k+1}=0$. Notice that in order to get a predecessor of $\lambda$ we can remove $q$ boxes from its $i$th row only if $\Delta \lambda_i\geq q$. Moreover notice that $\Delta \lambda_1+\ldots + \Delta \lambda_k = \lambda_1\geq t$.

\begin{lemma}\label{onepredecessor}
Let $\lambda=(\lambda_1,\ldots , \lambda_k)\vdash td$ be a partition such that $\Vl\subseteq \bigotimes^d \bigwedge^t  V$, with $d>1$. Then $\lambda$ has a unique predecessor if and only if either $\lambda_1=\ldots =\lambda_k$ ($\lambda$ is a rectangle) or there exist $i$ such that $\lambda_1=\ldots =\lambda_i>\lambda_{i+1}=\ldots =\lambda_k$ and $k=d$ ($\lambda$ is a fat hook).
\end{lemma}
\begin{proof}
If $\lambda$ is a rectangle its difference sequence is
\[ \Delta \lambda = (\underbrace{0,0,\ldots ,0}_{{\mbox{\footnotesize $k-1$ times}}},\lambda_1).\]
This means  that the only way to remove $t$ boxes from different columns of $\lambda$ is to do it from its last row. Therefore $\lambda$ has a unique predecessor, namely
\[\lambda'=(\underbrace{\lambda_1,\lambda_1,\ldots ,\lambda_1}_{{\mbox{\footnotesize $k-1$ times}}},\lambda_1-t).\]
If $\lambda$ is a fat hook, instead, its difference sequence is
\[ \Delta \lambda = (\underbrace{0,0,\ldots ,0}_{{\mbox{\footnotesize $i-1$ times}}},\lambda_1-\lambda_k,\underbrace{0,0,\ldots ,0}_{{\mbox{\footnotesize $k-i-1$ times}}},\lambda_k).\]
Since $k=d$, to get a predecessor of $\ll$ we must remove completely the last row. So the only nonzero entry which remains is $\Delta \lambda_i$. Therefore the only predecessor of $\lambda$ is
\[ \lambda'=(\underbrace{\lambda_1,\lambda_1,\ldots ,\lambda_1}_{{\mbox{\footnotesize $i-1$ times}}},\lambda_1+\lambda_k-t,\underbrace{\lambda_k,\lambda_k,\ldots ,\lambda_k}_{{\mbox{\footnotesize $k-i-1$ times}}}).\]

\vskip1mm

For the converse, first suppose that $k<d$. Then in order to get a predecessor of $\lambda$  we have not to care about removing the last row. Moreover, in this case, $\Delta \lambda_1+\ldots + \Delta \lambda_k=\ll_1>t$. This implies that if two of the $\Delta \lambda_i$'s were bigger than $0$, more than one choice would be possible: So $\ll$ would have more than one predecessor. So there must be just one $i$ such that $\Delta\ll_i> 0$. In other words, $\lambda$ has to be a rectangle.

Now suppose $k=d$. In this case we must remove the $k$th row. If $\lambda_k=t$, then $\lambda$ will be a rectangle. So we can assume that $\lambda_k<t$ (consequently $\ll_1>t$). This means that after removing the last row we can remove freely $t-\lambda_k$ boxes from the other. Because $\Delta \lambda_1+\ldots + \Delta \lambda_{k-1}=\ll_1-\ll_k > t-\ll_k$, the choice of removing some boxes to get a predecessor will be unique only if just one among the $\Delta \lambda_i$ is bigger than $0$. This means that $\lambda$ has to be a fat hook.
\end{proof}

\begin{corollary}\label{mult1}
Given a partition $\lambda=(\lambda_1,\ldots , \lambda_k)\vdash td$ (with $d\geq 2$), the irreducible $\GL(V)$-representation $\Vl$ appears with multiplicity one in $\bigotimes^d \bigwedge^t  V$ if and only if $\lambda$ has only one predecessor which either is $(t)$, or, in turn, has only one predecessor. These facts are equivalent to the fact that $\ll$ is a diagram of the following list:
\begin{compactenum}
\item a rectangle with one row ($k=1$);
\item a rectangle with $d-1$ rows;
\item a rectangle with $d$ rows;
\item a fat hook of the type $\lambda_1>\lambda_2=\ldots = \lambda_k$ with $k=d$;
\item a fat hook of the type $\lambda_1=\ldots =\lambda_{k-1}> \lambda_k$ with $k=d$;
\end{compactenum}
\end{corollary}
\begin{proof}
By Pieri's formula (Theorem \ref{pieri}),  $\Vl$ has multiplicity one in $\bigotimes^d\bigwedge^tV$ if and only if $\ll$ has only a predecessor $\ll'$ such that $L_{\ll'}V$ has multiplicity one in $\bigotimes^{d-1}\bigwedge^t V$. Thus we can argue by induction on $d$. For $d=2$, the only predecessor of $\ll$ is $\ll'=(t)$, and $L_{(t)}V\cong \bigwedge^t V$ has obviously multiplicity one in $\bigwedge^tV$. For $d>2$, the $\ll$'s appearing in the list of the statement have just one predecessor by Lemma \ref{onepredecessor}. Furthermore, it is easy to check that if $\ll$ is a partition in the list, its unique predecessor $\ll'$ is in the list as well as $\ll$ (of course we mean with $d$ replaced by $d-1$). So $L_{\ll'}V$  has multiplicity one in $\bigotimes^{d-1}\bigwedge^t V$ by induction. To see that there are no other $\ll$ but those in the list, we have to check that: If a partition $\ll$ is not in the list and has only one predecessor $\ll'$, then $\ll'$ has more than one predecessor. This is easily checkable using Lemma \ref{onepredecessor}.
\end{proof}

Now we are ready to prove the announced fact that the only minimal relations for reasons of shape are in degree $2$ and $3$.

\begin{prop}\label{noothershape}
Let $(\gamma | \lambda)$ be an asymmetric bi-diagram such that $L_{\gamma}W\otimes \Vl^*$ appears in $T_d$ with multiplicity one and such that its only predecessor is symmetric. Then $d=2$ or $d=3$ and $(\gamma | \lambda)$ is a shape relation.
\end{prop}
\begin{proof}
Since $(\gamma | \lambda)$ is an irreducible $G$-subrepresentation of $T_d$, we have $|\gamma|=|\lambda|=td$. For the proof we have to reason by cases.
Set:
$$A_1:=\{\alpha \mbox{ $d$-admissible } \ : \ \alpha \mbox{ is a rectangle with $1$ row}\}$$
$$A_2:=\{\alpha \mbox{ $d$-admissible } \ : \ \alpha \mbox{ is a rectangle with $d-1$ rows}\}$$
$$A_3:=\{\alpha \mbox{ $d$-admissible } \ : \ \alpha \mbox{ is a rectangle with $d$ rows}\}$$
$$A_4:=\{\alpha \mbox{ $d$-admissible } \ : \ \alpha \mbox{ is a fat hook of the type $\lambda_1>\lambda_2=\ldots = \lambda_k$ with $k=d$}\}$$
$$A_5:=\{\alpha \mbox{ $d$-admissible } \ : \ \alpha \mbox{ is a fat hook of the type $\lambda_1=\ldots =\lambda_{k-1}> \lambda_k$ with $k=d$}\}$$

Since $(\gamma|\lambda)$ has multiplicity one, by Corollary \ref{mult1} there exist $i,j$ such that $\gamma\in A_i$ and $\lambda \in A_j$. Denote by $(\gamma'|\lambda')$ the only predecessor of $(\gamma|\lambda)$. We can assume that $d\geq 3$.
\begin{enumerate}
\item $\gamma \in A_1$.
\begin{compactenum}
\item If $\lambda\in A_1$ then $(\gamma|\lambda)$ would not be asymmetric;
\item If $\lambda\in A_2$ then $(\gamma'|\lambda')$ would not be symmetric;
\item If $\lambda\in A_3$ then $(\gamma'|\lambda')$ would not be symmetric;
\item If $\lambda\in A_4$ then $(\gamma'|\lambda')$ would not be symmetric;
\item If $\lambda\in A_5$ then $(\gamma'|\lambda')$ would not be symmetric;
\end{compactenum}
\item $\gamma\in A_2$.
\begin{compactenum}
\item If $\lambda\in A_2$ then $(\gamma|\lambda)$ would not be asymmetric;
\item If $\lambda\in A_3$ then $(\gamma'|\lambda')$ would not be symmetric;
\item If $\lambda\in A_4$ and if $d\geq 4$ then $(\gamma'|\lambda')$ would not be symmetric. If $d=3$ then actually $(\gamma|\lambda)$ is a shape relation;
\item If $\lambda\in A_5$ then $(\gamma'|\lambda')$ would not be symmetric ($\gamma'_1>\lambda'_1$);
\end{compactenum}
\item $\gamma\in A_3$.
\begin{compactenum}
\item If $\lambda\in A_3$ then $(\gamma|\lambda)$ would not be asymmetric;
\item If $\lambda\in A_4$ then $(\gamma'|\lambda')$ would not be symmetric ($\gamma'_2=t>\lambda'_2$);
\item If $\lambda\in A_5$ then $(\gamma'|\lambda')$ would not be symmetric ($\gamma'_1=t<\lambda'_1$);
\end{compactenum}
\item $\gamma\in A_4$.
\begin{compactenum}
\item If $\lambda\in A_4$ then, since $(\gamma'|\lambda')$ has to be symmetric, $(\gamma|\lambda)$ would have to be symmetric as well;
\item If $\lambda\in A_5$ and $d\geq 4$, then $(\gamma'|\lambda')$ would not be symmetric; if $d=3$, then $(\gamma|\ll)$ is a shape relation.
\end{compactenum}
\item $\gamma\in A_5$.
\begin{compactenum}
\item If $\lambda\in A_5$ then, since $(\gamma'|\lambda')$ has to be symmetric, $(\gamma|\lambda)$ would have to be symmetric as well.
\end{compactenum}
\end{enumerate}

\end{proof}
\index{relations!shape|)}\index{predecessor|)}

\section{The initial algebra and the algebra of $U$-invariants of $A_t$}

In this section we will exhibit a system of generators for: (i). The initial algebra of $A_t$ with respect to a diagonal term order. (ii). The subalgebra of $U$-invariants of $A_t$. In general both the initial algebra and the ring of invariants of a finitely generated $\kkk$-algebra might be not finitely generated. We will prove that the two $\kkk$-algebras considered above are finitely generated, but especially we will give a finite system of generators of them.

\subsection{A finite Sagbi basis of $A_t$}\index{initial algebra|(}\index{Sagbi bases|(}

In \cite[Theorem 3.10]{BC3}, Bruns and Conca described a Gr\"obner basis, with respect to any diagonal term order, for every powers of the determinantal ideals $I_t$. This allows us to describe when a monomial of $R=\kkk[X]$ belongs to $\init(A_d)$. In fact, we have
\[\init(A_t) = \bigoplus_{d\in \NN}\init(I_t^d \cap R_{td}).\]

\begin{lemma}\label{initialalgebra}
A monomial $M\in R$ belongs to $\init(A_t)_d$ if and only if $M=M_1\cdots M_k$ where the $M_q$'s are monomials of $R$ such that
\begin{compactenum}
\item For any $q=1,\ldots , k$ the monomial $M_q$ is the initial term of an $r_q$-minor.
\item The partition $(r_1,r_2,\ldots ,r_k)$ is $d$-admissible.
\end{compactenum}
\end{lemma}

Although in \cite[Theorem 3.11]{BC3} the authors showed that $\init(A_t)$ is a finitely generated $\kkk$-algebra, they could not specify a finite Sagbi basis of $A_t$ (see \cite[Remark 7.11 (b)]{BC2}). Actually they could not even give an upper bound for the degree of a minimal generator of $\init(A_t)$. We will be able to do it. To this aim, first, we need some notions about partitions of integers (to delve more into such an argument see the sixth chapter of the book of Sturmfels \cite{sturmfels2}).

A {\it partition identity}\index{partition identity} is any identity of the form
\begin{equation}\label{pi}
a_1+a_2+\ldots +a_k = b_1+b_2+\ldots +b_l
\end{equation}
where $a_i,b_j \geq 1$ are integers. Fixed a positive integer $q$ we will say that \eqref{pi} is a {\it partition identity with entries in $[q]$}\index{partition identity!with entries in $[q]$} if furthermore $a_i,b_j\leq q$.
The partition identity \eqref{pi} is called {\it primitive}\index{partition identity!primitive} if there is no proper subidentity
\[a_{i_1}+a_{i_2}+\ldots +a_{i_r}=b_{j_1}+b_{j_2}+\ldots +b_{j_s}\]
with $r+s<k+l$.
We say that the partition identity \ref{pi} is {\it homogeneous}\index{partition identity!homogeneous} if $k=l$.
It is {\it homogeneous primitive}\index{partition identity!homogeneous!primitive} if there is no proper subidentity
\[a_{i_1}+a_{i_2}+\ldots +a_{i_r}=b_{j_1}+b_{j_2}+\ldots +b_{j_r}\]
with $r<k$.

\begin{thm}\label{sagbi}
Let $d$ denote the maximum degree of a minimal generator lying in $\init(A_t)$. Then $d\leq m-1$. Furthermore $d\leq m-2$ if and only if $\GCD(m-1,t-1)=1$.
\end{thm}
\begin{proof}
Let $M:=M_1\cdots M_k$ be a product of initial terms of minors, say $M_i:=\init(\delta_i)$ where $\delta_i$ is an $m_i$-minor of $X$. Let $td$ be the degree of $M$, so that $\sum_{i=1}^k m_i=td$.

If $k<d$ the monomial $M$ cannot be a minimal generator of $A_t$: In fact, since $m_1>t$,
\[ M_1=\init(\delta_1')\cdot \init(\delta_1''),\]
where $\delta_1'$ is an $m_1-t$-minor and $\delta_1''$ is a $t$-minor. Since $k\leq d-1$, the monomial
\[ M':=\init(\delta_1')\cdot M_2 \cdots M_k\]
belongs to $\init(A_t)_{d-1}$ by Lemma \ref{initialalgebra}. Moreover $\init(\delta_1'')$ obviously belongs to $\init(A_t)_1$. Thus, as $M=M'\cdot \init(\delta_1'')$, it is not a minimal generator.

So we can assume that $k=d$. This means that we have a homogeneous partition identity with entries in $[m]$ of the kind
\[ m_1+m_2+ \ldots + m_d=\underbrace{t+t+\ldots +t}_{d \mbox{ times}}. \]
If $d\geq m$ then the above is not a homogeneous primitive partition identity by \cite[Theorem 6.4]{sturmfels2}. Therefore there is a subset $S\subsetneq [d]$ such that                          \[ \sum_{i\in S}m_i=|S|\cdot t, \]
and as a consequence
\[ \sum_{i\in [d]\setminus S}m_i=(d-|S|)\cdot t.\]
Therefore by Lemma \ref{initialalgebra} $M':=\prod_{i\in S}M_i\in \init(A_t)_{|S|}$ and $M'':=\prod_{i\in [d]\setminus S}M_i\in \init(A_t)_{d-|S|}$, so that $M=M'\cdot M''$ is not a minimal generator of $\init(A)$.

For the second part of the statement, arguing as in the proof of \cite[Theorem 6.4]{sturmfels2} one can deduce that the only possible homogeneous primitive partition identity with entries in $[m]$ of the kind
\[ \lambda_1+\lambda_2+\ldots +\lambda_{m-1}=\underbrace{t+t+\ldots +t}_{m-1 \mbox{ times}} \]
is the following one:
\[\underbrace{1+1+\ldots +1}_{m-t \mbox{ times}} + \underbrace{m+m+\ldots +m}_{t-1 \mbox{ times}}=\underbrace{t+t+\ldots +t}_{m-1 \mbox{ times}}.\]
The above is not a homogeneous primitive partition identiy if and only if there exist two natural numbers, $p\leq m-t$ and $q\leq t-1$, such that $q(m-t) = p(t-1)$ and $p+q<m-1$. This is possible if and only if $\GCD(m-t,t-1)=\GCD(m-1,t-1)> 1$.
\end{proof}

\begin{corollary}
A finite Sagbi basis of $A_t$ is formed by the product of minors belonging to it and of degree at most $m-1$.
\end{corollary}

\begin{remark}
An upper bound for the degree of the minimal generators of $J_t$ might be gotten by studying the defining ideal of the initial algebra of $A_t$. The minimal generators of such an ideal are binomials, so it could be not impossible to find them. Although there is no hope that this ideal is generated in degree at most $3$, there is a hope to obtain, in this way, an upper bound linear in $m$ (the one we got in Corollary \ref{generalupper} is quadratic in $m$). Furthermore passing  through the initial algebra would save us from using representation theory, so the results would hold in any not exceptional characteristic ($\operatorname{char}(\kkk)=0$ or $\operatorname{char}(\kkk)>\min\{t,m-t\}$). In this direction we can prove that the ideal defining the initial algebra of $A_2(3,n)$ is generated in degree $2$, $3$ and $4$. However we will not include this result in the thesis, since its proof is boring and it would extend too weakly (with respect to the effort we would do proving it) our knowledge about the relations.
\end{remark}

\index{initial algebra|)}\index{Sagbi bases|)}

\subsection{A finite system of generators for $A_t^U$}\index{ring of invariants|(}\index{U-invariant@$U$-invariant|(}

Techniques similar to those used in the previous subsection can be also used to exhibit a finite  system of generators for an important subalgebra of $A_t$. Let $U_-(W)$ be the subgroup of the lower triangular matrices of $\GL(W)$ with $1$'s on the diagonal and $U_+(V)$ be the subgroup of the upper triangular matrices of $\GL(V)$ with $1$'s on the diagonal. Then, set 
\[U:=U_-(W)\times U_+(V)\subseteq G.\] 
The group $U$ plays an important role: In fact one can show that a polynomial $f\in R$ is $U$-invariant if and only if it belongs to the $\kkk$-vector space generated by the highest bi-weight vectors with respect to $B$ (for the definition of highest bi-weight vector and of $B$ see \ref{sacminors}, for the proof of the above fact see \cite[Proposition 11.22]{BrVe}). Therefore, we have a description for the ring of $U$-invariants of $A_t$, namely
\[A_t^U=\kkk \oplus <[c_{\ll}|c_{\ll}] \ : \ \ll \mbox{ is an admissible partition}>.\]
In this subsection we will find a finite set of ``basic invariants" of $A_t$, that is a finite system of $\kkk$-algebra generators of $A_t^U$.
In other words, we will solve what is known as the {\it first main problem} of invariant theory in the case of $A_t^U$.

\begin{thm}\label{invariant}
The ring of invariants $A_t^U$ is generated by the product of initial minors $[c_{\lambda}|c_{\lambda}]$ where $\lambda$ is an admissible partition with at most $m-1$ parts. Furthermore is generated up to degree $m$.
\end{thm}
\begin{proof}
Let $A$ be the subalgebra of $A_t^U$ generated by $[c_{\lambda}|c_{\lambda}]$ where $\lambda$ is an admissible partition with at most $m-1$ parts. We have to prove that actually $A=A_t^U$.

Thus let $\lambda = (\lambda_1,\ldots ,\lambda_k)$ be an admissible partition with $k\geq m$ such that, by contradiction, $[c_{\lambda}|c_{\lambda}]$ does not belong to $A$. Furthermore assume that among such partitions $\lambda$ is minimal with respect to $k$. First, let us suppose that $|\lambda|=kt$. Then, as in the proof of Theorem \ref{sagbi}, the following
\[ \lambda_1+\lambda_2+ \ldots + \lambda_k=\underbrace{t+t+\ldots +t}_{k \mbox{ times}} \]
is not a homogeneous primitive partition identity by \cite[Theorem 6.4]{sturmfels2}. Therefore there is a subset $S\subsetneq [k]$ such that $\sum_{i\in S}\lambda_i=|S|\cdot t$ and $\sum_{i\in [k]\setminus S}\lambda_i=(k-|S|)\cdot t$.
We set $\lambda_S:=(\lambda_i:i\in S)$ and $\lambda_{[k]\setminus S}:=(\lambda_i:i\in [k]\setminus S)$. These are both admissible partitions so $[c_{\lambda_S}|c_{\lambda_S}]$ and $[c_{\lambda_{[k]\setminus S]}}|c_{\lambda_{[k]\setminus S}}]$ are elements of $A$.  We claim that
\[[c_{\lambda_{S}}|c_{\lambda_{S}}]\cdot [c_{\lambda_{[k]\setminus S}}|c_{\lambda_{[k]\setminus S}}] = [c_{\lambda}|c_{\lambda}].\]
In fact $[c_{\lambda_{S}}|c_{\lambda_{S}}]\cdot [c_{\lambda_{[k]\setminus S}}|c_{\lambda_{[k]\setminus S}}]$ is obviously $U$-invariant. Moreover it has bi-weight $((\tl\ll_1,\ldots ,\tl\ll_m)|(-\tl\ll_n,\ldots ,-\tl\ll_1))$, so it is actually $[c_{\lambda}|c_{\lambda}]$.

It remains the case in which $|\lambda|=dt$ with $d>k$. Notice that we can assume that $\ll_i\neq t$ for all $i=1,\ldots ,k$: Otherwise, putting $\ll':=(\ll_j:j\neq i)$, we would have $[c_{\ll}|c_{\ll}]=[c_{\ll'}|c_{\ll'}]\cdot [12|12]$. We have the partition identity with entries in $[m]$
\[\lambda_1+\lambda_2+\ldots \lambda_k=\underbrace{t+t+\ldots +t}_{d \mbox{ times}}.\]
Set $s:=\max\{i \ : \ \lambda_i > t\}$. We can consider the partition identity with entries in $[m]$
\begin{equation}
\delta_1+ \delta_2 + \ldots +\delta_s = \delta_{s+1} + \delta_{s+2}+\ldots +\delta_d \label{proofinv}
\end{equation}
where $\delta_i := |\lambda_i-t|$ for $i\leq k$ and $\delta_i=t$ for $k<i\leq d$. We claim that the above partition identity is primitive. By the contrary, suppose that there exist $\{i_1,\ldots , i_p\}\subseteq [s]$ and $\{j_1,\ldots ,j_q\}\subseteq \{s+1,\ldots ,d\}$ with $p+q<d$ such that
\[\delta_{i_1}+ \delta_{i_2} + \ldots +\delta_{i_p} = \delta_{j_1} + \delta_{j_2}+\ldots +\delta_{j_q}.\]
Therefore if $\{u_1,\ldots ,u_{s-p}\}=[s]\setminus \{i_1,\ldots ,i_p\}$ and $\{v_1,\ldots ,v_{d-s-q}\}=\{s+1,\ldots ,d\}\setminus \{j_1,\ldots ,j_q\}$ we also have
\[\delta_{u_1}+ \delta_{u_2} + \ldots +\delta_{u_{s-p}} = \delta_{v_1} + \delta_{v_2}+\ldots +\delta_{v_{d-s-q}}.\]
Setting $r:=\max\{a \ : \ j_a\leq k \}$ and $l:=\max\{b \ : v_b\leq k \ \}$ we get
\[\lambda_{i_1}+ \lambda_{i_2} + \ldots +\lambda_{i_p}+\lambda_{j_1}+\ldots + \lambda_{j_r} = \underbrace{t+t+\ldots +t}_{p+q \mbox{ times}}\]
and
\[\lambda_{u_1}+ \lambda_{u_2} + \ldots +\lambda_{u_{s-p}}+\lambda_{v_1}+\ldots + \lambda_{v_{l}} = \underbrace{t+t+\ldots +t}_{d-p-q \mbox{ times}}.\]
Let $\alpha:=(\lambda_{i_1}, \ldots ,\lambda_{i_p}, \lambda_{j_1},\ldots , \lambda_{j_r})$ and $\beta:=(\lambda_{u_1}, \ldots ,\lambda_{u_{s-p}},\lambda_{v_1},\ldots , \lambda_{v_{l}} )$ be the respective two partitions. Since $p+r\leq p+q$ and $s-p+l\leq d-p-q$ both $\alpha$ and $\beta$ are admissible partitions. So, as above $ [c_{\lambda}|c_{\lambda}] = [c_{\alpha}|c_{\alpha}]\cdot [c_{\beta}|c_{\beta}]$ belongs to $A$, which is a contradiction. Therefore the partition \eqref{proofinv} must be primitive. So, using \cite[Corollary 6.2]{sturmfels2} we get, $d\leq m-t+t=m$, and therefore $k\leq m-1$.
\end{proof}

As a consequence of Theorem \ref{invariant} we have an upper bound on the maximum degree of a generator of any prime ideal of $A_t$ which is a $G$-subrepresentation, for short a {\it $G$-stable prime ideal}.

\begin{corollary}
A $G$-stable prime ideal of $A_t(m,n)$ is generated in degree less than or equal to $m$.
\end{corollary}
\begin{proof}
Suppose to have a minimal generator of degree greater than $m$ in a $G$-stable prime ideal $\wp\subseteq A_t$. Since $\wp$ is a $G$-subrepresentation of $A_t$, than we can suppose that such an element is of the type $[c_{\ll} |c_{\ll}]$ with $|\ll |>mt$. By Theorem \ref{invariant} there exist two admissible partitions $\gamma$ and $\delta$ such that $|\gamma|,|\delta| < |\ll |$ and  $[c_{\gamma}|c_{\gamma}]\cdot[c_{\delta}|c_{\delta}]=[c_{\ll}|c_{\ll}]$. Since $[c_{\ll}|c_{\ll}]$ is a minimal generator, this contradicts the primeness of $\wp$.
\end{proof}

\index{ring of invariants|)}\index{U-invariant@$U$-invariant|)}
\index{generic matrix|)}\index{generic matrix!minors of|)}\index{relations|)}\index{relations!between minors|)}

\chapter{Symbolic powers and Matroids}\label{chapter4}\index{symbolic power|(}\index{matroid|(}

This chapter is shaped on our  paper \cite{va3}. It is easy to show that, if $S:=\kkk[x_1,\ldots ,x_n]$ is a polynomial ring in $n$ variables over a field $\kkk$ and $I$ is a complete intersection ideal, then $S/I^k$ is Cohen-Macaulay for all positive integer $k$. A result of Cowsik and Nori, see \cite{cowsiknori}, implies that the converse holds true, provided the ideal is homogeneous and radical. Therefore, somehow the above result says that there are no ideals with this property but the trivial ones. On the other hand, if $S/I^k$ is Cohen-Macaulay, then $I^k$ has not any embedded prime ideals, so by definition it is equal to the $k$th symbolic power of $I$, namely $I^k=I^{(k)}$ (see \ref{symbolicpowers}). So, it is natural to ask:

\vskip1mm

{\it For which $I\subseteq S$ is the ring $S/I^{(k)}$ Cohen-Macaulay for any positive integer $k$?}

\vskip1mm

We will give a complete answer to the above question in the case in which $I$ is a square-free monomial ideal. Notice that when $n=4$ this problem was studied by Francisco in \cite{Fra}. These kind of ideals, whose basic properties are summarized in \ref{appendixd}, supply a bridge between combinatorics and commutative algebra, given by attaching to any simplicial complex $\D$ on $n$ vertices the so-called {\it Stanley-Reisner ideal} $\Id$ and {\it Stanley-Reisner ring} $\kkk[\D]=S/\Id$. One of the most interesting parts of this theory is finding relationships between combinatorial or topological properties of $\D$ and ring-theoretic ones of $\kkk[\D]$. For instance, Reisner gave a topological characterization of those simplicial compexes $\Delta$ for which $\kkk[\D]$ is Cohen-Macaulay (for example see Miller and Sturmfels \cite[Theorem 5.53]{MS}). Such a characterization depends on certain singular homology groups with coefficients in $\kkk$ of topological spaces related to $\D$. In fact, the Cohen-Macaulayness of $\D$ may depend on $\chara(\kkk)$. At the contrary, a wide open problem is, once fixed the characteristic, to characterize in a purely graph-theoretic fashion those graphs $G$ for which $\kkk[\D(G)]$ is Cohen-Macaulay, where $\D(G)$ denotes the independence complex of $G$. For partial results around this problem see, for instance, Herzog and Hibi \cite{HH}, Herzog, Hibi and Zheng \cite{HHZ}, Kummini \cite{Ku} and our paper joint with Constantinescu \cite{CV}. Thus, in general, even if a topological characterization of the ``Cohen-Macaulay simplicial complex" is known, a purely combinatorial one is still a mystery. In this chapter we are going to give a combinatorial characterization of those simplicial complexes $\D$ such that $S/\Idm$ is Cohen-Macaulay for all positive integer $k$. The characterization is quite amazing, in fact the combinatorial counter-party of the algebraic one ``$S/\Idm$ Cohen-Macaulay" is a well studied class of simplicial complexes, whose interest comes besides this result. Precisely we will show in Theorem \ref{main}: 

\vskip1mm

{\it The ring $S/\Idm$ is Cohen-Macaulay for all $k>0$ $\iff$ $\D$ is a matroid.}

\vskip1mm

Matroids have been introduced as an abstraction of the concept of linear independence. We briefly discuss some basic properties of them in Section \ref{appdmatroids}, however there are entire books treating this subject, such as that of Oxley \cite{Ox} or the one of Welsh \cite{Wel}. It must be said that Theorem \ref{main} has been proved independently and with different methods by Minh and Trung in \cite[Theorem 3.5]{MT}. 

Actually, we will show a more general version of Theorem \ref{main} described above, namely Theorem \ref{maingeneral}, regarding a class of monomial ideals more general than the square-free ones. Let us say  that this generalization appears for the first time, since it was not present in our original paper \cite{va3}. As a first tool for our proof, we need a {\it duality for matroids} (see Theorem \ref{matroidduality}), which will allow us to switch the problem from $\Id$ to the cover ideal $J(\D)$. The {\it if-part} of the proof, at this point, is based on the investigation of the symbolic fiber cone of $\JD$, the so-called algebra of basic covers $\AD$ (more generally $\ADo$, where $\oo$ is a weighted function on the facets of $\D$). Such an algebra was introduced by Herzog during the summer school {\it Pragmatic 2008}, in Catania. Among other things, he asked for its dimension. Since then, even if a complete answer is not yet known, some progresses have been done: In our paper \cite{CV}, a combinatorial characterization of the dimension of $\AD$ is given for one-dimensional simplicial complexes $\D$. Here we will show that $\dim(\AD)$ is minimal whenever $\D$ is a matroid. To this purpose an {\it exchange property for matroids}, namely \eqref{exchangeproperty}, is fundamental. Eventually, Proposition \ref{thekey} implies the if-part of Theorem \ref{main}. 

In \cite{va3}, we actually showed that $\dim(\D)$ is minimal {\it exactly} when $\D$ is a matroid, getting also the {\it only-if part} of Theorem \ref{main}. Here, however, we decided to prove this part in an other way. Namely, we study the associated primes of the polarization of $\Idm$ (Corollary \ref{asspol1}), showing that if $\D$ is not a matroid, then there is a lack of connectedness of certain ideals related to $\D$, obstructing their Cohen-Macaulayness (Theorem \ref{cmforcesmatroid}). In particular, this will imply that if $S/\Id^{n-\dim(\D)+2}$ is Cohen-Macaulay, then $\D$ has to  be a matroid, a stronger statement than the only-if part of Theorem \ref{main}.

It turns out that Theorem \ref{main} has some interesting consequences. For instance, in Corollary \ref{multdim} we deduce that $\dim(\AD)=\kkk[\D]$ if and only if $\D$ is a matroid, and that, in this situation, the ``multiplicity of $\AD$" is bounded above from the multiplicity of $\kkk[\D]$ (here the commas are due to the fact that, in general, $\AD$ is not standard graded). An other consequence regards the problem of set-theoretic complete intersections (Corollary \ref{sci}): {\it After localizing at the maximal irrelevant ideal, $\Id$ is a set-theoretic complete intersection whenever $\D$ is a matroid.}

\section{Towards the proof of the main result}

In this section we prove the main theorem of the chapter. For the notation and the definitions of the basic objects we remind to Appendix \ref{appendixd}. In particular, $\kkk$ will denote a field, $S:=\kkk[x_1,\ldots ,x_n]$ will be the polynomial ring in $n$ variables over $\kkk$, and $\mm := (x_1,\ldots ,x_n)\subseteq S$ will denote the maximal irrelevant ideal of $S$.

\begin{thm}\label{main}
Let $\D$ be a simplicial complex on $[n]$. The following are equivalent:
\begin{compactitem} 
\item[(i)] $S/\Idm$ is Cohen-Macaulay for any $k\in \NN_{\geq 1}$.
\item[(ii)] $S/\JDm$ is Cohen-Macaulay for any $k\in \NN_{\geq 1}$.
\item[(iii)] $\D$ is a matroid.
\end{compactitem}
\end{thm}

%\begin{remark}
%In light of Theorem \ref{main} one might guess that if $\D$ is a matroid, then the simplicial complex associated to the polarization of $\Id^{(m)}$ is a matroid as well as $\D$. However this is not true in general, for instance consider $\D$ to be three isolated points. Then the simplicial complex associated to the second symbolic power of $\Id$ is not a matroid. 
%\end{remark}

\begin{remark}
Notice that condition (iii) of Theorem \ref{main} does not depend on the characteristic of $\kkk$. Thus, as a consequence of Theorem \ref{main}, conditions (i) and (ii) do not depend on $\chara(\kkk)$ as well as (iii). This fact was not clear a priori.
\end{remark}

\begin{remark}
If $\D$ is the $m$-skeleton of the $(n-1)$-simplex, $-1\leq m\leq n-1$, then $\D$ is a matroid. So Theorem \ref{main} implies that all the symbolic powers of $\Id$ are Cohen-Macaulay.
\end{remark}

Actually, we will prove a slightly more general version of Theorem \ref{main}, since this does not require much more effort. More precisely, we will show Theorem \ref{main} for a larger class of pure monomial ideals than the pure square-free ones. We are going to introduce them below: For a simplicial complex $\D$ on $[n]$, a function 
\begin{displaymath}
\begin{array}{rcl}
\omega : \F(\D) & \rightarrow & \NN \setminus \{0\} \\
F & \mapsto & \omega_F
\end{array}
\end{displaymath} 
is called {\it weighted function}\index{simplicial complex!weighted}\index{weighted function} (this concept has been introduced by Herzog, Hibi and Trung in \cite{HHT}). Moreover, the pair $(\Delta ,\omega)$ is called a {\it weighted simplicial complex}. The authors of \cite{HHT} studied the properties of the {\it weighted monomial ideal}\index{monomial ideal!weighted} 
\[J(\D ,\omega) := \bigcup_{F\in \F(\D)}\wp_F^{\omega_F}.\]
If $\omega$ is the {\it canonical weighted function}\index{weighted function!canonical}, that is $\omega_F = 1$ for any $F\in \F(\D)$, it turns out than $J(\D ,\omega)=J(\D)$ is nothing but than the cover ideal of $\D$. In particular the class of all the weighted monomial ideals contains the square-free ones. The class we want to define stays between the pure square-free monomial ideals and the weighted pure monomial ideals. We say that a weighted function $\omega$ is a {\it good-weighted function} if it is induced by a {\it weight on the variables}, namely if there exists a function $\ll : [n] \rightarrow \RR_{>0}$ such that $\omega = \omega^{\ll}$, where
\[\omega^{\ll}_F := \sum_{i\in F} \ll(i) \ \ \ \forall \ F\in \F(\D).\]
In this case the pair $(\D ,\omega)$ will be called a {\it good-weighted simplicial complex}\index{simplicial complex!good weighted} and the ideal $J(\D,\omega)$ a {\it good-weighted monomial ideal}\index{monomial ideal!good-weighted}.

\begin{remark}\label{sfagw}
If $I\subseteq S$ is a pure square-free monomial ideal, then it is a good-weighted monomial ideal. In fact, we have that $I=J(\D)$ for some pure simplicial complex $\D$. Let $d-1$ be the dimension of $\D$. Then, because $\D$ is pure, $|F|=d$ for all facets $F\in \FD$. Defining the function $\ll : [n]\rightarrow \RR_{>0}$ as $\ll(i):=1/d$ \ for any $i\in \NN$, we have that the canonical weighted function is induced by $\ll$, and so $J(\D)=J(\D,\omega^{\ll})$ is a good-weighted monomial ideal.
\end{remark}

The assumption that $I$ is pure, in Remark \ref{sfagw}, is necessary, as we are going to show in the next example.

\begin{example}
Let $\D$ be the simplicial complex on $\{1,\ldots ,5\}$ such that
\[\F(\D)=\{\{1,2\},\{1,3\},\{1,4\},\{3,4\},\{2,3,5\}\}.\]
If the canonical weighted function on $\D$ were induced by some weight $\ll$ on the variables, then we should have 
$\ll(1)=\ll(2)=\ll(3)=\ll(4)=1/2$. This, since $\{2,3,5\}\in \F(\D)$, would imply $\ll(5)=0$, a contradiction.
\end{example}

We will prove the following theorem which, as we are going to show just below, implies Theorem \ref{main}.

\begin{thm}\label{maingeneral}
Let $J=J(\D,\omega)\subseteq S$ be a good-weighted monomial ideal. Then $S/J^{(k)}$ is Cohen-Macaulay for any $k\in \NN_{\geq 1}$ if and only if $\D$ is a matroid.
\end{thm}

\begin{proof} (of Theorem \ref{main} from Theorem \ref{maingeneral}). (ii) $\implies$ (iii). Since $S/J(\D)$ is Cohen-Macaulay, $\D$ has to be pure. Therefore $J(\D)$ is good-weighted by Remark \ref{sfagw}, and Theorem \ref{maingeneral} implies that $\D$ is a matroid.

(iii) $\implies$ (ii). $\D$ is pure because it is a matroid (see \ref{appdmatroids}). Thus $J(\D)$ is good-weighted by Remark \ref{sfagw}, and Theorem \ref{maingeneral} yields that $S/\JDm$ is Cohen-Macaulay for all $k\in \NN_{\geq 1}$.

(i) $\iff$ (ii). By Theorem \ref{matroidduality} a simplicial complex $\D$ is a matroid if and only if $\D^c$ is a matroid. Since $\Id = J(\D^c)$ and $I_{\D^c}=J(\D)$, the equivalence between (ii) and (iii) yields the equivalence between (i) and (ii).
\end{proof}

\subsection{The algebra of basic $k$-covers and its dimension}

In Appendix \ref{appendixd} we introduced the concept of vertex cover of a simplicial complex, emphasizing how it lends the name to the cover ideal. The concept of vertex cover can be extended in the right way to a weighted simplicial complex as follows: For all natural number $k$, a nonzero function $\a : [n]\rightarrow \NN$ is a {\it $k$-cover}\index{simplicial complex!$k$-cover}\index{simplicial complex!weighted!$k$-cover} of a weighted simplicial complex $(\D ,\omega)$ on $[n]$ if 
\[\sum_{i\in F}\a(i)\geq k\omega_F \ \ \ \forall \ F\in \FD .\] 
If $\omega$ is the canonical weighted function, i.e. if $(\D ,\omega)$ is an ordinary simplicial complex, then vertex covers are in one-to-one correspondence with $1$-covers with entries in $\{0,1\}$. Moreover, it is easy to show
\[J(\D ,\omega)=(x_1^{\a(1)}\cdots x_n^{\a(n)} \ : \ \mbox{$\a$ is a $1$-cover}).\]
Actually $k$-covers come in handy to describe a set of generators of all the $k$th symbolic powers of $J(\D ,\omega)$. In fact, by \eqref{strucmonsymb}, we get
\[J(\D ,\omega)^{(k)}=\bigcap_{F\in \FD}\wp_F^{k\omega_F}.\]
Therefore, as before, one can show
\[J(\D ,\omega)^{(k)}=(x_1^{\a(1)}\cdots x_n^{\a(n)} \ : \ \mbox{$\a$ is a $k$-cover}).\]
%\begin{lemma}\label{mirrormatroid}
%A simplicial complex $\D$ on $[n]$ is a matroid if and only if $\Dc$ is a matroid.
%\end{lemma}
%\begin{proof}
%Let $F$ and $G$ be two facets of $\Dc$. Thus $F=[n]\setminus F'$ and $G=[n]\setminus G'$ for some facets $F'$ and $G'$ of $\D$.
%Pick a vertex $i$ in $F\setminus G = G'\setminus F'$. Since $\D$ is a matroid, using the exchange property there exists a vertex $j\in F'$ such that $(F'\setminus \{j\})\cup \{i\}$ is a facet of $\D$. But then $j\in G$ and $(F\setminus \{i\})\cup\{j\}=[n]\setminus ((F'\setminus \{j\})\cup \{i\})$ is a facet of $\Dc$. So if $\D$ is a matroid $\Dc$ is a matroid as well. Because $\D=(\Dc)^c$ the other direction follows too.
%\end{proof}^{\a(1)}\cdots x_n^{\a(n)}:\a \mbox{ is an $m$-cover of }\D).\]
A $k$-cover $\a$ of $(\D,\omega)$ is said to be {\it basic}\index{simplicial complex!$k$-cover!basic}\index{simplicial complex!weighted!$k$-cover!basic} if for any $k$-cover $\beta$ of $(\D ,\omega)$ with $\beta(i)\leq \a(i)$ for any $i\in [n]$, we have $\beta = \a$. Of course to the basic $k$-covers of $(\D,\omega)$ corresponds a minimal system of generators of $J(\D ,\omega)^{(k)}$.  

Now let us consider the multiplicative filtration $\J(\D,\omega):=\{\JDom\}_{k\in \NN_{\geq 0}}$ (for any ideal $I$ in a ring $R$, we set $I^{(0)}:=I^0=R$). We can form the Rees algebra of $S$ with respect to the filtration $\J(\D,\omega)$, namely
\[A(\D,\omega) := \bigoplus_{k\in \NN}\JDom.\]
Actually $A(\D,\oo)$ is nothing but than the symbolic Rees algebra of $J(\D , \omega)$. In \cite[Theorem 3.2]{HHT}, Herzog, Hibi and Trung proved that $A(\D,\oo)$ is noetherian. In particular, the associated graded ring of $S$ with respect to $\J(\D,\oo)$
\[G(\D , \omega) := \bigoplus_{k\in \NN}\JDom/J(\D , \omega)^{(k+1)}\]
and the special fiber
\[\ADo := A(\D , \omega)/\mm A(\D , \omega) = G(\D ,\omega)/\mm G(\D , \omega)\]
are noetherian too. The algebra $A(\D ,\omega)$ is known as the {\it vertex cover algebra}\index{vertex cover algebra} of $(\D ,\omega)$, and its properties have been intensively studied in \cite{HHT}. The name comes from the fact that, writing 
\[A(\D , \omega)=\bigoplus_{k\in \NN}\JDom \cdot t^k \subseteq S[t]\]
and denoting by $A(\D,\omega)_k = \JDom \cdot t^k$, it turns out that, for $k\geq 1$, a (infinite) basis for $A(\D,\omega)_k$ as a $\kkk$-vector space is 
\[\{x_1^{\a(1)}\cdots x_n^{\a(n)}\cdot t^k : \a \mbox{ is a $k$-cover of }(\D,\oo)\}.\]
The algebra $\ADo$ is more subtle to study with respect to the vertex cover algebra: For instance it is not evan clear which is its dimension. It is called the {\it algebra of basic covers}\index{algebra of basic covers} of $(\D,\omega)$, and its properties have been studied for the first time by the author with Benedetti and Constantinescu in \cite{BCV} when $\D$ is a bipartite graph. More generally, in \cite{CV}, we studied it for any $1$-dimensional simplicial complex $\D$. Clearly, the grading defined above on $A(\D,\oo)$ induces a grading on $\ADo$, and it turns out that a basis for $\ADo_k$, for $k\geq 1$, as a $\kkk$-vector space is
\[\{x_1^{\a(1)}\cdots x_n^{\a(n)}\cdot t^k : \a \mbox{ is a basic $k$-cover of }(\D,\oo)\}.\]
Notice that if $\a$ is a basic $k$-cover of $(\D,\oo)$, then $\a(i)\leq \ll k$ for any $i\in [n]$, where $\ll := \max\{\oo_F: F\in \FD\}$. This implies that $\ADo_k$ is a finite $\kkk$-vector space for any $k\in \NN$. So we can speak about the Hilbert function of $\ADo$, denoted by $\HFo$, and from what said above we have, for $k\geq 1$, 
\[\HFo(k)=|\{\mbox{basic $k$-covers of }(\D,\oo)\}|.\] 
The key to prove Theorem \ref{main} is to compute the dimension of $\ADo$. So we need a combinatorial description of $\dim(\ADo)$. Being in general non-standard graded, the algebra $\ADo$ could not have a Hilbert polynomial\index{Hilbert polynomial|(}. However, by \cite[Corollary 2.2]{HHT} we know that there exists $h\in \NN$ such that $(\JDo^{(h)})^k=\JDo^{(hk)}$ for all $k\geq 1$. It follows that the $h$th Veronese of the algebra of basic covers, namely 
\[\ADo^{(h)} :=  \bigoplus_{k\in \NN} \ADo_{hk},\]
is a standard graded $\kkk$-algebra. Notice that if a set $\{ f_1,\ldots ,f_q \}$ generates $\ADo$ as a $\kkk$-algebra, then the set $\{f_1^{i_1}\cdots f_q^{i_q} \ : \ 0\leq  i_1,\ldots ,i_q \leq h-1\}$ generates $\ADo$ as a $\ADo^{(h)}$-module. Thus $\dim (\ADo) = \dim (\ADo^{(h)})$. Since $\ADo^{(h)}$ has a Hilbert polynomial, we get a useful criterion to compute the dimension of $\ADo$. First let us remind that, for two functions $f,g:\NN \rightarrow \mathbb{R}$, the writing $f(k)=O(g(k))$ means that there exists a positive real number $\lambda$ such that $f(k)\leq \lambda \cdot g(k)$ for $k\gg 0$. Similarly, $f(k)=\Omega(g(k))$ if there is a positive real number $\lambda$ such that $f(k)\geq \lambda \cdot g(k)$ for $k\gg 0$

\vspace{1mm}

{\it Criterion for detecting the dimension of $\ADo$}. If $\HFo(k)=O(k^{d-1})$, then $\dim(\ADo)\leq d$. If $\HFo(k)=\Omega(k^{d-1})$, then $\dim(\ADo)\geq d$. 

\vspace{1mm}

The following proposition justifies the introduction of $\AD$. It is an alternative version of a result got by Eisenbud and Huneke in \cite{eisenbudhuneke}.
\begin{prop}\label{thekey}
For any simplicial complex $\D$ on $[n]$ we have
\[ \dim (\ADo) = n - \min \{\depth(S/\JDom) : k\in \NN_{\geq 1}\}\]
\end{prop}
\begin{proof}
Consider $G(\D,\oo)$, the associated graded ring of $S$ with respect to $\J(\D,\oo)$. Since $G(\D,\oo)$ is noetherian, it follows by Bruns and Vetter \cite[Proposition 9.23]{BrVe} that 
\[\min \{\depth(S/\JDom) : k\in \NN_{\geq 1}\} = \grade (\mm G(\D,\oo)).\]
We claim that {\it $G(\D,\oo)$ is Cohen-Macaulay.} In fact the Rees ring of $S$ with respect to the filtration $\J(\D,\oo)$, namely $A(\D,\oo)$, is Cohen-Macaulay by \cite[Theorem 4.2]{HHT}.
Let us denote by $\M := \mm \oplus A(\D,\oo)_+$ the unique bi-graded maximal ideal of $A(\D,\oo)$. The following short exact sequence
\[0\longrightarrow A(\D,\oo)_+ \longrightarrow A(\D,\oo) \longrightarrow S \longrightarrow 0\]
yields the long exact sequence on local cohomology
\[\ldots \longrightarrow H_{\M}^{i-1}(S)\longrightarrow H_{\M}^{i}(A(\D,\oo)_+)\longrightarrow H_{\M}^{i}(A(\D,\oo))\longrightarrow \ldots.\]
By the independence of the base in computing local cohomology modules (see Lemma \ref{basetheorems} (ii)) we have $H_{\M}^i(S)=H_{\mm}^i(S)=0$ for any $i<n$ by \eqref{depthloccoh}. Furthermore, using once again \eqref{depthloccoh}, $H_{\M}^i(A(\D,\oo))=0$ for any $i\leq n$ since $A(\D,\oo)$ is a Cohen-Macaulay $(n+1)$-dimensional (see Bruns and Herzog \cite[Theorem 4.5.6]{BH}) ring. Thus $H_{\M}^i(A(\D,\oo)_+)=0$ for any $i\leq n$ by the above long exact sequence. Now let us look at the other short exact sequence
\[0\longrightarrow A(\D,\oo)_+(1) \longrightarrow A(\D,\oo) \longrightarrow G(\D,\oo) \longrightarrow 0,\]
where $A(\D,\oo)_+(1)$ means $A(\D,\oo)_+$ with the degrees shifted by $1$, and the corresponding long exact sequence on local cohomology
\[\ldots \longrightarrow H_{\M}^i(A(\D,\oo))\longrightarrow H_{\M}^i(G(\D,\oo))\longrightarrow H_{\M}^{i+1}(A(\D,\oo)_+(1))\longrightarrow  \ldots.\]
Because $A(\D,\oo)_+$ and $A(\D,\oo)_+(1)$ are isomorphic $A(\D,\oo)$-module, we have that $H_{\M}^i(A(\D,\oo)_+(1))=0$ for any $i\leq n$. Thus $H_{\M}^i(G(\D,\oo))=0$ for any $i<n$. Since $G(\D,\oo)$ is an $n$-dimensional ring (see \cite[Theorem 4.5.6]{BH}) this implies, using once again Lemma \ref{basetheorems} (ii) and \eqref{depthloccoh}, that $G(\D,\oo)$ is Cohen-Macaulay. 

Since $G(\D,\oo)$ is Cohen-Macaulay, $\grade (\mm G(\D,\oo)) = \height(\mm G(\D,\oo))$ (for instance see Matsumura \cite[Theorem 17.4]{matsu}). So, because $\ADo = G(\D,\oo)/\mm G(\D,\oo)$, we get
\[ \dim (\ADo) = \dim(G(\D,\oo))-\height(\mm G(\D,\oo))=n-\grade(\mm G(\D,\oo)), \]
and the statement follows at once.
\end{proof}

\subsection{If-part of Theorem \ref{maingeneral}}

In this subsection we show the {\it if-part} of Theorem \ref{maingeneral}, namely: {\it If $\D$ is a matroid, then $S/\JDom$ is Cohen-Macaulay for each good-weighted function $\oo$ on $\D$}. To this aim we need the following lemma, which can be interpreted as a sort of rigidity property of the basic covers of a good-weighted matroid.
\begin{lemma}\label{rigiditymatroids}
Let $\D$ be a matroid on $[n]$, $\oo$ a good-weighted function on it and $k$ a positive integer. Let us fix a facet $F\in \FD$, and let $\a : F\rightarrow \NN$ be a function such that $\sum_{i\in F}\a(i)=k\oo_F$. Then there exists an unique way to extend $\a$ to a basic $k$-cover of $(\D,\omega)$.
\end{lemma}
\begin{proof}
That such an extension exists is easy to prove, and to get it we do not even need that $\D$ is a matroid nor that $\oo$ is a good-weighted function. Set $\a' : [n] \rightarrow \NN$ the nonzero function such that $\a'\mid_F=\a$ and $\a'(i):=M k$ if $i\in [n]\setminus F$, where $M := \max\{\oo_F: F\in \FD\}$. Certainly $\a'$ will be a $k$-cover of $(\D ,\oo)$. If it is not basic, then there exists a vertex $i\in [n]$ such that $\sum_{j\in G}\a'(j)>k\oo_G$ for any $G$ containing $i$. Thus we can lower the value $\a'(i)$ of one, getting a new $k$-cover. Such a new $k$-cover is still an extension of $\a$, since obviously $i\notin F$. Going on in such a way we will eventually find a basic $k$-cover extending $\a$.

The uniqueness, instead, is a peculiarity of matroids. Let $i_0\in [n]$ be a vertex which does not belong to $F$, and let us denote by $\a'$ a basic $k$-cover of $(\D ,\oo)$ extending $\a$. Since $\a'$ is basic, there must exist a facet $G$ of $\D$ such that $i_0\in G$ and $\sum_{i\in G}\a'(i)=k\oo_G$. By the exchange property for matroids \eqref{exchangeproperty}, there exists a vertex $j_0\in F$ such that 
\[G':=(G\setminus \{i_0\})\cup \{j_0\}\in \FD \mbox{ \ \ \ and \ \ \ } F':=(F\setminus \{j_0\})\cup \{i_0\}\in \FD.\] 
So, denoting by $\ll$ the weight on the variables inducing $\a'$, we have  
\[\sum_{i\in G'}\a'(i) \geq k\oo_{G'}=\sum_{i\in G'}k\ll(i),\]
which yields $\a'(j_0)-k\ll(j_0)\geq \a(i_0)-k\ll(i_0)$, and
\[\sum_{j\in F'}\a'(j) \geq k\oo_{F'}=\sum_{j\in F'}k\ll(j),\]
which yields $\a'(j_0)-k\ll(j_0)\leq \a(i_0)-k\ll(i_0)$. Therefore there is only one possible value to assign to $j_0$, namely: 
\[\a'(j_0)=\a(i_0)-k\ll(i_0)+k\ll(j_0).\]
\end{proof}

Now we are ready to prove the if-part of Theorem \ref{maingeneral}.

\begin{proof}
{\it If-part of Theorem \ref{maingeneral}.} Let us consider a basic $k$-cover $\a$ of our good-weighted matroid $(\D,\oo)$. Since $\a$ is basic there is a facet $F$ of $\D$ such that $\sum_{j\in F}\a(j)=k\oo_F$. Set $d:=|F|$. By Lemma \ref{rigiditymatroids} we deduce that the values assumed by $\a$ on $[n]$ are completely determined by those on $F$. Furthermore, the ways to give values on the vertices of $F$ in such a manner that $\sum_{i\in F}\a(i)=k\oo_F$, are exactly: 
\[\displaystyle \binom{k\oo_F+d-1}{d-1},\]
which is a polynomial in $k$ of degree $d-1$. This implies that, for $k\geq 1$,
\[\displaystyle \HFo(k) = |\{\mbox{basic $k$-covers of }(\D,\oo)\}|\leq |\FD| \cdot \binom{kM+d-1}{d-1},\]
where $M := \max\{\oo_F: F\in \FD\}$. Since $|\FD|$ does not depend on $h$, we get 
\[\HFo(k)=O(k^{d-1}).\]
So $\dim(\ADo)\leq d= \dim(\D)+1$ (the last equality is because $\D$ is pure, see \ref{appdmatroids}). Since $\dim(S/J(\D))=n-d$, by Proposition \ref{thekey} we get
\[d \geq \dim(\ADo) = n - \min \{\depth(S/\JDom):k\in \NN_{\geq 1}\} \geq d,\]
from which $S/\JDom$ is Cohen-Macaulay for any $k\in \NN_{\geq 1}$.
\end{proof}

\begin{remark}
If $\D$ is a matroid on $[n]$, may exist a weighted function $\oo$ of $\D$ such that $S/\JDom$ is Cohen-Macaulay for all $k\in \NN_{\geq 1}$ and $\oo$ is not good. For instance, consider the complete graph $\D=K_4$ on $4$ vertices, which clearly is a matroid. Furthermore consider the following ideal:
\[J(K_4,\oo):=(x_1,x_2)^2\cap (x_1,x_3)\cap (x_1,x_4)\cap (x_2,x_3)\cap (x_2,x_4)\cap (x_3,x_4).\]
The corresponding weighted function is clearly not good, however one can show that $S/J(K_4,\oo)^{(k)}$ is Cohen-Macaulay for all $k\in \NN_{\geq 1}$ using the result of Francisco \cite[Theorem 5.3 (iii)]{Fra}.
\end{remark}
\index{Hilbert polynomial|)}
\subsection{Only if-part of Theorem \ref{maingeneral}}\index{polarization|(}

To this aim we need to describe the associated prime ideals of the polarization (see \ref{polarization}) of a weighted monomial ideal $\JDo$, namely 
\[ \widetilde{\JDo}\subseteq \widetilde{S}.\]
Since polarization commutes with intersections (see \eqref{intersectionpol1}), we get:
\[\widetilde{\JDo}=\bigcap_{F\in \FD}\widetilde{\wp_F^{\oo_F}}.\]
Thus, to understand the associated prime ideals of $\widetilde{\JDo}$ we can focus on the description of $\Ass(\widetilde{\wp_F^{\oo_F}})$.
So let us fix a subset $F\subseteq [n]$ and a positive integer $k$. We have:
\[\widetilde{\wp_F^k}\subseteq \widetilde{S}=\kkk[x_{i,j} \ : \ i=1,\ldots ,n, \ j=1,\ldots ,k].\]
Being a monomial ideal, the associated prime ideals of $\widetilde{\wp_F^k}$ are ideals of variables. For this reason is convenient to introduce the  following notation: For a subset $G:=\{i_1,\ldots ,i_d\}\subseteq [n]$ and a vector $\aaa:=(a_{1},\ldots ,a_{d})\subseteq \NN^d$ with $1\leq a_i\leq k$ for any $i=1,\ldots ,d$, we set
\[\wp_{G,\aaa}:=(x_{i_1,a_1}, x_{i_2,a_2},\ldots ,x_{i_d,a_d}).\]

\begin{lemma}\label{asspol}
Let $F:=\{i_1,\ldots ,i_d\}\subseteq [n]$ and $k$ be a positive integer. Then, a prime ideal $\wp\subseteq \widetilde{S}$ is associated to $\widetilde{\wp_F^k}$ if and only if $\wp=\wp_{F,\aaa}$ with $|\aaa|:=a_1+\ldots +a_d\leq k+d-1$.
\end{lemma}
\begin{proof}
A minimal generator of $\widetilde{\wp_F^k}$ is of the form
\[x_{F,\bbb}:=\prod_{p=1}^{b_1}x_{i_1,p}\cdot \prod_{p=1}^{b_2}x_{i_2,p}\cdots \prod_{p=1}^{b_d}x_{i_d,p},\]
where $\bbb := (b_1,\ldots ,b_d)\in \NN^d$ is such that $|\bbb|=k$. Let us call $B_k\subseteq \NN^d$ the set of such vectors. An associated prime of $\widetilde{\wp_F}$ is forced to be generated by $d$ variables: In fact, $\widetilde{\wp_F^k}$ is Cohen-Macaulay of height $d$ like $\wp_F^k$ from Theorem \ref{cmpol}. Moreover, it is easy to check that a prime ideal of variables
\[\wp:=(x_{j_1,c_1}, \ x_{j_2,c_2}, \ \ldots , \ x_{j_d,c_d})\subseteq \widetilde{S},\]
 where ${\bf c}:=(c_1,\ldots ,c_d)\in[k]^d$, is associated to $\widetilde{\wp_F^k}$ if and only if
\begin{equation}\label{proofasspol}
\forall \ \bbb\in B_k \ \exists \ p\in [d] \ \mbox{ such that } \ x_{j_p,c_p}|x_{F,\bbb}.
\end{equation}
So, if we choose $\bbb:=(0,0,\ldots,0,k,0,\ldots ,0)$, where the nonzero entry is at the place $p$, we get that $i_p$ is in $\{j_1,\ldots ,j_d\}$. Letting vary $p\in [d]$, we eventually get $\{j_1,\ldots ,j_d\}=F$, i.e. $\wp=\wp_{F,{\bf c}}$. Moreover notice that
\begin{equation}\label{proofasspol1}
x_{i_p,c_p}|x_{F,\bbb} \ \iff \ c_p\leq b_p.
\end{equation}
Suppose by contradiction that $|{\bf c}|\geq k+d$ and set $p_0:=\min\{p:\sum_{i=1}^p c_i\geq k+p\}\leq d$. Then choose $\bbb:=(b_1,\ldots,b_d)\in B_k$ in this way:
\[b_p:= \begin{cases} c_p-1  & \mbox{if \ } p<p_0 \\
k-\sum_{i=1}^{p_0-1}(c_i-1) & \mbox{if \ } p=p_0 \\
0 & \mbox{if \ } p>p_0
\end{cases}
\]
Notice that the property of $p_0$ implies that $b_{p_0}<c_{p_0}$. Moreover its minimality implies $b_{p_0}>0$. Since $b_p<c_p$ for all $p=1,\ldots ,d$, it cannot exist any $p\in [d]$  such that $x_{i_p,c_p}|x_{F,\bbb}$ by \eqref{proofasspol1}. Therefore $\wp_{F,{\bf c}}\notin \Ass(\widetilde{\wp_F^k})$ by \eqref{proofasspol}.

On the other hand, if ${\bf c}\in [k]^d$ is such that $|{\bf c}|\leq k+d-1$, then for all $\bbb\in B_k$ there exists $p\in [d]$ such that $c_p\leq b_p$, otherwise $|\bbb|\leq |{\bf c}|-d\leq k-1$. Therefore $x_{i_p,c_p}|x_{F,\bbb}$ by \eqref{proofasspol1} and $\wp_{F,{\bf c}}\in \Ass(\widetilde{\wp_F^k})$ by \eqref{proofasspol}.
\end{proof}

\begin{corollary}\label{asspol1}
Let $\D$ be a simplicial complex on $[n]$ and $\oo$ a weighted function on it. A prime ideal $\wp\subseteq \widetilde{S}$ is associated to $\widetilde{\JDo}$ if and only if $\wp=\wp_{F,\aaa}$ with $F\in\FD$ and $|\aaa|\leq \oo_F+|F|-1$.
\end{corollary}
\begin{proof}
Recall that $\JDo = \bigcap_{F\in \FD}\wp^{\oo_F}$. By \eqref{intersectionpol1}, we have 
\[\widetilde{\JDo}=\bigcap_{F\in \FD}\widetilde{\wp_F^{\oo_F}}.\]
Therefore, by Lemma \ref{asspol}, we get
\[\displaystyle \widetilde{\JDo}=\bigcap_{F\in \FD}\left(\bigcap_{|\aaa|\leq \oo_F+|F|-1}\wp_{F,\aaa}\right).\]
Thus we conclude.
\end{proof}

\begin{thm}\label{cmforcesmatroid}
Let $\D$ be a simplicial complex on $[n]$ and $\oo$ a weighted function on it. Furthermore assume that \ $\oo_F\geq \dim(\D)+2$ \ for all facet $F$ of $\D$. \ If \ $S/\JDo$ \ is Cohen-Macaulay, then $\D$ is a matroid.
\end{thm}
\begin{proof}
Suppose, by contradiction, that $\D$ is not a matroid. Then there exist two facets $F,G \in \FD$ and a vertex $i\in F$ such that
\begin{equation}\label{proofpol}
(F\setminus \{i\})\cup\{j\} \notin \FD \ \ \ \forall \ j \in G.
\end{equation}
Being $S/\JDo$ Cohen-Macaulay, then $\D$ has to be pure. So, setting $\dim(\D):=d-1$, we can assume:
\[F=\{i_1,\ldots ,i_d\}, \ \ G=\{j_1,\ldots ,j_d\} \ \ \mbox{and} \ \ i=i_1.\]
Let us define:
\[\aaa:=(d+1,1,1,\ldots ,1)\in \NN^d \ \ \mbox{and} \ \ \bbb:=(2,2,\ldots ,2)\in \NN^d.\]
Notice that $|\aaa|=|\bbb|=2d$. So, since by assumption $\omega_F$ and $\omega_G$ are at least $d+1$, Corollary \ref{asspol1} implies:
\[\wp_{F,\aaa}, \ \wp_{G,\bbb}\in \Ass(\widetilde{J(\D,\oo)}).\]
Since $S/\JDo$ is Cohen-Macaulay, $\widetilde{S}/\widetilde{\JDo}$ has to be Cohen-Macaulay as well by Theorem \ref{cmpol}. So, the localization 
\[(\widetilde{S}/\widetilde{\JDo})_{\wp_{F,\aaa}+\wp_{G,\bbb}}\]
is Cohen-Macaulay. Particularly it is connected in codimension $1$ (see Proposition \ref{barney}). This translates into the existence of a sequence of prime ideals 
\[\wp_{F,\aaa}=\wp_0,\ \wp_1, \ \ldots , \ \wp_s=\wp_{G,\bbb}\]
such that $\wp_i\in \Ass(\widetilde{\JDo})$, \ $\wp_i\subseteq \wp_{F,\aaa}+\wp_{G,\bbb}$ and $\height(\wp_i+\wp_{i-1})\leq d+1$ for each index $i$ making sense (see Lemma \ref{hilbert}). In particular, $\height(\wp_{F,\aaa}+\wp_1)\leq d+1$. Since $\wp_1\in \Ass(\widetilde{\JDo})$ and it is contained in $ \wp_{F,\aaa}+\wp_{G,\bbb}$, there must exist $p,q\in [d]$ such that ${\bf c}=(a_1,a_2,\ldots , a_{p-1},b_q,a_{p+1},\ldots ,a_d)\in\NN^d$ and
\[\wp_1=(x_{i_1,a_1}, \ \ldots , \ x_{i_{p-1},a_{p-1}}, \ x_{j_q,b_q}, \ x_{i_{p+1},a_{p+1}}, \ \ldots , \ x_{i_d,a_d})=\wp_{(F\setminus \{i_p\})\cup\{j_q\},{\bf c}}.\]
But this is eventually a contradiction: In fact, if $p=1$, then $(F\setminus \{i\})\cup \{j_q\}$ would be a facet of $\D$, a contradiction to \eqref{proofpol}. If $p\neq 1$, then $|{\bf c}|=2d+1$. This contradicts Lemma \ref{asspol}, since $\wp_1=\wp_{(F\setminus \{i_p\})\cup\{j_q\},{\bf c}}$ is associated to $\widetilde{\JDo}$.
\end{proof}

Eventually, Theorem \ref{cmforcesmatroid} implies the only-if part of Theorem \ref{maingeneral}.

\begin{proof} {\it Only if-part of Theorem \ref{maingeneral}}.
Since $\oo_F$ is a positive integer for any facet $F\in \FD$, then $(\dim(\D)+2)\oo_F\geq \dim(\D)+2$ for all facets $F\in \FD$. So, since $S/\JDo^{(\dim(\D)+2)}$ is Cohen-Macaulay, $\D$ has to be a matroid by Theorem \ref{cmforcesmatroid}.
\end{proof}

\begin{remark}
Actually Theorem \ref{cmforcesmatroid} is much stronger than the only-if part of Theorem \ref{maingeneral}: In fact there are not any assumptions on the weighted function $\oo$. Moreover, Theorem \ref{cmforcesmatroid} implies that it is enough that an opportune symbolic power of $\JDo$ is Cohen-Macaulay to force $\D$ to be a matroid!
\end{remark}

\index{polarization|)}
\section{Two consequences}

We end the the chapter stating some applications of Theorem \ref{maingeneral}. We recall that we already introduced in Chapter \ref{chapter2} the concept of multiplicity\index{multiplicity} of a standard graded $\kkk$-algebra $R$, namely $e(R)$, in terms of the Hilbert series of $R$. It is well known that $e(R)$ can be also defined as the leading  coefficient of the Hilbert polynomial\index{Hilbert polynomial} times $(\dim(R)-1)!$, see \cite[Proposition 4.1.9]{BH}. Geometrically, let $\Proj R \subseteq \PP^N$, i.e. $R=K[X_0,\ldots ,X_N]/J$ for a homogeneous ideal $J$. The multiplicity $e(R)$ counts the number of distinct points of $\Proj R \cap H$, where $H$ is a generic linear subspace of $\PP^N$ of dimension $N - \dim (\Proj R)$. Before stating the next result, let us say that for a simplicial complex $\D$, if $\oo$ is the canonical weighted function, then we denote $\ADo$ simply by $\AD$.

\begin{corollary}\label{multdim}
Let $\D$ be a simplicial complex and $\oo$ a good-weighted function on it. Then $\D$ is a $(d-1)$-dimensional matroid if and only if:
\[\dim(\ADo) = \dim (\kkk[\D]) = d.\]
Moreover, if $\D$ is a matroid, then 
\[\displaystyle \HF(k)\leq \frac{e(\kkk[\D])}{(\dim(\AD)-1)!}k^{\dim(\AD)-1}+O(k^{\dim(\AD)-2}).\]
\end{corollary}
\begin{proof}
The first fact follows putting together Theorem \ref{maingeneral} and Proposition \ref{thekey}.
For the second fact, we have to recall that, during the proof of the if-part of Theorem \ref{maingeneral}, we showed that for a $(d-1)$-dimensional matroid $\D$ we have the inequality
\[\displaystyle \HF(k)\leq |\FD| \cdot \binom{k+d-1}{d-1}.\]
It is well known that if $\D$ is a pure simplicial complex then $|\FD|=e(K[\D])$ (for instance see \cite[Corollary 5.1.9]{BH}), so we get the conclusion.
\end{proof}

\begin{example}
If $\D$ is not a matroid the inequality of Corollary \ref{multdim} may not be true. For instance, take $\D:=C_{10}$ the decagon (thus it is a $1$-dimensional simplicial complex). Since $C_{10}$ is a bipartite graph $\bar{A}(C_{10})$ is a standard graded $\kkk$-algebra by \cite[Theorem 5.1]{HHT}. In particular it admits a Hilbert polynomial, and for $k\gg 0$ we have
\[\displaystyle \operatorname{HF}_{\bar{A}(C_{10})}(k)=\frac{e(\bar{A}(C_{10}))}{(\dim(\bar{A}(C_{10}))-1)!}k^{\dim(\bar{A}(C_{10}))-1}+O(k^{\dim(\bar{A}(C_{10}))-2}).\]
In \cite{CV} it is proved that for any bipartite graph $G$ the algebra $\bar{A}(G)$ is a homogeneous ASL\index{Algebra with Straightening Laws} on a poset described in terms of the minimal vertex covers of $G$. So the multiplicity of $\bar{A}(G)$ can be easily read from the above poset. In our case it is easy to check that $e(\bar{A}(C_{10}))=20$, whereas $e(\kkk[C_{10}])=10$.
\end{example}

\index{arithmetical rank|(}\index{set-theoretic complete intersection|(}
For the next result we need the concepts of ``arithmetical rank" of an ideal and ``set-thoeretic complete intersections", introduced in \ref{secara}. By \ref{aragcd}, if an ideal $\aa$ of a ring $R$ is a set-theoretical complete intersection, then $\cd(R,\aa)=\height(\aa)$. In the case in which $R=S$ and $\aa = \Id$ where $\D$ is some simplicial complex on $[n]$, by a result of Lyubeznik got in \cite{Ly} (see Theorem \ref{lyubmon}), it turns out that $\cd(S,\Id) = n - \depth(\kkk[\D])$. So, if $\Id$ is a set-theoretic complete intersection, then $\kkk[\D]$ will be Cohen-Macaulay.

\begin{remark}
In general, even if $\kkk[\D]$ is Cohen-Maculay, then $\Id$ might not be a set-theoretic complete intersection. For instance, if $\D$ is the triangulation of the real projective plane with 6 vertices described in \cite[p. 236]{BH}, then $\kkk[\D]$ is Cohen-Macaulay whenever $\chara(\kkk)\neq 2$. However, for any characteristic of $\kkk$, $\Id$ need at least (actually exactly) $4$ polynomials of $\kkk[x_1,\ldots ,x_6]$ to be defined up to radical (see the paper of Yan \cite[p. 317, Example 2]{Ya}). Since $\height(\Id)=3$, this means that $\Id$ is not a complete intersection for any field.
\end{remark}

\begin{corollary}\label{sci}
Let $\kkk$ be an infinite field. For any matroid $\D$, the ideal $\Id S_{\mm}$ is a set-theoretic complete intersection in $S_{\mm}$.
\end{corollary}
\begin{proof}
By the duality on matroids it is enough to prove that $\JD S_{\mm}$ is a set-theoretic complete intersection.
For $h \gg 0$ it follows by \cite[Corollary 2.2]{HHT} that the $h$th Veronese of $\AD$,  
\[ \AD^{(h)}=\bigoplus_{m\geq 0}\AD_{hm}, \]
is standard graded. Therefore $\AD^{(h)}$ is the ordinary fiber cone of $\JD^{(h)}$. Moreover $\AD$ is finite as a $\AD^{(h)}$-module. So the dimensions of $\AD$ and of $\AD^{(h)}$ are the same. Therefore, using Corollary \ref{multdim}, we get
\[ \height(\JD S_{\mm})=\height(\JD)=\dim \AD^{(h)}=\ell(\JD^{(h)})=\ell((\JD S_{\mm})^{(h)}), \]
where $\ell(\cdot)$ is the analytic spread of an ideal, i.e. the Krull dimension of its ordinary fiber cone.
From a result by Northcott and Rees in \cite[p. 151]{NR}, since $\kkk$ is infinite, it follows that the analytic spread of $(\JD S_{\mm})^{(h)}$ is the cardinality of a set of minimal generators of a minimal reduction of $(\JD S_{\mm})^{(h)}$. 
Clearly the radical of such a reduction is the same as the radical of $(\JD S_{\mm})^{(h)}$, i.e. $\JD S_{\mm}$, so we get the statement.
\end{proof}

\begin{remark}
Notice that a reduction of $I S_{\mm}$, where $I$ is a homogeneous ideal of $S$, might not provide a reduction of $I$. So localizing at the maximal irrelevant ideal is a crucial assumption of Corollary \ref{sci}. It would be interesting to know whether $\Id$ is a set-theoretic complete intersection in $S$ whenever $\D$ is a matroid.
\end{remark}
\index{arithmetical rank|)}\index{set-theoretic complete intersection|)}
\section{Comments}

One of the keys to get our proof of Theorem \ref{main} is to characterize the simplicial complexes $\D$ on $[n]$ for which the Krull dimension of $\AD$ is the least possible. We succeeded in this task showing that this is the case if and only if $\D$ is a matroid (Corollary \ref{multdim}). In view of this, it would be interesting to characterize in general the dimension of $\AD$ in terms of the combinatorics of $\D$. In \cite[Theorem 3.7]{BCV}, we located  the following range for the dimension of the algebra of basic covers:
\[\dim \D +1 \leq \dim \AD \leq n - \displaystyle \left\lfloor \frac{n-1}{\dim \D +1} \right\rfloor.\]
In \cite{CV}, we characterized in a combinatorial fashion the dimension of $\bar{A}(G)$ for a graph $G$, that is a $1$-dimensional simplicial complex. However, already in that case, matters have been not easy at all. We will not say here which is the graph-theoretical invariant allowing to read off the dimension of $\bar{A}(G)$, since it is a bit technical. It suffices to say that, for reasonable graphs, one can immediately compute such an invariant, and therefore the dimension of the algebra of basic covers. The following is an interesting example.

\begin{example}\label{esagono}
Let $G=C_6$ the hexagon, namely:
\[
\setlength{\unitlength}{2mm}
\begin{picture}(40,18)
\put(5.5, 9.5){$C_6 : $}
\put(12,10){\circle*{0.75}}
\put(15,5){\circle*{0.75}}
\put(15,15){\circle*{0.75}}
\put(20,5){\circle*{0.75}}
\put(20,15){\circle*{0.75}}
\put(23,10){\circle*{0.75}}

\put (12,10){\line(3,-5){3.15}}
\put (12,10){\line(3,5){3.15}}
\put (20,5){\line(-1,0){4.625}}
\put (20,5){\line(3,5){3.15}}
\put (20,15){\line(-1,0){4.625}}
\put (20,15){\line(3,-5){3.15}}

{\footnotesize
\put(10.5,9.5){$6$}
\put(19.5,16){$2$}
\put(19.5,2.9){$4$}
\put(14.5,16){$1$}
\put(14.5,2.9){$5$}
\put(24,9.5){$3$}
}

\end{picture}
\]
By \cite[Theorem 3.8]{CV}, it follows that $\dim \bar{A}(C_6)=3$. So, Proposition \ref{thekey} yields:
\[\min\{\depth(S/J(C_6)^{(k)}):k\in\NN_{\geq 1}\}=6-3=3,\]
where $J(C_6)$ is the cover ideal:
\[J(C_6)=(x_1,x_2)\cap (x_2,x_3)\cap (x_3,x_4)\cap (x_4,x_5)\cap (x_5,x_6)\cap (x_6,x_1).\]
Notice that $S/J(C_6)$ is a $4$-dimensional $\kkk$-algebra which is not Cohen-Macaulay. To see this it suffices to consider the localization of $S/J(C_6)$ at the residue class of the prime ideal $(x_1,x_2,x_4,x_5)\subseteq S=\kkk[x_1,\ldots ,x_6]$. The resulting $\kkk$-algebra is a $2$-dimensional local ring not $1$-connected. Therefore, Proposition \ref{barney} implies that it is not Cohen-Macaulay. A fortiori, the original ring $S/J(C_6)$ is not Cohen-Macaulay too. So we get $\depth(S/J(C_6))<\dim S/J(C_6) =4$, which implies $\depth(S/J(C_6))=3$ (from what said just below the picture). Moreover, it follows from a general fact (see Herzog, Takayama and Terai \cite[Theorem 2.6]{HTT}) that $\depth(S/J(C_6))\geq \depth(S/J(C_6)^{(k)})$ for all integers $k\geq 1$. Thus, eventually, we have:
\[\depth(S/J(C_6)^{(k)})=3<\dim(S/J(C_6)) \ \ \ \ \ \ \forall \ k\in \NN_{\geq 1}\]
\end{example}

Example \ref{esagono} suggests us the following question: {\it Which are the simplicial complexes $\D$ on $[n]$ such that $\depth(S/\Idm)$ is constant with $k$ varying among the positive integers?} Equivalently, which are the simplicial complexes $\D$ such that $\depth(S/\JD)=\dim \AD$? Theorem \ref{main} shows that matroids are among these simplicial complexes, however Example \ref{esagono} guarantees that there are others.

\index{symbolic power|)}\index{matroid|)}

\appendix
\chapter{Other cohomology theories}

Besides local cohomology, in mathematics many other cohomologies are available. In this section we want to introduce some cohomology theories we will use in Chapter \ref{chapter1}, underlying the relationships between them. 

\section{Sheaf cohomology}

In algebraic geometry the most used cohomology is {\it sheaf cohomology}\index{sheaf cohomology|(}. It can be defined once given a topological space $X$ and a sheaf $\F$ of abelian groups on $X$. When $X$ is an open subset of an affine or of a projective scheme and $\F$ is a quasi-coherent sheaf, then sheaf cohomology and local cohomology are strictly related, as we are going to show.

Actually to define sheaf cohomology it is not necessary to have a topological space, but just to have a special category $\CCC$: For a topological space $X$ such a category is simply $\Op(X)$ (see Example \ref{opx} below). The definition in this more general context does not involve much more effort and is useful since this thesis makes use of  \'etale cohomology, so we decided to present sheaf cohomology in such a generality. \'Etale cohomology was introduced by Grothendieck in \cite{sga4}.

Let $\CCC$ be a category. For each object $U$ of $\CCC$ suppose to have a distinguished set of family of morphisms $(U_i \rightarrow U)_{i\in I}$, called {\it coverings} of $U$, satisfying the following conditions:
\begin{compactitem}
\item[(i)] For any covering $(U_i \rightarrow U)_{i\in I}$ and any morphism $V\rightarrow U$ the fiber products $U_i \times_U V$ exist, and $(U_i\times_U V)_{i\in I}$ is a covering of $V$.
\item[(ii)] If $(U_i \rightarrow U)_{i\in I}$ is a covering of $U$ and $(U_{ij} \rightarrow U_i)_{j\in J_i}$ is a covering of $U_i$, then $(U_{ij} \rightarrow U)_{(i,j)\in I\times J_i}$ is a covering of $U$.
\item[(iii)] For each object $U$ of $\CCC$ the family $(U\xrightarrow{\mbox{id}_U} U)$ consisting of one morphism is a covering of $U$.
\end{compactitem}
The system of coverings is called {\it Grothendieck topology} , and the category $\CCC$ together with it is called {\it site}\index{site}.

\begin{example}\label{opx}
Given a topological space $X$, we can consider the category $\Op(X)$ in which the objects are the open subsets of $X$ and whose morphisms are the inclusion maps: So $\Hom(V,U)$ is non-empty if and only if $V\subseteq U$. Moreover, in such a case it consists in only one element. There is a natural Grothendieck topology on $\Op(X)$, namely the coverings of an object $U$ are its open coverings. Clearly, if $V$ and $W$ are open subsets of an open subset $U$, we have $V\times_{U}W=V\cap W$.
\end{example}

A {\it presheaf (of Abelian groups)}\index{sheaf}\index{presheaf} on a cite $\CCC$ is a contravariant functor $\F : \CCC \rightarrow \Ab$, where $\Ab$ denotes the category of the Abelian groups. If $\F$ is a presheaf of Abelian groups and $V\xrightarrow{\phi} U$ is a morphism in $\CCC$, for any $s\in \F(U)$ the element $\F(\phi)(s)\in V$ is denoted by $s|_V$, although this can be confusing because there may be more than one morphism from $V$ to $U$. A presheaf $\F$ of Abelian groups is a {\it sheaf} if for any object $U$ of $\CCC$ and any covering $(U_i\rightarrow U)_{i\in I}$, we have:
\begin{compactitem}
\item[(i)] For any $s\in \F(U)$, if $s|_{U_i}=0$ for any $i\in I$, then $s=0$.
\item[(ii)] If $s_i\in \F(U_i)$ for any $i\in I$ are such that $s_i|_{U_i\times_U U_j} = s_j|_{U_i\times_U U_j}$ for all $(i,j)\in I\times I$, then there exists $s\in \F(U)$ such that $s|_{U_i} = s_i$ for every $i\in I$.
\end{compactitem}
A {\it morphism of presheaves} is a natural transformation between functors. A {\it morphism of sheaves} is a morphism of presheaves.

The definition of Grothendieck topology gives rise to the study of various cohomology theories: (Zariski) cohomology, \'etale cohomology, flat cohomology etc. Actually all of these cohomologies are constructed in the same way: What changes is the site.

\begin{example}\label{sites}
In this thesis we are interested just in two sites, both related to a scheme $X$.
\begin{compactitem}
\item[(i)] The {\it Zariski site}\index{site!Zariski} on $X$, denoted by $X_{Zar}$, is the site on $X$ regarded as a topological space as in Example \ref{opx}. Since this is the most natural structure of site on $X$, we will not denote it by $X_{Zar}$, but simply by $X$.
\item[(ii)] The {\it \'etale site}\index{site!etale@\'etale} on $X$, denoted by $X_{\et}$, is constructed as follows: The objects are the \'etale morphisms $U\rightarrow X$. The arrows are the $X$-morphisms $U\rightarrow V$, and the coverings are the families of \'etale $X$-morphisms $(U_i\xrightarrow{\phi_i} U)_{i\in I}$ such that $U=\cup\phi_i(U_i)$.
\end{compactitem}
\end{example}

The functor from the category of sheaves on a site $\CCC$ to $\Ab$, defined by $\Gamma(U,\cdot) \ : \ \F \mapsto \Gamma(U,\F):=\F(U)$, is left exact for every object $U$ of $\CCC$, and the category of sheaves on $\CCC$ has enough injective objects (see Artin \cite[p. 33]{artin}). Therefore it is possible to define the right derived functors of $\Gamma(U,\cdot)$, which are denoted by $H^i(U,\cdot)$.  For a scheme $X$ and a sheaf  $\F$ on $X$, the {$i$th cohomology of $\F$} is $H^i(X,\F)$. If $\F$ is a sheaf on $X_{\et}$, the {\it $i$th \'etale cohomology of $\F$} is \index{etale cohomology@\'etale cohomology} $H^i(X\rightarrow X,\F)$, and it is denoted by $H^i(X_{\et},\F)$. 

As we already anticipated, the cohomology of a quasi-coherent sheaf on an open subset of a noetherian affine scheme $X$ is related to local cohomology: Let $X:=\Spec(R)$, $\aa \subseteq R$ an ideal and $U:=X\setminus \V(\aa)$. Recall that if $\F$ is a quasi-coherent sheaf on $X$ then $\F$ is the sheafication of the $R$-module $M:=\Gamma(X,\F)$. By combining  \cite[Theorem 20.3.11]{BS} and \cite[Theorem 2.2.4]{BS}, shown in the book of Brodmann and Sharp, we have that there is an exact sequence
\begin{equation}\label{shafloc1}
0\rightarrow \Ga(M)\rightarrow M \rightarrow H^0(U,\F)\rightarrow \Ha^1(M)\rightarrow 0,
\end{equation}
and, for all $i\in \NN_{\geq 1}$, isomorphisms
\begin{equation}\label{shafloc2}
H^i(U,\F)\cong \Ha^{i+1}(M).
\end{equation}
There is also a similar correspondence in the graded case: Let $R$ be a graded ring and $\aa$ be a graded ideal of $R$. Set $X:=\Proj(R)$ and $U:=X\setminus \V_+(\aa)$. Let $M$ be a $\ZZ$-graded $R$-module. In this¤ case the homomorphism appearing in \cite[Theorem 2.2.4]{BS} are graded. So by \cite[Theorem 20.3.15]{BS} we have an exact sequence of graded $R$-modules
\begin{equation}\label{shafloc3} 
0\rightarrow \Ga(M)\rightarrow M \rightarrow \bigoplus_{d\in \ZZ}H^0(U,\widetilde{M}(d))\rightarrow \Ha^1(M)\rightarrow 0,
\end{equation}
and, for all $i\in \NN_{\geq 1}$, graded isomorphisms
\begin{equation}\label{shafloc4}
\bigoplus_{d\in \ZZ}H^i(U,\widetilde{M}(d))\cong \Ha^{i+1}(M).
\end{equation}
(The notation \ $\widetilde{\cdot}$ \ above means the graded sheafication, see Hartshorne's book \cite[p. 116]{hart}. We used the same notation also for the ``ordinary" sheafication, however it should always be clear from the context which sheafication is meant). In particular, if $\mm:=R_+$,
\[\bigoplus_{d\in \ZZ}H^i(X,\widetilde{M}(d))\cong H_{\mm}^{i+1}(M) \ \ \forall \ i\in \NN_{\geq 1}.\]

Given a scheme $X$, the {\it cohomological dimension}\index{cohomological dimension} of $X$ is
\[\cd(X):=\min \{i \ : \ H^j(X,\F)=0 \mbox{ for all quasi-coherent sheaves $\F$ and $j>i$}\}.\]
Suppose that $\aa$ is an ideal of a ring $R$ which is not nilpotent. Then \eqref{shafloc1} and \eqref{shafloc2} imply
\begin{equation}\label{cdcd1}
\cd(R,\aa)-1=\cd(\Spec(R)\setminus \V(\aa)).
\end{equation}
Moreover, if $R$ is standard graded over a field $\kkk$ and $\aa$ is a homogeneous ideal which is not nilpotent, using \eqref{shafloc3} and \eqref{shafloc4} we also have
\begin{equation}\label{cdcd2}
\cd(R,\aa)-1=\cd(\Spec(R)\setminus \V(\aa))=\cd(\Proj(R)\setminus \V_+(\aa)).
\end{equation}

For a scheme $X$ it is also available the {\it \'etale cohomological dimension} \index{etale cohomological dimension@\'etale cohomological dimension}. To give its definition we recall that a sheaf $\F$ on $X_{\et}$ is {\it torsion} if for any \'etale morphism $U\rightarrow X$ such that $U$ is quasi-compact, $\F(U\rightarrow X)$ is a torsion Abelian group.
\begin{example}
We can associate to any abelian group $G$ the {\it constant sheaf} on a topological space $X$ \index{constant sheaf}. It is denoted by $\underline{G}$, and for any open subset $U\subseteq X$ it is defined as
\[\underline{G}(U):=G^{\pi_0(U)},\]
where $\pi_0(U)$ is the number of connected components of $U$. If $X$ is a scheme, we can associate to $G$ the constant sheaf on $X_{\et}$ in the same way:
\[\underline{G}(U\rightarrow X):=G^{\pi_0(U)}.\]
Clearly, if $G$ is a finite group, then $\underline{G}$ is a torsion sheaf on $X_{\et}$.
\end{example}

We can define the \'etale cohomological dimension of a scheme $X$ as:
\[ \ecd(X):=\min \{i \ : \ H^j(X_{\et},\F)=0 \mbox{ for all torsion sheaves $\F$ and $j>i$}\}.\] 
If $X$ is a $n$-dimensional scheme of finite type over a separably closed field, then $\ecd(X)$ is bounded above by $2n$ (see Milne's book \cite[Chapter VI, Theorem 1.1]{milne}). If moreover $X$ is affine, then $\ecd(X) \leq n$ (\cite[Chapter VI, Theorem 7.2]{milne}).

We have seen in \eqref{aragcd} that the cohomological dimension provides a lower bound for the arithmetical rank. An analog fact holds true also for the \'etale cohomological dimension. 
Assume that $X$ is a scheme and pick a closed subscheme $Y \subseteq X$. Suppose that $U:=X \setminus Y$ can be cover by $k$ affine subsets of $X$. The \'etale cohomological dimension of these affine subsets is less than or equal to $n$ for what said above. So, using repetitively the Mayer-Vietoris sequence (\cite[Chapter III, Exercise 2.24]{milne}), it is easy to prove that
\begin{equation}\label{aragecd}
\ecd(U) \leq n+k -1
\end{equation} 
Notice that if $X:=\Spec(R)$ and $Y:=\V(\aa)$, where $\aa$ is an ideal of the ring $R$, then $X\setminus Y$ can be covered by $\ara(\aa)$\index{arithmetical rank} affine subsets of $R$. The same thing happens in the graded setting, with $\ara(\aa)$ replaced by $\ara_h(\aa)$. Therefore \eqref{aragcd} actually provides a lower bound for the arithmetical rank!

We want to finish the section noticing the formula for \'etale cohomoligical dimension analog to \eqref{cdcd2}. Let $R$ be standard graded over a separably closed field $\kkk$ and $\aa$ a homogeneous ideal of $R$. Then, using a result of Lyubeznik \cite[Proposition 10.1]{ly3}, one can prove that:\index{arithmetical rank!homogeneous}
\begin{equation}\label{ecdecd}
\ecd(\Spec(R)\setminus \V(\aa))=\ecd(\Proj(R)\setminus \V_+(\aa))+1.
\end{equation}
 
\index{sheaf cohomology|)}

\section{GAGA}\index{GAGA|(}\label{GAGAsec}

To compare the cohomology theories we are going to define we need to introduce some notions from a foundamental paper by Serre \cite{GAGA}. When we consider a quasi-projective scheme over the field of complex numbers $\CC$, two topologies are available: The euclideian one and the Zariski one. Serre's work allows us to use results from complex analysis for the algebraic study of such varieties.
The name GAGA comes from the title of the paper, ``Geometrie Algebrique et Geometrie Analytique".

If $X$ is a quasi-projective scheme over $\CC$, we can associate to it an analytic space $X^h$. Roughly speaking, we have to consider an affine covering $\{U_i\}_{i\in I}$ of $X$. Let us embed every $U_i$ as a closed subspace of a suitable $\AA_{\CC}^N$. The ideal defining each $U_i$ is generated by a set of polynomials of $\CC[x_1,\ldots ,x_N]$. Because a polynomial is a holomorphic function, such a set defines a closed analytic subspace of $\CC^N$, which we denote by $U_i^h$. We obtain $X^h$ gluing the $U_i^h$'s together. Below follow some expected properties of this construction we will use during the thesis:
\begin{compactitem}
\item[(i)] $(\PP^n)^h\cong \PP(\CC^{n+1})$.
\item[(ii)] If $Y$ is another quasi-projective scheme over $\CC$, then $X^h\times Y^h\cong (X\times Y)^h$.
\end{compactitem} 
Furthermore, to any sheaf $\F$ of $\O_X$-modules we can associate functorially a sheaf $\F^h$ of $\O_{X^h}$-modules (\cite[Definition 2]{GAGA}). The sheaf $\F^h$ is, in many cases, the one we expect: For instance, denoting by $\Omega_{X/\CC}$ the sheaf of K\"ahler differentials of $X$ over $\CC$, the sheaf $\Omega_{X/\CC}^h$ is nothing but than the sheaf of holomorphic $1$-forms on $X^h$, namely $\Omega_{X^h}$. More generally, setting $\Omega_{X/\CC}^p := \Lambda^p \Omega_{X/\CC}$, the sheaf $(\Omega_{X/\CC}^p)^h$ is  $\Omega_{X^h}^p$, the sheaf of holomorphic $p$-forms on $X^h$. Another nice property is that $\F^h$ is coherent whenever $\F$ is coherent (\cite[Proposition 10 c)]{GAGA}). One of the main results of \cite{GAGA} is the following (\cite[Th\'eorem\`e 1]{GAGA}):

\begin{thm}\label{GAGA}
Let $X$ be a projective scheme over $\CC$ and $\F$ be a coherent sheaf on $X$. For any $i\in \NN$ there are functorial isomorphisms
\[H^i(X,\F)\cong H^i(X^h,\F^h).\]
\end{thm}

Thanks to GAGA, we can borrow techniques from complex analysis when we study algebraic varieties over a field of characteristic $0$. Often this is useful because euclideian topology is much finer than the Zariski one: For instance, if $X$ is irreducible over $\CC$, the cohomology groups $H^i(X,\underline{\CC})$ vanish for any $i>0$. Instead, $H^i(X^h,\underline{\CC})\cong H_{Sing}^i(X^h,\CC)$ (see Subsection \ref{subssing}), which in general are nonzero. When the characteristic of the base field is positive these methods are not available. However, a valuable analog of euclideian topology is the \'etale site. The following comparison theorem of Grothendieck (see Milne's notes \cite[Theorem 21.1]{milnel}) confirms the effectiveness of the \'etale site.

\begin{thm}\label{etvsangro}
Let $X$ be an affine or a projective scheme smooth over $\CC$ and let $G$ be a finite abelian group. For any $i\in \NN$ there are isomorphisms
\[H^i(X_{\et},\underline{G})\cong H^i(X^h,\underline{G}).\] 
\end{thm} 
\index{GAGA|)}
\section{Singular homology and cohomology}\label{subssing}

If $X$ is a topological space and $G$ an abelian group we will denote by $H^{Sing}_i(X,G)$ the {\it $i$-th singular homology group of $X$ with coefficients in $G$} \index{Singular homology} (for instance see Hatcher's book \cite[p. 153]{hatcher}). When $G=\ZZ$, we just write $H^{Sing}_i(X)$ for $H^{Sing}_i(X,\ZZ)$. We recall that $H^{Sing}_i(X,G)$ and $H^{Sing}_i(X)$ are related by the universal coefficient theorem for homology \index{universal coefficient theorem for homology} (\cite[Theorem 3A.3]{hatcher}):
\begin{thm}\label{unihom}
If $X$ is a topological space and $G$ an abelian group, for any $i\in \NN$ there is a split exact sequence
\[0\longrightarrow H^{Sing}_i(X)\otimes G \longrightarrow H^{Sing}_i(X,G)\longrightarrow \Tor_1(H^{Sing}_{i-1}(X),G)\longrightarrow 0\]
\end{thm} 
The {\it $i$-th singular cohomology group of $X$ with coefficients in $G$} \index{Singular cohomology} will be denoted by $H_{Sing}^i(X,G)$ (see \cite[p. 197]{hatcher}). Singular cohomology and singular homology are related by the universal coefficient theorem for cohomology \index{universal coefficient theorem for cohomology} (\cite[Theorem 3.2]{hatcher}):
\begin{thm}\label{unicohom}
If $X$ is a topological space and $G$ an abelian group, for any $i\in \NN$ there is a split exact sequence
\[0\longrightarrow \Ext^1(H^{Sing}_{i-1}(X),G) \longrightarrow H_{Sing}^i(X,G)\longrightarrow \Hom(H^{Sing}_{i}(X),G)\longrightarrow 0\]
\end{thm} 

\begin{example}\label{bettiproj}
Let us consider the $n$-dimensional projective space over the complex numbers, $\PP(\CC^{n+1})$, supplied with the euclideian topology. Then it is possible to compute its singular homology groups:
\begin{align*}
H^{Sing}_i(\PP(\CC^{n+1}))\cong 
\begin{cases}
\ZZ & \text {if } i=0,2,4,\ldots ,2n\\
0& \text {otherwise}
\end{cases}
\end{align*}
Using Theorems \ref{unihom} and \ref{unicohom}, since $\CC$ is a free abelian group, we get 
\begin{align*}
H^{Sing}_i(\PP(\CC^{n+1}),\CC)\cong H^{Sing}_i(\PP(\CC^{n+1}))\otimes \CC \cong 
\begin{cases}
\CC & \text {if } i=0,2,4,\ldots ,2n\\
0& \text {otherwise}
\end{cases}
\end{align*}
and
\begin{align*}
H_{Sing}^i(\PP(\CC^{n+1}),\CC)\cong \Hom(H_{Sing}^i(\PP(\CC^{n+1})),\CC) \cong 
\begin{cases}
\CC & \text {if } i=0,2,4,\ldots ,2n\\
0& \text {otherwise}
\end{cases}
\end{align*}
\end{example}
Suppose that $X$ is an analytic space and that $G=\CC$. Then, for any $i\in \NN$, we have a functorial isomorphism
\begin{equation}\label{singsheaf}
H_{Sing}^i(X,\CC)\cong H^i(X,\underline{\CC}),
\end{equation}
for instance see the notes of Deligne \cite[Proposition 1.1]{deligne}.

\section{Algebraic De Rham cohomology}\label{subderham}\index{De Rham cohomology}\index{De Rham cohomology!algebraic}

For any regular projective scheme $X$ over a field $\kkk$ of characteristic $0$, and actually for much more generals schemes, it is possible to define its {\it algebraic De Rham cohomology groups}, which we will denote by $H_{DR}^i(X)$ as shown by Grothendieck  in \cite{gro}. Consider the complex of sheaves
\[\Omega_{X/\kkk}^{\bullet} \ : \ \O_{X} \longrightarrow \Omega_{X/\kkk}\longrightarrow \Omega_{X/\kkk}^2\longrightarrow \ldots . \]
The De Rham cohomology is defined to be the hypercohomology of the complex $\Omega_{X/\kkk}^{\bullet}$. In symbols $H_{DR}^i(X):=\mathbb{H}^i(X,\Omega_{X/\kkk}^{\bullet})$. The De Rham cohomology theory can also be developed in the singular case, see Hartshorne \cite{hartder}, but we do not need it in this thesis.

If $\kkk = \CC$, we can consider the complex 
\[\Omega_{X^h}^{\bullet} \ : \ \O_{X^h} \longrightarrow \Omega_{X^h}\longrightarrow \Omega_{X^h}^2\longrightarrow \ldots . \]
A theorem of Grothendieck \cite[Theorem 1']{gro} tells us that, under the above hypothesis,
\begin{equation}\label{groder}
H_{DR}^i(X)\cong \mathbb{H}^i(X^h,\Omega_{X^h}).
\end{equation}
The complex form of Poincar\'e's lemma shows that $\Omega_{X^h}^{\bullet}$ is a resolution of the constant sheaf $\underline{\CC}$. Therefore $\mathbb{H}^i(X^h,\Omega_{X^h})$ is nothing but than $H^i(X^h,\underline{\CC})$. So \eqref{groder} and \eqref{singsheaf} yield functorial isomorphisms
\begin{equation}\label{groder2}
H_{DR}^i(X)\cong H_{Sing}^i(X^h,\CC)
\end{equation}  
for all $i\in \NN$.

\chapter{Connectedness in noetherian topological spaces}\label{connectednesssection}
\index{connected|(}\index{connected!r-@$r$-|(}\index{connected!in codimension $1$|(}\index{connected!in codimension $d$|(}
%Let us recall that a topological space $T$ is called {\it Noetherian} if all increasing chains of closed subsets are stationary. The dimension of a Noetherian topological space is denoted by $\dim Z$, and it is the supremum of the lengths of strictly increasing chains of irreducible closed subsets of $T$. Some examples of Noetherian topological space are:
%\begin{compactitem}
%\item[(i)] $\Spec(R)$, where $R$ is a Noetherian ring.
%\item[(ii)] $\Proj(R)$, where $R$ is an $\NN$-graded Noetherian ring.
%\item[(iii)] $\Spec(R)\setminus \{\mm\}$, where $(R,\mm)$ is a Noetherian local ring (the topology is the one induced by the Zariski topology) This space is known as the {\it punctured spectrum} of $R$.
%\end{compactitem}

In this appendix, we discuss a notion which generalizes that of connectedness on Noetherian topological spaces. As example of Noetherian topological spaces, we suggest to keep in mind schemes.

\begin{remark}
Let $X$ be an affine or a projective scheme smooth over the complex numbers. Then $X^h$ is endowed with a finer topology than the Zariski one. Therefore, if $X^h$ is connected, then $X$ is connected as well. What about the converse implication? Quite surprisingly, $X^h$ is connected if and only if $X$ is connected. This can be shown using Theorem \ref{etvsangro}. In fact, if $G$ is a finite abelian group, then
\[G^{\pi_0(X)}\cong H^0(X_{\et},\underline{G})\cong H^0(X^h,\underline{G})\cong G^{\pi_0(X^h)},\]
where $\pi_0$ counts the connected components of a topological space. Therefore, to figure a connected projective scheme, one can trust the Euclideian perception. However, we recommend carefulness about the fact that this occurrence fails as soon as $X$ is not an affine or a projective scheme. For instance, if $X$ is the affine line without a point, then it is connected, but $X^h$ is obviously not.
\end{remark}

For simplicity, from now on it will always be implied that our noetherian topological spaces have finite dimension.

\begin{definition}\label{r-conn} 
A noetherian topological space $T$ is said to be {\it $r$-connected} \index{r-connected|see{connected, $r$-}} if the following holds:
if $Z$ is a closed subset of $T$ such that $T\setminus Z$ is disconnected, then $\dim Z \geq r$.
(We use the convention that the emptyset is disconnected of dimension $-1$.)
\end{definition}

\begin{remark}\label{bprc}
Let us list some easy facts about Definition \ref{r-conn}:
\begin{compactitem}
\item[(i)] If $T$ is $r$-connected, then $r\leq \dim T$. Moreover $T$ is always $(-1)$-connected.
\item[(ii)] If $T$ is $r$-connected, then it is $s$-connected for any $s\leq r$.
\item[(iii)] $T$ is connected if and only if it is $0$-connected.
\item[(iv)] $T$ is irreducible if and only if it is $(\dim T)$-connected.
\end{compactitem}
\end{remark} 

If, for a positive integer $d$, $T$ is $(\dim(T)-d)$-connected we
say that $T$ is {\it connected in codimension $d$}. 
%Notice that
%this definition is slightly different from that given in the paper
%of Hartshorne \cite{hartshorne}. However in the cases which we are going to
%discuss in this chapter, thanks to ``catenariety", the two notions are
%the same.
If $R$ is a ring, we say that $R$ is $r$-connected if $\Spec (R)$ is. 
%More generally, we say that an $R$-module $M$ is $r$-connected if $\Supp(M)$ is. 

\begin{example}\label{punctureddu}
Besides $\Spec(R)$ we will consider other two noetherian topological spaces related to special noetherian rings $R$.
\begin{compactitem}
\item[(i)] If $(R,\mm)$ is a local ring the {\it punctured spectrum}\index{punctured spectrum} of $R$ is $\Spec(R)\setminus \{\mm\}$, where the topology is induced by the Zariski topology on $\Spec(R)$. Notice that the punctured spectrum of $R$ is $r$-connected if and only if $R$ is $(r+1)$-connected (a local ring is obviously always $0$-connected).
\item[(ii)] If $R$ is graded, we can consider the noetherian topological space $\Proj(R)$. One can easily show that if $R_+$ is a prime ideal, then $\Proj(R)$ is $r$-connected  if and only if the punctured spectrum of $R_{R_+}$ is $r$-connected if and only if $R$ is $(r+1)$-connected.
\end{compactitem}
\end{example}

We list three useful lemmas which allow us to interpret in different ways the concept introduced in Definition \ref{r-conn}.

\begin{lemma}\label{milhouse}
Let $T$ be a nonempty noetherian $r$-connected topological space. Let $T_1, \ldots, T_m$ be the irreducible components of $T$:
\begin{compactitem}
\item[(i)] If $A$ and $B$ are disjoint nonempty subsets of $\{1,\ldots ,m\}$ such that $A\cup B = \{1,\ldots ,m\}$, then
\begin{equation}\label{connirr} 
\dim (( \bigcup_{i \in A}T_i) \cap( \bigcup_{j \in B}T_j))\geq r. 
\end{equation}
Moreover, if $T$ is not $(r+1)$-connected, then it is possible to choose $A$ and $B$ in such a way that equality holds in \eqref{connirr}.
\item[(ii)]  If $T$ is not $(r+1)$-connected, then the dimension of each $T_i$ is at least $r$. Moreover, if $T$ has is reducible, then $\dim T_i > r$ for any $i=1,\ldots ,m$ (we recall our assumption $\dim(T)<\infty$).
\end{compactitem}
\end{lemma}
\begin{proof}
For (i) see Brodmann and Sharp \cite[Lemma 19.1.15]{BS}, for (ii) look at \cite[Lemma 19.2.2]{BS}).
\end{proof}

\begin{lemma}\label{lemmaconn}
A connected ring $R$ with more than one minimal prime ideals is not $r$-connected if and only if there exist two ideals $\aa$ and $\bb$ such that:
\begin{compactitem}
\item[(i)] $\sqrt{\aa}$ and $\sqrt{\bb}$ are incomparable.
\item[(ii)] $\aa \cap \bb$ is nilpotent.
\item[(iii)] $\dim R/(\aa + \bb) < r$.
\end{compactitem}
\end{lemma}
\begin{proof}
Set $T:=\Spec(R)$.

If $\aa$ and $\bb$ are ideals of $R$ satisfying (i), (ii) and (iii), then consider $Z:=\V(\aa+\bb)$. It is clear that $\V(\aa)\cap (T\setminus Z)$ and $\V(\bb)\cap (T\setminus Z)$ provide a disconnection for $T\setminus Z$. Since $\dim Z < r$ we have that $R$ is not $r$-connected.

If $R$ is not $r$-connected, then there exists a closed subset $Z\subseteq T$ such that $T\setminus Z$ is disconnected and $\dim Z =s < r$. We can assume that $R$ is $s$-connected. Because $T\setminus Z$ is disconnected, there are two closed subsets $A$ and $B$ of $T\setminus Z$ such that:
\begin{compactitem}
\item[(i)] $A$ and $B$ are nonempty.
\item[(ii)] $A\cup B = T\setminus Z$.
\item[(iii)] $A\cap B = \emptyset$.
\end{compactitem}
Let us choose two ideals $\aa$ and $\bb$ of $R$  such that $A = \V(\aa) \setminus Z$ and $B = \V(\bb) \setminus Z$. The fact that $A\cap B = \emptyset$ implies $\V(\aa+\bb)\subseteq Z$. Therefore $\dim R /(\aa+\bb) \leq s < r$. 
Moreover, since $A\cup B = T\setminus Z$, we have that $\V(\aa \cap \bb) \supseteq T\setminus Z$. If $\wp$ is a minimal prime ideal of $R$, by Lemma \ref{milhouse} (ii) $\dim R/\wp > s$. Then $\wp \notin Z$, so it belongs to $\V(\aa \cap \bb)$. This implies that $\aa \cap \bb$ is nilpotent. Finally, since $A$ and $B$ are nonempty, the radicals of $\aa$ and $\bb$ cannot be comparable: For instance, if $\sqrt{\aa}$ were contained in $\sqrt{\bb}$, we would get $\V(\aa)\supseteq \V(\bb)$. Since $A\cap B = \emptyset$, we would deduce $\V(\bb)\subseteq Z$, but this would contradict the fact that $B$ is non-empty.
\end{proof}

\begin{lemma}\label{hilbert}
For a noetherian topological space $T$, the following are
equivalent:
\begin{compactitem}
\item[(i)] $T$ is $r$-connected; 
\item[(ii)] For each $T'$ and $T''$ irreducible
components of $T$, there exists a sequence $T'=T_{i_0}, T_{i_1}, \ldots,
T_{i_s}=T''$ such that $T_i$ is an irreducible component of $T$ for
all $i=0, \ldots, s$ and $\dim (T_{i_j} \cap T_{i_{j-1}}) \geq r$ for all
$j= 1, \ldots, s$.
\end{compactitem}
A sequence like in (ii) will be referred as an {\it $r$-connected sequence}\index{connected!r-@$r$-!sequence}.
\end{lemma}
\begin{proof}
Let $T_1,\ldots ,T_m$ be the irreducible components of $T$.

(ii) $\implies$ (i). Suppose that $T$ is not $r$-connected. Pick two disjoint nonempty subsets $A,B\subseteq \{1,\ldots ,m\}$ such that $A\cup B = \{1,\ldots ,m\}$, and choose $p\in A$ and $q\in B$. Set $T':=T_{p}$ and $T'':=T_{q}$. By the hypothesis there is an $r$-connected sequence between $T'$ and $T''$, namely:
\[T'=T_{i_0}, T_{i_1}, \ldots,T_{i_s}=T''.\]
Of course there exists $k\in \{1,\ldots ,s\}$ such that $T_{i_{k-1}}\in A$ and $T_{i_k}\in B$. Therefore,
\[\dim (( \bigcup_{i \in A}T_i) \cap( \bigcup_{j \in B}T_j))\geq \dim(T_{i_{k-1}}\cap T_{i_k})\geq r.\]
This contradicts Lemma \ref{milhouse} (i).

(i) $\implies$ (ii). Set $T'=T_p$ for some $p\in\{1,\ldots ,m\}$. By Lemma \ref{milhouse} (i), 
\[\dim(T'\cap (\bigcup_{i\neq p}T_i))\geq r.\]
Therefore there exists an irreducible closed subset $S_1\subseteq T'\cap (\bigcup_{i\neq p}T_i)$ of dimension bigger than or equal to $r$. Being $S_1$ irreducible, there exists $j_1\neq p$ such that $S_1\subseteq T'\cap T_{j_1}$. So
\[\dim(T'\cap T_{j_1})\geq r.\]
Set $A_1:=\{p,j_1\}$ and $B_1:=\{1,\ldots ,m\}\setminus \{p,j_1\}$. Arguing as above, Lemma \ref{milhouse} implies that there is $j_2\in B_1$ such that:
 \[\dim(T'\cap T_{j_2})\geq r \ \ \ \mbox{or} \ \ \ \dim(T_{j_1}\cap T_{j_2})\geq r.\]
 Going on this way, we can show that the graph whose vertices are the $T_i$'s and such that $\{T_i,T_j\}$ is an edge if and only if $\dim(T_i\cap T_j)\geq r$ is a connected graph. This is in turn equivalent to (ii).
\end{proof}

As we saw in Equation \eqref{cohomcompl1}, the cohomological dimension does not change under completion. This is not the case for the connectedness. Since in Chapter \ref{chapter1} we compared the cohomological dimension $\cd(R,\aa)$ with the connectedness of $R/\aa$, the following lemma is necessary.

\begin{lemma}\label{boe}
Let $R$ be a local ring. The following hold:
\begin{compactitem}
\item[(i)] If $\widehat{R}$ is $r$-connected, then $R$ is $r$-connected as well. 
\item[(ii)] if $\wp \widehat{R} \in \Spec(\widehat{R})$ for all minimal prime ideals
$\wp$ of $R$, then $\widehat{R}$ is $r$-connected if and only if $R$ is $r$-connected.
\end{compactitem}
\end{lemma}
\begin{proof}
The reader can find the proof in \cite[Lemma
19.3.1]{BS}.
\end{proof}

Because Lemma \ref{boe} (ii), it is interesting to know local rings whose minimal prime ideals extend to prime ideals of the completion.  The following Lemma provides a good class of such rings. Even if the proof is standard, we include it here for the convenience of the reader.

\begin{lemma}\label{krusty}
Let $R$ be a graded ring and $\mathfrak{m} :=
R_+$ denote the irrelevant ideal of $R$. If $\wp$ is a graded prime of $R$, then $\wp
\widehat{R^{\mathfrak{m}}}\in \Spec(\widehat{R^{\mathfrak{m}}})$.
In particular, if $R$ is a domain, $\widehat{R^{\mathfrak{m}}}$ is
a domain as well.
\end{lemma}
\begin{proof}
We prove first that if $R$ is a domain then
$\widehat{R^{\mathfrak{m}}}$ is a domain as well. Consider the multiplicative
filtration $\mathcal{F}:=(I_m)_{m \in \mathbb{N}}$ of ideals of $R$, where $I_m$ are defined as $I_m = (\{ f \in R_j: j \geq m
\})$. Let us consider the associated graded ring associated to $\F$, namely $\mbox{gr}_{\mathcal{F}}(R) =
\oplus_{m=0}^{\infty} I_m/I_{m + 1}$. Obviously there is a graded
isomorphism of $R_0$-algebras between $R$ and
$\mbox{gr}_{\mathcal{F}}(R)$. Now, let $\widehat{R^{\mathcal{F}}}$
denote the completion of $R$ with respect to the filtration
$\mathcal{F}$, and let $\mathcal{G}$ be the filtration $(I_m
\widehat{R^{\mathcal{F}}})_{m \in \mathbb{N}}$. It is well known
that \[ \mbox{gr}_{\mathcal{F}}(R) \cong
\mbox{gr}_{\mathcal{G}}(\widehat{R^{\mathcal{F}}}).\] 
By these considerations we can assert
that $\mbox{gr}_{\mathcal{G}}(\widehat{R^{\mathcal{F}}})$ is a
domain; since $\cap_{m \in \mathbb{N}}I_m
\widehat{R^{\mathcal{F}}}=0$, $\widehat{R^{\mathcal{F}}}$ is a
domain as well. Since
the inverse families of ideals $(I_j)_{j \in \mathbb{N}}$ and
$(\mathfrak{m}^j)_{j \in \mathbb{N}}$ are cofinal, $\widehat{R^{\mm}}$ is a domain.
For the more general claim of the lemma we have only to note that,
if $\wp$ is a graded prime of $R$, then $R/\wp$ is an graded 
domain and use the
previous part of the proof.
\end{proof}

\begin{remark}
Thanks to Lemma \ref{krusty}, we can apply Lemma \ref{boe} (ii) to (localizations of) graded rings $R$. In fact, every minimal prime ideals of $R$ is graded (see Bruns and Herzog \cite[Lemma 1.5.6]{BH}). 
\end{remark}

\section{Depth and connectedness}\index{depth|(}

A result of Hartshorne in \cite{hartshorne} (see also Eisenbud's book \cite[Theorem
18.12]{eisenbud}), asserts  that a Cohen-Macaulay ring is
connected in codimension 1. In fact there is a relationship between depth and connectedness. To explain it we need the following lemma.

\begin{lemma}\label{maggie}
Let $(R,\mathfrak{m})$ be local. Then
\[ \dim R/\mathfrak{a} \geq \depth(R) - \grade(\mathfrak{a},R). \]
\end{lemma}
\begin{proof}
Set $k:= \depth(R)$ and $g:=\grade(\mathfrak{a},R)$. Let $f_1,
\ldots, f_g \in \mathfrak{a}$ be an $R$-sequence. If $J:=(f_1,
\ldots, f_g)$ we must have 
\[ \displaystyle \mathfrak{a} \subseteq \bigcup_{\wp \in
Ass(R/J)} \wp ,\]
so there exists $\wp \in \mbox{Ass}(R/J)$ such that
$\mathfrak{a} \subseteq \wp$.
Since $\depth(R/J) = k - g$, moreover, we have $\dim R/ \wp \geq k - g$ (for instance see Matsumura \cite[Theorem
17.2]{matsu}). 
\end{proof}

The relation between connectedness and depth is expressed by the following two results.

\begin{prop}\label{barney}
If $R$ is a local ring such that $\depth(R)=r+1$, then it is $r$-connected.
\end{prop}
\begin{proof}
If $R$ has only one minimal prime, then the proposition is obvious, therefore we can suppose that $R$ has at least two minimal prime ideals.
If $R$ were not $r$-connected, by Lemma \ref{lemmaconn} there would exist two ideals $\aa$ and $\bb$ whose radicals are incomparable, such that $\aa \cap \bb$ is nilpotent and such that $\dim R/ (\aa + \bb) < r$.
The first two conditions together with
(\cite[Theorem 18.12]{eisenbud}) imply $\grade(\mathfrak{a}
+ \mathfrak{b}, R) \leq 1$. Then, from Lemma \ref{maggie}, we would have
$\dim R/(\mathfrak{a} + \mathfrak{b}) \geq r$, which is a
contradiction.
\end{proof}

\begin{corollary}\label{serre2}
Let $R$ be a catenary local ring satisfying $S_2$ Serre's condition. Then $R$ is connected in codimension 1.
\end{corollary}
\begin{proof}
If $R$ has only one minimal prime, then it is clearly connected in codimension 1, therefore we can assume that $R$ has at least two minimal prime ideals. If $R$ were not connected in codimension 1, by Lemma \ref{lemmaconn} there would exist ideals $\aa$ and $\bb$ whose radicals are incomparable, such that $\aa \cap \bb$ is nilpotent and such that $\dim R/ (\aa + \bb) < \dim R-1$. Let us localize at a minimal prime $\wp$ of $\aa + \bb$: Since $R$ is catenary $\height(\wp)\geq 2$. It follows by the assumption that $\depth(R_{\wp}) \geq 2$. But $\mathcal{V}(\aa R_{\wp})$ and $\mathcal{V}(\bb R_{\wp})$ provide a disconnection for the punctured spectrum of $R_{\wp}$.  Therefore $R_{\wp}$ is not $1$-connected, and this contradicts Proposition \ref{barney}. 
\end{proof}
\index{connected|)}\index{connected!r-@$r$-|)}\index{connected!in codimension $1$|)}\index{connected!in codimension $d$|)}\index{depth|)}
%
%\begin{corollary}\label{smithers}
%
%Let $R$ be a finitely generated $k$-algebra positively graded,
%$\mathfrak{m}$ be the irrelevant ideal of $R$ and $\mathfrak{a}$ a
%graded ideal. Then
%\[ \c(R/\mathfrak{a}) \geq \depth(R/\mathfrak{a}) - 1 \]
%where
%$\depth(R/\mathfrak{a})=\depth(\mathfrak{m},R/\mathfrak{a})$.
%
%\end{corollary}
%
%\begin{proof}
%
%We have
%$\depth(\mathfrak{m},R/\mathfrak{a})=\depth(R_{\mathfrak{m}}/\mathfrak{a}
%R_{\mathfrak{m}})$ (\cite[Proposition 1.5.15 (e)]{B-H}). Moreover,
%for all minimal prime ideals $\wp$ of $\mathfrak{a}$, we have $\wp
%\subseteq \mathfrak{m}$, so by Lemma \ref{milhouse} (a) we deduce
%$\c(R/\mathfrak{a})=\c(R_{\mathfrak{m}}/\mathfrak{a}
%R_{\mathfrak{m}})$. Then we conclude using Proposition
%\ref{barney}.
%
%\end{proof}
%
%

\chapter{Gr\"obner deformations}

There are a lot of references concerning Gr\"obner deformations. We decided to follow, for our treatment, especially the lecture notes by Bruns and Conca \cite{BC5}.

\section{Initial objects with respect to monomial orders}\label{monord}\index{initial ideal|(}\index{initial algebra|(}\index{space of initial monomials|(}

Let $S:=\kkk[x_1,\ldots ,x_n]$ be the polynomial ring in $n$ variables over a field $\kkk$. A {\it monomial}\index{monomial} of $S$ is an element of $S$ of the form $x^{\a}:=x_1^{\a_1}\cdots x_n^{\a_n}$ with $\a\in \NN^n$. We will denote by $\MM(S)$ the set of all monomials of $S$. A total order $\prec$ on $\MM(S)$ is called a {\it monomial order}\index{monomial order} if: 
\begin{compactitem}
\item[(i)] For all $m\in \MM(S)\setminus \{1\}$, $1\prec m$.
\item[(ii)] For all $m_1,m_2,n\in \MM(S)$, if $m_1\prec m_2$, then $m_1n\prec m_2n$.
\end{compactitem}
Given a monomial order $\prec$, a nonzero polynomial $f\in S$ has a unique representation:
\[f=\ll_1m_1+\ldots+\ll_km_k,\]
where $\ll_i\in \kkk\setminus \{0\}$ for all $i=1,\ldots ,k$ and $m_1\succ m_2 \succ \ldots \succ m_k$. The {\it initial monomial}\index{initial monomial} with respect to $\prec$ of $f$, denoted by $\ini(f)$, is, by definition, $m_1$. Furthermore, if $V\subseteq S$ is a nonzero $\kkk$-vector space, then we will call the {\it space of initial monomials} of $V$ the following $\kkk$-vector space:
\[\ini(V):=<\ini(f):f\in V\setminus \{0\}>\subseteq S.\]
\begin{remark}\label{inalgid}
Actually, we will usually deal with initial objects of more sophisticated structures than vector spaces, namely algebras and ideals. The justifications of the following statements can be found in \cite[Remark/Definition 1.5]{BC5}. 
\begin{compactitem}
\item[(i)] If $A$ is a $\kkk$-subalgebra of $S$, then $\ini(A)$ is a $\kkk$-subalgebra of $S$ as well, and it is called the {\it initial algebra of $A$ with respect to $\prec$}\index{initial algebra!with respect to a monomial order}. However, even if $A$ is finitely generated, $\ini(A)$ might be not (see \cite[Example 1.7]{BC5}).
\item[(ii)] If $A$ is a $\kkk$-subalgebra of $S$ and $I$ is an ideal of $A$, then $\ini(I)$ is an ideal of $\ini(A)$, and it is called the {\it initial  ideal of $I$ with respect to $\prec$}\index{initial ideal!with respect to a monomial order}. In particular, since $\ini(S)=S$, we have that $\ini(I)$ is an ideal of $S$ whenever $I$ is an ideal of $S$.
\end{compactitem}
\end{remark}
A subset $\G$ of a $\kkk$-subalgebra $A\subseteq S$ is called {\it Sagbi bases}\index{Sagbi bases} with respect to $\prec$ if 
\[\ini(A)=\kkk[\ini(g):g\in \G]\subseteq S.\]
As it is easy to see, a Sagbi basis of $A$ must generate $A$ as a $\kkk$-algebra. Of course, a Sagbi basis of $A$ always exists, but, unfortunately, it might do not exist finite: This is the case if and only if $\ini(A)$ is not finitely generated. Analogously, a subset $\G$ of an ideal $I$ of a $\kkk$-subalgebra $A\subseteq S$ is said a {\it Gr\"obner basis}\index{Gr\"obner basis} with respect to $\prec$ of $I$ if 
\[\ini(I)=(\ini(g):g\in \G)\subseteq \ini(A).\]
Also in this case, one can show that a Gr\"obner basis of $I$ must generate $I$ as an ideal of $A$. Moreover, one can show that if $\ini(A)$ is finitely generated, then Noetherianity implies the existence of a finite Gr\"obner basis of any ideal $I\subseteq A$. In particular, any ideal of the polynomial ring $S$ admits a finite Gr\"obner basis. From now on, we will omit the subindex $\prec$ when it is clear which is the monomial order, writing $\init(\cdot)$ in place of $\ini(\cdot)$.

\section{Initial objects with respect to weights}

In this section we introduce the notion of initial objects with respect to weights. A vector $\oo:=(\oo_1,\ldots ,\oo_n)\in \NN_{\geq 1}^n$ supplies an alternative graded structure on $S$, namely the one induced by putting $\deg_{\oo}(x_i):=\oo_i$. Therefore, the degree with respect to $\oo$ of a monomial $x^{\a}=x_1^{\a_1}\cdots x_n^{\a_n}\in \MM(S)\setminus \{1\}$ will be $\oo_1\a_1+\ldots +\oo_n\a_n\geq 1$. For a nonzero polynomial $f\in S$, we call the {\it initial form}\index{initial form} with respect to $\oo$ of $f$ the part of maximum degree of $f$. I.e., if $f=\ll_1m_1+\ldots+\ll_km_k$ with $\ll_i\in \kkk\setminus \{0\}$ and $m_i\in \MM(S)$ for all $i=1,\ldots ,k$, then 
\[\inito(f)=\ll_{j_1}m_{j_1}+\ldots +\ll_{j_h}m_{j_h},\]
where $\deg_{\oo}(m_{j_1})=\ldots =\deg_{\oo}(m_{j_h})=\deg_{\oo}(f):=\max\{\deg_{\oo}(m_{i}):i=1,\ldots ,k\}$. The $\oo$-homogenization of $f$ is the polynomial $\homo(f)$ of the polynomial ring $S[t]$ with one more variable, defined as
\[\homo(f):=\sum_{i=1}^k \ll_i m_it^{\deg_{\oo}(f)-\dego(m_i)}\in S[t].\]
Notice that $\homo(f)$ is homogeneous with respect to the {\it $\oo$-graduation}\index{omega-graduation@$\oo$-graduation|(}\index{omega-homogenization@$\oo$-homogenization|(}\index{homogenization|see{$\oo$-homogenization}} on $S[t]$, given by setting $\deg_{\oo}(x_i):=\oo_i$ and $\deg(t):=1$. Moreover, note that we have that 
\[\inito(f)(x_1,\ldots ,x_n)=\homo(f)(x_1,\ldots ,x_n ,0).\]
Analogously to Section \ref{monord}, for a nonzero $\kkk$-vector space $V\subseteq S$, we define the $\kkk$-vector space
\[\inito(V):=<\inito(f):f\in V\setminus \{0\}>\subseteq S\]
and the $\kkk[t]$-module
\[\homo(V):=\kkk[t]<\homo(f):f\in V\setminus \{0\}>\subseteq S[t].\]
\begin{remark}\label{inalgid1}
The present remark is parallel to \ref{inalgid}.
\begin{compactitem}
\item[(i)] If $A$ is a $\kkk$-subalgebra of $S$, then $\inito(A)$ is a $\kkk$-subalgebra of $S$, called the {\it initial algebra of $A$ with respect to $\oo$}\index{initial algebra!with respect to a weight}, and $\homo(A)$ is a $\kkk$-subalgebra of $S[t]$. However, even if $A$ is finitely generated, both $\inito(A)$ and $\homo(A)$ might be not.
\item[(ii)] If $A$ is a $\kkk$-subalgebra of $S$ and $I$ is an ideal of $A$, then $\inito(I)$ is an ideal of $\inito(A)$, called the {\it initial ideal of $I$ with respect to $\oo$}\index{initial ideal!with respect to a weight}, and $\homo(I)$ is an ideal of $\homo(A)$. In particular, since $\inito(S)=S$ and $\homo(S)=S[t]$, we have that $\inito(I)$ is an ideal of $S$ and $\homo(I)$ is an ideal of $S[t]$ whenever $I$ is an ideal of $S$.
\end{compactitem}
\end{remark}
The two application $\inito$ and $\homo$ are related, roughly speaking, by the fact that $\homo$ supplies a flat family whose fiber at $0$ is $\inito$. More precisely:
\begin{prop}\label{flatdef}
Let $I\subseteq S$ be an ideal and $\oo$ a vector in $\NN_{\geq 1}^n$. The ring $S[t]/\homo(I)$ is a free $\kkk[t]$-module. In particular, $t-\ll$ is a nonzero divisor of $S[t]/\homo(I)$ for any $\ll\in\kkk$. Furthermore, we have:
\[S[t]/(\homo(I)+(t-\ll))\cong \begin{cases} S/I  & \mbox{if \ } \ll\neq 0 \\
S/\inito(I) & \mbox{if \ } \ll=0
\end{cases}.\]
\end{prop}
Proposition \ref{flatdef}, whose proof can be found in \cite[Proposition 2.4]{BC5}, is the main tool to pass to the initial ideal retaining properties from the original ideal, or viceversa. Moreover, the next result says that initial objects with respect to monomial orders are a particular case of those with respect to weights. Thus, Proposition \ref{flatdef} is available also in the context of Section \ref{monord}. 
\begin{thm}\label{telespallabob}
Let $A$ be a $\kkk$-subalgebra of $S$ and $I_i$ ideals of $A$ for $i=1,\ldots ,k$. If $\prec$ is a monomial order such that $\ini(A)$ is finitely generated, then there exists a vector $\oo\in\NN_{\geq 1}^n$ such that $\ini(A)=\inito(A)$ and $\ini(I_i)=\inito(I_i)$ for all $i=1,\ldots ,k$.
\end{thm}
\begin{proof}
See \cite[Proposition 3.8]{BC5}.
\end{proof}
If $\prec$ and $\oo$ are like in Theorem \ref{telespallabob}, then we say that $\oo$ {\it represents} $\prec$ for $A$ and $I_i$. 

\section{Some properties of the homogenization}

In this section we discuss some properties which will be useful in Chapter \ref{chapter2}. First of all we need to introduce the following operation of dehomogenization:
\begin{equation*}
\begin{array}{rcl}
\pi: S[t] & \longrightarrow & S \\
F(x_1, \ldots , x_n , t) & \mapsto & F(x_1 , \ldots , x_n , 1)
\end{array}
\end{equation*}
\begin{remark}\label{lennie}
Let $\omega \in \mathbb{N}_{\geq 1}^n$. Then the following properties are easy to verify:
\begin{compactitem}
\item[(i)] For all $f \in S$ we have $\pi(\homo(f))=f$.
\item[(ii)] Let $F \in S[t]$ be an homogeneous polynomial (with
respect to the $\omega$-graduation) such that $F \notin (t)$. Then
$\homo(\pi(F))=F$; moreover, for all $k \in \mathbb{N}$, if
$G=t^k F$ we have $\homo(\pi(G)) t^k=G$.
\item[(iii)] If $F \in \homo(I)$, then $\pi(F) \in I$.
\end{compactitem}
\end{remark}
In the next lemma we collect some easy and well known facts:
\begin{lemma}\label{ralph}
Let $\omega \in \mathbb{N}_{\geq 1}^n$ and $I$ and $J$ two ideals of $S$.
Then:
\begin{compactitem}
\item[(i)] $\homo(I \cap J)= \homo(I) \cap \homo(J)$.
\item[(ii)]$I$ is prime if and only if $\homo(I)$ is prime.
\item[(iii)]$\homo(\sqrt{I})=\sqrt{\homo(I)}$.
\item[(iv)]$I\subseteq J$ if and only if $\homo(I)\subseteq \homo(J)$. 
\item[(v)]$\wp_1, \ldots, \wp_s$ are the minimal primes of $I$ if and only if $\homo(\wp_1), \ldots,
\homo(\wp_s)$ are the minimal primes of $\homo(I)$;
\item[(vi)]$\dim S/I +1=\dim S[t]/\homo(I) $.
\end{compactitem}
\end{lemma}

\begin{proof}
For (i), (ii) and (iii) see the book of Kreuzer and Robbiano \cite[Proposition 4.3.10]{K-R} (for (ii) see also the lecture notes of Huneke and Taylor \cite[Lemma 7.3, (1)]{hu-ta}).

(iv). This follows easily from Remark \ref{lennie}.

(v). If $\wp_1, \ldots, \wp_s$ are the minimal primes of $I$, then
$\bigcap_{i=1}^s \wp_i = \sqrt{I}$. So (i), (iii) and (iv) imply
\[\bigcap_{i=1}^s \homo(\wp_i) =
\sqrt{\homo(I)}.\]
Then (ii) implies that all
minimal primes of $\homo(I)$ are contained in the set
\[\{\homo(\wp_1), \ldots,\homo(\wp_s) \}.\] 
Moreover, by
(iv), all the primes in this set are minimal for $\homo(I)$.
Conversely, if $\homo(\wp_1), \ldots, \homo(\wp_s)$ are
the minimal primes of $\homo(I)$,  then $\bigcap_{i=1}^s \homo(\wp_i) = \sqrt{\homo(I)}$. So, from
(i), (iii) and (iv), it follows that $\bigcap_{i=1}^s \wp_i = \sqrt{I}$. Therefore, (ii) yields that all the minimal primes of $I$ are contained in the set
$\{ \wp_1, \ldots, \wp_s \}$. Again using (iv), the
primes in this set are actually all minimal for $I$.

(vi). If $\wp_0 \varsubsetneq \wp_1 \varsubsetneq \ldots
\varsubsetneq \wp_d$ is a strictly increasing chain of prime
ideals such that $I \subseteq \wp_0$, then (ii)  and (iv) get that
$\homo(\wp_0) \varsubsetneq \homo(\wp_1) \varsubsetneq \ldots \varsubsetneq \homo(\wp_d) \varsubsetneq (x_1, \ldots, x_n, t)$ is a
strictly increasing chain of prime ideals containing $\homo(I)$. The last inclusion holds true because $(x_1,\ldots ,x_n,t)$ is the unique maximal ideal of $S[t]$ which is $\oo$-homogeneous. Furthermore it is a strict inclusion because obviously $t\notin \homo(H)$ for any ideal $H\subseteq S$.
So, $\dim S[t]/
\homo(I) \geq \dim S/I +1$. Similarly, $\height(\homo(I))\geq \height(I)$, thus we conclude.
\end{proof}

\index{initial ideal|)}\index{initial algebra|)}\index{space of initial monomials|)}\index{omega-graduation@$\oo$-graduation|)}\index{omega-homogenization@$\oo$-homogenization|)}

\chapter{Some facts of representation theory}\label{appendixc}\index{representation!theory|(}

In this appendix we want to recall some facts of Representation Theory we used in Chapter \ref{chapter3}.

\section{General facts on representations of group}

Let $\kkk$ be a field, $E$ a $\kkk$-vector space (possibly not finite dimensional) and $G$ a group. By $\GL(E)$ we denote the group of all $\kkk$-automorphisms of $E$ where the multiplication is $\phi \cdot \psi = \psi \circ \phi$. A {\it representation}\index{representation|(}\index{G-representation@$G$-representation|see{representation}} of $G$ on $V$ is a homomorphism of groups
\[\rho : G \longrightarrow \GL(E).\]
We will say that $(E,\rho)$ is a $G$-representation of dimension $\dim_{\kkk} E$. When it is clear what is $\rho$ we will omit it, writing just ``$E$ is a $G$-representation". If it is clear also what is $G$, we will call $E$ simply a representation. Moreover, we will often write just $gv$ for $\rho(g)(v)$ ($g\in G$ and $v\in E$). A map between two $G$-representations $(E,\rho)$ and $(E',\rho')$ is a homomorphism of $\kkk$-vector spaces, say $\phi : E \rightarrow E'$, such that 
\[ \phi \circ \rho(g) = \rho'(g) \circ \phi \ \ \forall \ g\in G.\] 
We will often call $\phi$ a {\it $G$-equivariant map}.\index{G-equivariant map@$G$-equivariant map} It is straightforward to check that if $\phi$ is bijective than $\phi^{-1}$ is $G$-equivariant. In such a case, therefore, we will say that $V$ and $W$ are isomorphic $G$-representations. 
A {\it subrepresentation} of a representation $V$ is a $\kkk$-subspace $W\subseteq V$ which is invariant under the action of $G$. A representation $V$ is called {\it irreducible} \index{representation!irreducible} if its subrepresentations are just itself and $<0>$. Equivariant maps and irreducible representations are the ingredients of the, so easy to prove as fundamental, Schur's Lemma.

\begin{lemma}\label{Schur Lemma}(Schur's Lemma)
Let $V$ be an irreducible $G$-representation and $V'$ be a $G$-representation. If
\[\phi : V \longrightarrow V'\]
is a $G$-equivariant map, then it is either zero or injective. If also $V'$ is irreducible, and $\phi$ is not the zero map, then $V$ and $V'$ are isomorphic $G$-representations.
\end{lemma}
\begin{proof}
We get the statement because both $\Ker(\phi)$ and $\Im(\phi)$ are $G$-representations.
\end{proof}
Given two $G$-representations $V$ and $W$, starting from them we can construct many other $G$-representations:
\begin{compactitem}
\item[(i)] The $\kkk$-vector space $V\oplus W$ inherits a natural structure of $G$-representation setting 
\[g(v+w):=gv+gw \ \ \ \forall \ g \in G, \ \forall \ v\in V \mbox{ and } \forall \ w\in W.\]
\item[(ii)] The tensor product $V\otimes W$ ($\otimes$ stands for $\otimes_{\kkk}$) becomes a $G$-representation via 
\[g(v\otimes w):=gv\otimes gw \ \ \ \forall \ g\in G, \ \forall \ v\in V \mbox{ and } \forall \ w\in W.\]
In particular, the $d$th tensor product 
\[\bigotimes^d V:= \underbrace{V\otimes V \otimes \ldots \otimes V}_{d \mbox{ {\small times}}}\]
becomes a $G$-representation in a natural way. 
\item[(iv)] If $\chara(\kkk)=0$, the exterior power $\bigwedge^d V$ and the symmetric power $\Sym^d V$ can both be realized as $\kkk$-subspaces of $\bigotimes^d V$. As it is easy to check, it turns out that actually they are $G$-subrepresentations of $\bigotimes^d V$. Particularly $\bigotimes^d V$ is not irreducible provided that $d\geq 2$.
\item[(iv)] The dual space $V^*=\Hom(V,\kkk)$ has a privileged structure of $G$-representation too:
\[(gv^*)(v):=v^*(g^{-1}v) \ \ \ \forall \ g\in G, \ \forall \ v\in V \mbox{ and }\forall \  v^*\in V^*.\]
This definition comes from the fact that, this way, we have $(gv^*)(gv)=v^*(v)$.
\end{compactitem} 
A representation $V$ is said {\it decomposable}\index{representation!decomposable} if there exist two nonzero subrepresentations $W$ and $U$ of $V$ such that $W\oplus U = V$. It is {\it indecomposable}\index{representation!indecomposable} if it is not decomposable. Obviously, an irreducible representation is indecomposable. Mashcke showed that the reverse implication holds true, provided that the group $G$ is finite and that $\chara(\kkk)$ does not divide the order of $G$ (for instance see the book of Fulton and Harris \cite[Proposition 1.5]{FH}). Actually, it is true the following more general fact:
\begin{thm}\label{mashke}
The following are equivalent:
\begin{compactitem}
\item[(i)] Every indecomposable $G$-representations is irreducible.
\item[(ii)] $G$ is finite and $\chara(\kkk)$ does not divide $|G|$.
\end{compactitem}
\end{thm}
\begin{proof}
See the notes of Del Padrone \cite[Theorem 3.1]{delpadro}.
\end{proof}
Thus, if $G$ is a finite group whose order is not a multiple of $\chara(\kkk)$, then any finite dimensional $G$-representation $V$ admits a decomposition
\[V=V_1^{a_1}\oplus \ldots \oplus V_k^{a_k},\]
where the $V_i$'s are distinct irreducible subrepresentations of $V$. Moreover Schur's lemma \cite[Lemma 1.7]{FH} ensures us that such a decomposition is unique.  

\section{Representation theory of the general linear group}\label{reptheorygl}
\index{representation!theory of the general linear group|(}

From now on, in this appendix we assume $\chara(\kkk)=0$. Let $V$ be an $n$-dimensional $\kkk$-vector space. We want to investigate on the representations of the {\it general linear group}\index{general linear group} $G=\GL(V)$. 
\begin{remark}
After choosing a $\kkk$-basis of $V$, we can identify $\GL(V)$ with the group of invertible $n\times n$ invertible matrices with coefficient in $\kkk$. We will often speak of these two groups without distinctions. However, we will punctually remark those situations which depend on the choice of a basis.
\end{remark}
Since $\GL(V)$ is infinite, Theorem \ref{mashke} implies that there are indecomposable $\GL(V)$-representations which are not irreducible. Below is an example of such a representation.
\begin{example}\label{indecred}
Let $\kkk = \RR$ be the field of real numbers, $G:=\GL(\RR)$ and $E$ be a $2$-dimensional $\RR$-vector space. So, to supply $E$ with a structure of $G$-representation we need to give a homomorphism $\RR^*\rightarrow \GL(E)$, where $\RR^*$ denotes the multiplicative group of nonzero real numbers. Set
\begin{displaymath}
\begin{array}{rll}
\rho : \RR^* & \rightarrow & \GL(E) \\
a & \mapsto & \left( \begin{array}{cc} 1 & \operatorname{log}|a| \\ 0 & 1 \end{array} \right)
\end{array}.
\end{displaymath}
Since $\rho(a)(x,y)=(x+\operatorname{log}|a|y,y)$, it is clear that the subspace 
\[F=\{(x , 0) \ : \ x\in \RR\}\subseteq E\]
is a subrepresentation of $E$. So $E$ is not irreducible. However, one can easily check that $F$ is the only $G$-invariant $1$-dimensional subspace of $E$, so $E$ is indecomposable.
\end{example}
To avoid an inconvenient like that in Example \ref{indecred}, we introduce a particular kind of $\GL(V)$-representations. An $N$-dimensional $\GL(V)$-representation $E$ is called {\it rational} \index{representation!rational} if in the homomorphism
\[\kkk^{n^2}\supseteq \GL(V)\rightarrow \GL(E)\subseteq \kkk^{N^2}\]
each of the $N^2$ coordinate function is a rational function in the $n^2$ variables. Analogously we define a {\it polynomial representation}. \index{representation!polynomial} Notice that the representation of Example \ref{indecred} was not rational. In fact it turns out that the analog of Mashcke's theorem holds true for rational representations of $\GL(V)$ (for instance see Weyman \cite[Theorem 2.2.10]{We}).
\begin{thm}\label{reductivityofgln}
Any indecomposable rational representation of $\GL(V)$ is irreducible (recall that $\chara(\kkk)=0$).
\end{thm}
\begin{remark}
Theorem \ref{reductivityofgln} does not hold in positive characteristic. For instance, let $\kkk$ be a field of characteristic $2$, $V$ a $2$-dimensional vector space and $E=\Sym^2(V)$ with the natural action. Let $\{x,y\}$ a basis of $V$. With respect to the basis $\{x^2,xy,y^2\}$ of $E$, the action is
\begin{displaymath}
\begin{array}{rll}
\rho : \GL(V) & \rightarrow & \GL(E) \\
\left( \begin{array}{cc} a & b \\ c & d \end{array} \right) & \mapsto & \left( \begin{array}{ccc} a^2 & 0 & b^2 \\ ac & ad+bc & bd \\ c^2 & 0 & d^2 \end{array} \right)
\end{array}.
\end{displaymath}
So $\rho$ is a polynomial representation. Particularly it is rational. Moreover, since the subspace $F=<x^2,y^2>\subseteq E$ is invariant, $E$ is not irreducible. However, it is easy to check that there are no invariant subspaces of $E$ but $\{0\}$, $F$ and $E$. Therefore $E$ is indecomposable.
\end{remark}
A crucial concept in representation theory is that of weight vectors.\index{weight}\index{weight!vector} Let us call $H\subseteq \GL(V)$ the subgroup of diagonal matrix, and, for elements $\xxx = x_1,\ldots ,x_n \in \kkk\setminus \{0\}$, let us denote by $\diag(\mathbf{x})\in H$ the diagonal matrix
\begin{displaymath}
\diag(\xxx) := \left(\begin{array}{cccc} x_1 & 0 & \cdots & 0 \\
0 & x_2 & \cdots & 0 \\
\vdots & \vdots & \ddots & \vdots \\
0 & 0 & \cdots & x_n
\end{array} \right).
\end{displaymath}
If $E$ is a rational $\GL(V)$-representation, a nonzero element $v\in E$ is called a {\it weight vector} of {\it weight} $\a=(\a_1,\ldots, \a_n)\in \ZZ^n$ if 
\[\diag(\xxx) v = \xxx^{\a}v, \ \ \ \xxx^{\a}:=x_1^{\a_1}\cdots x_n^{\a_n}\] 
for all diagonal matrices $\diag(\xxx)\in H$. 
\begin{remark}\label{dependencebasis}
Let $E$ be a rational $\GL(V)$-representation, $\B = e_1,\ldots ,e_n$ a $\kkk$-basis of $V$, $\sigma$ a permutation of the symmetric group $\SS_n$ and $\B_{\sigma}$ the $\kkk$-basis $e_{\sigma(1)},\ldots ,e_{\sigma(n)}$. We can think to $E$ as a representation with respect to $\B$ or to $\B_{\sigma}$. However a weight vector $v\in E$ of weight $(\a_1,\ldots ,\a_n)$ with respect to $\B$ will have weight $(\a_{\sigma(1)},\ldots ,\a_{\sigma(n)})$ with respect to $\B_{\sigma}$. This is not a big deal, but it is better to be aware that such situations can happen.
Sometimes it will also happen that we consider weights $\beta$ with $\beta\in \ZZ^p$ and $p<n$. In such cases we always mean $\beta$ with zeroe-entries added up to $n$. For instance, if $\beta = (4,1,-1,-2,-3)$ and $n=7$, for weight $\beta$ we mean $(4,1,0,0,-1,-2,-3)$.
\end{remark} 
As one can show, it turns out that $E$ is the direct sum of its {\it weight spaces} \index{weight!space}:
\[E = \bigoplus_{\a\in \ZZ^n}E_{\a}, \ \ \ E_{\a}:=\{v\in E \ : \ \diag(\xxx)v=\xxx^{\a}v \ \ \forall \ \diag(\xxx)\in H\}.\]
Let $B_-(V)\subseteq \GL(V)$ be the subgroup of the lower triangular matrices. A nonzero element $v\in E$ is called a {\it highest weight vector} \index{weight!vector!highest} if $B_-(V) v \subseteq <v>$. It is straightforward to check that a highest weight vector actually is a weight vector. (One can also define the highest weight vector with respect to the subgroup $B_+(V)\subseteq \GL(V)$ of all the upper triangular matrices: the theory does not change). Highest weight vectors supply the key to classify the irreducible rational representations of $\GL(V)$. In fact:

{\it (i) A rational representation is irreducible $\iff$ it has only one highest weight vector (up to multiplication by scalars).}\\
By what said in Remark \ref{dependencebasis}, we can, and from now on we do, suppose that the only highest weight of an irreducible rational representation is $\a = (\a_1,\ldots ,\a_n)$ with $\a_1\geq \ldots \geq \a_n$. After this observation, we can state a second crucial property of highest weight:

{\it (ii) Two rational irreducible representations are isomorphic $\iff$ their highest weight vectors have the same weight.}\\
The third and last fact we want to let the reader know is:

{\it (iii) Given $\a = (\a_1,\ldots , \a_n)\in \ZZ^n$ with $\a_1\geq \ldots \geq \a_n$, there exists an irreducible rational representation whose highest weight vector has weight $\a$.}

\subsection{Schur modules}\label{subsecshurmodules}

In this subsection we are going to give an explicit way to construct all the irreducible rational representations of $\GL(V)$. The plan is first to describe such representations when $\a\in \NN^n$. In this case the corresponding representation is actually polynomial. To get the general case it will be enough to tensorize the polynomial representations by a suitable negative power of the determinant representation. We will soon be more precise. Let us start introducing the notion of partition: A vector $\lambda = (\lambda_1,\ldots , \lambda_k)\in \NN^k$ is a {\it partition} \index{partition} of a natural number $m$, written $\lambda \vdash m$, if $\lambda_1 \geq \lambda_2 \geq \ldots \geq \lambda_k \geq 1$ and $\lambda_1 + \lambda_2 + \ldots + \lambda_k = m$. We will say that the partition $\lambda$ has $k$ {\it parts} 
\index{partition!parts of} and {\it height} \index{partition!height of} $\height(\lambda)=\lambda_1$. Sometimes will be convenient to group together the equal terms of a partition: For instance we will write $(k^d)$ for the partition $(k,k,\ldots ,k)\in \NN^d$, or $(3^3,2^2,1^4)$ for $(3,3,3,2,2,1,1,1,1)$.
%In this section we summarize some known facts about the representation theory of the general linear group. Such theory is going to be crucial in many parts of the paper. The reader can find in the book of Fulton \cite{Fu} or in that of Fulton and Harris \cite{FH} the proofs of any we are going to mention without giving references. Also the second section of the book of Weyman \cite{We} contains the essential results.
From a representation $E$, we can build new representations for any partition $\ll$, namely $L_{\lambda}E$. If $E=V$ is the representation with the obvious action of $\GL(V)$ the obtained representations $\Vl$, called {\it Schur module}\index{Schur module}, will be polynomial and irreducible.
%\footnote{Our definition and notation will be consistent with that in \cite{We}. In \cite{Fu} and \cite{FH}, the authors work with the Weyl module associated to a partition $\lambda$, denoted respectively by $V^{\lambda}$ and $\Sl V$. Anyway this is not a big problem: Set $\lambda'$ the dual partition of $\lambda$, i.e. $\lambda'_i:=| \{j \ : \ \lambda_j\geq i\}|$. Then
%$V^{\lambda}=\Sl V \cong L_{\lambda'}V$. This means that the results of \cite{Fu} and \cite{FH} can be easily translated in our setting.}.
The $\kkk$-vector space $\Vl$ is obtained as a suitable quotient of
\[ \bigwedge^{\lambda_1}V\otimes \bigwedge^{\lambda_2}V \otimes \ldots \otimes \bigwedge^{\lambda_k}V,\]
and the action of $\GL(V)$ is the one induced by the natural action on $V$. To see the precise definition look at \cite[Chapter 2]{We}. We can immediately notice that $\Vl = 0$ if $\height(\lambda)>n=\dim_{\kkk} V$. However in characteristic $0$, that is our case, the above Schur module can be defined as a subrepresentation of $\bigotimes^mV$: Since we think that for this thesis is more useful the last interpretation, we decided to give the details for it.

We can feature a partition $\lambda$ as a {\it (Young) diagram} \index{Young diagram}\index{diagram|see{Young diagram}}, that we will still denote by $\lambda$, namely:
\[\lambda := \{(i,j)\in \NN \setminus \{0\}\times \NN \setminus \{0\} \ : \ i\leq k \mbox{ and }j\leq \lambda_i\}.\]
It is convenient to think at a diagram as a sequence of rows of boxes, for instance the diagram associated to the partition $\lambda=(6,5,5,3,1)$ features as
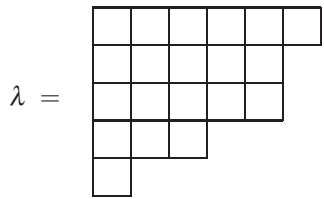
\[
{\setlength{\unitlength}{1mm}
\begin{picture}(30,25)(-5,0)

\put(-11,12){$\lambda \ =$}

\put(0,25){\line(1,0){30}}
\put(0,20){\line(1,0){30}}
\put(0,15){\line(1,0){25}}
\put(0,10){\line(1,0){25}}
\put(0,5){\line(1,0){15}}
\put(0,0){\line(1,0){5}}

\put(0,25){\line(0,-1){25}}
\put(5,25){\line(0,-1){25}}
\put(10,25){\line(0,-1){20}}
\put(15,25){\line(0,-1){20}}
\put(20,25){\line(0,-1){15}}
\put(25,25){\line(0,-1){15}}
\put(30,25){\line(0,-1){5}}

\end{picture}}
\]
Obviously we can also recover the partition $\lambda$ from its diagram, that is why we will speak indifferently about diagrams or partitions. Sometimes will be useful to use the partial order on diagram given by inclusion. That is, given two partitions $\gamma = (\gamma_1,\ldots ,\gamma_h)$ and $\ll = (\ll_1,\ldots ,\ll_k)$, by $\gamma \subseteq \ll$ we mean $h\leq k$ and $\gamma_i\leq \ll_i$ for all $i=1,\ldots ,h$. We will denote by $|\lambda|$ the number of boxes of a diagram $\lambda$, that is $|\lambda|:=\lambda_1+\ldots +\lambda_k$. So $\ll \vdash |\ll |$. Given a diagram $\lambda$, a {\it (Young) tableu} \index{Young tableu}\index{tableu|see{Young tableu}} $\Lambda$ of shape $\lambda$ on $[r]:=\{1,\ldots ,r\}$ is a filling of the boxes of $\lambda$ by letters in the alphabet $[r]$. For instance the following is a tableu of shape $(6,5,5,3,1)$ on $\{1,\ldots ,r\}$, provided $r\geq7$ .
\[
{\setlength{\unitlength}{1mm}
\begin{picture}(30,25)(-5,0)

\put(-11,12){$\Lambda \ =$}

\put(0,25){\line(1,0){30}}
\put(0,20){\line(1,0){30}}
\put(0,15){\line(1,0){25}}
\put(0,10){\line(1,0){25}}
\put(0,5){\line(1,0){15}}
\put(0,0){\line(1,0){5}}

\put(0,25){\line(0,-1){25}}
\put(5,25){\line(0,-1){25}}
\put(10,25){\line(0,-1){20}}
\put(15,25){\line(0,-1){20}}
\put(20,25){\line(0,-1){15}}
\put(25,25){\line(0,-1){15}}
\put(30,25){\line(0,-1){5}}

\put(1.8,21.3){\small{$3$}}
\put(6.8,21.3){\small{$5$}}
\put(11.8,21.3){\small{$4$}}
\put(16.8,21.3){\small{$3$}}
\put(21.8,21.3){\small{$2$}}
\put(26.8,21.3){\small{$7$}}
\put(1.8,16.3){\small{$2$}}
\put(6.8,16.3){\small{$1$}}
\put(11.8,16.3){\small{$7$}}
\put(16.8,16.3){\small{$6$}}
\put(21.8,16.3){\small{$4$}}
\put(1.8,11.3){\small{$2$}}
\put(6.8,11.3){\small{$2$}}
\put(11.8,11.3){\small{$3$}}
\put(16.8,11.3){\small{$1$}}
\put(21.8,11.3){\small{$2$}}
\put(1.8,6.3){\small{$5$}}
\put(6.8,6.3){\small{$6$}}
\put(11.8,6.3){\small{$7$}}
\put(1.8,1.3){\small{$1$}}

\end{picture}}
\]
Formally, a tableu $\Lambda$ of shape $\lambda$ on $[r]$ is a map $\Lambda:\lambda \rightarrow [r]$. The {\it content} \index{Young tableu!content of} of $\Lambda$ is the vector $c(\Lambda)=(c(\Lambda)_1,\ldots ,c(\Lambda)_r)\in \NN^r$ such that $c(\Lambda)_p:=|\{(i,j) \ : \ \Lambda(i,j)=p\}|$. In order to define the Schur modules we need to introduce the Young symmetrizers: Let $\lambda$ be a partition of $m$, and $\Lambda$ be a tableu of shape $\ll$ such that $c(\Lambda)=(1,1,\ldots ,1)\in \NN^m$. Let $\SS_m$ be the symmetric group on $m$ elements, and let us define the following subsets of it:
\[\C_{\Lambda}:=\{\sigma \in \SS_m \ : \ \sigma \mbox{ preserves each column of }\Lambda\},\]
\[\R_{\Lambda}:=\{\tau \in \SS_m \ : \ \tau \mbox{ preserves each row of }\Lambda\}.\]
The symmetric group $\SS_m$ acts on $\otimes^mV$ extending by $\kkk$-linearity the rule
\[\sigma(v_1\otimes \ldots \otimes v_m):=v_{\sigma(1)}\otimes \ldots \otimes v_{\sigma(m)}, \ \ \ \sigma \in \SS_m, \ v_i\in V. \]
With these notation, the {\it Young symmetrizer} \index{Young symmetrizer} is the following map:
\begin{displaymath}
\begin{array}{rll}
\displaystyle
e(\Lambda) : \bigotimes^m V & \rightarrow & \displaystyle \bigotimes^m V\\ 
v = v_1\otimes \ldots \otimes v_m & \mapsto & \sum_{\sigma \in \C_{\Lambda}}\sum_{\tau \in \R_{\Lambda}}\operatorname{sgn}(\tau)\sigma \tau (v)
\end{array}.
\end{displaymath}
One can check that the image of $e(\Lambda)$ is a $\GL(V)$-subrepresenation of $\bigotimes^m V$. Moreover, up to $\GL(V)$-isomorphism, it just depends from the partition $\lambda$, and not from the particular chosen tableu. 
Thus we define the Schur module $\Vl$ as
\[\Vl:=e(\Lambda)(\bigotimes^m V).\]
To see that this definition coincides with the one given at the beginning of this subsection see \cite[Lemma 2.2.13 (b)]{We}. (Actually in \cite{We} the definition of the Young symmetrizers is different from the one given here, and the statement of \cite[Lemma 2.2.13 (b)]{We} is wrong: However, the argument of its proof is the correct one). 
\begin{example}
The Schur modules somehow fill the gap between exterior and symmetric powers. This example should clarify what we mean.
\begin{compactitem}
\item[(i)] Let $\ll := (m)$ and $\Lambda(1,j):=j$ for any $j=1,\ldots ,m$. In this case $\C_{\Lambda} = \{\id_{[m]}\}$ and $\R_{\Lambda} = \SS_m$. Therefore, if $v = v_1 \otimes \ldots \otimes v_m \in \bigotimes^m V$, then
\[e(\Lambda)(v) = \sum_{\tau \in \SS_m}\operatorname{sgn}(\tau) \tau (v).\]
The image of such a map is exactly $\bigwedge^m V \subseteq \bigotimes^m V$. So $L_{(m)}V \cong \bigwedge^m V$.
\item[(ii)] Let $\ll := (1^m)\in \NN^m$ and $\Lambda(j,1):=j$ for any $j=1,\ldots ,m$. This is the opposite case of the above one, in fact we have $\C_{\Lambda} = \SS_m$ and $\R_{\Lambda} = \{\id_{[m]}\}$. Therefore, if $v = v_1 \otimes \ldots \otimes v_m \in \bigotimes^m V$, then
\[e(\Lambda)(v) = \sum_{\sigma \in \SS_m}\sigma (v).\]
The image of such a map is exactly $\Sym^m V \subseteq \bigotimes^m V$. So $L_{(1^m)}V \cong \Sym^m V$.
\end{compactitem}
\end{example}
Before stating the following crucial theorem, let us recall that the {\it transpose partition} \index{partition!transpose} of a partition $\ll = (\ll_1,\ldots ,\ll_k)$ is the partition 
\[\tl\ll = (\tl\ll_1,\ldots , \tl\ll_{\ll_1} ) \ \mbox{where} \ \tl\ll_i := |\{j:\ll_j\geq i\}|.\]
Notice that $|\ll|=|\tl\ll|$, $\tl\ll$ has $\height(\ll)$ parts, $\ll$ has $\height(\tl\ll)$ parts and $\tl(\tl\ll)=\ll$.
\begin{thm}\label{schur}
The Schur module $\Vl$ is an irreducible polynomial representation with highest weight $\tl\lambda$. So all the irreducible polynomial representations are of this kind.
\end{thm}
From Theorem \ref{schur} it is quite simple to get all the irreducible rational $\GL(V)$-representations. For $k\in \ZZ$, let us define the $k$th {\it determinant representation} \index{representation!determinant} $D^k$ to be the $1$-dimensional representation $\GL(V)\rightarrow \kkk^*$ which to a matrix $g\in \GL(V)$ associates the $k$th power of its determinant, namely $\det(g)^k$. It urns out that $D:=D^1$ is the determinant representation $\bigwedge^n V$. 
\begin{remark}
In this remark we list some basic properties of the determinant representations.
\begin{compactitem}
\item[(i)] Clearly $D^k$ is a rational representation for all $k\in \ZZ$. Moreover it is polynomial precisely when $k\in \NN$.
\item[(ii)] Being $1$-dimensional, $D^k$ is obviously irreducible. Furthermore it is straightforward to check that its highest weight is $(k,k,\ldots ,k)\in \ZZ^n$. Therefore, if $k\geq 0$ Theorem \ref{schur} implies that $D^k\cong \Vl$ where $\lambda = (n,n,\ldots ,n)\in \ZZ^k$.
\item[(iii)] If $E$ is a irreducible rational representation with highest weight $(\alpha_1,\ldots ,\alpha_n)$, one can easily verify that $E\otimes D^k$ is a rational irreducible representation with highest weight $(\a_1+k,\ldots ,\a_n+k)$.
\end{compactitem}
\end{remark}
Eventually we are able to state the result which underlies the representation theory of the general linear group.
\begin{thm}\label{irrrepgl}
A rational representation $E$ is irreducible if and only if there exists a partition $\lambda$ and $k\in \ZZ$ such that $E\cong \Vl\otimes D^k$. In particular:
\begin{compactitem}
\item[(i)] For each vector $\a = (\a_1,\ldots ,\a_n)\in \ZZ^n$ with $\a_1\geq \ldots\geq \a_n$ the unique irreducible rational representation of weight $\a$ is given by $L_{ \ \tl\ll}V\otimes D^k$ where $\ll_i=\a_i-k\geq 0$ for $i=1,\ldots ,n$.
\item[(ii)] Two irreducible rational representations $\Vl\otimes D^k$ and $L_{\gamma}V\otimes D^h$ are isomorphic if and only if $\tl\ll_i+k = \tl\gamma_i+h$ for all $i=1,\ldots ,n$.
\end{compactitem}
\end{thm}
\begin{remark}\label{dualrepresentation}
If $E$ is an irreducible rational representation with highest weight $\a = $ $(\a_1,\ldots ,\a_n)$, then it easy to check that its dual representation $E^*$ is an irreducible rational representation with highest weight $(-\a_n,\ldots ,-\a_1)$. In particular, if $\ll = (\ll_1,\ldots ,\ll_k)$ is a partition, we denote by $\ll^*=(\ll_1^*,\ldots ,\ll_k^*)$ the partition such that $\ll_i^* = n-\ll_{k-i+1}$. By Theorem \ref{irrrepgl} we have a $\GL(V)$-isomorphism
\[(\Vl)^* \cong L_{\ll^*} \otimes D^{-k}, \]
Moreover one can show that $(\Vl)^*\cong L_{\ll}(V^*)$ (see the book of Procesi \cite[Chapter 9, Section 7.1]{procesi}). For this reason, from now on we will write $\Vl^*$ for $(\Vl)^*$: Each interpretation of such a notation is correct!
\end{remark}

We want to end this subsection describing a $\kkk$-basis of the Schur modules $\Vl$ in terms of the tableux of shape $\ll$. A tableu $\Lambda$ is said to be {\it standard}\index{Young tableu!standard} if its rows are increasing ($\Lambda(i,j)<\Lambda(i,j+1)$) and its columns are not decreasing ($\Lambda(i,j)\leq \Lambda(i+1,j)$). Among the standard tableux of fixed shape $\lambda$ one plays a crucial role: The {\it canonical tableu}\index{Young tableu!canonical} $c_{\lambda}$. For any $i,j$ such that $j\leq \lambda_i$, we have $\Kl(i,j):=j$. For instance the canonical tableu of shape $\lambda=(6,5,5,3,1)$ is:
\[
{\setlength{\unitlength}{1mm}
\begin{picture}(30,25)(-5,0)

\put(-11,12){$\Lambda \ =$}

\put(0,25){\line(1,0){30}}
\put(0,20){\line(1,0){30}}
\put(0,15){\line(1,0){25}}
\put(0,10){\line(1,0){25}}
\put(0,5){\line(1,0){15}}
\put(0,0){\line(1,0){5}}

\put(0,25){\line(0,-1){25}}
\put(5,25){\line(0,-1){25}}
\put(10,25){\line(0,-1){20}}
\put(15,25){\line(0,-1){20}}
\put(20,25){\line(0,-1){15}}
\put(25,25){\line(0,-1){15}}
\put(30,25){\line(0,-1){5}}

\put(1.8,21.3){\small{$1$}}
\put(6.8,21.3){\small{$2$}}
\put(11.8,21.3){\small{$3$}}
\put(16.8,21.3){\small{$4$}}
\put(21.8,21.3){\small{$5$}}
\put(26.8,21.3){\small{$6$}}
\put(1.8,16.3){\small{$1$}}
\put(6.8,16.3){\small{$2$}}
\put(11.8,16.3){\small{$3$}}
\put(16.8,16.3){\small{$4$}}
\put(21.8,16.3){\small{$5$}}
\put(1.8,11.3){\small{$1$}}
\put(6.8,11.3){\small{$2$}}
\put(11.8,11.3){\small{$3$}}
\put(16.8,11.3){\small{$4$}}
\put(21.8,11.3){\small{$5$}}
\put(1.8,6.3){\small{$1$}}
\put(6.8,6.3){\small{$2$}}
\put(11.8,6.3){\small{$3$}}
\put(1.8,1.3){\small{$1$}}

\end{picture}}
\]
Notice that the content of the canonical tableu $c_{\lambda}$ is the transpose partition $\tl\lambda$ of $\lambda$. If not already guessed, the reason of the importance of the canonical tableux will soon be clear. Fixed a $\kkk$-basis $\{e_1,\ldots ,e_n\}$ of $V$, it turns out that there is a $1$-$1$ correspondence between standard tableux of shape $\lambda$ on $[n]$ and a $\kkk$-basis of $\Vl$. The correspondence associates to $\Lambda$ the equivalence class of the element
\[(e_{\Lambda(1,1)}\wedge \ldots \wedge e_{\Lambda(1,\lambda_1)})\otimes \ldots \otimes (e_{\Lambda(k,1)}\wedge \ldots \wedge e_{\Lambda(k,\lambda_k)}).\]
By meaning of the Young symmetrizers, once fixed a tableu $\Gamma$ of shape $\ll$ on $[| \lambda |]$ of content $(1,1,\ldots ,1)\in \ZZ^{|\ll |}$, the correspondence is given by 
\[\Lambda \mapsto e(\Gamma)(e_{\Lambda(1,1)}\otimes \ldots \otimes e_{\Lambda(1,\lambda_1)}\otimes \ldots \otimes e_{\Lambda(k,1)}\otimes \ldots \otimes e_{\Lambda(k,\lambda_k)}).\]
With respect to both the correspondences above, any tableu $\Lambda$ is a weight vector of weight $c(\Lambda)$. Moreover the canonical tableu $c_{\ll}$ corresponds to the highest weight vector of $\Vl$. 

\subsection{Plethysms}

What said up to now implies that, given a finite dimensional rational $GL(V)$-representation $E$, there is a unique decomposition of it in irreducible representations, namely
\begin{equation}\label{plethysm}
E \cong \bigoplus_{{\height(\ll)<n}\atop {k\in \ZZ}} (\Vl \otimes D^k)^{m(\ll ,k)}.
\end{equation}
The numbers $m(\ll ,k)$ are the {\it multiplicities}\index{representation!multiplicities} of the irreducible representation $\Vl \otimes D^k$ appears in $E$ with. To find these numbers is a fascinating problem, still open even for some very natural representations. When the representation $E$ is polynomial,
the decomposition, rather than as in \eqref{plethysm}, is usually written as
\begin{equation}\label{plethysmpoly}
E \cong \bigoplus_{\height(\ll)\leq n} \Vl^{m(\ll)}
\end{equation}
For instance, the decomposition of $\Sym^p(\bigwedge^q V)$ is unknown in general, and to find it falls in the so-called {\ plethysm}'s problems\index{plethysm}. Unfortunately, this fact caused some obstructions to our investigations in Chapter \ref{chapter3}. At the contrary, an help from representation theory came from Pieri's formula\index{Pieri's formula}, a special case of the Littlewood-Richardson rule.
\begin{thm}\label{pieri}(Pieri's formula)
Let $\lambda=(\ll_1,\ldots,\ll_k)\vdash r$ and $\lambda(j)=(\ll_1+j,\ll_1,\ll_2,\ll_3,\ldots , \ll_k)$. Then
\[\Vl \otimes \bigwedge^jV \cong \bigoplus_{{\mu \vdash r+j, \ \height(\mu)\leq n}\atop {\ll \subseteq \mu \subseteq \ll(j)}}L_{\mu}V\]
and
\[\Vl^*\otimes \bigwedge^j V^* \cong \bigoplus_{{\mu \vdash r+j, \ \height(\mu)\leq n}\atop {\ll \subseteq \mu \subseteq \ll(j)}}L_{\mu}V^*\]
\end{thm}
\begin{proof}
For the proof of the first formula see \cite[Chapter 9, Section 10.2]{procesi}. The dual formula is straightforward to get from the first one.
\end{proof}
\index{representation!theory of the general linear group|)}

\section{Minors of a matrix and representations}\label{sacminors}

Let $\kkk$ be a field of characteristic $0$, $m$ and $n$ two positive integers such that $m\leq n$ and
\begin{displaymath}
X := \left(\begin{array}{ccccc} x_{11} & x_{12} & \cdots & \cdots &  x_{1n} \\
x_{21} & x_{22} & \cdots & \cdots & x_{2n} \\
\vdots & \vdots & \ddots & \ddots & \vdots \\
x_{m1} & x_{m2} & \cdots & \cdots & x_{mn}
\end{array} \right)
\end{displaymath}
a $m\times n$ matrix of indeterminates over $\kkk$. Moreover let
\[R:=\kkk[x_{ij} \ : \ i=1,\ldots ,m, \ j=1,\ldots ,n ]\]
be the polynomial ring in $m\cdot n$ variables over $\kkk$. Let $W$ and $V$ be $\kkk$-vector spaces of dimension $m$ and $n$, and set $G:=\GL(W)\times \GL(V)$. The rule
\[(A,B)*X := A\cdot X\cdot B^{-1} \ \ \]
induces an action of the group $G$ on $R_1$: Namely $(A,B)*x_{ij}=x'_{ij}$ where $x'_{ij}$ denotes the $(ij)$th entry of the matrix $(A,B)*X$. Extending this action $R$ becomes a $G$-representation. Actually, each graded component $R_d$ of $R$ is a (finite dimensional) rational $G$-representation. 

In Chapter \ref{chapter3} we dealt with the above action of $G=\GL(W)\times \GL(V)$: Therefore we need to introduce some notation about $G$-representations of this kind. Let $F$ be a $\GL(W)$-representation, and $E$ be a $\GL(V)$-representation. Then $T := F\otimes E$ becomes a $G$-representation. Furthermore, if $F$ and $E$ are irreducible,  $T$ is irreducible as well. More generally, two potential decompositions in irreducible representations
\[F = \bigoplus_i F_i \mbox{ \ \ \ and \ \ \ }E = \bigoplus_j E_j,\]
yield a decomposition in irreducible $G$-representation of $F\otimes E$, namely
\[T = \bigoplus_{i,j}F_i\otimes E_j.\]
Let us assume that $F$ and $E$ are both rational. If $f\in F$ is a weight vector of weight $\beta\in \ZZ^m$ and $e\in E$ is a weight vector of weight $\a\in \ZZ^n$, then we say that $t := f\otimes e$ is a {\it bi-weight vector} of {\it bi-weight} $(\beta |\a)$.\index{bi-weight}\index{bi-weight!vector} This is equivalent to say that, for any $\diag(\yyy)\in \GL(W)$ and $\diag(\xxx)\in \GL(V)$, we have
\[(\diag(\yyy),\diag(\xxx))\cdot t = \yyy^{\beta}\xxx^\a t.\]
Let $f$ and $e$ be highest weight vectors of, respectively, $F$ and $E$. 
For questions of notation we suppose that they are invariant, respectively, with respect to $B_-(W)$ and to $B_+(V)$. Setting $B = B_-(W)\times B_+(V)$, we have that
\[B \cdot t \subseteq <t>.\]
We call $t$ an {\it highest bi-weight vector}\index{bi-weight!vector!highest} of $T$. 

Let us consider the symmetric algebra 
\[ \Sym(W\otimes V^*) = \bigoplus_{d\in \NN} \Sym^d (W\otimes V^*).\]
Let $\{e_1,\ldots ,e_m\}$ be a basis of $W$ and $\{f_1,\ldots ,f_n\}$ be a basis of $V$. Denoting the dual basis of $\{f_1,\ldots ,f_n\}$ by $\{f_1^*,\ldots ,f_n^*\}$, consider the isomorphism of graded $\kkk$-algebras
\begin{displaymath}
\begin{array}{rll}
\phi : R & \rightarrow & \Sym(W\otimes V^*) \\
x_{ij} & \mapsto & e_i\otimes f_j^*
\end{array}.
\end{displaymath}
As one can check, $\phi$ is $G$-equivariant. Thus, it is an isomorphism of $G$-representations. Furthermore the $G$-representation $\Sym^d(W\otimes V^*)$ can be decomposed in irreducible representations: There is an explicit formula for such a decomposition, known as the {\it Cauchy formula}\index{Cauchy formula} \cite[Chapter 9, Section 7.1]{procesi}:
\begin{equation}\label{cauchy}
\Sym^d(W\otimes V^*)\cong \displaystyle{\bigoplus_{{\lambda \vdash d}\atop {\height(\ll)\leq m}}} \Wl \otimes \Vl^*.
\end{equation}
Since $R_d$ and $\Sym^d(W\otimes V^*)$ are isomorphic $G$-representations, we will describe which polynomials belong to the isomorphic copy of $\Wl \otimes \Vl^*$ in $R$, underlining which one is $U$-invariant. We need a notation do denote the minors of the matrix $X$. Given two sequences
$1\leq i_1<\ldots <i_s\leq m$ and $1\leq j_1<\ldots <j_s\leq n$, we write
\[[i_1,\ldots ,i_s|j_1,\ldots ,j_s]\]
for the $s$-minor which insists on the rows $i_1,\ldots ,i_s$ and the columns $j_1,\ldots ,j_s$ of $X$.
%First, let us describe a $\kkk$-basis of $\Wc \otimes \Vl^*$ in terms of {\it bi-tableux}\index{bi-tableux}. We will denote by $(\gamma | \lambda)$ the {\it bi-diagram}\index{bi-diagrams} made from two diagrams $\gamma$ and $\ll$. A {\it bi-tableu} $(\Gamma | \Lambda)$ of shape $(\gamma | \lambda)$ on $(\{1,\ldots ,m\}|\{-1,\ldots ,-n\})$ consists in a tableu $\Gamma$ of shape $\gamma$ on $\{1,\ldots ,m\}$ and in a tableu $\Lambda$ of shape $\ll$ on $\{-1,\ldots ,-n\}$. It is {\it standard}\index{bi-tableux!standard} if both $\Lambda$ and $\Gamma$ are standard. As in the end of Subsection \ref{subsecshurmodules}, we let standard bi-tableux parametrize a $\kkk$-basis of $\Wc\otimes \Vl^*$ consisting in bi-weight-vectors.
Since an element of $G$ takes an $s$-minor in a linear combination of $s$-minors, the $\kkk$-vector space spanned by the $s$-minors is a finite dimensional $G$-representation. So it is plausible to expect a description of the irreducible representations $\Wl \otimes \Vl^*$ in terms of minors. For any pair of standard tableu $\Lambda$ and $\Gamma$ of shape $\ll=(\ll_1,\ldots ,\ll_k)$, respectively on $[m]$ and on $[n]$, we define the following polynomial of $R$:
\[ [\Lambda | \Gamma]:=\delta_1 \cdots \delta_k\]
where
\[\delta_i := [\Lambda(i,1),\ldots ,\Lambda(i,\lambda_i)|\Gamma(i,1),\ldots ,\Gamma(i,\lambda_i)] \ \ \ \forall \ i=1,\ldots ,k\]
We say that the product of minors $[\Lambda |\Gamma]$ has {\it shape} $\lambda$. One can show that the product of minors $[c_{\ll}|c_{\ll}]$ is a highest bi-weight vector of bi-weight $((\tl\ll_1,\ldots ,\tl\ll_m)|(-\tl\ll_n,\ldots ,-\tl\ll_1))$, so there is an isomorphism of $G$-representations
\[M_{\ll}:=G*[c_{\lambda}|\c_{\lambda}]\cong \Wl\otimes \Vl^*.\]
%We define a {\it left (right) canonical product of minors of shape $\lambda$} to be a product of minors of the type $[\Kl |\Lambda]$ \ ($[\Lambda |\Kl]$). Setting $\lL$ \ ($\Ll$) the $K$-vector subspace of $R$ generated by all the left (right) canonical product of minors, we get the isomorphism of $G$-representations
%\begin{equation}\label{minorsandrep} \Ll \otimes \lL \cong \Wl \otimes \Vl
%\end{equation}
Actually, it turns out that the set 
\[\{[\Lambda|\Gamma] \ : \ \Lambda \mbox{ and $\Gamma$ are tableu of shape $\ll$, respectively on $[m]$ and on $[n]$}\}\]
is a $\kkk$-basis of $M_{\ll}$ (see the paper of DeConcini, Eisenbud and Procesi \cite{DEP1} or the book of Bruns and Vetter \cite[Section 11]{BrVe}). 
%In particular $\Ll \otimes \lL$ is an irreducible representation of $G$. Furthermore the isomorphism \eqref{minorsandrep} respects also a graded structure.
%On $R$, in fact, beyond that given by the usual degree of polynomials, there is another important graduation, which fits better with the structure of $R$ as a $G$-representation.
%Namely, given a monomial $M\in R$, we will say that $M$ has {\it multidegree} $(a,b)\in \NN^m \times \NN^n$ if
%\begin{compactitem}
%\item[-] For any $i_0=1,\ldots ,m$, $a_{i_0}$ is the maximum integer such that there exists a monomial of $K[x_{i_0j} \ : \ j=1,\ldots ,n]$ of (usual) degree $a_{i_0}$ which divides $M$.
%\item[-] For any $j_0=1,\ldots ,n$, $b_{j_0}$ is the maximum integer such that there exists a monomial of $K[x_{ij_0} \ : \ i=1,\ldots ,m]$ of (usual) degree $b_{j_0}$ which divides $M$.
%\end{compactitem}
%For instance, if $m=n=3$, the monomial $M=x_{12}x_{13}^2x_{23}^3x_{32}$ has multidegree $((3,3,1),(0,2,5))$. We will say that a polynomial $F\in R$ is {\it multihomogeneous} of multidegree $(a,b)$ if any monomial in the support of $F$ has multidegree $(a,b)$.  It turns out that any product of minors is multihomogeneous. The isomorphism \eqref{minorsandrep} is then multigraded, meaning that a polynomial of multidegree $(a,b)$ is sent to a linear combination of standard bi-tableux of bi-content $(a,b)$, and viceversa.

In Chapter \ref{chapter3} we were especially interested in the study of the $\kkk$-subalgebra of $R$
\[A_t=A_t(m,n)\]
generated by the $t$-minors of $X$, where $t$ is a positive integer smaller than or equal to $m$. These algebras are known as {\it algebras of minors}\index{algebra of minors}. It turns out that the algebra of minors $A_t$ is a $G$-subrepresentation of $R$. Therefore it is natural to ask about its decomposition in irreducibles. Equivalently, which $M_{\ll}$ are in $A_t$? The answer to this question is known. Before describing it, we fix a definition.
\begin{definition}
Given a partition $\lambda=(\lambda_1,\ldots , \lambda_k)$ we say that it is {\it admissible}\index{partition!admissible} if $t$ divides $|\lambda|$ and $tk \leq |\lambda|$. Furthermore we say that $\lambda$ is {\it $d$-admissible}\index{partition!d-admissible@$d$-admissible} if $|\lambda|=td$.
\end{definition}
Let us denote by $[A_t]_d$ the $\kkk$-vector space consisting in the elements of $A_t$ of degree $td$ in $R$. This way we are defining a new grading over $A_t$ such that the $t$-minors have degree $1$. The below result follows at once by the description of the powers of determinantal ideals got in \cite{DEP1}. 
\begin{thm}\label{decat}
For each natural number $d$ and for each $1\leq t\leq m$ we have
\[[A_t]_d = \bigoplus_{{\mbox{ {\footnotesize $\ll$ is $d$-admissible}}}\atop {\mbox{ {\footnotesize $\height(\ll)\leq m$}}}} M_{\ll} \cong \bigoplus_{{\mbox{ {\footnotesize $\ll$ is $d$-admissible}}}\atop {\mbox{ {\footnotesize $\height(\ll)\leq m$}}}} \Wl \otimes \Vl^*.\]
\end{thm}

\index{representation|)}
\index{representation!theory|)}

\chapter{Combinatorial commutative algebra}\label{appendixd}

Below we recall some basic facts concerning Stanley-Reisner rings. Standard references for this topic are Bruns and Herzog \cite[Chapter 5]{BH}, Stanley \cite{St} or Miller and Sturmfels \cite{MS}

%First of all we define the basic objects involved in the statement. For the part concerning commutative algebra and Stanley-Reisner rings, we refer to Bruns and Herzog \cite{BH}, Stanley \cite{St} or Miller and Sturmfels \cite{MS}. For what concerns the theory of matroids, some references are the book of Welsh \cite{We} or that of Oxley \cite{Ox}.
 
Let $\kkk$ be a field, $n$ a positive integer and $S:=\kkk[x_1,\ldots ,x_n]$ the polynomial ring on $n$ variables over $\kkk$. Moreover, let us denote by $\mm:=(x_1,\ldots ,x_n)$ the maximal irrelevant ideal of $S$. We write $[n]$ for $\{1,\ldots ,n\}$. By a {\it simplicial complex}\index{simplicial complex} $\D$ on $[n]$ we mean a collection of subsets of $[n]$ such that for any $F \in \D$, if $G\subseteq F$ then $G\in \D$. An element $F\in \D$ is called a {\it face}\index{simplicial complex!faces of} of $\D$. The dimension of a face $F$ is $\dim F := |F|-1$ and the dimension of $\D$ is $\dim \D := \max\{\dim F : F\in \D\}$. The faces of $\D$ which are maximal under inclusion are called {\it facets}\index{simplicial complex!facets of}. We denote the set of the facets of $\D$ by $\FD$. Obviously a simplicial complex on $[n]$ is univocally determined by its set of facets. A simplicial complex $\D$ is {\it pure}\index{simplicial complex!pure} if all its facets have the same dimension. It is {\it strongly connected} if \index{simplicial complex!strongly connected} for any two facets $F$ and $G$ there exists a sequence $F=F_0,F_1,\ldots ,F_s=G$ such that $F_i\in \FD$ and $|F_i\setminus (F_i\cap F_{i-1})|=|F_{i-1}\setminus (F_i\cap F_{i-1})|=1$ for any $i=1,\ldots ,s$. For a simplicial complex $\D$ we can consider a square-free monomial ideal, known as the {\it Stanley-Reisner ideal}\index{Stanley-Reisner ideal} of $\D$,
\begin{equation}
\Id := (x_{i_1}\cdots x_{i_s} : \{i_1,\ldots ,i_s\}\notin \D).
\end{equation}
Such a correspondence turns out to be one-to-one between simplicial complexes and square-free monomial ideals\index{monomial ideal}\index{monomial ideal!square-free}, its inverse being
\[I \mapsto \D(I):=\{F\in [n] \ : \ \prod_{i\notin F}x_i\in I\}\]
for any square-free monomial ideal $I\subseteq S$.
The $\kkk$-algebra $\kkk[\D]:=S/\Id$ is called the {\it Stanley-Reisner ring}\index{Stanley-Reisner ring} of $\D$, and it turns out that 
\[ \dim (\kkk[\D]) = \dim \D + 1.\] 
More precisely, with the convention of denoting by $\wp_A:=(x_i:i\in A)$ the prime ideal of $S$ generated by the variables correspondent  to a given subset $A\subseteq [n]$, we have
\[\Id = \bigcap_{F \in \FD} \wp_{[n]\setminus F}. \]
\begin{remark}\label{sccd1}
Thanks to the above interpretation and to Lemma \ref{hilbert}, we have that a simplicial complex $\D$ on $[n]$ is strongly connected if and only if $\kkk[\D]$ is connected in codimension $1$.
\end{remark}
We will use another correspondence between simplicial complexes and square-free monomial ideals, namely
\begin{equation}\label{coveridealdef}
\D \mapsto J(\D):=\bigcap_{F\in \FD} \wp_F.
\end{equation}
The ideal $J(\D)$ is called the {\it cover ideal}\index{cover ideal} of $\D$. The name ``cover ideal" comes from the following fact: A subset $A\subseteq [n]$ is called a {\it vertex cover}\index{simplicial complex!vertex cover} of $\D$ if $A\cap F \neq \emptyset$ for any $F\in \FD$. Then it is easy to see that
\[\JD = (x_{i_1}\cdots x_{i_s} : \{i_1,\ldots ,i_s\}\mbox{ is a vertex cover of }\D).\]
Let $\Delta^c$ be the simplicial complex on $[n]$ whose facets are $[n]\setminus F$ such that $F\in \FD$. Clearly we have $I_{\D^c}=\JD$ and $I_{\D}=J(\D^c)$. Furthermore $(\D^c)^c=\D$, therefore also the correspondence \eqref{coveridealdef} is one-to-one between simplicial complexes and square-free monomial ideals.
\index{vertex cover|see{simplicial complex}}\index{k-cover@$k$-cover|see{simplicial complex}}

\section{Symbolic powers}\label{symbolicpowers}

If $R$ is a ring, and $\aa\subseteq R$ is an ideal, then the $m$th {\it symbolic power}\index{symbolic power} of $\aa$ is the ideal of $R$
\[\aa^{(m)}:=(\aa^m R_W)\cap R \subseteq R,\]
where the multiplicative system $W$ is the complement in $R$ of the union of the associated prime ideals of $\aa$. If $\aa =\wp$ is a prime ideal, then $\wp^{(m)}=(\wp^mR_{\wp})\cap R$. One can show that $\wp^{(m)}$ is the $\wp$-primary component of $\wp^m$, so that $\wp^{(m)}=\wp^m$ if and only if $\wp^m$ is primary. Furthermore, one can show the following:

\begin{prop}\label{decomposition of symbolic powers}
Let $\aa$ be an ideal of $R$ with no embedded primes, that is $\Ass(\aa)=\Min(\aa)$. If $\aa=\bigcap_{i=1}^k \qq_i$ is a minimal primary decomposition of $\aa$, then
\[\aa^{(m)}=\bigcap_{i=1}^k \qq_i^{(m)}.\]
\end{prop}

In the case in which $\aa$ is a power of a prime monomial ideal of the polynomial ring $S$, i.e. there exists $F\subseteq [n]$ and $k\in \NN$ such that $\aa=\wp_F^k$, then it is easy to show that
\[\aa^m=\aa^{(m)} \ \ \forall \ m\in \NN.\]
Therefore if $\D$ is a simplicial complex on $[n]$ and $I=\bigcap_{F\in \FD}\wp_F^{\omega_F}$,
where $\omega_F$ are some positive integers, then Proposition \ref{decomposition of symbolic powers} yields
\begin{equation}\label{strucmonsymb}
I^{(m)}=\bigcap_{F\in \F(\D)}\wp^{m\omega_F}.
\end{equation}
In particular 
\[\Id^{(m)} = \bigcap_{F \in \FD} \wp_{[n]\setminus F}^m \ \ \ \mbox{and} \ \ \ \JD^{(m)} = \bigcap_{F \in \FD} \wp_{F}^m. \]

\section{Matroids}\label{appdmatroids}

A simplicial complex $\D$ on $[n]$ is a {\it matroid}\index{matroid} if, for any two facets $F$ and $G$ of $\D$ and any $i\in F$, there exists a $j\in G$ such that $(F\setminus \{i\})\cup \{j\}$ is a facet of $\D$. 

\begin{example}
The following is the most classical example of matroid. Let $V$ be a $\kkk$-vector space and let $A := \{v_1,\ldots ,v_n\}$ a set of distinct vectors of $V$. We define a simplicial complex on $[n]$ as follows: A subset $F\subseteq [n]$ is a face of $\D$ if and only if $\dim_{\kkk}(<v_i \ : \ i\in F>) = |F|$. This way the facets of $\D$ are the subsets of $[n]$ corresponding to the bases in $A$ of the $\kkk$-vector subspace $<v_1,\ldots ,v_m>\subseteq V$. Actually, the concept of matroid was born as an ``abstraction of the bases of a $\kkk$-vector space". 
\end{example}

A matroid $\D$ has very good properties. For an exhaustive account the reader can see the book of Welsh \cite{Wel} or the one of Oxley \cite{Ox}. We use matroids in Chapter \ref{chapter4} and, essentially, we need three results about them. The first one, the easier to show, is that a matroid is a pure simplicial complex (\cite[Lemma 1.2.1]{Ox}). The second one, known as the {\it exchange property} of matroids\index{matroid!exchange property of}, states that for any matroid $\D$, we have
\begin{equation}\label{exchangeproperty}
\forall \ F, G \in \FD, \ \forall \ i \in F, \ \ \exists \ j\in G \ : \ (F\setminus \{i\})\cup \{j\}, \ (G\setminus \{j\})\cup \{i\} \ \in \FD,
\end{equation}
(this result is of Brualdi \cite{Br}, see also \cite[p. 22, Exercise 11]{Ox}). The last fact is a basic result of matroid theory, which is a kind of duality. For the proof see \cite[Theorem 2.1.1]{Ox}.
\begin{thm}\label{matroidduality}
A simplicial complex $\D$ on $[n]$ is a matroid if and only if $\D^c$ is a matroid.
\end{thm}
If $\D$ is a matroid, $\D^c$ is called its {\it dual matroid}\index{matroid!dual}.

\section{Polarization and distractions}\index{polarization}\label{polarization}

Sometimes, problems regarding (not necessarily square-free) monomial ideals $I\subseteq S$ can be faced passing to a square-free monomial ideal associated to $I$, namely its {\it polarization} $\widetilde{I}$, which preserves many invariants of $I$, such as its minimal free resolution.

More precisely, the polarization of a monomial $t=x_1^{a_1}\cdots x_n^{a_n}$ is
\[\widetilde{t}:=\prod_{j=1}^{a_1}x_{1,j}\cdot \prod_{j=1}^{a_2}x_{2,j} \cdots \prod_{j=1}^{a_n}x_{n,j}\subseteq \Spol , \]
where the $x_{i,j}$'s are new variables over $\kkk$ and $\Spol := \kkk[x_{i,j}: i=1,\ldots ,n \ j=1,\ldots ,a_i]$. Actually it is not so important that $\Spol$ is generated by variables $x_{i,j}$ with $j\leq a_i$. The significant thing is that it contains all such variables, in such a way that $\widetilde{t}\in \Spol$. For example, it is often convenient to consider $\Spol=\kkk[x_{i,j}:i=1,\ldots ,n \ j=1,\ldots ,d]$ where $d$ is the degree of $t$.  

The concept of polarization fits in a more general context, that of {\it distractions}\index{distractions} (see Bigatti, Conca and Robbiano \cite{bcr}), which is useful to introduce. Let $P:=\kkk[y_1,\ldots ,y_N]$ be a polynomial ring in $N$ variables over $\kkk$. An infinite matrix $\L=(L_{i,j})_{i\in [N],j\in \NN}$ with entries $L_{i,j}\in P_1$ is called a {\it distraction matrix} if $<L_{1,j_1},\ldots ,L_{N,j_N}>=P_1$ for all $j_i\in \NN$ and there exists $k\in \NN$ such that $L_{i,j}=L_{i,k}$ for any $j>k$. If $t=y_1^{a_1}\cdots y_N^{a_N}$ is a monomial of $P$, the {\it $\L$-distraction} of $t$ is the monomial
\[D_{\L}(t):=\prod_{j=1}^{a_1}L_{1,j}\cdot \prod_{j=1}^{a_2}L_{2,j} \cdots \prod_{j=1}^{a_N}L_{N,j}\subseteq P.\]

\begin{remark}\label{polsubdis}
In this remark we want to let the reader noticing how the polarization is a particular distraction. Let $t=x_1^{a_1}\cdots x_n^{a_n}$ be a monomial of $S$ and $\widetilde{t}$ is polarization in the polynomial ring $\Spol =\kkk[x_{i,j}:i\in [n], \ j\in[d]]$, where $d$ is the degree of $t$. Let us think at $\Spol$ as the polynomial ring $P$ given for the definition of distractions. So $N=nd$. Moreover, we can look at $t$ as an element of $P$, namely
\[t=x_{1,1}^{a_1}\cdots x_{n,1}^{a_n}.\] 
Let us consider the following matrix $\L$:
\[L_{i,j}:= \begin{cases} x_{i,j}  & \mbox{if \ } i\leq n \mbox{ and }j\leq d \\
x_{r,q}+x_{r,q+1} & \mbox{if \ } n<i\leq N, \ j\leq d \mbox{ \ and \ } i=qn+r \mbox{ with $0< r\leq n$} \\
L_{i,d} & \mbox{if \ } j>d
\end{cases}\]
One can easily check that $\L$ is a distraction matrix, and that
\[D_{\L}(t)=\widetilde{t}.\]
\end{remark}

By Remark \ref{polsubdis}, we can state the results we need during the thesis in the more general context of distractions, even if we will actually use them just for the case of the polarization. 

Fixed a distraction matrix $\L$, we can extend $\kkk$-linearly $D_{\L}$, getting a $\kkk$-linear map:
\[D_{\L}:P\longrightarrow P.\]
Of course $D_{\L}$ is not a ring homomorphism, however we have that $D_{\L}(I)$ is an ideal of $P$ for any monomial ideal $I\subseteq P$ (\cite[Corollary 2.10 (a)]{bcr}). Moreover, if $I=(t_1,\ldots ,t_m)$, then $D_{\L}(I)=(D_{\L}(t_1),\ldots ,D_{\L}(t_m))$ (\cite[Corollary 2.10 (b)]{bcr}) and $\height(D_{\L}(I))=\height(I)$  (\cite[Corollary 2.10 (c)]{bcr}). Furthermore, \cite[Proposition 2.9 (d)]{bcr} implies that, if $I=\cap_{i=1}^p I_i$, where the $I_i$'s are monomial ideals of $P$, then:
\begin{equation}\label{intersectionpol}
D_{\L}(I)=\bigcap_{i=1}^pD_{\L}(I_i).
\end{equation}

\begin{remark}\label{remarkintpol}
If $I=(t_1,\ldots ,t_m)\subseteq S$, the polarization of $I$ is defined to be:
\[\widetilde{I}=(\widetilde{t_1},\ldots ,\widetilde{t_m})\subseteq \Spol = \kkk[x_{i,j}:i\in [n], \ j\in [d]],\] 
where $d$ is the maximum of the degrees of the $t_i$'s. By the discussion previous to the remark we deduce that $\widetilde{I}$ and $D_{\L}(I)$ are basically the same object. The equality \eqref{intersectionpol} can be interpreted as
\begin{equation}\label{intersectionpol1}
\widetilde{I}=\bigcap_{i=1}^p\widetilde{I_i}\subseteq \Spol.
\end{equation}
where all the ideal involved are polarized in the same polynomial ring $\Spol$.
\end{remark}

\begin{remark}\label{dimunderpol}
As it already got out by Remark \ref{remarkintpol}, among the things which distinguish polarization and distractions, is that the ambient ring, contrary to what happens for the former, does not change under the latter operation. Of course, this is just a superficial problem, since we can add variables before polarizing, as we did in Remark \ref{polsubdis}. However the reader should play attention to this fact: If $J$ is a monomial ideal of $P$ and $\L$ is a distraction matrix, then we have the equality $\Hf_{P/J}=\Hf_{P/D_{\L}(J)}$ (\cite[Corollary 2.10 (c)]{bcr}). Of course, it is not true the same fact for polarization: Namely, in general, if $I$ is a monomial ideal of $S$, it might happen that $\Hf_{S/I}\neq \Hf_{\Spol/\widetilde{I}}$. Actually, even $\dim S/I$ is in general different from $\dim \Spol/\widetilde{I}$. However, in view of Remark \ref{polsubdis}, one should expect that the properties of distractions hold true also for polarization. This is actually the case, but the right way to think is ``for codimension": In fact, we have $\height(I)=\height(\widetilde{I})$; moreover, suppose that the Hilbert series\index{Hilbert series} of $S/I$ is $\Hs_{S/I}(z)=h(z)/(1-z)^d$, where $h(z)\in \ZZ[z]$ is such that $h(1)\neq 0$ and $d$ is the dimension of $S/I$ (for instance see the book of Bruns and Herzog \cite[Corollary 4.1.8]{BH}). Then, $\Hs_{\Spol/\widetilde{I}(z)}=h(z)/(1-z)^e$, where $e=\dim(\Spol/\widetilde{I})$.
\end{remark}

In \cite[Theorem 2.19]{bcr}, the authors showed that the minimal free resolution of a monomial ideal $I\subseteq P$ can be carried to a minimal free resolution of $D_{\L}(I)$. In particular, we get the following:

\begin{thm}\label{cmpol}
Given a distraction matrix $\L$, the graded Betti numbers of $P/I$ and those of $P/D_\L(I)$ are the same. Particularly, $P/I$ is Cohen-Macaulay if and only if $P/D_{\L}(I)$  is Cohen-Macaulay.
\end{thm}

%\section{The independence complex of a general ideal}
%
%Actually, it is possible to attach a simplicial complex on $[n]$ to any ideal $I\subseteq S$, not necessarily monomial. A subset $F\subseteq [n]$ is said to be{\it independent} modulo $I$ if
%\[\kkk[x_i:i\in F]\cap I =\{0\}.\]
%The {\it independence complex} of $I$ is:
%\[\D(I):=\{F\subseteq [n]:F \mbox{ is independent modulo }I\}.\]
%As the reader probably noticed, we did not change the notation we used in the case in which $I$ is a square-free monomial ideal: This is because the two definitions agree in this case.

\printindex

\end{document}